\documentclass{memo-l}

\usepackage{tipa}
\usepackage{relsize}
\usepackage{cite}
\usepackage{amssymb}
\usepackage{graphicx}
\usepackage{caption}
\usepackage{subcaption}
\usepackage{lmodern}
\usepackage{bm}
\usepackage{url}

\newtheorem{theorem}{Theorem}[chapter]
\newtheorem{lemma}[theorem]{Lemma}

\theoremstyle{definition}
\newtheorem{definition}[theorem]{Definition}
\newtheorem{example}[theorem]{Example}

\theoremstyle{remark}
\newtheorem{remark}[theorem]{Remark}

\numberwithin{section}{chapter}
\numberwithin{equation}{chapter}

\theoremstyle{plain}
\newtheorem{prop}[theorem]{Proposition}

\newtheorem{cor}[theorem]{Corollary}

\newtheorem{theorem_without}{Theorem}
\newtheorem{theorem_without_2}{Theorem}

\newcommand*{\N}{\mathbb{N}}  
\newcommand*{\Z}{\mathbb{Z}}  
\newcommand*{\Q}{\mathbb{Q}}  
\newcommand*{\R}{\mathbb{R}}
\newcommand*{\eps}{\varepsilon}
\newcommand{\cA}{\mathcal A}
\newcommand{\cB}{\mathcal B}
\newcommand{\cC}{\mathcal C}
\newcommand{\cD}{\mathcal D}
\newcommand{\cG}{\mathcal G}
\newcommand{\cH}{\mathcal H}
\newcommand{\cI}{\mathcal I}
\newcommand{\cJ}{\mathcal J}
\newcommand{\cL}{\mathcal L}
\newcommand{\cM}{\mathcal M}
\newcommand{\cU}{\mathcal U}
\newcommand{\cW}{\mathcal W}
\newcommand*{\norm}[1]{\left\lVert#1\right\rVert}
\newcommand{\Cext}[1]{\Pi_E^{\mathrm{ext}}(#1)}
\newcommand{\Xext}[1]{\Xi_{#1}^{\mathrm{ext}}(E_\eps)}
\newcommand{\Pext}[1]{\mathcal{P}_{#1}^{\mathrm{ext}}(E_\eps)}
\newcommand{\Unp}[1]{\mathrm{Unp_\eps}(#1)}
\newcommand{\Sing}{S(E_\eps)}

\makeindex

\begin{document}

\frontmatter

\title[On Boundaries of {\larger{\textepsilon}}-neighbourhoods]{On Boundaries of {\larger{\textepsilon}}-neighbourhoods of Planar Sets: Singularities, Global Structure, and Curvature}

\author[Lamb]{{J.} {S.} {W.} Lamb}
\address{Jeroen Lamb \\ Department of Mathematics \\ Imperial College London \\ 180 Queen's Gate, South Kensington, London SW7 2AZ, United Kingdom\\ \& 
International Research Center for Neurointelligence (IRCN), The University of Tokyo, 7-3-1 Hongo Bunkyo-ku, Tokyo, 
113-0033 Japan\\ \& Centre for Applied Mathematics and Bioinformatics, Department of Mathematics and Natural Sciences, Gulf University for Science and Technology, Halwally, 32093 Kuwait.}

\author[Rasmussen]{{M.} Rasmussen}
\address{Martin Rasmussen \\ Department of Mathematics \\ Imperial College London \\ 180 Queen's Gate, South Kensington, London SW7 2AZ, United Kingdom}

\author[Timperi]{{K.} {G.} Timperi}
\address{Kalle Timperi \\ Center for Ubiquitous Computing \\ University of Oulu \\ Erkki Koiso-Kanttilan katu 3, door E, Oulu, Finland}
\email{kalle.timperi@oulu.fi}

\thanks{
This paper is based on the PhD thesis~\cite{PhD_Kalle_Timperi} of KT under the supervision of JSWL and MR at Imperial College London, supported by the EU Marie Sklodowska-Curie ITN Critical Transitions in Complex Systems (H2020-MSCA-ITN-2014 643073
CRITICS).
JSWL and MR have been supported by the EPSRC grants EP/W009455/1 and  EP/Y020669/1.  JSWL also acknowledges support from the EPSRC Centre for Doctoral Training in Mathematics of Random Systems: Analysis, Modelling and Simulation (EP/S023925/1).  JSWL is grateful for support from JST Moonshot R \& D Grant Number JPMJMS2021 (IRCN, University of Tokyo) and thanks GUST (Kuwait) for their support.
}
\thanks{The authors express their gratitude to Gabriel Fuhrmann, Konstantinos Kourliouros, Tuomas Orponen, Sebastian van Strien, Dmitry Turaev and Vadim Weinstein for many useful discussions and valuable comments regarding earlier versions of this paper.}





\begin{abstract}
We study the geometry, topological properties and smoothness of the boundaries of closed $\eps$-neighbourhoods $E_\eps = \{x \in \R^2 \, : \, \textrm{dist}(x, E) \leq \eps \}$ of compact planar sets $E \subset \R^2$. We develop a novel technique for analysing the boundary, and use this to obtain a classification of singularities (i.e.~non-smooth points) on $\partial E_\eps$ into eight categories. We show that the set of singularities is either countable or the disjoint union of a countable set and a closed, totally disconnected, nowhere dense set. Furthermore, we characterise, in terms of local geometry, those $\eps$-neighbourhoods whose complement $\overline{\R^2 \setminus E_\eps}$ is a set with positive reach. It is known that for all bounded $E \subset \R^d$ and all $\eps > 0$, the boundary $\partial E_\eps$ is $(d-1)$-rectifiable. Improving on this, we identify a sufficient condition for the boundary to be uniformly rectifiable, and provide an example of a planar $\eps$-neighbourhood that is not Ahlfors regular. In terms of the topological structure, we show that for a compact set $E$ and $\eps > 0$ the boundary $\partial E_\eps$ can be expressed as a disjoint union of an at most countably infinite union of Jordan curves and a possibly uncountable, totally disconnected set of singularities. Finally, we show that curvature is defined almost everywhere on the Jordan curve subsets of the boundary.
\end{abstract}

\maketitle

\tableofcontents


\mainmatter
\chapter{Introduction}
\section{Motivation} \label{Sec_Motivation}
For a given set $E \subset \R^d$ and radius $\eps > 0$, the closed \emph{$\eps$-neigh\-bour\-hood} of $E$ is the set
\begin{equation} \label{Def_Tubular_Neighbourhood}
E_\eps := \overline{B_\eps(E)} := \overline{ \bigcup_{x \in E} B_\eps(x)},                 
\end{equation}
where the overline denotes closure and $B_\eps(\cdot)$ is an open ball of radius $\eps$ in the Euclidean metric.
The sets $E_\eps$ are also known in the literature as \emph{tubular neighbourhoods} \cite{Fu_Tubular_neighborhoods}, \emph{collars} \cite{Przeworski_An_Upper_Bound} or \emph{parallel sets} \cite{Rataj_Winter_On_Volume, Stacho_On_the_volume}. The boundary $\partial E_\eps$ is a subset of the set $\partial E_{<\eps} := \partial \left( \bigcup_{x \in E} B_\eps(x) \right)$, which is sometimes referred to as the \emph{$\eps$-boundary} \cite{Ferry_When_epsilon_boundaries} or \emph{$\eps$-level set} \cite{Oleksiv_Pesin_Finiteness} of $E$.\footnote{Apart from the exceptions of~\cite{Fu_Tubular_neighborhoods} and~\cite{Rataj_Winter_On_Volume} which explicitly discuss closed $\eps$-neighbourhoods, the existing literature~\cite{Erdos_Some_remarks, Brown_Sets_of_constant, Ferry_When_epsilon_boundaries, Gariepy_Pepe_On_the_level, Blokh_Misiurewicz_Oversteegen_Set_of_Constant, Rataj_Zajicek_Critical_Values_and_Level_Sets_of_Distance} concerns the properties of $\eps$-level sets (open $\eps$-neighbourhoods). The difference between the two definitions is illustrated in Figure~\ref{Figure_Extremal_Interior}.}

The central question addressed in this paper concerns the geometric and topological properties of such sets $E_\varepsilon$, with a focus on properties of its boundary $\partial E_\varepsilon$. This is not only a very natural and fundamental question in (Euclidean) geometry, but it is also relevant in specific settings where $\varepsilon$-neighbourhoods naturally arise. For instance, we are motivated by the classification and bifurcation of minimal invariant sets in random dynamical systems with bounded noise \cite{TopoBif_of_MinInvSets}, but $\varepsilon$-neighbourhoods also naturally feature for instance in control theory \cite{Colonius_Kliemann_Dynamics_of_Control}.

In spite of significant progress in charting the theoretical properties of $\eps$-neigh\-bour\-hoods during the last decades~\cite{Ferry_When_epsilon_boundaries, Fu_Tubular_neighborhoods, R-S-M_Approximations_of, Rataj_Schmidt_Spodarev_On_The_Expected, Rataj_Winter_On_Volume, Rataj_Zajicek_Critical_Values_and_Level_Sets_of_Distance,  Rataj_Zajicek_Properties_of_distance, Rataj_Zajicek_Smallness}, the geometric classification of possible boundaries $\partial E_\eps$ has remained open, even in dimension two.

\section{Main results}

\begin{figure}[ht]
      \centering  \vspace{-1mm}
      \captionsetup{margin=0.75cm}
                \includegraphics[width = \textwidth]{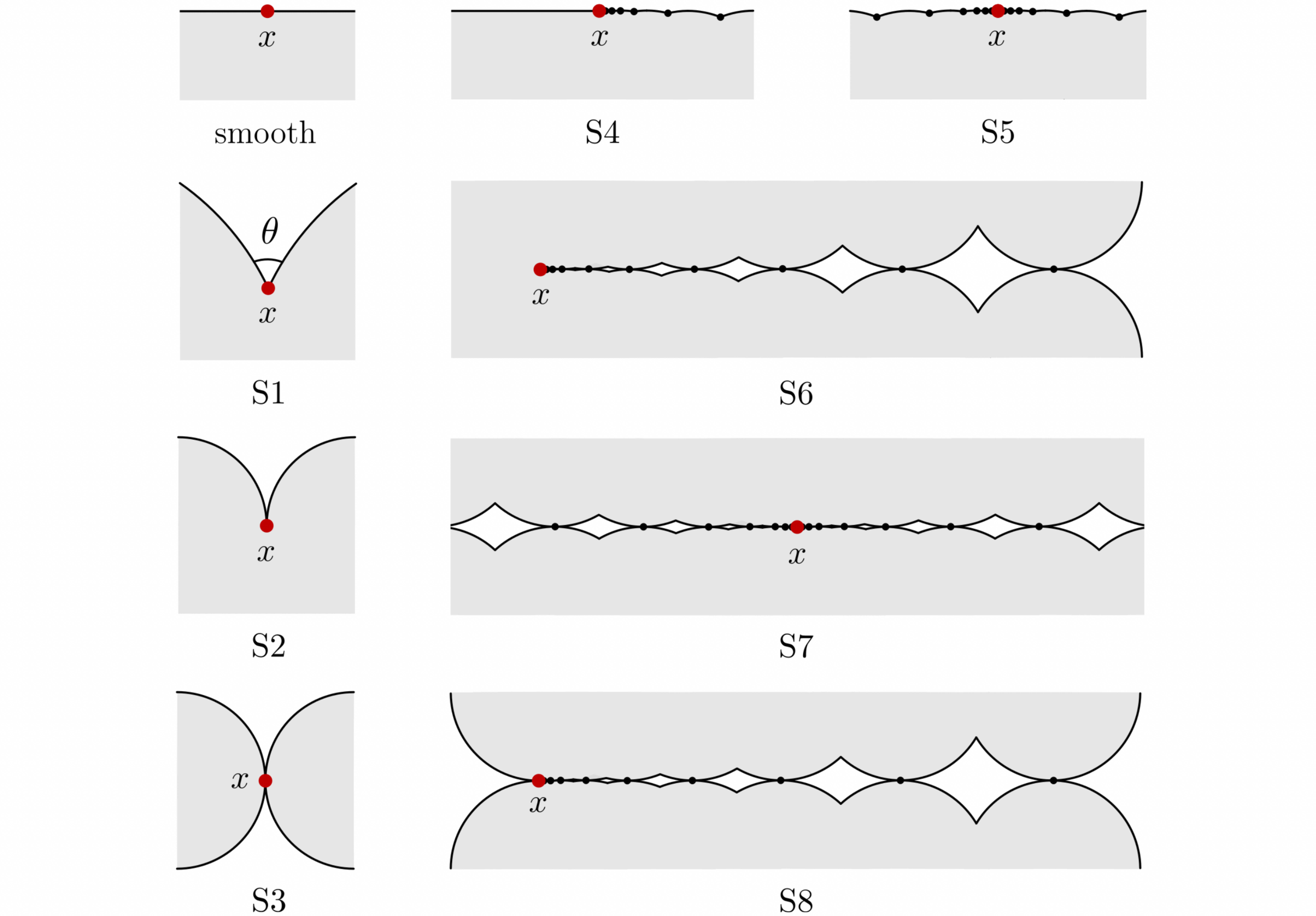} \\[0mm]   
                \caption{Schematic illustration of the types of singularities mentioned in Theorem~\ref{Thm_Main_1}. The grey area represents the $\eps$-neighbourhood $E_\eps$ and the white area the complement $\R^2 \setminus E_\eps$. Every boundary point $x \in \partial E_\eps$ either is a smooth point or belongs to exactly one of eight categories of singularities. At a \emph{wedge} (type S1) the one-sided tangents form an angle $0 < \theta < \pi$. A \emph{sharp singularity} (type S2) and a \emph{sharp-sharp singularity} (type S3) can be thought of as extremal cases of a wedge, with $\theta = 0$. A \emph{shallow singularity} (type S4) and a \emph{shallow-shallow singularity} (type S5) have a well-defined tangent, but they are accumulation points (from one or two directions, respectively) of sequences of increasingly obtuse wedges (black dots). A \emph{chain singularity} (type S6), a \emph{chain-chain singularity} (type S7) and a \emph{sharp-chain singularity} (type S8) share the geometric property of being accumulation points of sequences of increasingly acute wedges (black dots). This turns out to be equivalent (see Proposition~\ref{Prop_Topological_Characterisation_of_Chain_Singularities}) to the topological property of being the limit with respect to Hausdorff distance 
                of a sequence of disjoint connected components of the complement $\R^2 \setminus E_\eps$. See also Figure~\ref{Figure_Three_Basic_Cases}.
                 }
                \label{Figure_Types_of_Singularities}
\end{figure}

Our main achievement is the development of a novel technique for analysing the geometric properties of the boundary $\partial E_\eps$. This approach enables a local representation of the boundary around every boundary point $x \in \partial E_\eps$ using graphs of Lipschitz-continuous functions. We employ this representation to obtain a classification of the boundary points of $\varepsilon$-neighbourhoods of compact planar sets. The strength of our elementary approach compared to existing literature is that it requires no restrictions on the radius $\eps$. See Section~\ref{Section_Context} for a discussion of previous work on $\eps$-neighbourhoods.

Our first main result establishes that for any compact set $E \in\R^2$ and $\eps > 0$, each boundary point $ x \in \partial E_\eps$ is either a smooth point (in the sense that the restriction of $\partial E_\eps$ into a neighbourhood of $x$ is a $C^1$-curve) or falls into exactly one of eight distinct categories of singularities.
\begin{theorem_without}  \label{Thm_Main_1}
Let $E \subset \R^2$ be compact, $\eps > 0$, and let $x \in \partial E_\eps$ be a boundary point of $E_\varepsilon$ that is not smooth.
Then $x$ belongs to precisely one of the following eight categories:

\vspace{1.5mm}
\begin{tabular}{llll}
 (S1) & wedge, & (S5) & shallow-shallow singularity,  \\
 (S2) & sharp singularity, &  (S6) & chain singularity,  \\
 (S3) & sharp-sharp singularity, & (S7) & chain-chain singularity,   \\
 (S4) & shallow singularity, & (S8) & sharp-chain singularity.   \\
\end{tabular}
\end{theorem_without}
The precise definitions of these categories are given in Definition~\ref{Def_classification_of_singularities}. For indicative sketches of the different types of singularities, see Figure~\ref{Figure_Types_of_Singularities}.

The proof of Theorem~\ref{Thm_Main_1} is based on the construction of a local boundary representation, given in Proposition~\ref{Prop_local_representation_exists}, that allows us to treat small parts of the boundary $\partial E_\eps$ as finite unions of graphs of continuous functions. This representation is akin to the one provided in~\cite[Definition~6.3]{Rataj_Zajicek_On_the_structure_of_sets_with_PR} for the boundaries of sets with positive reach. Our representation is independent of this, and instead directly utilises the specific geometry of the $\eps$-neighbourhood boundary. It relies on a local contribution property, Proposition~\ref{Prop_local_contribution}, which states that the geometry of the boundary near each boundary point $x \in \partial E_\eps$ essentially depends on contributions from points $y \in E$ in at most two directions.

Building further on these ideas, we establish our second main result regarding the cardinalities of the different types of singularities.

\begin{figure}
      \centering
      \captionsetup{margin=0.75cm}
                \includegraphics[width = \textwidth]{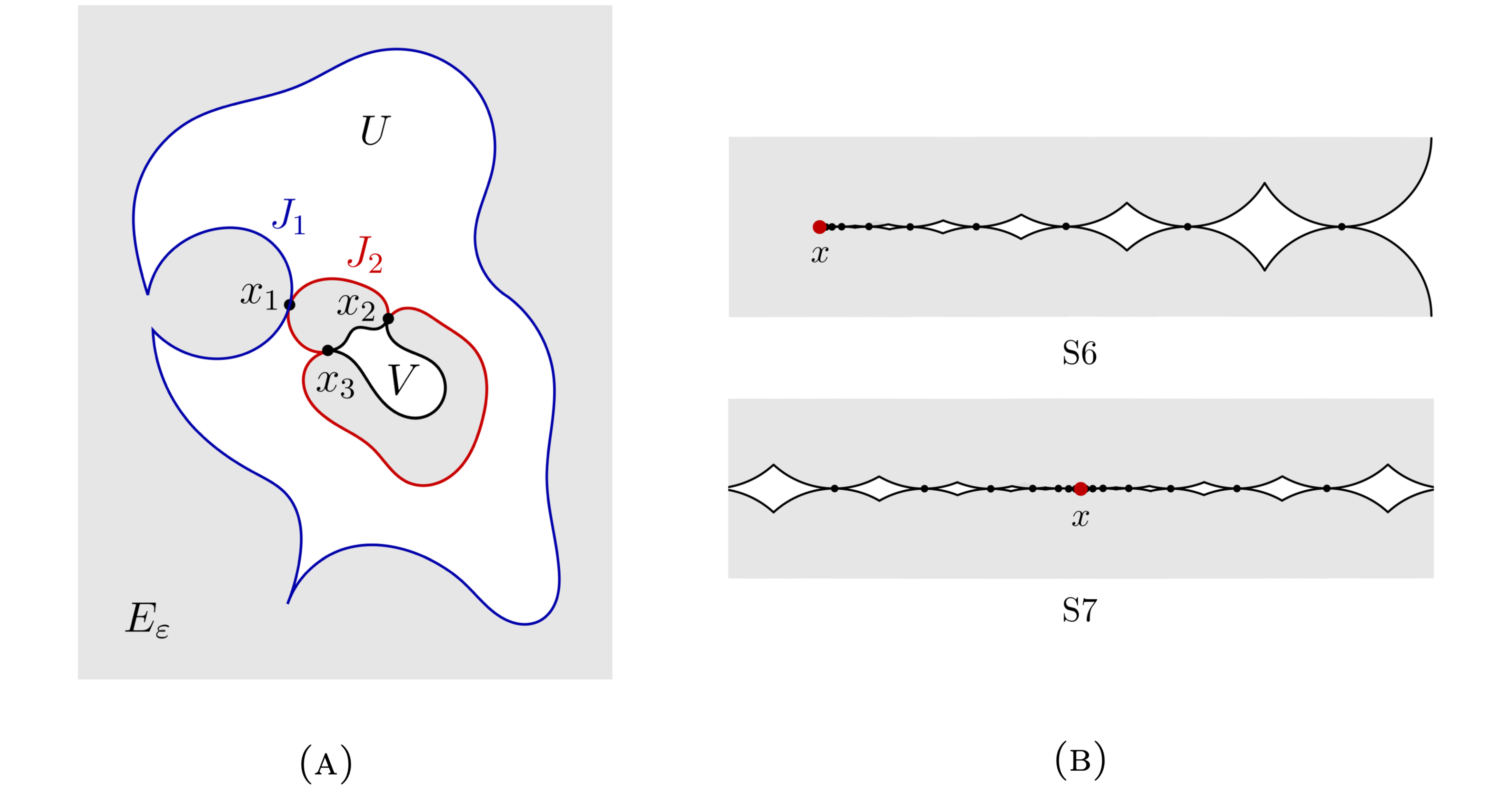}
                \caption{Schematic illustration of Theorem~\ref{Thm_global_structure}.
                ({\footnotesize A}): Jordan curve subsets of the boundary of a connected component $U$ of the complement $\R^2 \setminus E_\eps$. Here $\partial U = J_1 \cup J_2$ and $J_1 \cap J_2 = \{x_1\}$. At the centre of the figure there is another connected component $V \subset \R^2 \setminus E_\eps$, with $\partial V \cap \partial U = \{x_2, x_3\}$. Note that only the sharp-sharp singularities $x_1, x_2$ and $x_3$ present a choice regarding how to continue the boundary curves.
                ({\footnotesize B}): The two types of inaccessible singularities. A chain singularity (type S6) is the limit (in terms of Hausdorff distance) of a sequence of mutually disjoint connected components of the complement. Such a sequence approaches a chain-chain singularity (type S7) on both sides. These are the only types of boundary points that do not lie on Jordan curve subsets on the boundary.}
                \label{Figure_Theorem_1}
\end{figure}

\begin{theorem_without} \label{Thm_Main_2}
For any compact set $E \subset \R^2$, the corresponding $\eps$-neighbourhood boundary $\partial E_\eps$ contains at most a countably infinite number of wedges (type S1), sharp singularities (types S2, S3, S8), one-sided shallow singularities (type S4) and chain singularities (type S6).
\end{theorem_without}

In addition, we present examples which illustrate that the sets of shallow-shallow singularities (type S5) and chain-chain singularities (type S7) may be uncountable on $\partial E_\eps$, see Examples~\ref{Ex_Shallow_can_be_dense} and~\ref{Example_Pos_Measure_Set_of_Chain_Sing}. We refer to the union of categories S6, S7 and S8 as \emph{chain singularities} and denote the set of chain singularities on $\partial E_\varepsilon$ by $\cC(\partial E_\varepsilon)$. Even though the set of chain-chain singularities (type S7) may in general be uncountable and can have a positive Hausdorff measure on the boundary $\partial E_\varepsilon$, our third main finding is that $\cC(\partial E_\varepsilon)$ is closed and totally disconnected.

\begin{theorem_without} \label{Thm_Main_3}
For any compact set $E \subset \R^2$ and $\eps > 0$, the set $\cC(\partial E_\varepsilon)$ of chain singularities is closed and totally disconnected.
\end{theorem_without}
As a corollary, Theorem~\ref{Thm_Main_3} implies that $\cC(\partial E_\varepsilon)$ is nowhere dense on $\partial E_\eps$, and is hence small in the topological sense.

Our fourth main result characterises, in terms of local boundary geometry, those $\eps$-neighbourhoods $E_\eps$ whose complement $\overline{\R^2 \setminus E_\eps}$ is a set with positive reach, see Definition~\ref{Def_Positive_Reach}.

\begin{theorem_without} \label{Thm_Main_4}
Let $E \subset \R^2$ be compact and $\eps > 0$. Then the following are equivalent:
\begin{itemize}
\item[(i)] the complement $\overline{\R^2 \setminus E_\eps}$ is a set with positive reach,
\item[(ii)] for all $x \in \partial E_\eps$, there exists some radius $r(x) > 0$ for which the set $B_{r(x)}(x) \cap \overline{\R^2 \setminus E_\eps}$ is connected,
\item[(iii)] for all chain singularities $x \in \cC(\partial E_\eps)$, there exists some radius $r(x) > 0$ for which the set $B_{r(x)}(x) \cap \overline{\R^2 \setminus E_\eps}$ is connected.
\end{itemize}
\end{theorem_without}

In particular, this result implies that the complement $\overline{\R^2 \setminus E_\eps}$ may have positive reach even when the boundary $\partial E_\eps$ contains critical points of the distance function $d_E(x) := \inf_{y \in E} \norm{x - y}$. This is an improvement to the results in the existing literature~\cite{Fu_Tubular_neighborhoods}.

Building on the techniques developed in Chapters~\ref{Sec_Epsilon_Neighbourhoods}--\ref{Sec_Singularities} allows us to make the topology of $\partial E_\eps$ more precise in two ways.
First, it is possible to give a description that applies for all $\eps > 0$, without the need to ignore a countable set of $\eps$-values, as was done in previous studies of open\footnote{The boundaries $\partial E_{<\eps}$ of \emph{open} $\eps$-neighbourhoods $E_{<\eps} :=  \bigcup_{x \in E} B_\eps(x)$ have been referred to as the \emph{$\eps$-boundaries} \cite{Ferry_When_epsilon_boundaries} or \emph{$\eps$-level sets} \cite{Oleksiv_Pesin_Finiteness} of $E$. We note that $\partial E_{<\eps} = \{ x \in \R^2 \, : \, \mathrm{dist}(x, E) = \eps \}$ and that in general $\partial E_\eps$ is a closed subset of $\partial E_{<\eps}$. In this article we deal exclusively with the closed $\eps$-neighbourhoods~\eqref{Def_Tubular_Neighbourhood}.} $\eps$-neighbourhoods~\cite{Blokh_Misiurewicz_Oversteegen_Set_of_Constant, Brown_Sets_of_constant}. Second, the local geometry of the boundaries of closed $\eps$-neighbourhoods reveals that no simple curves exist on $\partial E_\eps$ for any $\eps > 0$, and that point components appear only in a very specific way.

Point components correspond to boundary points $x \in \partial E_\eps$ that do not lie on the boundary of any connected component of the complement $\R^2 \setminus E_\eps$. 
We call such points \emph{inaccessible singularities}. Our fifth main result is that, apart from the set of inaccessible singularities, the boundary $\partial E_\eps$ consists of a countable (possibly finite) union of Jordan curves.

\begin{theorem_without} \label{Thm_global_structure}
For any compact set $E \subset \R^2$ and $\eps > 0$, the boundary $\partial E_\eps$ is a disjoint union
\[
\partial E_\eps = \cI \cup J,
\]
where $\cI$ is the set of inaccessible singularities and $J = \bigcup_{i \in I} J_i$ is a countable (possibly finite) union of Jordan curves $J_i$. Furthermore there is a unique representation with the property that each Jordan curve $J_i$ satisfies $J_i \subset \partial U$ for some connected component $U$ of the complement $\R^2 \setminus E_\eps$.
\end{theorem_without}

In general, the Jordan curves $J_i$ in Theorem~\ref{Thm_global_structure} need not be disjoint. However, the bounded regions in $\R^2$ defined by these Jordan curves are mutually disjoint, and the intersection of any two Jordan curves $J_i$ is either empty, or contains finitely many points.

According to Theorem~\ref{Thm_Main_1}, each boundary point $x  \in \partial E_\eps$ is either smooth (in the sense that, in a neighbourhood of $x$, $\partial E_\eps$ is a $C^1$-curve) or belongs to exactly one of eight distinct types of singularities. From the topological point of view, the singularities can further be divided into three groups:
\begin{itemize}
\item[(i)] Point components of the boundary: the \emph{inaccessible singularities}.
\item[(ii)] Boundary points, around which the boundary $\partial E_\eps$ can be uniquely represented as a simple curve: these are the smooth points as well as the so-called \emph{wedges}, \emph{sharp} and \emph{sharp-chain singularities}\footnote{More precisely, the boundary of the unique connected component of the complement that touches the sharp-chain singularity can be uniquely represented by a simple curve in any sufficiently small neighbourhood, cf. Lemma~\ref{Lemma_Finite_Repr_for_Boundary}.} and \emph{shallow singularities}.
\item[(iii)] Boundary points, around which there is more than one way to locally represent the boundary as a union of simple curves: the so-called \emph{sharp-sharp singularities}.
\end{itemize}
Whenever sharp-sharp singularities exist on the boundary $\partial E_\eps$, there are several ways of representing the boundary as a union of curves. It turns out, however, that the boundary of each individual connected component of the complement $\R^2 \setminus E_\eps$ has a unique representation as a finite union of Jordan curves. In this article we show how to construct such representations, and use them to arrive at a global representation of the boundary $\partial E_\eps$ as a union of Jordan curves and point components. See~Figure~\ref{Figure_Theorem_1} for a schematic illustration of such decompositions of the boundary.

In Chapter~\ref{Sec_Ahlfors_regularity}, we recall the definitions of rectifiability and Ahlfors regularity, and show that the Jordan curve subsets of the boundary are
uniformly rectifiable.
We also give an example of an $\eps$-neighbourhood whose boundary is not an Ahlfors-regular set.
\begin{theorem_without} \label{Thm_suff_cond_for_uni_rect_intro}
Let $E \subset \R^2$ be compact, let $\eps > 0$, and assume that the boundary $\partial E_\eps$ contains no chain or chain-chain singularities. Then $\partial E_\eps$ is uniformly rectifiable.
\end{theorem_without}

Another fundamental property of interest is the smoothness of the boundary. It is known that for a fixed compact set $E \in \R^2$, the components of the boundary $\partial E_\eps$ are Lipschitz manifolds except for a zero-measure set of radii $\eps > 0$~\cite{Ferry_When_epsilon_boundaries, Fu_Tubular_neighborhoods, Rataj_Zajicek_Smallness}. In light of Theorem~\ref{Thm_global_structure}, and the geometric results obtained in Chapters~\ref{Sec_Epsilon_Neighbourhoods}--\ref{Sec_Singularities}, this property can be interpreted as the existence of well-defined tangents almost everywhere on the Jordan curve components of the boundary. This holds true despite the fact that singularities of wedge type may be dense on subsets of $\partial E_\eps$ that have positive one-dimensional Hausdorff measure, see Example~\ref{Ex_Shallow_can_be_dense}.

We analyse the curvature of the boundary $\partial E_\eps$ by using local representations in the form of graphs of Lipschitz-continuous functions. The construction of such \emph{local boundary representations} around all boundary points $x \in \partial E_\eps$ is one of the key technical results of this paper, and is presented in Chapter~\ref{Sec_Local_Structure}. According to Corollary~\ref{Prop_LBR_Differentiable_AE}, the first derivative of a local boundary representation exists almost everywhere, and it follows from Propositions~\ref{Prop_structure_of_set_of_outward_directions} and~\ref{Prop_Tangents_are_Defined} that one-sided derivatives exist at every point. We establish the existence of the second derivatives of these functions almost everywhere by showing that the right and left first derivatives are functions of bounded variation.

\begin{theorem_without} \label{Thm_existence_of_curvature}
Let $E \subset \R^2$ be compact and let $\cJ = \partial E_\eps \setminus \cI$ where $\cI$ is the set of inaccessible singularities. Then (signed) curvature $\kappa(x)$ exists for $\cH^1$-almost all $x \in \cJ$, where $\cH^1$ denotes the one-dimensional Hausdorff measure.
\end{theorem_without}

This paper has arisen from our interest in bifurcations of minimal invariant sets of random dynamical systems with bounded noise, which naturally appear as dynamically defined $\varepsilon$-neighbourhoods. In this context, the aim is to develop a theory which allows
the classification of qualitative changes of minimal invariant sets in (generic) parametrised families of random dynamical systems with bounded noise. The results in this paper provide a characterisation of boundaries at fixed values of the parameters, including $\varepsilon$, which is a first step towards this more general direction.

\section{Context} \label{Section_Context}
Systematic investigations into the regularity of the boundary $\partial E_\eps$ were initiated in the landmark papers of Ferry~\cite{Ferry_When_epsilon_boundaries} and Fu~\cite{Fu_Tubular_neighborhoods}, and were continued in the recent work of Rataj and {Z}aj\'{i}\v{c}ek~\cite{Rataj_Zajicek_Critical_Values_and_Level_Sets_of_Distance, Rataj_Zajicek_Smallness}. These efforts have focused on developing analytical techniques, such as the theory of semiconcave and DC functions. A particular aim has been to narrow down the exceptional set of radii $\eps > 0$ for which the boundary $\partial E_\eps$ contains critical points of the distance function
\[
d_E(x) := \inf_{y \in E} \norm{x - y}.
\]

According to a key result of Fu~\cite[Theorem 4.1]{Fu_Tubular_neighborhoods}, the absence of critical points on the boundary $\partial E_\eps$ implies that the complement $\overline{\R^2 \setminus E_\eps}$ has positive reach, as defined by Federer in~\cite{Federer_Curvature_measures}. In addition, the set of critical values of the distance function $d_E$ is known to be small in measure~\cite{Fu_Tubular_neighborhoods, Rataj_Zajicek_Smallness}, although in general uncountable~\cite[Example 1]{Ferry_When_epsilon_boundaries}. Combining these results implies that the general properties enjoyed by boundaries of sets with positive reach also apply to the boundary $\partial E_\eps$ for most values of the radius $\eps$.

Such properties were recently investigated by Rataj and {Z}aj\'{i}\v{c}ek~\cite{Rataj_Zajicek_On_the_structure_of_sets_with_PR}. For instance, it is shown in~\cite[Theorem 5.9]{Rataj_Zajicek_On_the_structure_of_sets_with_PR} that the boundary of a set in $\R^d$ with positive reach can be locally covered with finitely many semiconvex hypersurfaces. Our Proposition~\ref{Prop_local_representation_exists} analogously states that each $x \in \partial E_\eps$ has a neighbourhood in which the boundary can be represented using a finite number of graphs of Lipschitz-continuous functions.\footnote{It follows from~\cite[Theorem 5.9]{Rataj_Zajicek_On_the_structure_of_sets_with_PR} that these functions are in fact semiconvex.} In addition, it is known~\cite[Theorem 6.4]{Rataj_Zajicek_On_the_structure_of_sets_with_PR} that the possible geometries of the boundaries of sets with positive reach can be classified into three distinct categories. Our classification of singularities given in Theorem~\ref{Thm_Main_1} is the counterpart of this result in the context of $\eps$-neighbourhoods.

The analysis in~\cite{Rataj_Zajicek_On_the_structure_of_sets_with_PR} is based on the theory of semiconcave and DC functions and enables a more general level of applicability compared to the present paper. The drawback in the context of $\eps$-neighbourhoods is that~\cite[Theorem 6.4]{Rataj_Zajicek_On_the_structure_of_sets_with_PR} provides information about $\partial E_\eps$ only for almost all $\eps > 0$ and therefore gives an incomplete description of the boundary geometry. In contrast, our approach directly relies on the characteristic geometry of $\eps$-neighbourhoods. This brings the key benefit that we do not need to assume that the complement $\overline{\R^2 \setminus E_\eps}$ has positive reach, or that the boundary $\partial E_\eps$ contains no critical points of the distance function $d_E$. We are thus able to prove, for the first time, a complete classification of possible local boundary geometries of $\partial E_\eps$ for all $\eps > 0$.

In~\cite[Theorem 6.4]{Rataj_Zajicek_On_the_structure_of_sets_with_PR} a classification of local boundary geometries was provided in the case where the complement $\overline{\R^2 \setminus E_\eps}$ has positive reach.  Here, we extend these results beyond this assumption. When the complement $\overline{\R^2 \setminus E_\eps}$ is a set with positive reach, our classification of singularities given in Definition~\ref{Def_classification_of_singularities} can be directly compared with~\cite[Definition 6.3]{Rataj_Zajicek_On_the_structure_of_sets_with_PR}. The latter defines the possible boundary geometries in terms of hypo- and epigraphs of semiconcave and semiconvex functions, whereas our definition is based on the characteristic geometry of $\eps$-neighbourhoods and the local connectedness properties of the complement. However, these two definitions relate closely to each other via the local boundary representations established in Proposition~\ref{Prop_local_representation_exists}. The classification given in Definition~\ref{Def_classification_of_singularities} is a refinement of the one in~\cite[Definition 6.3]{Rataj_Zajicek_On_the_structure_of_sets_with_PR}, and extends it in the context of $\eps$-neighbourhoods. A detailed correspondence between the two is given at the end of Section~\ref{Subsect_Types_of_Singularities}. To clarify the scope of the different approaches, we identify in Theorem~\ref{Thm_Main_4} a local connectedness property that characterises those $\eps$-neighbourhoods whose complement $\overline{\R^2 \setminus E_\eps}$ is a set with positive reach.

{\larger{\larger{\textepsilon}}}-neighbourhoods arise in many branches of mathematics from convex analysis and manifold theory \cite{Przeworski_An_Upper_Bound} to fractal geometry \cite{Kaenmaki_Lehrback_Vuorinen_Dimensions}. In stochastic processes, so-called Wiener sausages represent smoothed-out Brownian motion trajectories \cite{Donsker_Varadhan_Wiener_Asymptotics, Hamana_On_The_Expected, Klartag_Eldans_Stochastic_Localization, R-S-M_Approximations_of, Rataj_Schmidt_Spodarev_On_The_Expected}, with applications in theoretical physics \cite{Kac-Luttinger_Bose-Einstein, Nekovar_Pruessner_A_Field_Theoretic_Approach, Spitzer_Electrostatic}. The relations between the surface area, volume and dimension of $\eps$-neighbourhoods \cite{Kaenmaki_Lehrback_Vuorinen_Dimensions, Rataj_Winter_On_Volume, Stacho_On_the_volume}, as well as the dependence of the manifold structure of the boundary $\partial  E_\eps$ on the radius $\eps$~\cite{Federer_Curvature_measures, Ferry_When_epsilon_boundaries, Fu_Tubular_neighborhoods, R-S-M_Approximations_of, Rataj_Schmidt_Spodarev_On_The_Expected, Rataj_Winter_On_Volume, Rataj_Zajicek_Critical_Values_and_Level_Sets_of_Distance,  Rataj_Zajicek_Properties_of_distance, Rataj_Zajicek_Smallness,Vellis_Sets_of_constant} have all received attention in the last decades.

The earliest reference to $\eps$-neighbourhoods
appears to be Erd\H{o}s's brief note \cite{Erdos_Some_remarks}. Here it was shown that for any $E \subset \R^d$, the $d$-dimensional Hausdorff measure of the boundary $\cH^d (\partial E_\eps)$ vanishes and that for any compact set $E$, the $(d-1)$-dimensional measure of the boundary is finite, i.e.~$\cH^{d-1} (\partial E_\eps) < \infty$. This result was also later documented in~\cite{Oleksiv_Pesin_Finiteness}.

After this early work, two branches of investigation into the structure of the boundary $\partial E_\eps$ developed, largely independently of each other. We outline here the main results of each line of inquiry, which we call the \emph{topological approach} and the \emph{analytical approach}.

\subsection*{The topological approach} The topological properties of the boundaries of $\eps$-neigh\-bour\-hoods were first studied in the 1970s~\cite{Brown_Sets_of_constant}. This work contains an analysis of the structure of the connected components of the \emph{$\eps$-boundary}
\begin{equation}  \label{Eq_Def_Eps_boundary}
\partial E_{<\eps} := \partial \big\{x \in \R^2 \, : \, d_E(x) < \eps \big\} = \big\{x \in \R^2 \, : \, d_E(x) = \eps \big\}.
\end{equation}
It was shown~\cite[Theorems 2 and 3]{Brown_Sets_of_constant} that for any compact planar set $E \subset \R^2$, the connected components of the $\eps$-boundary $\partial E_{<\eps}$ are locally connected and for all but countably many $\eps > 0$, each component of $\partial E_{<\eps}$ is either a point, a simple curve, or a simple closed curve (a Jordan curve).

More recently, it was shown in~\cite{Blokh_Misiurewicz_Oversteegen_Set_of_Constant} that the above result can be generalised to 2-manifolds with a geodesic metric (with or without boundary). This result too relies on~\cite{Moore_Concerning_triods}, but the first part of the above proof is replaced by an argument showing that the set $\partial E_{<\eps}$ is finitely Suslinian\footnote{A compactum is said to be \emph{finitely Suslinian} if for each $\eps > 0$, each collection of pairwise disjoint subcontinua of diameter larger than $\eps$ is finite.}, which implies local and arcwise connectedness.

\begin{figure}[ht]
  \begin{minipage}[c]{0.4\textwidth}
  \mbox{}\\[-\baselineskip] \vspace{-6mm}
    \includegraphics[width=1.19\textwidth]{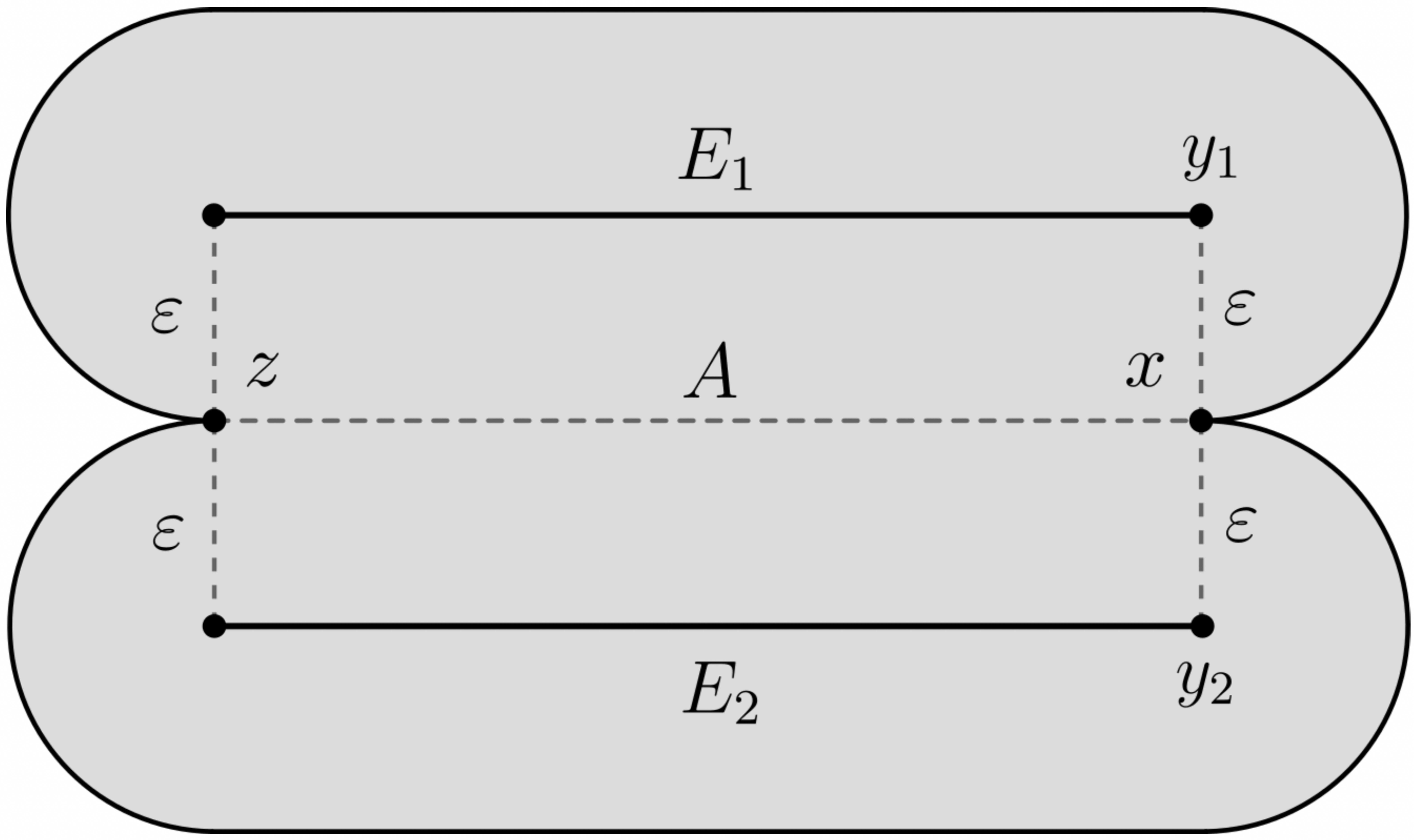}
  \end{minipage}
  \begin{minipage}[c]{0.05\textwidth}
   \mbox{}\\[-\baselineskip] \vspace{2mm}
  \end{minipage}
  \begin{minipage}[c]{0.54\textwidth}
  \mbox{}\\[-\baselineskip]  
    \caption[Difference between the $\eps$-boundary $\partial E_{<\eps}$ and the boundary of the closed $\eps$-neighbourhood $\partial E_\eps$.]{Schematic illustration of the difference between the $\eps$-boundary $\partial E_{<\eps}$ and the boundary of the closed $\eps$-neighbourhood $\partial E_\eps$. For $E = E_1 \cup E_2$, the set $A := \partial E_{<\eps} \setminus \partial E_\eps$ is non-empty. See also the discussion following Lemma~\ref{Lemma_OD_versus_dot_product_ineq} and~\cite[Example 2.1]{Rataj_Winter_On_Volume}.
                 } \label{Figure_Extremal_Interior}
  \end{minipage}
\end{figure}
\vspace{2mm}

\subsection*{The analytical approach} This approach uses tools from (set-valued) analysis in order to study, for a given set $E \in \R^d$, the \emph{critical points} and \emph{critical values} of the distance function $d_E : \R^d \to \R$,
\begin{equation} \label{Eq_Def_Distance_Function}
d_E(x) := \inf_{y \in E} \norm{x - y}.
\end{equation}
The rationale of this approach is that even for an arbitrary set $E \in \R^2$, the distance function $d_E$ only has few critical values $\eps$, and that for regular (non-critical) values, the $\eps$-boundary $\partial E_{<\eps}$ turns out to be a geometrically well-behaved set.

The first step in this direction was the clarification of the manifold structure of the $\eps$-boundary~\eqref{Eq_Def_Eps_boundary} in~\cite{Ferry_When_epsilon_boundaries}. The main result~\cite[Corollary 2.3]{Ferry_When_epsilon_boundaries}\footnote{In addition to his main result, Ferry also provides counterexamples showing that~\cite[Corollary 2.3]{Ferry_When_epsilon_boundaries} cannot be  improved to hold true apart from a countable set of radii \hspace{-0.1mm}\cite[Example 1]{Ferry_When_epsilon_boundaries} or in dimensions $d > 3$ \hspace{-0.1mm}\cite[Example 3]{Ferry_When_epsilon_boundaries}, and that Brown's results~\cite[Theorems 2 and 3]{Brown_Sets_of_constant} cannot be generalised to $d = 3$ \hspace{-0.1mm}\cite[Example 2]{Ferry_When_epsilon_boundaries}.}  is that for $d \in \{2,3\}$ and $E \subset \R^d$, the $\eps$-boundary $\partial E_{<\eps}$ is a $(d-1)$-manifold for almost all $\eps > 0$. In the case $d = 2$, this result had been obtained a few years earlier~\cite{Gariepy_Pepe_On_the_level}. Although this result provides a more specific statement compared to~\cite{Brown_Sets_of_constant}, this comes at the expense of replacing 'for all but countably many $\eps > 0$' in~\cite[Theorems 2 and 3]{Brown_Sets_of_constant} with 'for almost all $\eps > 0$'. In terms of the cardinality of the exceptional set of radii, the results are shown to be optimal, see~\cite[Examples 1 and 2]{Ferry_When_epsilon_boundaries}. Thus, it has become a central aim in subsequent work to obtain more refined measure theoretic and other bounds on the size of the set of exceptional radii in~\cite[Corollary 2.3]{Ferry_When_epsilon_boundaries}.

Such a bound was provided roughly a decade later in 1985 by Fu~\cite{Fu_Tubular_neighborhoods} as part of his doctoral dissertation. Let $E \subset \R^d$ be compact and let $\mathrm{cv}(d_E)$ denote the set of critical values of the distance function $d_E$, see Definition~\ref{Def_Critical_points_and_critical_values} below. Then, for each regular value $\eps \notin \mathrm{cv} (d_E)$, the boundary $\partial E_\eps$ is a Lipschitz manifold and the complement $\overline{\R^d \setminus E_\eps}$ is a set with positive reach~\cite[Theorem 4.1]{Fu_Tubular_neighborhoods}. Furthermore, the set of critical values $\mathrm{cv} (d_E)$ is compact and is contained in the interval $\big[0, \sqrt{d / 2(d + 1)} \mathrm{diam} (E)\big]$, in which $\textrm{diam}(E)$ denotes the diameter of the set $E$.

In~\cite{Fu_Tubular_neighborhoods}, the critical points and critical values of $d_E$ are defined in terms of the \emph{generalised (Clarke) gradient} $\partial d_E$, which is a set-valued generalisation of the classical gradient.\footnote{For almost everywhere differentiable functions $f$, the generalised gradient $\partial f(x)$ is the convex hull of the set of limits of the form $\lim_{i \to \infty} \nabla f(x + h_i)$, where $h_i \to 0$ as $i \to \infty$, and $\nabla f$ denotes the classical gradient, see~\cite[Definition~1.1]{Clarke_Generalized}. However, the importance of this notion stems from the fact that it can be defined for a much larger class of functions, see for instance~\cite{Clarke_Generalized,Clarke_Optimization}.} We provide here the definition of critical points and critical values. We also provide a geometric characterisation of critical points that we will refer to in Chapter~\ref{Sec_Pos_Reach} as we analyse conditions that guarantee that the complement $\overline{\R^2 \setminus E_\eps}$ of the $\eps$-neighbourhood is a set with positive reach.

\begin{definition}[Critical points, critical values, and regular points of the distance function~{\cite[Definition 3.2]{Fu_Tubular_neighborhoods}}] \label{Def_Critical_points_and_critical_values}
The set of \emph{critical points} of the distance function $d_E$ in~\eqref{Eq_Def_Distance_Function} is given by
\begin{equation} \label{Eq_Def_Critical_Points_Clarke}
\mathrm{crit}(d_E) := \big\{x \in \R^d \, : \, 0 \in \partial d_E(x) \big\}.
\end{equation}
The set of \emph{critical values} of $d_E$ is given by
\begin{equation} \label{Eq_Def_Critical_Values}
\mathrm{cv} (d_E) := d_E(\mathrm{crit}(d_E)) \subset \R_+ .
\end{equation}
All other points $x \in \mathrm{reg}(d_E) := \R^d \setminus \mathrm{crit}(d_E)$ are called \emph{regular points} of $d_E$.
\end{definition}

It turns out that $x \in \R^d \setminus E$ is a critical point of $d_E$ in the sense of Definition~\ref{Eq_Def_Critical_Points_Clarke} if and only if it satisfies the geometric property
\begin{equation} \label{Eq_Def_Critical_Points_Ferry}
x \in \mathrm{Conv} (\Pi_E(x)),
\end{equation}
where $\mathrm{Conv}(\cdot)$ denotes the convex hull and $\Pi_E(x)$ is the metric projection
\begin{equation} \label{Def_metric_projection}
\Pi_E(x) := \{y \in E \, : \, d_E(x) = \norm{x - y} \}
\end{equation}
of $x$ onto $E$, see~\cite[Lemma 4.2]{Fu_Tubular_neighborhoods}.
In the earlier work~\cite[Definition 1.3]{Ferry_When_epsilon_boundaries} the same geometric property~\eqref{Eq_Def_Critical_Points_Ferry} was used directly as the definition of a critical point.

A key result in~\cite{{Fu_Tubular_neighborhoods}} is the connection between the regular values of the distance function $d_E$ and the geometric property of the closure of the complement $\overline{\R^2 \setminus E_\eps}$ being a \emph{set with positive reach}.\footnote{This concept was introduced in~\cite{Federer_Curvature_measures} as a natural generalisation of both convex sets and smooth manifolds.
For sets with positive reach, it is possible to formulate and prove generalised analogues for the Steiner and kinematic formulas from convex analysis via the notion of \emph{curvature measures}.} The reach of a set $E \subset \R^d$ is defined in terms of points $x \in \R^d$ whose metric projection~\eqref{Def_metric_projection} onto $E$ is a singleton. We denote the set of such points by $\mathrm{Unp}(E)$.
\begin{definition}[The reach of a set {\hspace{-0.1mm}{\cite[Definition 4.1]{Federer_Curvature_measures}}}] \label{Def_Positive_Reach}
For $E \subset \R^d$ the (local) reach at $y \in E$ is given by
\begin{equation} \label{Eq_Def_Reach}
\mathrm{reach} (E, y) := \sup\{r \geq 0 \, : \, B(y,r) \subset \mathrm{Unp}(E)\}.
\end{equation}
The \emph{reach} of the set $E \subset \R^d$ is given by
\begin{equation} \label{Eq_Def_Positive_Reach}
\mathrm{reach} \, E := \inf_{y \in E} \mathrm{reach} (E, y).
\end{equation}
\end{definition}

In light of~\cite[Theorem 4.1]{Fu_Tubular_neighborhoods} it is natural to inquire into the structure and size of the set of critical values $\mathrm{cv} (d_E)$ of the distance function. Recent work~\cite{Rataj_Zajicek_Smallness} introduced the notion of $BT_\alpha$-sets for subsets of the real line, to further quantify the size of this set.

\begin{definition}[{\hspace{-0.1mm}{\cite[Definition 2.2]{Rataj_Zajicek_Smallness}}}] \label{Def_BT_alpha_sets}
Given a compact set $K \subset \R$ and $\alpha > 0$, the \emph{degree-$\alpha$-gap sum} of $K$ is
\[
G_\alpha(K) := \sum_{I \subset \cG_K} |I|^\alpha,
\]
where $\cG_K$ is the set of all bounded components of $\R \setminus K$. The set $K$ is a \emph{$BT_\alpha$-set} if it is compact, has zero Lebesgue measure and $G_\alpha(K) < \infty$.
\end{definition}
Let $L_E \subset \R$ be the set of radii $\eps > 0$ for which the $\eps$-boundary $\partial E_{<\eps}$ fails to be a Lipschitz manifold. The first main result of~\cite{Rataj_Zajicek_Smallness} states that for any $\delta > 0$ and $A \subset [\delta, \infty)$ the properties
\begin{itemize}
\item[(i)] $A \subset L_E$ for some compact set $E \subset \R^2$;
\item[(ii)] $A \subset \mathrm{cv}(d_E)$ for some compact set $E \subset \R^2$;
\item[(iii)] $\overline{A}$ is a $BT_{1/2}$-set
\end{itemize}
are equivalent~\cite[Theorem~1.1]{Rataj_Zajicek_Smallness}. This result thus provides a characterisation of the smallness of the set $\mathrm{cv}(d_E) \cap [\delta, \infty)$, i.e.~the set of critical values of the distance function bounded away from $0$. In addition to this, the authors provide a characterisation of the non-truncated set of critical values~\cite[Theorem~1.2]{Rataj_Zajicek_Smallness}, and show these results to be optimal. Notably, the methods used in the proofs are independent of those employed in~\cite{Fu_Tubular_neighborhoods}, and in fact provide an alternative proof for~\cite[Theorem~4.1]{Fu_Tubular_neighborhoods}. The argument used in~\cite{Rataj_Zajicek_Smallness} is instead based on the analysis of critical points of DC functions (i.e., differences of convex functions) and an inequality due to Ferry~\cite[Proposition~1.5]{Ferry_When_epsilon_boundaries}.

\section{Structure of the paper}
The rest of the paper is structured as follows. Chapter~\ref{Sec_Epsilon_Neighbourhoods} sets up the conceptual framework and terminology that will be used throughout the paper. In Section~\ref{Subsec_Properties_of_Outward_Dir}, we introduce the notions of \emph{contributors} (Definition~\ref{Def_Contributor}) and \emph{outward directions} (Definition~\ref{Def_outward_direction}), and their basic properties. We also establish their relationship with the tangential properties of the boundary $\partial E_\eps$. 

In Chapter~\ref{Sec_Local_Structure}, we investigate the local properties of boundary points $x \in \partial E_\eps$ in small neighbourhoods $\overline{B_r(x)}$. The key result is Proposition~\ref{Prop_local_contribution}, which states that for $E \subset \R^2$ the boundary geometry near each $x \in \partial E_\eps$ is defined solely by those $y \in \partial E$ that lie near the \emph{extremal contributors} (Definition~\ref{Def_Extremal_Outward_Directions_and_Contributors}) of $x$. Building on this observation and a related approximation scheme (Definition~\ref{Def_Finite_Approximation_Scheme}), we show that the boundary $\partial E_\eps$ can be represented locally as a finite union of continuous graphs (Proposition~\ref{Prop_local_representation_exists}). This representation plays a central role in the proofs of the subsequent results. Section~\ref{Section_Local_Structure_of_the_Complement} contains an analysis of the topological and geometric structure of the complement $\R^2 \setminus E_\eps$ near smooth points, wedges (type S1) and shallow singularities (types S4 and S5).

Chapters~\ref{Sec_Singularities} and~\ref{Sect_Topo_Structure} contain the first main results of this paper. We present the different types of singularities encountered on the boundary $\partial E_\eps$ (Definition~\ref{Def_classification_of_singularities}) and show how these can be characterised in terms of the local topological structure of the complement $\R^2 \setminus E_\eps$ and the geometric properties of the boundary $\partial E_\eps$ (Propositions~\ref{Prop_Sharp_Singularity_Types} and~\ref{Prop_Topological_Characterisation_of_Chain_Singularities}, Corollary~\ref{Cor_Inaccessible_Singularities}). Chapter~\ref{Sec_Singularities} concludes with the proof of our first main result, Theorem~\ref{Thm_Main_1}, which states that the classification given in Definition~\ref{Def_classification_of_singularities} defines a partition of $\partial E_\eps$.

In Chapter~\ref{Sect_Topo_Structure}, we study the cardinalities of the subsets on the boundary $\partial E_\eps$ that correspond to different types of singularities. We show that the number of singularities of types S1--S4, S6 and S8 on the boundary $\partial E_\eps$ is at most countably infinite (Theorem~\ref{Thm_Main_2}). Furthermore, we show that the set of chain singularities is closed and totally disconnected (Theorem~\ref{Thm_Main_3}).

Chapter~\ref{Sec_Pos_Reach} contains an analysis of the relationship between the local geometry of the boundary $\partial E_\eps$ and the reach of the complement $\overline{\R^2 \setminus E_\eps}$. In particular, we show that $\textrm{reach}\big( \overline{\R^2 \setminus E_\eps}, x \big) = 0$ if and only if $x$ is a chain-singularity for which the set $B_r(x) \cap \overline{\R^2 \setminus E_\eps}$ is disconnected for all $r > 0$. The chapter concludes with a characterisation of those $\eps$-neighbourhoods whose complement is a set with positve reach (Theorem~\ref{Thm_Main_4}).

In Chapter~\ref{Sec_Boundaries_as_Jordan_Curves}, we turn our attention to the global topological structure of the boundary $\partial E_\eps$. We utilise the local boundary representations (Proposition~\ref{Prop_local_representation_exists}) in order to construct Jordan curve covers for the boundaries of the connected components of the complement $\R^2 \setminus E_\eps$ (Lemma~\ref{Lemma_Finite_Repr_for_Boundary}) and demonstrate that these covers always have a subcover of order two (Lemma~\ref{Lemma_Order_Two_Cover}). We use these partial results to prove Proposition~\ref{Prop_Jordan_curve_boundaries}, which states that the boundary of each connected component of the complement $\R^2 \setminus E_\eps$ can be expressed as a finite union of Jordan curves. We close this chapter with the proof of Theorem~\ref{Thm_global_structure}, which provides a complete topological description of the boundary $\partial E_\eps$ in terms of Jordan curve subsets and inaccessible singularities.

In Chapter~\ref{Sec_Ahlfors_regularity} we recall the definitions of rectifiability (\ref{Def_Rectifiability}), Ahlfors regularity (\ref{Def_Ahlfors_Regularity}), and uniform rectifiability (\ref{Def_uniform_rect}). We show that all Jordan curve subsets of the boundary $\partial E_\eps$ are Ahlfors regular (Proposition~\ref{Prop_Jordan_curves_are_Ahlfors_regular}). The main result of this chapter is Theorem~\ref{Thm_suff_cond_for_uni_rect_intro}, which states that those $\eps$-neighbourhoods whose boundary does not contain any chain singularities, are uniformly rectifiable. To augment this result, in Example~\ref{Ex_counterexample_to_Ahlfors} we provide a construction of an $\eps$-neighbourhood whose boundary contains a chain singularity and is not an Ahlfors-regular set, and hence not uniformly rectifiable.

In the final Chapter~\ref{Sec_Existence_of_Curvature}, we investigate the existence of curvature on the Jordan curve subsets of the boundary. We again make use of the local boundary representations given by Proposition~\ref{Prop_local_representation_exists}, this time as a way of analysing the local smoothness of the boundary. 
We formulate an elementary sufficient condition for bounded variation, which we apply to the local boundary representations. This enables us to prove our last main result, Theorem~\ref{Thm_existence_of_curvature}, which states that curvature is well-defined almost everywhere on $\partial E_\eps$ with respect to the one-dimensional Hausdorff measure. The result is based on three technical lemmas,
which allow us to show that the functions constituting the local boundary representations satisfy a specific lower bound for tangential directions.

\section{Outlook}
In this paper we have discussed a number of properties of boundaries of $\eps$-neighbourhoods of compact sets in two dimensions. However, some interesting problems remain open. For instance, in Chapter~5, measure theoretical and topological characterisations have been obtained for the sets of different types of singularities on the boundary. However, it remains open to clarify how the properties of the underlying set $E$ lead to different categories of singularities on the boundary $\partial E_\eps$.

This work was motivated by our interest in two-dimensional set-valued dynamical systems with spherical reach, in particular the singularity structure of attractor boundaries in this setting. While the classification of singularities in this paper applies to this problem, it does not address the likelihood (co-dimension) with which these singularities are observed. Preliminary numerical studies indicate that -- as expected -- wedge singularities can arise in a persistent manner, but surprisingly so can the more exotic-looking shallow singularity~\cite{Lamb_Rasmussen_Tey_Henon}.

The dynamical systems angle also provides a natural connection with singularity theory \cite{Kourliouros2023_Persistence}. Indeed, the singularities we find are Legendrian  \cite{Arnold_Singularities}.
However, the singularity-theoretical classification does not align with ours: some Legendrian singularities do not arise on $\partial E_\varepsilon$ and singularities on $\partial E_\varepsilon$ that are accumulated by wedges have co-dimension infinity from the singularity theory point of view (and are normally not really discussed in singularity theory). 

Finally, it is of obvious interest how our results extend to three- and higher- dimensional settings. Ideally, an approach may be developed that is less dependent on the ambient dimension. 
\chapter{{\larger{\textepsilon}}-neighbourhoods} \label{Sec_Epsilon_Neighbourhoods}
The object of our study is the \emph{$\eps$-neighbourhood} $E_\eps = \overline{B_\eps(E)}$ of a closed subset $E \subset \R^d$. The main results of this paper concern $\eps$-neighbourhoods of planar sets $E \subset \R^2$, but we provide the basic definitions in a more general $d$-dimensional setting, for $d \in \N_+$. Throughout the paper we make the assumption that the underlying set $E \subset \R^d$ is closed\footnote{Note that this is not an actual restriction, since $d_E(x) = d_{\overline{E}}(x)$ for all $x \in \R^d$ and $E \subset \R^d$.} and $\eps > 0$. Many of the results require the stronger assumption of compactness; where necessary, this will be explicitly stated in the formulation of each result.

The most immediate observation regarding the structure of the set $E_\eps$ is that each $x \in \partial E_\eps$ necessarily lies on the boundary of a closed ball $\overline{B_\eps(y)}$ of radius $\eps$, centered at some $y \in \partial E$. On the other hand, for each $x \in \partial E_\eps$ there may exist more than one $y \in \partial E$ with $\norm{y - x} = \eps$. These considerations motivate the following definition.

\begin{definition}[{Contributor}] \label{Def_Contributor}
Let $E \subset \R^d$ be closed. For each $x \in \partial E_\eps$ we define the set of \emph{contributors} as the collection
\[
\Pi_E(x) = \big\{y \in \partial E \,\, : \,\, \norm{y - x} = \eps \big\}.
\]
Boundary points $x \in \partial E_\eps$ with only one contributor constitute the set
\[
\Unp{E} := \big\{x \in \partial E_\eps \, : \, \Pi_E(x) = \{y\} \,\, \textrm{for some} \,\, y \in \partial E \big\},
\]
where $\mathrm{Unp}$ stands for \emph{unique nearest point}, see \cite[Definition 4.1]{Federer_Curvature_measures}.
\end{definition}

The set of contributors $\Pi_E(x)$ consists of those points on $\partial E$ that minimise the distance from $\partial E$ to $x$. Hence, $\Pi_E$ can be interpreted as a restriction onto $\partial E_\eps$ of the metric projection $\mathrm{proj}_E: \R^2 \to E$, given by
\[
\mathrm{proj}_E(x) := \{y \in E \, : \, \norm{y - x} = \mathrm{dist}(x, E) \}.
\]
In our classification of boundary points, the set $\Unp{E}$ consists of smooth points and shallow singularities (types S4 and S5), see Definition~\ref{Def_classification_of_singularities} and Figures~\ref{Figure_Types_of_Singularities} and~\ref{Figure_Three_Basic_Cases}.

\begin{definition}[{Smooth point, singularity}] \label{Def_Smooth_Points}
A boundary point $x \in \partial E_\eps$ is $\emph{smooth}$, if there exists a neighbourhood $B_r(x)$ for which $\partial E_\eps \cap B_r(x) \subset \Unp{E}$.
If $x$ is not smooth, it is a \emph{singularity} and we write $x \in \Sing$.
\end{definition}

The rationale for 
the above formulation stems from the observation that each $x \in \partial E_\eps$ is smooth in terms of Definition~\ref{Def_Smooth_Points} if and only if $x$ has a neighbourhood $B_r(x)$ in which the boundary is a $C^1$-smooth curve, see Proposition~\ref{Prop_Characterisation_of_smooth_points}.

\section{Tangents via Outward Directions}
Our first objective is to analyse the tangential properties of individual boundary points $x \in \partial E_\eps$. Acknowledging that classical tangents do not necessarily exist everywhere on the boundary, we adopt a set-valued definition of tangency which allows for several tangential directions to exist at each point. Our definition is a restriction of \cite[Definition 4.3]{Federer_Curvature_measures} to the boundary $\partial E_\eps$.

\begin{definition}[{Tangent set}] \label{Def_one-sided-tangent}
Let $E \subset \R^d$ be closed and $x \in \partial E_\eps$. We define the set $T_x(E_\eps)$ of unit \emph{tangent vectors of $E_\eps$ at $x$} as all those points $v \in S^{d-1}$ for which there exists a sequence $(x_n)_{n=1}^\infty \subset \partial E_\eps$ of boundary points satisfying $x_n \to x$ and
\[
\frac{x_n - x}{\norm{x_n - x}} \rightarrow v, \quad \textrm{as } \, n \to \infty.
\]
\end{definition}

In order to study the existence of tangential directions at boundary points $x \in \partial E_\eps$, we relate the set $T_x(E_\eps)$ to what we call outward directions. Intuitively, the set of outward directions at each $x$ contains the angles at which $x$ can be approached from the complement $\R^d \setminus E_\eps$. It turns out that for an $\eps$-neighbourhood $E_\eps$, the extremal values of these angles coincide with the tangential directions as defined in Definition~\ref{Def_one-sided-tangent}.
The existence of tangents at each $x \in \partial E_\eps$ thus depends on the properties of the corresponding set of outward directions at $x$. These turn out to be easier to study due to their geometric relationship with the contributors $y \in \Pi_E(x)$.

We define outward directions as points on the unit sphere $S^{d-1} \subset \R^d$ but think of them rather as directional vectors in the ambient space $\R^d$, since we want to operate with them using the Euclidean scalar product $\langle \cdot, \cdot \rangle \, : \, \R^d \times \R^d \to \R$.

\begin{definition}[{Outward direction}] \label{Def_outward_direction}
Let $E \subset \R^d$ be closed, and let $\eps > 0$. We say that a point $\xi \in S^{d-1}$ is an \emph{outward direction} from $E_\eps$ at a boundary point $x \in \partial E_\eps$, if there exists a sequence $(x_n)_{n=1}^\infty \subset \R^d \setminus E_\eps$, for which
$x_n \to x$ and
\[
\xi_n := \frac{x_n - x}{\norm{x_n - x}} \longrightarrow \xi \in S^{d-1} \subset \R^d,
\]
as $n \to \infty$. We denote by $\Xi_x(E_\eps)$ the set of outward directions from $E_\eps$ at $x$.
\end{definition}

\begin{figure}
      \centering
      \captionsetup{margin=0.75cm}
                \includegraphics[width = \textwidth]{Basic_concepts_combined_v01.eps}
                \caption{Illustration of the relationship between outward directions and extremal contributors. ({\footnotesize A}) A singularity $x \in \partial E_\eps$ with $\Cext{x} = \{y_1, y_2\}$ and $\Xext{x} = \{\xi_1, \xi_2\}$. ({\footnotesize B}) The set $\Xi_x(E_\eps) \subset S^1$ of outward directions is geodesically convex with boundary $\partial_{S^1} \Xi_x(E_\eps) = \{\xi_1, \xi_2\}$.}
                \label{Figure_Outward_Directions_and_Contributors}
\end{figure}

\begin{remark} \label{Remark_Bouligand}
Our definition of the set of outward directions is a variation of the well-known contingent cone, introduced by Bouligand, see for instance \cite{Aubin_Frankowska_Set-valued_Analysis, Bigolin_Greco_Geometric_Characterization, Rockafellar_Variational_Analysis} and Bouligand's original work in \cite{Bouligand_Sur_les_surfaces, Bouligand_Introduction}. For a boundary point $x \in \partial E_\eps$, the \emph{contingent cone} or \emph{Bouligand cone} $C_x(E_\eps)$ consists of those vectors $v \in \R^d$, for which
there exist sequences $(h_n)_{n=1}^\infty \subset \R_+$ and $(v_n)_{n=1}^\infty \subset \R^d$ for which  $x + h_n v_n \in E_\eps$ for all $n \in \N$ and
\[
h_n \to 0, \quad \textrm{and} \quad v_n \to v
\]
as $n \to \infty$. In fact, it is easy to verify that for each $x \in \partial E_\eps$, the contingent cone $C_x\big(\R^2 \setminus E_\eps\big)$ for the complement coincides with the \emph{outward cone}
\[
\cW_x(E_\eps) := \left\{ s \xi  \, : \, \xi \in \Xi_x(E_\eps), \, s \geq 0 \right\}
\]
at $x$. We do not make use of this correspondence. For further information on tangent cones, see for instance \cite{Clarke_Generalized, Rockafellar_Clarke_tangent}.
\end{remark}

For each $x \in \partial E_\eps$, we single out those outward directions $\xi \in \Xi_x(E_\eps)$ that are perpendicular to some contributor $y \in \Pi_E(x)$. We call these the \emph{extremal outward directions} and \emph{extremal contributors}, respectively, see Figure~\ref{Figure_Outward_Directions_and_Contributors}. Definition~\ref{Def_Extremal_Outward_Directions_and_Contributors} below emphasises this geometric relationship. Proposition~\ref{Prop_structure_of_set_of_outward_directions} in Section~\ref{Subsec_Properties_of_Outward_Dir} confirms that extremal outward directions can equivalently be defined through the topological property of constituting the boundary of the set of outward directions $\Xi_x(E_\eps)$ on $S^{d-1}$.

\begin{definition}[{Extremal contributor, extremal outward direction}] \label{Def_Extremal_Outward_Directions_and_Contributors}
Let $E \subset \R^d$ be closed, $\eps > 0$, and $x \in \partial E_\eps$. If an outward direction $\xi \in \Xi_x(E_\eps)$ and a contributor $y \in \Pi_E(x)$ satisfy
\[
\langle y - x, \xi \rangle = 0,
\]
we call $\xi$ an \emph{extremal outward direction} and $y$ an \emph{extremal contributor}  at $x$.
For each $x \in \partial E_\eps$, we write $\Xext{x}$ and $\Cext{x}$ for the sets of extremal outward directions and extremal contributors, respectively.
\end{definition}

\section{Properties of Contributors and Outward Directions} \label{Subsec_Properties_of_Outward_Dir}

We collect in this section the basic properties of contributors and outward directions that are needed repeatedly in the rest of the paper. The existence of outward directions at each $x \in \partial E_\eps$ is established in Proposition~\ref{Prop_outward_exists_and_is_closed}. Their geometric relationship with the contributors is clarified in Lemma~\ref{Lemma_OD_versus_dot_product_ineq} and Proposition~\ref{Prop_structure_of_set_of_outward_directions}. The observation that the set of tangential directions $T_x(E)$ coincides with the set of extremal outward directions $\Xext{x}$ is established in Proposition~\ref{Prop_Tangents_are_Defined}.
As before, we assume that the set $E \subset \R^d$ is closed, and that $\eps > 0$.

\begin{prop}[{The set of outward directions is non-empty and closed}] \label{Prop_outward_exists_and_is_closed}
Let $E \subset \R^d$ be closed and $x \in \partial E_\eps$. Then the set $\Xi_x(E_\eps)$ of outward directions is non-empty and closed.
\end{prop}

\begin{proof}
This follows from the basic properties of contingent cones, see for instance~\cite[p.~121]{Aubin_Frankowska_Set-valued_Analysis} and~\cite[Section 3.1.21]{Federer_Geometric_Measure_Theory}.
\end{proof}

The following simple Lemmas~\ref{Lemma_limit_of_contributing_points} and~\ref{Lemma_Basic_Properties_of_Sequences_Converging_onto_Boundary} regarding the properties of contributors play an important role in the proofs of many of the subsequent results.

\begin{lemma}[{Convergence of contributors}] \label{Lemma_limit_of_contributing_points}
Let $E \in \R^d$ be compact, $(x_n)_{n=1}^\infty \subset E_\eps$ with $x_n \to x \in \partial E_\eps$, and $(y_n)_{n=1}^\infty \subset E$ with $x_n \in \overline{B_\eps(y_n)}$ for all $n \in \N$. Then there exists some $y \in \Pi_E(x)$ and a convergent subsequence $(y_{n_k})_{k=1}^\infty$, for which $y_{n_k} \to y$ as $k \to \infty$.
\end{lemma}

\begin{proof}
Since $E$ is compact, there exists a convergent subsequence $(y_{n_k})_{k=1}^\infty$ with $y_{n_k} \to y \in E$ as $k \to \infty$. We need to show that $y \in \Pi_E(x)$. For each $k \in \N$,
\[
\norm{y - x} \leq \norm{y - y_{n_k}} + \norm{y_{n_k} - x_{n_k}} + \norm{x_{n_k} - x}, 
\]
which implies $\norm{y - x} \leq \eps$ since $x_{n_k} \to x$, $y_{n_k} \to y$, and $\lim_{k\to\infty} \norm{y_{n_k} - x_{n_k}} \leq \eps$. On the other hand $\norm{y - x} \geq d_E(x) = \eps$. Hence $\norm{y - x} = \eps$ so that $y \in \Pi_E(x)$.
\end{proof}

\begin{lemma}[{Tails of directed sequences}] \label{Lemma_Basic_Properties_of_Sequences_Converging_onto_Boundary}
Let $E \subset \R^d$ be compact and let $(x_n)_{n=1}^\infty \subset \R^d$ with $x_n \to x \in \partial E_\eps$ and $x_n \neq x$ for all $n \in \N$. Assume that
\[
v_n := \frac{x_n - x}{\norm{x_n - x}} \longrightarrow v \in S^{d-1} 
\]
as $n \to \infty$. Then
\begin{enumerate}
\item[(i)] if $\langle y - x, v \rangle > 0$ for some $y \in \Pi_E(x)$, then there exists some $N \in \N$, for which $x_n \in B_\eps(y)$ 
for all $n \geq N$,
\item[(ii)] if $\langle y - x, v \rangle < 0$ for all $y \in \Pi_E(x)$, then there exists some $N \in \N$, for which $x_n \in \R^d \setminus E_\eps$ for all $n \geq N$.
\end{enumerate}
\end{lemma}

\begin{proof} {(i)} Assume $\langle y - x, v \rangle = p > 0$ for some $y \in \Pi_E(x)$. Then $\langle y - x, v_n \rangle \to p$ due to the continuity of the scalar product.
Hence there exists some $N \in \N$ for which $\langle y - x, v_n \rangle > \frac{1}{2}p$ and $\norm{x_n - x} < p$, whenever $n \geq N$. This implies
\[
\norm{x_n - x} < p < 2 \langle y - x, v_n \rangle
\]
for all $n \geq N$ so that
\begin{align*}
\norm{y - x_n}^2 &= \norm{(y - x) - (x_n - x)}^2 \\
&= \varepsilon^2 + \norm{x_n - x} \left( \norm{x_n - x} - 2 \langle y - x, v_n \rangle \right)
< \varepsilon^2.
\end{align*}

{(ii)} Let $\langle y - x, v \rangle < 0$ for all $y \in \Pi_E(x)$, and assume contrary to the claim that there exists a subsequence $(x_{n_k})_{k=1}^\infty \subset E_\eps$. Then, for each $k$ there exists some $y_k \in E$ for which $\norm{x_{n_k} - y_k} \leq \eps$. Since $E$ is compact, Lemma~\ref{Lemma_limit_of_contributing_points} implies the existence of some $y^* \in \Pi_E(x)$ for which $y_k \to y^*$ (if necessary, one can switch to a further convergent subsequence). By assumption $y^* \in \Pi_E(x)$ implies $\langle y^* - x, v \rangle < 0$. Let $q := |\langle y^* - x, v \rangle|$. Due to the continuity of the scalar product there exists some $K \in \N$, for which $\langle y_k - x_{n_k}, v_{n_k} \rangle \leq -q/2$ and $\norm{x_{n_k} - x} < q$, whenever $k \geq K$.
Then
\[
\norm{x_{n_k} - x}+ 2 \langle y_k - x_{n_k}, v_{n_k} \rangle < 0
\]
for all $k \geq K$, and applying the triangle-inequality with respect to the points $x_{n_k}$ yields
\begin{align*}
\norm{y_k - x}^2 
&\leq \varepsilon^2 + \norm{x_{n_k} - x} \left( \norm{x_{n_k} - x} + 2 \langle y_k - x_{n_k}, v_{n_k} \rangle \right) 
< \varepsilon^2.
\end{align*}
This implies the contradiction $x \in \textrm{int} (E_\eps)$.
\end{proof}


Lemma~\ref{Lemma_OD_versus_dot_product_ineq} below provides a partial characterisation of outward directions $\Xi_x(E_\eps)$ in terms of the contributors $\Pi_E(x)$. Geometrically, it implies that outward directions point away from the vectors $y - x$ for all $y \in \Pi_E(x)$, see Figure~\ref{Figure_Outward_Directions_and_Contributors}.

\begin{lemma}[{Orientation of outward directions relative to contributors}] \label{Lemma_OD_versus_dot_product_ineq}
Let $E \subset \R^d$ be compact, $x \in \partial E_\eps$ and $\xi \in S^{d-1}$. Then
\begin{enumerate}
\item[(i)] if $\xi \in \Xi_x(E_\eps)$, then $\langle y - x, \xi \rangle \leq 0$ for all $y \in \Pi_E(x)$,
\item[(ii)] if $\langle y - x, \xi \rangle < 0$ for all $y \in \Pi_E(x)$, then $\xi \in \Xi_x(E_\eps)$.
\end{enumerate}
\end{lemma}

\begin{proof}
{(i)} If $\xi \in \Xi_x(E_\eps)$, there exists a sequence $(x_n)_{n=1}^\infty \subset \R^d \setminus E_\eps$, for which $x_n \to x$ and
\[
\xi_n := \frac{x_n - x}{\norm{x_n - x}} \to \xi
\] as $n \to \infty$. Assume contrary to the claim tha there exists some $y \in \Pi_E(x)$ with $\langle y - x, \xi \rangle > 0$.
Substituting $v_n = \xi_n$ and $v = \xi$ in Lemma~\ref{Lemma_Basic_Properties_of_Sequences_Converging_onto_Boundary}(i) implies the existence of some $N \in \N$ for which $x_n \in \textrm{int}(E_\eps)$ for all $n \geq N$. This contradicts the claim.

{(ii)} Write $\xi_n = \frac{1}{n}\xi$, and define $x_n := x + \xi_n$, so that $(x_n - x) / \norm{x_n - x} = \xi$ for all $n \in \N$. Substituting $v := \xi$ and $v_n := \xi_n$ in Lemma~\ref{Lemma_Basic_Properties_of_Sequences_Converging_onto_Boundary}(ii) implies the existence of some $N \in \N$ for which $x_n \in \R^d \setminus E_\eps$ for all $n \geq N$. The sequence $(x_n)_{n = N}^\infty$ now defines the outward direction $\xi \in \Xi_x(E_\eps)$.
\end{proof}

Note that assuming the weaker condition $\langle y - x, \xi \rangle \leq 0$ for all contributors $y \in \Pi_E(x)$ is not sufficient in Lemma~\ref{Lemma_OD_versus_dot_product_ineq}(ii). For example, let $E := [2,3] \times \{0,1\}$ and consider the set $E_\eps$ with $\eps = 1/2$. Then $x := (3,1/2) \in \partial E_\eps$ with $\Pi(x) = \{(3,0), (3,1)\}$ and has only one outward direction $\xi = (1,0)$. Here also $\eta := (-1,0)$ satisfies $\langle y - x, \eta \rangle = 0$ for $y \in \{(3,0), (3,1)\}$, and yet $\eta \notin \Xi_{x} (E_\eps)$. This example illustrates the difference between $\partial E_\eps$ and the $\eps$-boundary $\partial E_{<\eps}$, since here $\partial E_{<\eps} \setminus \partial E_\eps = (2,3) \times \{1/2\} \subset \mathrm{int} \, E_\eps$. See also~\cite[Example 2.1]{Rataj_Winter_On_Volume}.

In order to describe the geometry of the sets of outward directions $\Xi_x(E_\eps)$ on the circle $S^1$, we introduce the concept of a geodesic arc-segment. Intuitively, a geodesic arc-segment is the shortest curve on $S^1$ that connects two points $v,w \in S^1$.

\begin{definition}[{Geodesic arc-segment}] \label{Def_geodesic_arc-segment}
Let $v, w \in S^1 \subset \R^2$ and let
\begin{equation} \label{Eq_Def_closed_geodesic_arc-segment}
[v, w]_{S^1} := \left\{u \in S^1 \, : \, u = a v + b w\, \textrm{ for some } a,b \geq 0\right\}.
\end{equation}
For $w \neq -v$, the set $[v, w]_{S^1}$ defines a \emph{geodesic arc-segment} between $v$ and $w$.
We also define the corresponding \emph{open} geodesic arc-segment $(v, w)_{S^1} \subset S^1$  as
\begin{equation} \label{Eq_Def_open_geodesic_arc-segment}
(v, w)_{S^1} := [v, w]_{S^1} \setminus \{v, w\}.
\end{equation}
\end{definition}
We use the notations $[v, w]_{S^1}$ and $(v, w)_{S^1}$ in accordance with~\eqref{Eq_Def_closed_geodesic_arc-segment} and~\eqref{Eq_Def_open_geodesic_arc-segment} also for the cases $v = w$ and $v = -w$, even though the corresponding sets in these cases are not arc-segments.

Unlike the previous results in this section, we formulate and prove the statements in Proposition~\ref{Prop_structure_of_set_of_outward_directions} and Lemma~\ref{Lemma_Convergence_of_Contributing_Points_of_Sequences_of_Boundary_Points} below only for the two-dimensional case. Note also that Proposition~\ref{Prop_structure_of_set_of_outward_directions} is formulated for a compact set $E \subset \R^2$, but essentially the same proof works for any closed set $E$ due to the local nature of the result.

\begin{prop}[{Structure of sets of outward directions}] \label{Prop_structure_of_set_of_outward_directions}
Let $E \subset \R^2$ be compact and $x \in \partial E_\eps$. Then the set of outward directions $\Xi_x(E_\eps)$ satisfies
\begin{enumerate}
    \item[(i)] if $x \in \Unp{E}$, then $\Xi_x(E_\eps) = \left\{\xi \in S^1 \, : \, \langle y - x, \xi \rangle \leq 0, \, \Pi_E(x) = \{y\} \right\}$,
    \item[(ii)] if $x \notin \Unp{E}$, then $\Xi_x(E_\eps) = [\xi_1, \xi_2]_{S^1}$, where $\xi_1, \xi_2$ are the \emph{only} extremal outward directions at $x$, possibly satisfying $\xi_1 = \xi_2$.
\end{enumerate}
\end{prop}

\begin{proof}
{(i)} Since $\Pi_E(x) = \{y\}$, Lemma~\ref{Lemma_OD_versus_dot_product_ineq}(ii) implies
\[
X := \{\xi \in S^1 \, : \, \langle y - x, \xi \rangle < 0 \} \subset \Xi_x(E_\eps).
\]
Then
\[
\partial_{S^1} X = \{\xi \in S^1 \, : \, \langle y - x, \xi \rangle = 0 \}
\]
due to continuity of the scalar product, and Lemmas~\ref{Prop_outward_exists_and_is_closed} and~\ref{Lemma_OD_versus_dot_product_ineq}(i) imply that $\overline{X} = \Xi_x(E_\eps)$, as claimed.

{(ii)} Assume then that $\Pi_E(x)$ contains at least two points. We assert that
\begin{enumerate}
\item[(a)]if $\xi, \eta \in \Xi_x(E_\eps)$ and $\gamma \in (\xi, \eta)_{S^1}$, then $\langle y - x, \gamma \rangle < 0$ for all $y \in \Pi_E(x)$,
\item[(b)]$\xi \in \textrm{int}_{S^1} \Xi_x(E_\eps)$ if and only if $\langle y - x, \xi \rangle < 0$ for all $y \in \Pi_E(x)$,
\item[(c)] $\Xi_x(E_\eps) = [\xi_1, \xi_2]_{S^1}$.
\end{enumerate}

{(a)} If $\Xi_x(E_\eps) = \{\xi\}$ or $\Xi_x(E_\eps) = \{\xi, -\xi\}$ for some $\xi \in S^1$, the claim is true since $(\xi, \xi)_{S^1} = (\xi, -\xi)_{S^1} = \varnothing$. Let then $\xi, \eta \in \Xi_x(E_\eps)$ with $\eta \notin \{\xi, -\xi\}$ and define a parametrised curve $\gamma: [0,1] \to S^1$ by
\[
\gamma(t) := \frac{t\eta + (1-t)\xi}{\norm{t\eta + (1-t)\xi}}.
\]
Clearly, $\gamma(0) = \xi, \gamma(1) = \eta$, and $\gamma((0,1)) = (\xi, \eta)_{S^1}$. For each $y \in \Pi_E(x)$, consider the scalar product
\begin{align}
P_y(t) := \langle y - x, t\eta + (1-t)\xi \rangle &= t \langle y - x, \eta \rangle + (1 - t) \langle y - x, \xi \rangle \label{P(t)_line_1} \\
 &= \langle y - x, \xi \rangle + t\langle y - x, \eta - \xi \rangle. \label{P(t)_line_2} 
\end{align}
We will show that $P_y(t) < 0$ for every $y \in \Pi_E(x)$ and all $t \in (0,1)$. We have $P_y(t) \leq 0$ for all $t \in [0,1]$ and all $y \in \Pi_E(x)$, since Lemma~\ref{Lemma_OD_versus_dot_product_ineq}(i) guarantees
\[
\langle y - x, \eta \rangle \leq 0 \quad \textrm{and} \quad \langle y - x, \xi \rangle \leq 0
\]
for all $y \in \Pi_E(x)$.
For $t \in (0,1)$, equation~\eqref{P(t)_line_1} implies $P_y(t) < 0$ when $\langle y - x, \xi \rangle < 0$. On the other hand, if $\langle y - x, \xi \rangle = 0$, equation~\eqref{P(t)_line_2} implies
\[
t\langle y - x, \eta - \xi \rangle = P_y(t) \leq 0.
\]
Hence the inequality $P_y(t) < 0$ is satisfied if and only if $\eta \notin \{\xi, -\xi\}$. This shows that $P_y(t) < 0$ for arbitrary $y \in \Pi_E(x)$, whenever $t \in (0,1)$.

{(b)} If $\xi \in S^1$ satisfies $\langle y - x, \xi \rangle < 0$ for all $y \in \Pi_E(x)$, then there exists some $n \in \N$ for which $\langle y - x, \xi \rangle < -1/n$ for all $y \in \Pi_E(x)$.
To show this, assume to the contrary that for each $n \in \N$ there exists some $y_n \in \Pi_E(x)$ for which
\[
\langle y_n - x, \xi \rangle \geq -1/n.
\]
Since $S^1$ is compact and $E$ is closed, the sequence $(y_n)_{n=1}^{\infty}$ has a convergent subsequence $(y_{n_k})_{k=1}^{\infty}$ for which $y_{n_k} \to y^* \in \Pi_E(x)$ as $k \to \infty$. On the other hand, due to the continuity of the scalar product, we have
\[
\langle y^* - x, \xi \rangle = \lim_{k \to \infty} \langle y_{n_k} - x, \xi \rangle \geq 0,
\]
which contradicts the assumption that $\langle y - x, \xi \rangle < 0$ for all $y \in \Pi_E(x)$.

Assume now that $\langle y - x, \xi \rangle < 0$ for all $y \in \Pi_E(x)$ and that $n \in \N$ has been chosen so that $\langle y - x, \xi \rangle < -1/n$ for all $y \in \Pi_E(x)$. The continuity of the scalar product implies that for some $\delta > 0$, 
one has $\big\langle y - x, \widehat{\xi} \big\rangle < 0$ for all $\widehat{\xi}$ that satisfy $\big\Vert \widehat{\xi} - \xi \big \Vert < \delta$. Hence, $\xi$ has an open neighbourhood $B_\delta(\xi)$ satisfying $B_\delta(\xi) \cap S^1 \subset \Xi_x(E_\eps)$, which implies $\xi \in \textrm{int}_{S^1} \Xi_x(E_\eps)$.

For the other direction, assume $\xi \in \textrm{int}_{S^1} \Xi_x(E_\eps)$. Then there exist $\eta_1, \eta_2 \in \Xi_x(E_\eps)$, for which
\[
\xi \in (\eta_1, \eta_2)_{S^1} \subset \textrm{int}_{S^1} \Xi_x(E_\eps).
\]
Step (a) consequently implies $\langle y - x, \xi \rangle < 0$ for all $y \in \Pi_E(x)$.

{(c)} It follows from steps (a) and (b) that $\textrm{int}_{S^1} \Xi_x(E_\eps) = (\xi_1, \xi_2)_{S^1}$. This in turn implies $\Xext{x} = \{\xi_1, \xi_2\}$, when $\textrm{int}_{S^1} \Xi_x(E_\eps) \neq \varnothing$, and $\Xi_x(E_\eps) = \{\xi\}$ (singleton), when $\xi_1 = \xi_2$.
\end{proof}

Proposition~\ref{Prop_structure_of_set_of_outward_directions} thus gives the following geometric picture of the set of extremal outward directions: in the case of a sharp singularity (type S2) or a chain singularity (type S6), the set of extremal outward directions is a singleton $\Xext{x} = \{\xi\}$ for some $\xi \in S^1$, see Figure~\ref{Figure_Types_of_Singularities} and Definition~\ref{Def_classification_of_singularities}. 
Otherwise $\Xext{x}$ contains two points, which may point directly away from each other or form an acute or obtuse angle.

Lemma~\ref{Lemma_Convergence_of_Contributing_Points_of_Sequences_of_Boundary_Points} below summarises the limiting behaviour of outward directions $\xi_n$ and contributors $y_n$ of points $x_n$ that appear in convergent sequences on the $\eps$-neighbourhood boundary. In particular, Lemma~\ref{Lemma_Convergence_of_Contributing_Points_of_Sequences_of_Boundary_Points}(ii)(a) establishes that for each $x \in \partial E_\eps$ the set of tangent vectors $T_x(E_\eps)$ is a subset of the set $\Xext{x}$ of extremal outward directions. We show in Proposition~\ref{Prop_Tangents_are_Defined} below that these sets in fact coincide for all $x \in \partial E_\eps$.

\begin{lemma}[{Orientation in converging sequences of boundary points}] \label{Lemma_Convergence_of_Contributing_Points_of_Sequences_of_Boundary_Points}
Let $E \subset \R^2$ be compact and let $x \in \partial E_\eps$. Furthermore, let $(x_n)_{n=1}^\infty$ be a sequence on $\partial E_\eps$ with $x_n \to x$ and define $\xi_n := (x_n - x)/\norm{x_n - x}$ for all $n \in \N$. Then
\begin{enumerate}
\item[(i)] the sequence $(\xi_n)_{n=1}^\infty$ can be split into two disjoint, convergent subsequences $(\xi_{1,k})_{k=1}^\infty$ and $(\xi_{2,k})_{k=1}^\infty$. For $i \in \{1,2\}$, the limit $\xi^{(i)} := \lim_{k \to \infty} \xi_{i,k}$ satisfies $\xi^{(i)} \in \Xext{x}$,
\item[(ii)] if the limit $\xi := \lim_{n\to\infty} \xi_n \in S^1$ exists, then
 \begin{enumerate}
 \item[(a)] every sequence $(y_n)_{n=1}^\infty$ in $E$ with $y_n \in \Pi(x_n)$ for all $n \in \N$ has a convergent subsequence $(y_{n_k})_{k=1}^\infty$ for which $y := \lim_{k\to\infty} y_{n_k} \in \Cext{x}$.
 Furthermore $\langle y - x, \xi \rangle = 0$ and consequently $\xi \in \Xext{x}$.
  \item[(b)] every sequence $(\eta_n)_{n=1}^\infty$ in $S^1$ with $\eta_n \in \Xext{x_n}$ for all $n \in \N$ satisfies
\[
\lim_{n \to \infty} \norm{\langle \eta_n, \xi \rangle} = 1.
\]
  \end{enumerate}
\end{enumerate}
\end{lemma}

\begin{proof}
We proceed by first proving the statements (ii)(a) and (ii)(b), and conclude with the proof of (i).

{(ii)(a)} Due to Lemma~\ref{Lemma_limit_of_contributing_points}, there exists some $y \in \Pi_E(x)$ and a convergent subsequence $(y_{n_k})_{k=1}^\infty \subset (y_n)_{n=1}^\infty$, for which $y_{n_k} \to y$. We break the proof into three steps.

\emph{Step 1.} $\langle y - x, \xi \rangle \leq 0$:
Assume contrary to the claim that $\langle y - x, \xi \rangle > 0$. Then substituting $x_k := x_{n_k}$ and $v := \xi$ in Lemma~\ref{Lemma_Basic_Properties_of_Sequences_Converging_onto_Boundary}(i) implies that for some $K \in \N$ we have $x_{n_k} \in \textrm{int} (E_\eps)$ for all $k \geq K$. This contradicts the assumption $x_n \in \partial E_\eps$ for all $n \in \N$.

\emph{Step 2.} $\langle y - x, \xi \rangle \geq 0$:
Assume contrary to the claim that $\langle y - x, \xi \rangle < 0$. The continuity of the scalar product then implies that there exists some $K \in \N$ for which
\[
\norm{x_{n_k} - x} + 2 \langle y_{n_k} - x_{n_k}, \xi_{n_k} \rangle < 0
\]
for all $k \geq K$. Applying the triangle-inequality with respect to the points $x_{n_k}$ yields
\begin{align*}
\norm{y_{n_k} - x}^2 
&= \varepsilon^2 + \norm{x_{n_k} - x} \left( \norm{x_{n_k} - x} + 2 \langle y_{n_k} - x_{n_k}, \xi_{n_k} \rangle \right)
< \varepsilon^2
\end{align*}
for all $k \geq K$. This implies the contradiction $x \in \textrm{int} (E_\eps)$.

\emph{Step 3.} $\xi \in \Xext{x}$:
For each $n \in \N$, write $r_n := \norm{x_n - x}$. Since $(x_n)_{n=1}^\infty \subset \partial E_\eps$ there exists for each $n \in \N$ some $z_n \in B_{r_n^2}(x_n) \setminus E_\eps$. Then $\xi_n^z := (z_n - x)/\norm{z_n - x} \to \xi$ so that $\xi \in \Xi_x(E_\eps)$. Steps 1 and 2 together imply $\langle y - x, \xi \rangle = 0$ so that $\xi \in \Xext{x}$, see Definition~\ref{Def_Extremal_Outward_Directions_and_Contributors}. This concludes the proof of (ii)(a).

{(ii)(b)} Assume contrary to the claim that there exists some $\delta > 0$ and a subsequence $(\eta_{n_k})_{k=1}^\infty$, for which $\norm{\langle \eta_{n_k}, \xi \rangle} < 1 - \delta$ for all $k \in \N$. This implies that if $y_k \in \Cext{x_{n_k}}$ with $\langle y_k - x_{n_k}, \eta_{n_k} \rangle = 0$, then there exists some $r > 0$, depending on $\delta$, for which
\begin{equation} \label{Eq_far_away_on_one_side_1}
\norm{\langle y_k - x_{n_k}, \xi \rangle} \geq r
\end{equation}
for infinitely many $k \in \N$. On the other hand, (ii)(a) implies the existence of a subsequence $(y_{k_j})_{j=1}^\infty$ for which the limit $y := \lim_{j\to\infty} y_{k_j} \in \Cext{x}$ exists and satisfies $\langle y - x, \xi \rangle = 0$. Inequality~\eqref{Eq_far_away_on_one_side_1} now leads to the contradiction
\begin{equation*}
\norm{\langle y - x, \xi \rangle} = \lim_{j \to \infty} \big \Vert \big\langle y_{k_j} - x_{n_{k_j}}, \xi \big\rangle \big \Vert \geq r > 0.
\end{equation*}

{(i)} According to (ii)(a), every convergent subsequence $(\xi_{n_k})_{k=1}^\infty$ satisfies $\xi_{n_k} \to \xi \in \Xext{x}$ as $k \to \infty$. For $\Xext{x} = \{\xi\}$ (a singleton) this implies $\xi_n \to \xi$ and the claim follows. In case $\Xext{x} = \left\{\xi^{(1)}, \xi^{(2)}\right\}$ for some $\xi^{(1)} \neq \xi^{(2)}$, write $r = \big\Vert \xi^{(1)}, \xi^{(2)} \big\Vert$. The compactness of $\partial E_\eps$ implies that there exists some $N \in \N$, for which
\[
\xi_n \in B_{r/3} \big(\xi^{(1)}\big) \cup B_{r/3} \big(\xi^{(2)}\big)
\]
for all $n \geq N$. For each $i \in \{1,2\}$, define $N_i := \left\{n \in \N \, : \, \xi_n \in B_{r/3} \left(\xi^{(i)}\right)\right\}$.
If $N_i$ is finite for some $i \in \{1,2\}$, we have $\lim_{n\to\infty} \xi_n = \xi^{(j)}$ for $j \in \{1,2\} \setminus \{i\}$ and the claim follows. Otherwise the sequences $(\xi_{1,k})_{k \in N_1}$ and $(\xi_{2,k})_{k \in \N \setminus N_1}$ are disjoint and satisfy $\lim_{k \to \infty} \xi_{i,k} = \xi^{(i)}$ for $i \in \{1,2\}$.
\end{proof}

The following result, combined with Proposition~\ref{Prop_structure_of_set_of_outward_directions}, implies that one-directional tangent vectors exist at every point on the boundary $\partial E_\eps$.

\begin{prop}[{Extremal outward directions coincide with tangents}] \label{Prop_Tangents_are_Defined}
Let $E \subset \R^2$ be compact and let $x \in \partial E_\eps$. Then $T_x(E_\eps) = \Xext{x}$.
\end{prop}

\begin{proof}
Lemma~\ref{Lemma_Convergence_of_Contributing_Points_of_Sequences_of_Boundary_Points}(ii)(a) implies $T_x(E_\eps) \subset \Xext{x}$, so we are left with proving the other direction.
Assume $\xi \in \Xext{x}$. Then there exists a sequence $(x_n)_{n=1}^\infty \subset \R^2 \setminus E_\eps$, for which
\[
\frac{\varphi_n}{\norm{\varphi_n}} := \frac{x_n - x}{\norm{x_n - x}} \longrightarrow \xi,
\]
as $n \to \infty$. Since $\xi$ is an extremal outward direction, there exists some extremal contributor $y \in \Cext{x}$ for which $\langle y - x, \xi \rangle = 0$. Let $\widehat{y} := (y - x)/\norm{y - x} = (y - x)/\eps$ and define $H_n := h_n \xi + r_n \widehat{y}$, where
\[
h_n := \norm{\varphi_n} \sqrt{1 - \frac{\norm{\varphi_n}^2}{4\eps^2}}, \quad \textrm{and} \quad r_n := \frac{\norm{\varphi_n}^2}{2\eps}.
\]
It follows from the orthogonality of $\xi$ and $\widehat{y}$ that $\norm{H_n} = \norm{\varphi_n}$.
Furthermore $x + H_n \in E_\eps$ for all $n \in \N$, since $\norm{\left(x + H_n\right) - y} = \eps$. See figure \ref{Figure_Tangents}.

Consider now the $\norm{\varphi_n}$-radius circle $\partial B_{\norm{\varphi_n}}(x)$ centered at $x$. The geodesic arc-segment (shortest path) on this circle that connects the points $x + H_n \in E_\eps$ and $x + \varphi_n = x_n \in \R^2 \setminus E_\eps$ must necessarily contain a boundary point $z_n \in \partial B_{\norm{\varphi_n}}(x) \cap \partial E_\eps$.

\begin{figure}
      \centering
      \captionsetup{margin=0.75cm}
                \includegraphics[width = \textwidth]{Outward_is_tangent_combined_v01.eps}
                \caption{Construction of the sequence $(z_n)_{n=1}^\infty$. The singularity $x$ is here depicted as a wedge, but the procedure is the same for other types of boundary points. ({\footnotesize A}) 
                For each $n \in \N$, the variables $h_n, r_n$ satisfy $r_n = 
                \eps - \sqrt{\eps^2 - h_n^2}$ and $\norm{H_n} = \norm{\varphi_n}$. ({\footnotesize B}) The point $z_n \in \partial E_\eps$ lies on a geodesic arc-segment on $\partial B_{\norm{\varphi_n}}(x)$ that connects the points $x_n$ and $x + H_n$.}
                \label{Figure_Tangents}
\end{figure}

Let $\delta > 0$. Since $x_n \to x$ and $\varphi_n/\norm{\varphi_n} \to \xi$, there exists some $N \in \N$ for which
\begin{equation} \label{ineq_varphi_estimates}
\norm{\varphi_n} \leq \frac{\delta}{3} \quad \textrm{and} \quad \norm{\frac{\varphi_n}{\norm{\varphi_n}} - \xi} \leq \frac{\delta}{3}
\end{equation}
whenever $n \geq N$. Since
\begin{align*}
\norm{\frac{H_n}{\norm{H_n}} - \xi}^2 
					&= 2 - \sqrt{4 - \frac{\norm{\varphi_n}^2}{\eps^2}} \longrightarrow 0
\end{align*}
as $\varphi_n \to 0$, we can choose some $N^* \geq N$ for which the inequality $\norm{H_n / \norm{H_n} - \xi} \leq \delta / 3$ as well as the estimates~\eqref{ineq_varphi_estimates} hold for all $n \geq N^*$. It follows from the definition of the points $z_n$ that $\norm{z_n - (x + \varphi_n)} \leq \norm{\varphi_n - H_n}$ for all $n \in \N$. This allows us to obtain the estimate
\begin{align*}
\norm{\frac{z_n - x}{\norm{z_n - x}} - \xi} 
				&\leq \norm{\frac{H_n}{\norm{H_n}} - \xi} + 2\norm{\frac{\varphi_n}{\norm{\varphi_n}} - \xi}
				\leq \delta,
\end{align*}
which is valid for all $n \geq N^*$. Since also $0 < \norm{z_n - x} = \norm{\varphi_n} < \delta$ for all $n \geq N^*$, we see that $\xi$ fulfils the requirements of Definition~\ref{Def_one-sided-tangent}, so that $\xi \in T_x(E_\eps)$.
\end{proof}

\chapter{Local Structure of the Boundary} \label{Sec_Local_Structure}
In this chapter we analyse the local properties of the boundary $\partial E_\eps$ using the results obtained in Chapter~\ref{Sec_Epsilon_Neighbourhoods} regarding outward directions and contributors.

We begin by proving a local contribution property, Proposition~\ref{Prop_local_contribution}. Intuitively, this result shows that the local geometry of the boundary of the $\eps$-neighbourhood $\partial E_\eps$ near a boundary point $x \in \partial E_\eps$ only depends on the geometry of the underlying boundary $\partial E$ around the extremal contributors $y \in \Cext{x}$.

In Section~\ref{Section_Finite_Approximation}, we approximate the $\eps$-neighbourhood $E_\eps$ with finite collections of balls $\{B_\eps(d_n) \, : \, d_n \in D^n \}$ that correspond to certain finite subsets $D^n \subset E$. Combining this idea with Proposition~\ref{Prop_local_contribution}, we show in Proposition~\ref{Prop_local_representation_exists} that it is possible to represent the boundary $\partial E_\eps$ locally in terms of finite collections of curves, which in turn can be expressed as graphs of continuous functions on a compact interval.

As the first application of Proposition~\ref{Prop_local_representation_exists} we show in Lemma~\ref{Lemma_Unique_Connected_Component} that for every $x \in \Unp{E}$ and every wedge (singularity type S1)
there exists a unique connected component $V$ of the complement  $\R^2 \setminus E_\eps$ for which $x \in \partial V$.

\begin{figure}
      \centering
      \captionsetup{margin=0.75cm}
                \includegraphics[width = \textwidth]{LC_at_Wedge_v01.eps}
                \caption{Local contribution at a wedge $x \in \partial E_\eps$. ({\footnotesize A}) The $\eps$-boundaries $\partial B_\eps(y_1)$ and $\partial B_\eps(y_2)$ around the extremal contributors $y_1, y_2$ approximate the local geometry of the boundary $\partial E_\eps$ inside the ball $B_r(x)$. ({\footnotesize B}) For an exact representation, one needs to consider all the contributors within some radius $\delta > 0$ from $y_1$ and $y_2$.}
                \label{Figure_LBR_idea}
\end{figure}

\section{Local Contribution} \label{Sec_Local_Contribution}
Intuitively, one can give a crude approximation for the boundary around each $x \in \partial E_\eps$ by considering the boundaries $\partial B_\eps (y)$ centered at the contributors $y \in \Pi_E(x)$, and zooming in on a suitably small neighbourhood $B_r(x)$, in which the local boundary $\partial E_\eps \cap B_r(x)$ is well approximated by the set
\begin{equation*} \label{Eq_Boundary_Approximation_Intuition}
\partial \Bigg( \bigcup_{y \in \Pi_E(x)} B_\eps(y) \Bigg) \cap B_r(x).
\end{equation*}
However, inside any ball $B_r(x)$, the geometry of the $\eps$-neighbourhood $E_\eps$ is not defined solely by the geometry of the $\eps$-balls around the contributors $y \in \Pi_E(x)$, see Figures~\ref{Figure_LBR_idea} and~\ref{fig:Local_Contribution}. Hence, one needs to consider at least all the contributors in some neighbourhood of the set of extremal contributors $\Cext{x}$. Proposition~\ref{Prop_local_contribution} below confirms that this is indeed sufficient.

We introduce here the following notation for open $x$-centered half-balls, oriented in the direction of some $v \in S^1$:
\begin{equation} \label{Def_oriented_half_ball}
U_r(x, v) := \{ z \in B_r(x) \, : \, \langle z - x, v \rangle > 0\}.
\end{equation}

\begin{prop}[{Local contribution}] \label{Prop_local_contribution}
Let $E \subset \R^2$ and let $x \in \partial E_\eps$ with $\Xext{x} = \{\xi_1, \xi_2\}$, where we allow $\xi_1 = \xi_2$. Then for all $\delta > 0$ there exists some $r > 0$ such that given $z \in B_r(x)$, we have $z \in E_\eps$ if and only if either
\begin{align}
z &\notin  \overline{U_r(x, \xi_1)} \cup \overline{U_r(x, \xi_2)} , \mathrm{ or} \label{Eq_condition_1} \\
z &\in \overline{B_\eps\big(E \cap B_\delta(\Cext{x})\big)}. \label{Eq_condition_2}
\end{align}
\end{prop}

\begin{proof}
Assume to the contrary that there exists some $\delta > 0$, for which the claim fails. This means that for all $r > 0$ there exists some $z \in B_r(x)$, for which either
\begin{enumerate}
\item $z \in E_\eps$ and both~\eqref{Eq_condition_1} and~\eqref{Eq_condition_2} fail, or
\item $z \notin E_\eps$ and one of the conditions~\eqref{Eq_condition_1} or~\eqref{Eq_condition_2} holds true.
\end{enumerate}
This implies that there exists a sequence $(z_k)_{k=1}^\infty$ with $z_k \in B_{1/k}(x)$, for which either condition (1) or (2) holds true for $z = z_k$ for infinitely many indices $k \in \N$. Some subsequence $(z_n)_{n=1}^\infty$ of $(z_k)_{k=1}^\infty$ then necessarily satisfies $z_n \to x$ as $n \to \infty$ and either
\begin{enumerate}
\item[(a)] $z_n \in E_\eps$ for all $n \in \N$, and the conditions~\eqref{Eq_condition_1} and~\eqref{Eq_condition_2} both fail for each $z = z_n$, or
\item[(b)] $z_n \notin E_\eps$ for all $n \in \N$, and either condition~\eqref{Eq_condition_1} or condition~\eqref{Eq_condition_2} is satisfied for each $z = z_n$.
\end{enumerate}
We proceed by showing that both of these statements lead to a contradiction. In both cases one can assume, without loss of generality, the existence of the limit
\[
v_z := \lim_{n \to \infty} (z_{n} - x) / \norm{z_{n} - x}.
\]

\begin{figure}
      \centering
      \captionsetup{margin=0.75cm}
                \includegraphics[width = \textwidth]{Local_Boundary_Ball_v01.eps}
                \caption{Schematic illustration of Proposition~\ref{Prop_local_contribution}. For a sufficiently small $r>0$, the balls $B_\delta(y_1)$ and $B_\delta(y_2)$ contain all the contributors $y \in \Pi_E(z)$ (red dotted lines) of those boundary points $z \in \partial E_\eps$ that lie inside the ball $B_r(x)$ (red solid lines). For $\rho > r$ the boundary segment generated by the point $y^* \in E$ would lie inside the larger ball $B_\rho(x)$ and would need to be accounted for separately. Note that the geometry of the boundary $\partial E_\eps$ inside $B_r(x)$ is not affected by whether or not the points $y$ on the blue dotted line satisfy $y \in E$.}
                \label{fig:Local_Contribution}
\end{figure}

{(i)} Assume first that (a) holds true. This means that
\begin{align} \label{Eq_counter_assumption}
z_n & \in E_\eps \cap \left( \overline{U_{1/n}(x, \xi_1)} \cup \overline{U_{1/n}(x, \xi_2)} \right), \\ \nonumber
z_n & \notin \overline{B_\eps\big(E \cap B_\delta(\Cext{x})\big)}
\end{align}
for all $n \in \N$. Since $z_n \in E_\eps$ for all $n \in \N$, one can choose a sequence $(y_n)_{n=1}^\infty \subset E$ with $z_n \in \overline{B_\eps(y_n)}$. In addition, since $z_n \to x$, Lemma~\ref{Lemma_limit_of_contributing_points} guarantees the existence of a convergent subsequence $(y_{n_k})_{k=1}^\infty$ with the limit $\widehat{y} := \lim_{k \to \infty} y_{n_k} \in \Pi_E(x)$. Note that it follows from~\eqref{Eq_counter_assumption} that $y_n \notin B_\delta(\Cext{x})$ for all $n \in \N$, which implies $\widehat{y} \notin \Cext{x}$, and consequently $\Pi_E(x) \setminus \Cext{x} \neq \varnothing$. This rules out the possibility $\xi_1 = -\xi_2$ for the extremal outward directions $\xi_1, \xi_2 \in \Xext{x}$, and therefore
\begin{equation} \label{Eq_two-sided_ruled_out}
\langle \xi_1, \xi_2 \rangle > -1.
\end{equation}
The relations~\eqref{Eq_counter_assumption} guarantee that $z_{n_k} \notin B_\eps(\Cext{x})$. Together with Lemma~\ref{Lemma_Basic_Properties_of_Sequences_Converging_onto_Boundary}(i), this implies $\langle y - x, v_z \rangle \leq 0$ for all extremal contributors $y \in \Cext{x}$. Combined with Proposition~\ref{Prop_structure_of_set_of_outward_directions} and the assumption that
\[
z_{n_k} \in \overline{U_{1/k}(x, \xi_1)} \cup \overline{U_{1/k}(x, \xi_2)}
\]
for all $k \in \N$, this further implies that $v_z \in [\xi_1, \xi_2]_{S^1} = \Xi_x(E_\eps)$, where $[\xi_1, \xi_2]_{S^1}$ is a geodesic arc-segment due to~\eqref{Eq_two-sided_ruled_out}. It follows that $\langle \widehat{y} - x, v_z \rangle < 0$, since $\widehat{y} \notin \Cext{x}$. But now a computation analogous to that presented in the proof of Lemma~\ref{Lemma_Basic_Properties_of_Sequences_Converging_onto_Boundary}(ii) leads to the contradiction $\norm{\widehat{y} - x} < \eps$.

{(ii)} Assume then that condition (b) is satisfied. Now $z_n \notin E_\eps$ for all $n \in \N$ so that $v_z \in \Xi_x(E_\eps)$. In addition, given that either~\eqref{Eq_condition_1} or~\eqref{Eq_condition_2} is satisfied for each $z_n$ and~\eqref{Eq_condition_2} implies $z_n \in E_\eps$, condition~\eqref{Eq_condition_1} necessarily holds true for all $z_n$.
Then $\xi_1 \neq -\xi_2$, since $\xi_1 = -\xi_2$ leads to the contradiction
\[
z_n \in B_{1/n}(x) \setminus \bigcup_{i \in \{1,2\}} \overline{U_{1/n}(x, \xi_i)} = B_{1/n}(x) \setminus B_{1/n}(x) = \varnothing.
\]
On the other hand, if $\xi_1 \neq -\xi_2$, Proposition~\ref{Prop_structure_of_set_of_outward_directions} states that $v_z \in \Xi_x(E_\eps)$ can be written as a convex combination $v_z = a\xi_1 + b\xi_2$. Note that at least one of the coefficients $a, b$ must be strictly positive, since $v_z \in S^1$. However, \eqref{Eq_condition_1} implies $\langle z_n - x, \xi_i \rangle < 0$ for $i \in \{1,2\}$ and all $n \in \N$, which leads to the contradiction $\langle v_z ,\xi_i \rangle \leq 0$ for $i \in \{1,2\}$.
\end{proof}

\section{Approximating the Boundary with Continuous Graphs} \label{Section_Finite_Approximation}
In order to study the properties of the boundary $\partial E_\eps$, we utilise a finite approximation scheme as a technical aid. The idea is to generate a non-decreasing sequence $(D^n)_{n=1}^\infty$ of finite subsets $D^n \subset E$ and use their $\eps$-boundaries $\partial B_\eps(D^n)$ to approximate the boundary $\partial E_\eps$ in small neighbourhoods around individual boundary points $x \in \partial E_\eps$.


Let $E \subset \R^2$ be compact, let $\eps > 0$ and let $n_0\in\N$ be sufficiently large so that $2^{-n_0} < \eps / 4$. Consider for natural numbers $n> n_0$ the partitions of $\R^2$ into squares
\begin{equation} \label{Eq_grid_sets}
  \cC^n := \left\{ C_{k,\ell}^n := \left[\frac{k}{2^n}, \frac{k+1}{2^n}\right] \times \left[\frac{\ell}{2^n},\frac{\ell+1}{2^n}\right] : k,\ell\in\Z \right\}.
\end{equation}
It follows from the axiom of choice 
that there exists a non-decreasing sequence $D^{n_0} \subset D^{n_0+1} \subset \dots$ of subsets of $\R^2$, in which the sets $D^n$ are defined for each $n \in \N$ by the property
  \begin{equation} \label{Eq_sets_D_n}
    D^n \cap C_{k,\ell}^n = \left\{\begin{array}{c@{\quad\text{if}\,\,}l}
      \big\{d_{k,\ell}^n\big\}, & E\cap C_{k,\ell}^n\not=\emptyset,\\
      \emptyset, & E\cap C_{k,\ell}^n=\emptyset,
    \end{array}\right.
  \end{equation}
  where the points $d_{k,\ell}^n$ are chosen arbitrarily from $E\cap C_{k,\ell}^n$.

\begin{figure}[h]
      \centering
      \captionsetup{margin=0.75cm}
                \includegraphics[width = \textwidth]{Finite_Approximating_Sets_Combined_v01.eps} \\[0mm]
               \caption{Schematic illustration of two consecutive finite approximating sets for the boundary $\partial E_\eps$ around a wedge $x$. ({\footnotesize A}) The boundaries $\partial B_\eps(d_{k, \ell}^n)$ for $d_{k,\ell}^n \in D^n$ (red dots) provide an approximation (red curve) for $\partial E_\eps$. ({\footnotesize B}) For $D^{n+1}$ the approximation improves, and the sets $\partial B_\eps(D^n) \cap B_r(x)$ converge to $\partial E_\eps \cap B_r(x)$ in Hausdroff distance as $n \to \infty$. 
                }
                \label{Figure_Finite_Approximating_sets}
\end{figure}

\begin{definition}[{Finite approximating sets}] \label{Def_Finite_Approximation_Scheme}
Let $E \subset \R^2$ be compact and let $(D^n)_{n=1}^\infty$ be a sequence of subsets $D^n \subset E$ as defined in~\eqref{Eq_grid_sets} and~\eqref{Eq_sets_D_n} above. The sets $D^n$ are then called \emph{finite approximating sets} for the set $E$.
\end{definition}

It is immediate from the definition of the sets $D^n$ that
$E_\eps = \overline{\bigcup_{n\ge n_0} B_\eps(D^n)}$. In fact, the sets $B_\eps(D^n)$ converge to $E_\eps$ in Hausdroff distance.\footnote{The \emph{Hausdorff distance} $\mathrm{dist}_H(X, Y)$ between two sets $X, Y \subset \R^d$ is defined as 
$\mathrm{dist}_H(X, Y) := \mathrm{inf} \{ \delta > 0 \, : \, X \subset Y_\delta \,\,\, \textrm{\emph{and}} \,\,\, Y \subset X_\delta \}$,
where $X_\delta := \overline{\bigcup_{x \in X} B_\delta(x)}$.}
However, one needs to exercise some caution because the boundaries $\partial B_\eps(D^n)$ do not in general converge to $\partial E_\eps$ in Hausdorff distance. This can be seen for instance by considering the $\eps$-neighbourhoods of finite approximating sets for the example given in Figure~\ref{Figure_Extremal_Interior} in Section~\ref{Sec_Motivation}. 

We now show that the sequence of boundaries $\partial B_\eps(D^n)$ 
can nevertheless be used to approximate the boundary $\partial E_\eps$ locally around individual boundary points $x \in \partial E_\eps$. 
%
%
%
%
The local contribution property, Proposition~\ref{Prop_local_contribution}, implies that in order to describe the boundary $\partial E_\eps$ near any $x \in \partial E_\eps$, one only needs to consider the boundaries of $\eps$-balls centered at contributors that lie close to the extremal contributors $y \in \Cext{x}$. To make this precise, we establish a local coordinate system that allows us to represent the boundary $\partial E_\eps$ locally as a union of function graphs.
\begin{definition}[{Extremal pairs}] \label{Def_Extremal_Pairs}
Let $E \subset \R^d$ be closed, let $x \in \partial E_\eps$ and denote by $\Xext{x}$ and  $\Cext{x}$ the sets of extremal outward directions and extremal contributors, respectively. The set of \emph{extremal pairs} at $x$ is the collection
\[
\Pext{x} := \big \{(\xi, y) \in \Xext{x} \times \Cext{x} \, : \, \langle x - y, \xi \rangle = 0 \big \}.
\]
\end{definition}
The idea of extremal pairs is that the vectors $\xi$ and $(x-y)/\eps$ serve as an orthonormal basis for a local coordinate system, aligned with the local boundary geometry.
Now, let $x \in \partial E_\eps$, 
%
and consider for each extremal pair $(\xi, y) \in \Pext{x}$ and $n \in \N$ the sets 
\begin{equation} \label{Eq_sets_D_n_xi_y}
E^{\xi,y} := E \cap U_{\eps/2}(y, \xi), \quad D_n^{\xi, y} := D^n \cap U_{\eps/2}(y, \xi),
\end{equation}
where $D^n$ is a finite approximating set for $E$ and $U_{\eps/2}(y, \xi)$ is a half-ball of radius $\eps/2$ centered at $y$, oriented towards $\xi$, see~\eqref{Def_oriented_half_ball}. 
Using the sets $D_n^{\xi, y}$, we may now define a sequence of functions $f_n^{\xi, y} : [0, \eps/2] \to \R$ by setting
\begin{equation} \label{Eq_definition_of_f_n_xi_y}
f_n^{\xi, y}(s) := \textrm{max} \left\{ t \in \R \, : \, \mathrm{dist}\left(x + s \xi + t\frac{x - y}{\eps},  D_n^{\xi, y} \right) \leq \eps \right\},
\end{equation}
where $\textrm{dist}\left(\cdot, D_n^{\xi, y}\right)$ denotes the Euclidean distance from the set $D_n^{\xi, y}$. Since $\overline{B_\eps\big(D_n^{\xi, y}\big)}$ is compact, the maximum in~\eqref{Eq_definition_of_f_n_xi_y} exists for all $s \in [0, \eps/2]$. Compactness furthermore implies that the functions $f_n^{\xi, y}$ are bounded for all $n \in \N$. In addition, it follows from the definition of the functions $f_n^{\xi, y}(s)$ that 
\begin{equation}
x + s \xi + f_n^{\xi, y}(s)(x - y)\eps^{-1} \in \partial \overline{B_\eps\big(D_n^{\xi, y}\big)}
\end{equation}
for all $s \in [0, \eps/2]$. Since for each $n \in \N$ the boundary $\partial \overline{B_\eps\big(D_n^{\xi, y} \big)}$ is composed of a finite collection of arc-segments, the corresponding functions $f_n^{\xi, y}$ are continuous.

\begin{lemma}[{Convergence of approximating functions}] \label{Lemma_Convergence_of_appr_functions}
Let $E$ be compact and let $\eps > 0$. Then, for each $x \in \partial E_\eps$
and each extremal pair $(\xi,y) \in \Pext{x}$,
the functions $f_n^{\xi,y}$ defined 
in~\eqref{Eq_definition_of_f_n_xi_y}, converge monotonically and uniformly to the continuous function $f^{\xi, y}: [0, \eps/2] \to \R$, given by
\begin{equation} \label{Eq_Def_f_xi_y}
f^{\xi,y}(s):= \mathrm{max} \left\{ t \in \R \, : \, \mathrm{dist}\left(x + s \xi + t\frac{x - y}{\eps},  E^{\xi, y} \right) \leq \eps \right\}
\end{equation}
as $n \to \infty$. Furthermore, 
\begin{equation} \label{Eq_g_xi_y_boundary}
\big\{ x + s \xi + f^{\xi, y}(s)\eps^{-1}(x - y) \, : \, s \in [0, \eps/2] \big \} = \partial \overline{B_\eps(E_{\xi,y})} \cap U_{\eps/2}(x, \xi),
\end{equation}
where $E_{\xi,y}$ is as in~\eqref{Eq_sets_D_n_xi_y}.
%
\end{lemma}

\begin{proof}
Let $x \in \partial E_\eps$ and $(\xi,y) \in \Pext{x}$. Since the sets $D_n^{\xi,y}$ form a non-decreasing sequence, it is clear that $f_{n+1}^{\xi,y}(s) \geq f_n^{\xi,y}(s)$ for all $s \in [0, \eps/2]$ and $n \in \N$. On the other hand, the value $f_n^{\xi, y}(s)$ is bounded from above by

Hence, the functions $f_n^{\xi,y}$ converge to a limiting function pointwise and monotonically on $[0, \eps/2]$. We now argue that the convergence is uniform and that $\lim_{n\to\infty} f_n^{\xi,y}(s) = f^{\xi,y}(s)$ for all $s \in [0, \eps/2]$.

Let $\delta > 0$ and consider a point $x_s := f^{\xi,y}(s)$ for some $s \in [0, \eps/2]$. Then $x_s \in \partial \overline{B_{\eps/2}(E^{\xi,y})}$ by~\eqref{Eq_Def_f_xi_y}. Let $y \in \Pext{x}$ be an extremal contributor of $x_s$. It follows from the definition of the finite approximating sets that there exists for each $n \in \N$ some $y_n \in D_n^{\xi,y}$ for which $\norm{y - y_n} \leq \sqrt{2}/n$.
Denote by $\textrm{proj}_y(x_s, \partial B_\eps(y_n))$ the projection of $x_s$ onto $\partial B_\eps(y_n)$ along the $x-y$ axis. Then, since the difference in the $\xi$-coordinate between $x_s$ and $y_n$ is less than $\eps/2$, we furthermore have
\begin{align*}
\big|f^{\xi, y}(s) - f_n^{\xi,y}(s)\big| &< \big|f^{\xi, y}(s) - \textrm{proj}_y(x_s, \partial B_\eps(y_n))\big| \\
&< \frac{2}{\sqrt{3}} \big(\norm{x_s - y_n} - \eps \big) \\
&\leq \frac{2}{\sqrt{3}} \norm{y - y_n}
\leq \frac{2\sqrt{2}}{\sqrt{3}n},
\end{align*}
so that $|f^{\xi, y}(s) - f_n^{\xi,y}(s)| < \delta$ whenever $n > 2\sqrt{2}/\sqrt{3}\delta$. Since this estimate is independent of $s$, it follows that $f_n^{\xi,y} \to f^{\xi,y}$ uniformly, as claimed. The last claim follows directly from the definition of $f^{\xi, y}$ in~\eqref{Eq_Def_f_xi_y}.
\end{proof}


We now apply Lemma~\ref{Lemma_Convergence_of_appr_functions} to obtain local representations for the boundary $\partial E_\eps$ around each boundary point $x \in \partial E_\eps$. The proof of Proposition~\ref{Prop_local_representation_exists} below makes use of Lemma~\ref{Lemma_Open_Cone_Touches_x} which we have placed in the subsequent Section~\ref{Section_Local_Structure_of_the_Complement}, together with other results on the geometry of the complement $\R^2 \setminus E_\eps$.

\begin{prop}[{Local boundary representation}] \label{Prop_local_representation_exists}
Let $E \subset \R^2$ be closed and let $x \in \partial E_\eps$. For each extremal pair $(\xi, y) \in \Pext{x}$ there exists a continuous function $f^{\xi, y}: [0, \eps/2] \to \R$ and a corresponding function $g_{\xi, y} : [0, \eps/2] \to \R^2$, given by
\begin{equation} \label{Eq_Canonical_LBR}
g_{\xi, y}(s) := x + s \xi + f^{\xi, y}(s)\frac{x - y}{\eps},
\end{equation}
so that the collection of functions $\cG(x) := \big\{ g_{\xi, y} \, : \, (\xi, y) \in \Pext{x} \big\}$ satisfies
\begin{equation} \label{Eq_local_boundary_repr_condition}
\partial E_\eps \cap \overline{B_r(x)} = \bigcup_{(\xi, y) \in \Pext{x}} g_{\xi, y}\left(A_{\xi,y}\right)
\end{equation}
for some $r > 0$ and some closed $A_{\xi, y} \subset [0, \eps/2]$. We call the collection $\cG(x)$ a \emph{local boundary representation (of radius $r$)} at $x$.  For each extremal pair $(\xi, y) \in \Pext{x}$ the corresponding subset $A_{\xi, y} \subset [0, \eps/2]$ is either
\begin{enumerate}
\item[(a)] an interval $[0, s_{\xi, y}]$ for some $0 < s_{\xi, y} \leq \eps/2$, or
\item[(b)] a closed set whose complement in $[0, \eps/2]$ contains a sequence of disjoint open intervals with $0$ as an accumulation point.
\end{enumerate}
For wedges (singularities of type S1) and $x \in \Unp{E}$, case (a) above holds true for all $(\xi, y) \in \Pext{x}$.
\end{prop}

\begin{figure}[h]
      \centering
      \captionsetup{margin=0.75cm}
                \includegraphics[width = \textwidth]{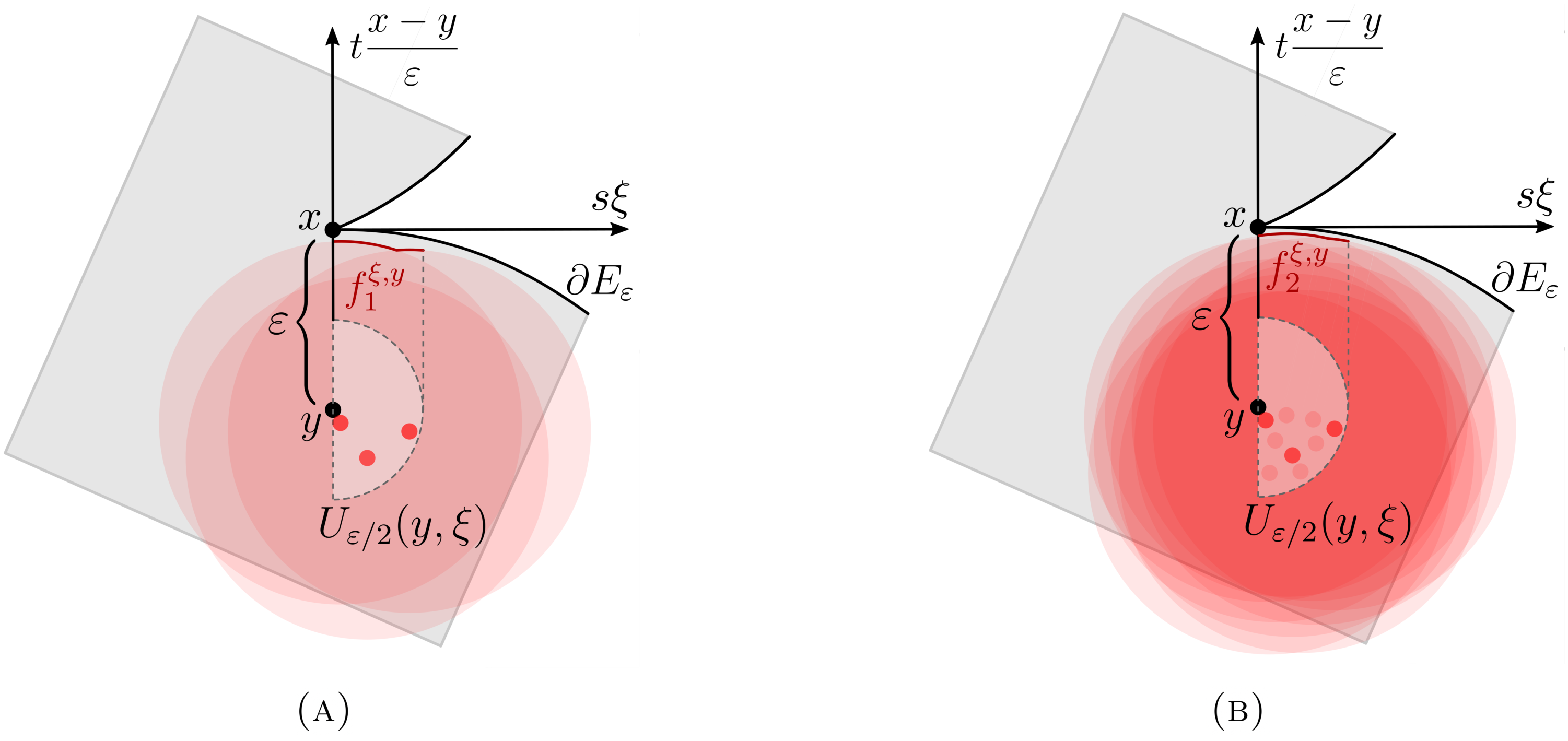} \\[0mm]
               \caption{Schematic illustration of the first functions in the sequence $\left(f_n^{\xi, y}\right)_{n=1}^\infty$. The red dots in ({\footnotesize A}) and ({\footnotesize B}) represent the points in two consecutive finite approximating sets $D_1^{\xi, y}$ and $D_2^{\xi,y}$, respectively.}
\end{figure}

\begin{proof}
Consider an extremal pair $(\xi, y) \in \Pext{x}$ and let $f^{\xi,y} : [0, \eps/2] \to \R$ be as in~\eqref{Eq_Def_f_xi_y}. Define $g_{\xi, y} : [0, \eps/2] \to \R^2$ by
\begin{equation} \label{Eq_definition_of_functions_g_xi_y}
g_{\xi, y}(s) := x + s \xi + f^{\xi, y}(s)\frac{x - y}{\eps}.
\end{equation}
%
%
%
%
%
%
%
Substituting $\delta = \eps / 2$ in Proposition~\ref{Prop_local_contribution} guarantees the existence of some $r > 0$, for which $z \in E_\eps \cap B_r(x)$ if and only if either
\[
\textrm{(i) } z \notin \overline{U_r(x, \xi)} \textrm{ for each } \xi \in \Xext{x}, \quad \textrm{or} \quad
\textrm{(ii) } z \in \overline{B_\eps\left(E \cap B_{\eps/2}(\Cext{x})\right)}.
\]
According to Proposition~\ref{Prop_Tangents_are_Defined}, the set of tangential directions $T_x(E_\eps)$
coincides with the set of extremal outward directions $\Xext{x}$. This implies that condition~(i) necessarily fails for all boundary points $z \in \partial E_\eps \cap B_r(x)$. Therefore, condition~(ii) holds true whenever $z \in \partial E_\eps \cap B_r(x)$, and in this case $z \in \partial \overline{B_\eps(E_{\xi, y})} \cap \overline{U_r(x, \xi)}$ for some $(\xi, y) \in \Pext{x}$. On the other hand, according to Lemma~\ref{Lemma_Convergence_of_appr_functions} we have
\begin{equation}
g_{\xi,y}([0, \eps/2]) = \partial \overline{B_\eps(E_{\xi,y})} \cap U_{\eps/2}(x, \xi),
\end{equation}
see~\eqref{Eq_g_xi_y_boundary}. Hence, there exists some $s_z \in [0, r]$, for which $z = g_{\xi, y}(s_z)$. In other words, every boundary point $z \in \partial E_\eps$ inside the neighbourhood $B_r(x)$ lies on a graph $g_{\xi, y}([0, r])$ for some $(\xi, y) \in \Pext{x}$.

It remains to be verified, for each $(\xi, y) \in \Pext{x}$, which arguments $s \in [0, r]$ satisfy $g_{\xi, y}(s) \in \partial E_\eps \cap B_r(x)$. We divide the rest of the proof into two cases, depending on whether or not there exist extremal contributors $y_1, y_2 \in \Cext{x}$ for which
\begin{equation} \label{Eq_opposing_contributors_1}
y_1 - x = -(y_2 - x).
\end{equation}
Condition~\eqref{Eq_opposing_contributors_1} is schematically illustrated in Figure~\ref{Figure_Three_Basic_Cases}({\footnotesize C}), in the beginning of Chapter~\ref{Sec_Singularities}.

{(i)} Assume~\eqref{Eq_opposing_contributors_1} is not satisfied by any $y_1, y_2 \in \Cext{x}$. Then there is either only one contributor so that $x \in \Unp{E}$, or else $x$ is a wedge, see Definition~\ref{Def_classification_of_singularities}. In both cases we have $\Pext{x} = \{(\xi_1, y_1), (\xi_2, y_2)\}$ for some $\xi_1, \xi_2 \in \Xi_{\textrm{ext}}(x)$ and $y_1, y_2 \in \Cext{x}$, where we allow for $y_1 = y_2$ in case $x \in \Unp{E}$. It follows from Lemma~\ref{Lemma_Open_Cone_Touches_x} below that for any outward directions $\eta_1, \eta_2 \in (\xi_1, \xi_2)_{S^1}$ there exists some radius $\rho$ and a corresponding neighbourhood $B_\rho(x) \subset B_r(x)$ in which the graphs of $g_{\xi_1, y_1}$ and $g_{\xi_2, y_2}$ are separated by the cone $V_\rho(x, \eta_1, \eta_2)$, see~\eqref{Eq_Def_Truncated_Cone}. Hence
\begin{equation} \label{Eq_graphs_dont_intersect}
g_{\xi_1, y_1}([0, r) \cap g_{\xi_2, y_2}([0, r]) \cap B_\rho(x) = \varnothing.
\end{equation}
For $i \in \{1,2\}$ we define the upper bounds
\[
s_{\xi_i, y_i} := \max \big\{s \in [0, r] \, : \, g_{\xi_i, y_i}\left([0,s]\right) \in \overline{B_{\rho}(x)} \big\}
\]
and the corresponding sets $A_{\xi_i, y_i} := \left[0, s_{\xi_i, y_i}\right]$. It follows then from~\eqref{Eq_graphs_dont_intersect} and Proposition~\ref{Prop_local_contribution} that $g_{\xi_i, y_i}(A_{\xi_i, y_i}) \subset \partial E_\eps$ for $i \in \{1,2\}$, which implies
\[
\partial E_\eps \cap \overline{B_\rho(x)} = \bigcup_{i \in \{1,2\}} g_{\xi_i, y_i}\left(A_{\xi_i,y_i}\right).
\]

{(ii)} Assume then that~\eqref{Eq_opposing_contributors_1} holds true for $y_1, y_2 \in \Cext{x}$
and consider some $\xi \in \Cext{x}$. In this case the graphs $g_{\xi, y_i}([0,s])$, $i \in \{1,2\}$ may generally intersect for arbitrarily small $s \in (0, r]$.
Due to~\eqref{Eq_opposing_contributors_1} one can write for $\{i,j\} = \{1,2\}$
\begin{align*}
g_{\xi, y_i}(s) &= x + s\xi + \left(f^{\xi, y_j}(s) - \alpha_\xi(s) \right) \frac{x - y_j}{\eps} \\
                    &= g_{\xi, y_j}(s) - \alpha_\xi(s) \frac{x - y_j}{\eps},
\end{align*}
where $\alpha_\xi(s) := f^{\xi, y_1}(s) + f^{\xi, y_2}(s)$. Hence it follows for $i \in \{1,2\}$ from Proposition~\ref{Prop_local_contribution} and the definitions of the functions $f_n^{\xi, y}$ and $f^{\xi, y}$ in~\eqref{Eq_definition_of_f_n_xi_y} and~\eqref{Eq_Def_f_xi_y} that
\begin{align*}
g_{\xi, y_i}(s) &\notin \partial E_\eps  \,\, \mathrm{whenever} \,\, \alpha_\xi(s) > 0, \, \mathrm{and} \\
g_{\xi, y_i}(s) &\in \partial E_\eps  \,\, \mathrm{whenever} \,\, \alpha_\xi(s) < 0.
\end{align*}
For $\alpha_\xi(s) = 0$ one has $g_{\xi, y_1}(s) = g_{\xi, y_2}(s) \in \partial E_\eps$ if and only if there exists a sequence $(s_n)_{n=1}^\infty$ with $s_n \to s$, for which $\alpha_\xi(s_n) < 0$ for all $n \in \N$.
This follows from the fact that
\[
\tau g_{\xi, y_1}(s) + (1 - \tau) g_{\xi, y_2}(s) \in \R^2 \setminus E_\eps
\]
whenever $s \in (0,r)$, $\alpha_\xi(s) < 0$ and $\tau \in (0,1)$. Equation~\eqref{Eq_local_boundary_repr_condition} is therefore satisfied for the neighbourhood $B_r(x)$ and the sets
\begin{align*}
A_{\xi, y_i} := A_\xi &:= \overline{\left\{s \in [0, r] \, : \, \alpha_\xi(s) < 0 \right\}} \\ \nonumber
  &\phantom{:}= \left\{s \in [0, r] \, : \, g_{\xi, y_1}(s), g_{\xi, y_2}(s) \in \partial E_\eps \right\},
\end{align*}
where $i \in \{1,2\}$ and $\xi \in \Xext{x}$.
Now either $[0, s_\xi] \subset A_\xi$ for some $s_{\xi} \in (0, r]$ 
or otherwise $[0, T] \setminus A_\xi \neq \varnothing$ for all $T > 0$. In either case $0$ is an accumulation point of $A_\xi$, since $x \in \partial E_\eps$ and $\xi \in T_x(E_\eps)$, see Proposition~\ref{Prop_Tangents_are_Defined}.
\end{proof}

Consider a boundary point $x \in \partial E_\eps$ and the corresponding local boundary representation $\cG(x)$ with radius $r>0$, and let $z \in \partial E_\eps \cap B_r(x)$ with
\[
z = g_{\xi, y}(s_z) = x + s_z\xi + f^{\xi,y}(s_z)\frac{x-y}{\eps}
\]
for some $(\xi, y) \in \Pext{x}$ and $s_z \in [0, r]$. The construction given in Proposition~\ref{Prop_local_representation_exists}
guarantees that, when written in the $(\xi, y)$-coordinates, the $\xi$-coordinate $s_w$ of each contributor $w \in \Pi_E(z)$ satisfies $0 \leq s_w \leq r \leq \eps/2$. This implies a lower and upper bound for the local growth-rate of the function $f^{\xi,y}$.
Combining this observation with Proposition~\ref{Prop_Tangents_are_Defined}
allows us to deduce that the functions $f^{\xi, y}$ in \eqref{Eq_Canonical_LBR} are in fact Lipschitz continuous on $[0, r]$ for all $(\xi, y) \in \Pext{x}$.

\begin{prop}[{Local boundary representation is Lipschitz}] \label{Prop_LBR_Lipschitz}
Let $E \subset \R^2$, let $x \in \partial E_\eps$ and let $\cG(x)$ be a local boundary representation at $x$ with radius $r>0$. Then, for each extremal pair $(\xi, y) \in \Pext{x}$, the function $f^{\xi, y}$ in \eqref{Eq_Canonical_LBR} is $1/\sqrt{3}$-Lipschitz, and the function $g_{\xi, y} \in \cG(x)$ is $2/\sqrt{3}$-Lipschitz on the interval $[0, r]$.
\end{prop}

\begin{proof}

\begin{figure}
  \begin{minipage}[c]{0.42\textwidth}
  \mbox{}\\[-\baselineskip] \vspace{2mm}
    \includegraphics[width=1.05\textwidth]{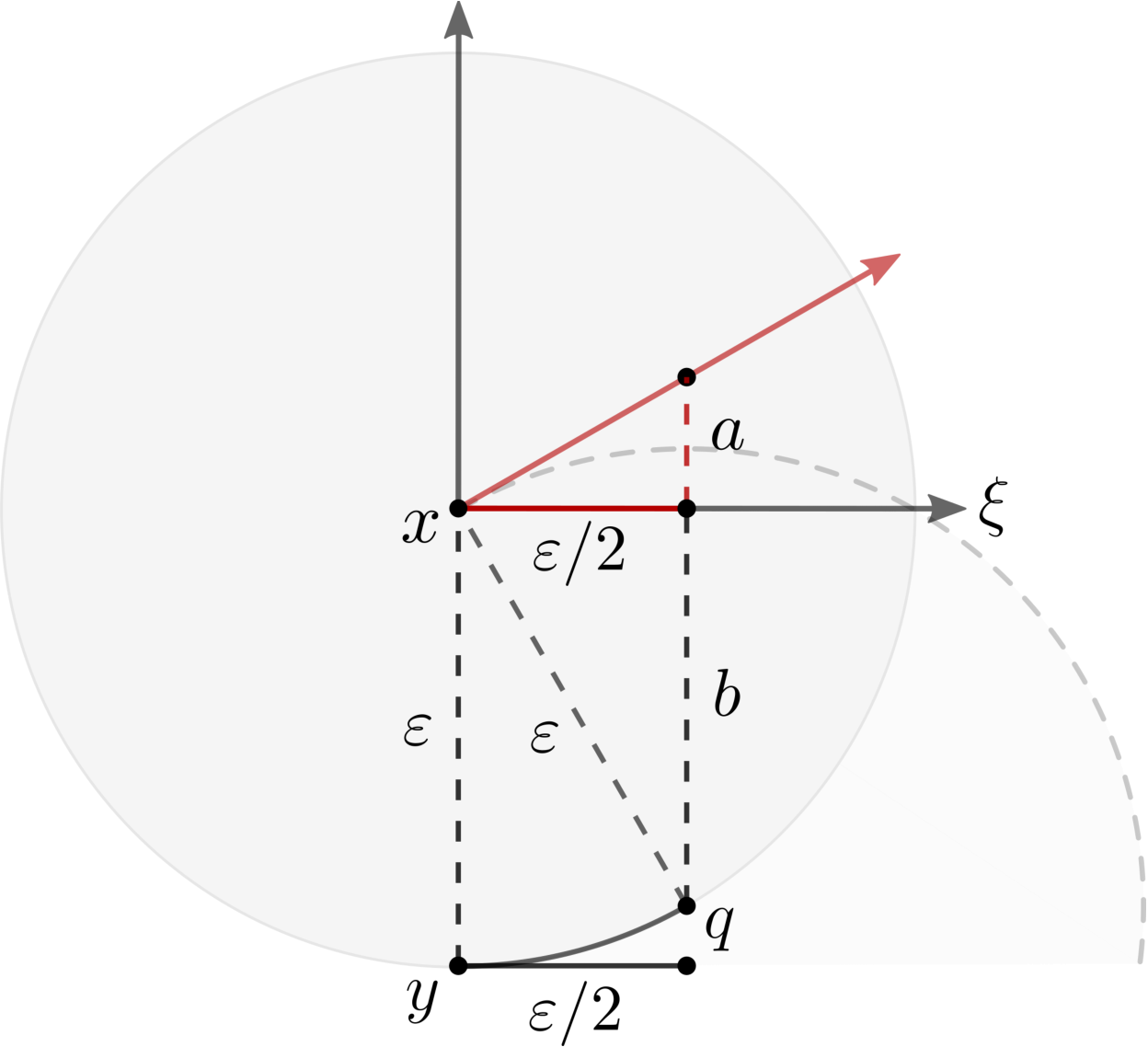}
  \end{minipage}
  \begin{minipage}[c]{0.57\textwidth}
  \mbox{}\\[-\baselineskip]  
    \caption{For all $s \in (0, r) \subset (0, \eps/2)$, the slopes $D^\pm f^{\xi, y}(s)$ are bounded from above by the ratio $D_\mathrm{max} = a/(\eps/2) = (\eps/2)/b = 1/\sqrt{3}$.
    Similarly, a lower bound is given by $D^\pm f^{\xi, y}(s) \geq D_\mathrm{min} = - 1/\sqrt{3}$. \vspace{2mm}}
  \label{Figure_Maximal_Tangent}
  \end{minipage}
\end{figure}

Assume contrary to the claim that for some $(\xi, y) \in \Pext{x}$ there exist some $p > 1/\sqrt{3}$ and $s, w \in [0, r]$ for which $|f^ {\xi, y}(s) - f^ {\xi, y}(w)| \geq p|s - w|$.  We present the argument for the case $w < s$ and $f^{\xi, y}(s) - f^ {\xi, y}(w) \geq p(s - w)$. Assuming $f^ {\xi, y}(s) - f^ {\xi, y}(w) \leq -p(s - w)$ leads to a contradiction through similar reasoning.
Write $m := (s + w)/2$ and define
  \begin{displaymath}
    (w_1, s_1) := \left\{\begin{array}{c@{,\quad\text{if}\,\,\,}l}
      (m,  s) & f^ {\xi, y}(s) - f^ {\xi, y}(m) > p(s - m)\\
      (w, m) & f^ {\xi, y}(s) - f^ {\xi, y}(m) \leq p(s - m)
    \end{array}\right.
  \end{displaymath}
so that $f^ {\xi, y}(s_1) - f^ {\xi, y}(w_1) \geq p(s_1 - w_1)$. For $n \in \N$ we define inductively $m_n := (s_n + w_n)/2$ and
  \begin{displaymath}
    (w_{n+1}, s_{n+1}) := \left\{\begin{array}{c@{,\quad\text{if}\,\,\,}l}
      (m_n,  s_n) & f^ {\xi, y}(s_n) - f^ {\xi, y}(m_n) > p(s_n - m_n)\\
      (w_n, m_n) & f^ {\xi, y}(s_n) - f^ {\xi, y}(m_n) \leq p(s_n - m_n)
    \end{array}\right.
  \end{displaymath}
so that for all $n \in \N$
\begin{equation} \label{Eq_inequality_1}
f^ {\xi, y}(s_n) - f^ {\xi, y}(w_n) \geq p(s_n - w_n).
\end{equation}
By construction, the sequence $(s_n)_{n=1}^\infty$ is non-increasing, and the sequence $(w_n)_{n=1}^\infty$ non-decreasing. Since $|s_n - w_n| = 2^{-n}|s-w| \to 0$ as $n \to \infty$ and $s_n > w_n$ for all $n \in \N$, there exists a unique limit $a = \lim_{n\to\infty} s_n = \lim_{n\to\infty} w_n$. From~\eqref{Eq_inequality_1} it follows that
\begin{equation} \label{Eq_estimate_1}
p \leq D_n := k_s(n) \frac{f^ {\xi, y}(s_n) - f^ {\xi, y}(a)}{s_n - a} + k_r(n) \frac{f^ {\xi, y}(a) - f^ {\xi, y}(w_n)}{a - w_n},
\end{equation}
where the coefficients
\[
k_s(n) := \frac{s_n - a}{s_n - w_n}, \quad k_w(n) := \frac{a - w_n}{s_n - w_n}
\]
satisfy $k_s(n) + k_w(n) = 1$ for all $n \in \N$. In case $s_N = a$ (resp.\ $w_N = a$) for some $N \in \N$, the coefficient $k_s(n)$ (resp.\ $k_w(n)$) vanishes for all $n \geq N$. Hence at least one and possibly both of the right and left derivatives of $f^ {\xi,y}$ at $a$ are given by
\begin{align*}
D^+f^ {\xi,y}(a) &:= \lim_{n \to \infty} \frac{f^ {\xi, y}(s_n) - f^ {\xi, y}(a)}{s_n - a}, \\ \nonumber 
D^-f^ {\xi,y}(a) &:= \lim_{n \to \infty} \frac{f^ {\xi, y}(a) - f^ {\xi, y}(w_n)}{a - w_n}
\end{align*}
and correspond to the tangential directions on $\partial E_\eps$ at
\[
x(a) := x + a\xi + f^ {\xi, y}(a)\frac{x - y}{\eps}.
\]
According to Proposition~\ref{Prop_Tangents_are_Defined} these in turn coincide with the extremal outward directions $\xi_a^+$ and $\xi_a^-$ at $x(a)$. Note that the $\xi$-coordinates of the corresponding extremal contributors $y_a^+, y_a^-$ lie on the interval $[0, \eps/2]$, see Figure~\ref{Figure_Maximal_Tangent}. Thus, the directional derivatives necessarily satisfy the upper bound
\[
D^\pm f^ {\xi,y}(a) \leq \frac{\frac{\eps}{2}}{\sqrt{\eps^2 - \big(\frac{\eps}{2}\big)^2}} = \frac{1}{\sqrt{3}}.
\]
The contradiction $p \leq \lim_{n\to\infty} D_n = 1/\sqrt{3} < p$ then follows by taking the limit $n \to \infty$ in~\eqref{Eq_estimate_1}.

It follows from the above reasoning that for all $x \in \partial E_\eps$ and each $(\xi, y) \in \Pext{x}$, the functions $g_{\xi, y} \in \cG(x)$ are $2/\sqrt{3}$-Lipschitz, since for any $w,s \in [0, r]$
\begin{align*}
\Vert g_{\xi, y}(s) - g_{\xi,y}(w) \Vert &=  \sqrt{|s - w|^2 + |f^{\xi, y}(s) - f^{\xi, y}(w)|^2 } \leq \frac{2}{\sqrt{3}} |s - w|.  \qedhere
\end{align*}
\end{proof}

Since a Lipschitz-function is differentiable almost everywhere, we obtain the following corollary.

\begin{cor}[{Local boundary representation is differentiable almost everywhere}] \label{Prop_LBR_Differentiable_AE}
Let $E \subset \R^2$, let $x \in \partial E_\eps$ and let $\cG(x)$ be a local boundary representation at $x$ with radius $r>0$. For each extremal pair $(\xi, y) \in \Pext{x}$, the function $f^{\xi, y}$ in \eqref{Eq_Canonical_LBR} is differentiable almost everywhere.
\end{cor}

\section{Local Structure of the Complement} \label{Section_Local_Structure_of_the_Complement}
In this section we analyse the connectedness of the complement $\R^2 \setminus E_\eps$ near points $x \in \Unp{E}$ and wedges. For these points, Lemma~\ref{Lemma_Unique_Connected_Component} guarantees the existence of a unique connected component $V \subset \R^2 \setminus E_\eps$ for which $x \in \partial V$, while Proposition \ref{Prop_regular_and_wedge_complement_cone} makes the stronger statement that in fact $B_r(x) \setminus E_\eps = V \cap B_r(x)$ for some connected set $V \subset \R^2 \setminus E_\eps$ and a neighbourhood $B_r(x)$.

\begin{lemma}[{Approximations of the outward cone}] \label{Lemma_Open_Cone_Touches_x}
Let $x \in \partial E_\eps$, let $\xi_1, \xi_2 \in \textrm{\emph{int}}_{S^1}\Xi_x(E_\eps)$ and define for each $r > 0$ the truncated cone
\begin{equation} \label{Eq_Def_Truncated_Cone}
V_r(x, \xi_1, \xi_2) := \left\{x + s v \, : \, v \in (\xi_1, \xi_2)_{S^1}, \, 0 < s < r \right\}. 
\end{equation}
Then there exists $r > 0$ for which
$V_r(x, \xi_1, \xi_2) \subset \R^2 \setminus E_\eps$.
\end{lemma}

\begin{proof}
Assume to the contrary that for each $n \in \N$ there exists $z_n \in E_\eps \cap V_{1/n}(x, \xi_1, \xi_2)$. Then $z_n \to x$ as $n \to \infty$ and we may assume without loss of generality that
\[ 
\frac{z_n - x}{\norm{z_n - x}} \to \xi \in [\xi_1, \xi_2]_{S^1}.
\]
Since $\xi_1, \xi_2 \in \textrm{int}_{S^1}\Xi_x(E_\eps)$, Proposition~\ref{Prop_structure_of_set_of_outward_directions} implies $\xi \in \textrm{int}_{S^1}\Xi_x(E_\eps)$, which together with Lemma~\ref{Lemma_OD_versus_dot_product_ineq}(i) gives $\langle y - x, \xi \rangle < 0$ for all $y \in \Pi_E(x)$. It follows then from Lemma~\ref{Lemma_Basic_Properties_of_Sequences_Converging_onto_Boundary}(ii) that there exists $N \in \N$, for which $z_n \notin E_\eps$ for all $n \geq N$, which contradicts the assumption.
\end{proof}

\begin{lemma}[{Unique connected component}] \label{Lemma_Unique_Connected_Component}
Let $E \subset \R^2$ and let $x \in \partial E_\eps$ either be a wedge (type S1) or $x \in \Unp{E}$. Then there exists a unique connected component $V \subset \R^2 \setminus E_\eps$, for which $x \in \partial V$.
\end{lemma}

\begin{proof}
Since $\textrm{int}_{S^1} \Xi (x) \neq \varnothing$ whenever $x$ is a wedge or $x \in \Unp{E}$, there exist outward directions $\xi_1, \xi_2 \in \textrm{int}_{S^1} \Xi_x(E_\eps)$. Lemma~\ref{Lemma_Open_Cone_Touches_x} then guarantees the existence of some $r > 0$ for which the truncated open cone
\begin{equation} \label{Eq_Open_r_cone}
V_r(x, \xi_1, \xi_2) := \left\{x + s v \, : \, v \in (\xi_1, \xi_2)_{S^1}, \, 0 < s < r \right\}
\end{equation}
satisfies $V_r(x, \xi_1, \xi_2) \subset \R^2 \setminus E_\eps$. Consequently, there exists a connected component $V$ of the complement $\R^2 \setminus E_\eps$ that satisfies $V_r(x, \xi_1, \xi_2) \subset V$ and $x \in \partial V$.

For each $(\xi, y) \in \Pext{x}$, let $f^{\xi, y} : [0, r] \to \R$ be the continuous function corresponding to the local boundary representation $\cG(x)$ with radius $0 < r < \eps/2$, see~\eqref{Eq_Canonical_LBR}. 
To prove uniqueness, assume contrary to the claim that there exists another connected component $W \subset \R^2 \setminus E_\eps$ with $W \neq V$ and $x \in \partial W$. Then for at least one extremal outward direction $\xi \in \Xext{x}$ there exists 
a sequence $\big(t_n^W\big)_{n=0}^\infty$ of coefficients and some $N \in \N$ for which
\[
w(n) := x + n^{-1} \xi + t_n^W \frac{x-y}{\eps} \in W
\]
and $t_n^W > f^{\xi,y}(n^{-1})$ for all $n > N$. Note that the outward directions $\xi_1, \xi_2$ in the definition of the cone $V_r(x, \xi_1, \xi_2)$ above may be chosen arbitrarily close to the extremal outward directions. Hence, it follows that there exist some outward directions $\xi_1', \xi_2'$, a corresponding radius $r' > 0$ and $N' \in \N$, such that for all $n > N'$ there exist coefficients $t_n^V > f^{\xi,y}(n^{-1})$ satisfying
\begin{equation} \label{Eq_points_v(n)}
v(n) := x + n^{-1} \xi + t_n^V \frac{x-y}{\eps} \in V_{r'}(x, \xi_1', \xi_2').
\end{equation}
Since $V_{r'}(x, \xi_1', \xi_2') \cap V_r(x, \xi_1, \xi_2) \neq \varnothing$, \eqref{Eq_points_v(n)} implies $v(n) \in V$ for all $n > N'$. 
It follows then from the definition of the local boundary representation that
\[
z(n, \alpha) := x + n^{-1} \xi + \big(\alpha t_n^W + (1 - \alpha) t_n^V \big) \frac{x-y}{\eps} \notin E_\eps
\]
for all $\alpha \in (0, 1)$. But then the line segment from $z(n, 0)$ to $z(n, 1)$ connects the sets $W$ and $V$ in $\R^2 \setminus E_\eps$, which contradicts the assumption that $V$ and $W$ are disjoint connected components of the complement $\R^2 \setminus E_\eps$.
\end{proof}

\begin{prop}[{Geometry of the complement}] \label{Prop_regular_and_wedge_complement_cone}
Let $E \subset \R^2$ and let $x \in \partial E_\eps$ either be a wedge or $x \in \Unp{E}$. Then there exists some $r > 0$, for which
\begin{equation} \label{Eq_Cone_repr_nro_1}
B_r(x) \setminus E_\eps = V \cap B_r(x) = \bigcup_{0 < \rho < r} x + A(\rho),
\end{equation}
where $V \subset \R^2 \setminus E_\eps$ is connected and 
for each $\rho \in (0,r)$ either $A(\rho) = \rho(\alpha_\rho, \beta_\rho)_{S^1}$ or $A(\rho) = \rho \big(S^1 \setminus [\alpha_\rho, \beta_\rho]_{S^1}\big)$ for $\alpha_\rho, \beta_\rho \in S^1$ 
and $\alpha_\rho \to \xi_\alpha$, $\beta_\rho \to \xi_\beta$, where $\Xi_{\textrm{ext}}(x) = \{\xi_\alpha, \xi_\beta\}$.
\end{prop}

\begin{proof}
Let $\cG(x)$ be a local boundary representation at $x$ with radius $r > 0$, see Proposition~\ref{Prop_local_representation_exists}. Then each $g_{\xi,y} \in \cG(x)$ is of the form
\[
g_{\xi, y}(s) = x + s\xi + f^{\xi, y}(s)\frac{x-y}{\eps}
\]
where the functions $f^{\xi, y} : [0, r] \to \R$ are continuous. We divide the proof into two parts according to whether (i) $x \in \Unp{E}$ or (ii) $x$ is a wedge.

{(i)} Assume first that $x \in \Unp{E}$ with $\Pi_E(x) = \{y\}$. Then $\cG(x) = \left\{g_{\xi_1, y}, g_{\xi_2, y}\right\}$ where $g_{\xi_i, y}  \, : \, [0,r] \to \R^2$ and
\[
g_{\xi_i, y}(s) = x + s\xi_i + f^{\xi_i, y}(s)\frac{x-y}{\eps}
\]
 for $i \in \{1,2\}$.
Consider for each $s \in (0,r)$ the distances
\[
\cD_1(s) := \norm{g_{\xi_1, y}(s) - x}, \quad \cD_2(s) := \norm{g_{\xi_2, y}(s) - x}.
\]
Due to Proposition~\ref{Prop_Tangents_are_Defined} and Lemma~\ref{Lemma_Convergence_of_Contributing_Points_of_Sequences_of_Boundary_Points}(i) we can assume $r$ to be small enough so that both $\cD_1$ and $\cD_2$ are strictly increasing on $(0, r)$.
One can thus define for each $\rho \in (0, r)$ the points $\alpha_\rho, \beta_\rho \in S^1$ by
\begin{align*}
\alpha_\rho := \frac{g_{\xi_1, y}\left(\cD_1^{-1}(\rho)\right) - x }{\rho}, \qquad
\beta_\rho := \frac{g_{\xi_2, y}\left(\cD_2^{-1}(\rho)\right) - x}{\rho}.
\end{align*}
According to Lemma~\ref{Lemma_Unique_Connected_Component} there exists a unique connected component $V$ of the complement $\R^2 \setminus E_\eps$ for which $x \in \partial V$. Due to Proposition~\ref{Prop_local_contribution}, we can assume $r$ to be sufficiently small so that
\[
B_r(x) \cap E \subset \overline{B_\eps \left(E \cap B_{\eps/2}(y) \right)}.
\]
This implies that for each $\rho \in (0, r)$ the geodesic curve segment
\[
A(\rho) := \rho (\alpha_\rho, (x - y) / \eps)_{S^1} \cup \{\rho (x - y) / \eps\}\cup \rho ((x - y) / \eps, \beta_\rho)_{S^1} \subset \rho S^1
\]
satisfies
\[
x + A(\rho) \subset V \cap B_r(x) \quad \mathrm{and} \quad
x + \rho S^1 \setminus A(\rho) \subset \overline{B_\eps \left(E \cap B_{\eps/2} (y)\right)} \subset E_\eps.
\]
Hence,
\[
B_r(x) \setminus E_\eps = V \cap B_r(x) = \bigcup_{0 < \rho < r} x + A(\rho).
\]

{(ii)} Let then $x$ be a wedge. As above, one can assume that the distances 
\[
\cD_1(s) := \norm{g_{\xi_1, y_1}(s) - x}, \quad \cD_2(s) := \norm{g_{\xi_2, y_2}(s) - x}
\]
are strictly increasing in $(0, r)$. Furthermore, since $\langle \xi_1, \xi_2 \rangle > -1$, one may define the average $\xi_{\mathrm{av}} := (\xi_1 + \xi_2)/\norm{\xi_1 + \xi_2}$. Due to Lemma~\ref{Lemma_Convergence_of_Contributing_Points_of_Sequences_of_Boundary_Points}(i) and the fact that $\xi_{\mathrm{av}} \notin \Xext{x}$, we may assume $r$ to be small enough so that the boundary segments represented by the functions $g_{\xi_i, y_i}$ are separated by the line segment $\{ x + \rho \xi_{\mathrm{av}} \, : \, \rho \in (0, r) \}$ in the neighbourhood $B_r(x)$.

For each $\rho \in (0, r)$, we once again define the points $\alpha_\rho, \beta_\rho \in S^1$ by
\begin{align*}
\alpha_\rho := \frac{g_{\xi_1, y}\left(\cD_1^{-1}(\rho)\right) - x}{\rho} \qquad \textrm{and} \qquad
\beta_\rho := \frac{g_{\xi_2, y}\left(\cD_2^{-1}(\rho)\right) - x}{\rho}.
\end{align*}
Analogously to the reasoning in the case $x \in \Unp{E}$, Proposition~\ref{Prop_local_contribution} and Lemma~\ref{Lemma_Unique_Connected_Component} guarantee that for all $\rho \in (0, r)$ the geodesic curve segment
\[
\widehat{A}(\rho) := \rho (\alpha_\rho, \xi_{\mathrm{av}})_{S^1} \cup \{\rho \xi_{\mathrm{av}}\}\cup \rho (\xi_{\mathrm{av}}, \beta_\rho)_{S^1} \subset \rho S^1
\]
satisfies
\[
x + \widehat{A}(\rho) \subset V \cap B_r(x) \quad \mathrm{and} \quad
x +  \rho S^1 \setminus \widehat{A}(\rho) \subset \overline{B_\eps \left(E \cap B_{\eps/2} (\{y_1,y_2\})\right)} \subset E_\epsilon,
\]
where $V$ is the unique connected component of the complement $\R^2 \setminus E_\eps$ for which $B_r(x) \setminus E_\eps = B_r(x) \cap V$. Hence
\[
B_r(x) \setminus E_\eps = V \cap B_r(x) = \bigcup_{0 < \rho < r} x + \widehat{A}(\rho). \qedhere
\]
\end{proof}
\chapter{Classification of Boundary Points} \label{Sec_Singularities}
In this chapter we present a classification of the boundary points $x \in \partial E_\eps$ based on their local geometric and topological properties. Using the results obtained in Chapters~\ref{Sec_Epsilon_Neighbourhoods} and~\ref{Sec_Local_Structure} above, we prove our first main result, Theorem~\ref{Thm_Main_1}, which states that the classification given in Definition~\ref{Def_classification_of_singularities} defines a partition of the boundary $\partial E_\eps$ into disjoint subsets.

The geometric aspect of the classification scheme relies on the orientation of the extremal contributors $y \in \Cext{x}$ at each boundary point $x \in \partial E_\eps$. In the planar case, there are essentially three different ways this orientation can be realised, depicted schematically in Figure~\ref{Figure_Three_Basic_Cases} below. The defining property $y_1 - x = -(y_2 - x)$ for the extremal contributors $y_1, y_2 \in \Cext{x}$ in case (c) can be equivalently expressed by $\big \langle (y_1 - x)/\eps, (y_2 - x)/\eps \big \rangle = -1$,
and we will make use of both formulations in what follows. Significantly, it follows from~\eqref{Eq_Def_Critical_Points_Ferry} and~\cite[Lemma 4.2]{Fu_Tubular_neighborhoods} that $x \in \partial E_\eps$ is a critical point of the distance function $d_E$ if and only if it is of type ({\footnotesize C}) in Figure~\ref{Figure_Three_Basic_Cases}.

\begin{figure}[h]
      \centering
      \captionsetup{margin=0.75cm}
                \includegraphics[width = \textwidth]{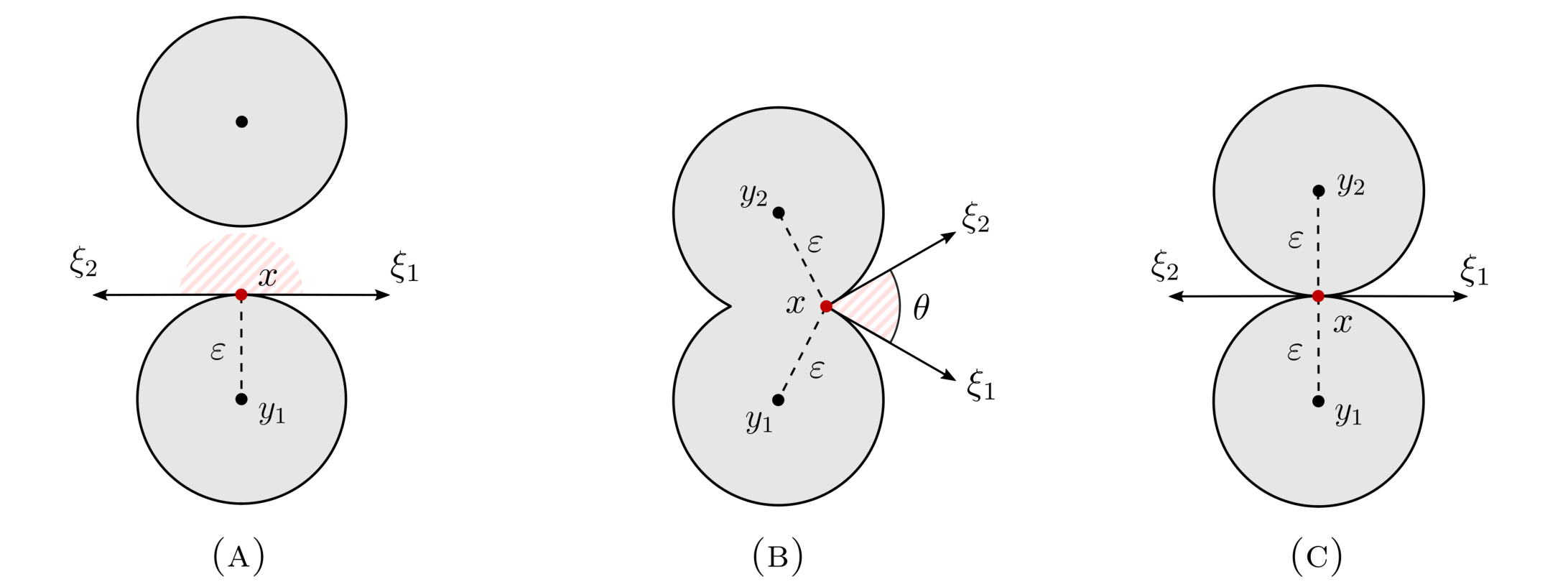} \\[0mm]
               \caption{Three basic scenarios for local boundary geometry in terms of the relative positions of the extremal contributors. In our classification of boundary points (Definition~\ref{Def_classification_of_singularities}), case ({\footnotesize A}) corresponds to smooth points and singularity types S4 and S5, case ({\footnotesize B}) to type S1, and case ({\footnotesize C}) to types S2, S3 and S6--S8. ({\footnotesize A}) At each $x \in \Unp{E}$ the extremal outward directions satisfy $\xi_1 = - \xi_2$, and the set $\Xi_x(E_\eps)$ spans a half-circle. ({\footnotesize B}) Each wedge $x$ has two extremal contributors $y_1, y_2$ and two extremal outward directions $\xi_1, \xi_2$, forming an angle $\theta = \sphericalangle(\xi_1, \xi_2)$. ({\footnotesize C}) For $\Cext{x} = \{y_1, y_2\}$ with $y_1 - x = - (y_2 - x)$, the set of extremal outward directions satisfies $\Xext{x} \subset \{\xi_1, \xi_2\}$.}
                \label{Figure_Three_Basic_Cases}
\end{figure}

\section{Types of Singularities} \label{Subsect_Types_of_Singularities}
Definition~\ref{Def_classification_of_singularities} below provides a classification of all possible boundary singularities into eight categories. For a schematic illustration of the different types of singularities, see Figure~\ref{Figure_Types_of_Singularities} in the Introduction. Recall that $U_r(x, v)$ denotes an open $x$-centered half-ball of radius $r$, oriented in the direction of $v \in S^1$, see~\eqref{Def_oriented_half_ball}. We denote by $\Sing$ the set of singularities on the boundary $\partial E_\eps$.

\begin{definition}[{Types of singularities}] \label{Def_classification_of_singularities}
Let $E \subset \R^2$ be closed, let $x \in \Sing$ and let $\Xext{x} = \{\xi_1, \xi_2\}$ be the set of extremal outward directions, where we allow for the possibility $\xi_1 = \xi_2$. We define the following eight types of singularities.

\begin{enumerate}
\item[\textbf{S1:}] $x$ is a \emph{wedge},
if $\xi_1 \notin \{\xi_2, -\xi_2\}$, i.e.~$0 < \theta < \pi$ for the angle $\theta$ between the vectors $\xi_1, \xi_2$;

\item[\textbf{S2:}] $x$ is a \emph{(one-sided) sharp singularity}, if $\xi_1 = \xi_2$, and there exists some $\delta > 0$ for which $B_\delta(x) \setminus E_\eps$ is a connected set;

\item[\textbf{S3:}] $x$ is a \emph{sharp-sharp singularity}, if $\xi_1 = -\xi_2$ and for each $i \in \{1,2\}$ there exists some $\delta_i > 0$ for which $U_{\delta_i}(x, \xi_i) \setminus E_\eps$ is a connected set;

\item[\textbf{S4:}] $x$ is a \emph{(one-sided) shallow singularity} if $x \in \Unp{E}$ and
\begin{enumerate}
\item[(i)]  $U_{\delta_1}(x, \xi_1) \cap \partial E_\eps \subset \Unp{E}$ for some $\delta_1 > 0$, and
\item[(ii)] $U_{\delta_2}(x, \xi_2) \cap \partial E_\eps \not \subset \Unp{E}$ for all $\delta_2 > 0$.
\end{enumerate}

\item[\textbf{S5:}] $x$ is a \emph{shallow-shallow singularity} if $x \in \Unp{E}$ and $U_\delta(x, \xi_i) \cap \partial E_\eps \not \subset \Unp{E}$ for all $\delta > 0$ and $i \in \{1,2\}$.

\item[\textbf{S6:}] $x$ is a \emph{(one-sided) chain singularity}, if $\xi_1 = \xi_2$ and
there exists a sequence of singularities $(x_n)_{n=1}^\infty \subset \Sing$,
for which $x_n \to x$ and
\[
\left\langle \frac{y_n^{(1)} - x_n}{\eps}, \frac{y_n^{(2)} - x_n}{\eps} \right \rangle \rightarrow -1,
\]
where $\big \{y_n^{(1)}, y_n^{(2)} \big \} = \Cext{x_n}$ is the set of extremal contributors at $x_n$;

\item[\textbf{S7:}] $x$ is a \emph{chain-chain singularity}, if $\xi_1 = -\xi_2$ and
for each $i \in \{1,2\}$ there exists some $\delta_i > 0$ and a sequence $(x_{i,n})_{n=1}^\infty \subset U_{\delta_i}(x, \xi_i) \cap \Sing$, for which $x_{i,n} \to x$ and
\[
\left\langle \frac{y_{i,n}^{(1)} - x_{i,n}}{\eps}, \frac{y_{i,n}^{(2)} - x_{i,n}}{\eps} \right \rangle \rightarrow -1,
\]
where $\big \{y_{i,n}^{(1)}, y_{i,n}^{(2)} \big \} = \Cext{x_{i,n}}$ is the set of extremal contributors at $x_{i,n}$;

\item[\textbf{S8:}] $x$ is a \emph{sharp-chain singularity}, if $\xi_1 = -\xi_2$ and
\begin{enumerate}
\item[(i)] there exists a $\delta_1 > 0$ for which $U_{\delta_1}(x, \xi_1) \setminus E_\eps$ is a connected set, and
\item[(ii)] there exists some $\delta_2 > 0$ and a sequence $(x_n)_{n=1}^\infty \subset U_{\delta_2}(x, \xi_2) \cap \Sing$, for which $x_n \to x$ and
\[
\left\langle \frac{y_n^{(1)} - x_n}{\eps}, \frac{y_n^{(2)} - x_n}{\eps} \right \rangle \rightarrow -1,
\]
where $\big \{y_n^{(1)}, y_n^{(2)} \big \} = \Cext{x_n}$ is the set of extremal contributors at $x_n$.
\end{enumerate}
\end{enumerate}
\end{definition}

Note that S8 may be interpreted both as a sharp singularity and as a chain singularity. Theorem~\ref{Thm_Main_3} below states that the set of \emph{chain singularities}
\[
\cC(\partial E_\eps) := \{x \in \partial E_\eps \, : \, x \textrm{ is of type S6--S8}\}
\]
is closed. On the other hand, the singularities of type S8 share an important property with those of types S1--S5: they all lie on the boundary $\partial V$ of some connected component $V$ of the complement $\R^2 \setminus E_\eps$. We show in Corollary~\ref{Cor_Inaccessible_Singularities} that this is exactly the property that is lacking from singularities of type S6 and S7. See also Remark~\ref{Remark_Inaccessible_Points}.

Motivated by these considerations we define a boundary point to be
\begin{itemize}
\item[(i)] a \emph{sharp singularity}, if it is of type S2, S3 or S8,
\item[(ii)] a \emph{chain singularity}, if it is of type S6, S7 or S8, and
\item[(iii)] an \emph{inaccessible singularity}, if it is of type S6 or S7.
\end{itemize}

The typology presented above is neither strictly topological nor strictly geometric. If one wanted to accomplish a strictly topological classification for neighbourhoods $\partial E_\eps \cap B_\delta(x)$ for some $\delta := \delta(x) > 0$, types S6--S8 would necessitate an infinite tree-like classification scheme, in order to account for the potentially accumulating chain and shallow singularities in arbitrarily small neighbourhoods $B_r(x)$ with $0 < r < \delta$, see Section \ref{Subsec_Topology_of_Chain} and Theorem \ref{Thm_Main_3}.

When the complement $\overline{\R^2 \setminus E_\eps}$ has positive reach, Definition~\ref{Def_classification_of_singularities} can be directly compared with~\cite[Definition 6.3]{Rataj_Zajicek_On_the_structure_of_sets_with_PR}.\footnote{We show in Chapter~\ref{Sec_Pos_Reach} that only the existence of certain kinds of chain singularities can cause the complement $\overline{\R^2 \setminus E_\eps}$ not to have positive reach.} The latter defines three distinct types of geometries $T^i$, for $i \in \{1,2,3\}$. Since these apply to the complement $\overline{\R^2 \setminus E_\eps}$, they represent 'mirror images' of the corresponding local geometry of $E_\eps$. This classification describes the geometry at a different level of detail compared to Definition~\ref{Def_classification_of_singularities}, which can be seen as a refinement of the former. Out of the types defined in~\cite[Definition 6.3]{Rataj_Zajicek_On_the_structure_of_sets_with_PR}, $T^1$ corresponds in our classification to smooth points, wedges (type S1) and shallow singularities (types S4 and S5), $T^2$ corresponds to sharp-sharp (type S3), chain-chain (type S7) and sharp-chain (type S8) singularities and $T^3$ corresponds to one-sided sharp (type S2) and chain (type S6) singularities.

\section{Classification of Singularities} \label{subsection_proof}
The main goal of this section is to prove that our classification of singularities in fact defines a partition on the boundary $\partial E_\eps$. 
We begin by characterising the topology and geometry of the complement $\R^2 \setminus E_\eps$ near those singularities $x \in \Sing$ whose extremal contributors $y_1,y_2 \in \Cext{x}$ satisfy $y_1 - x = -(y_2 - x)$. Geometrically these correspond to case ({\footnotesize C}) in Figure~\ref{Figure_Three_Basic_Cases}.

\begin{prop}[{Difference between sharp-type and chain-type geometry}] \label{Prop_Sharp_Singularity_Types}
Let $E \subset \R^2$, $x \in \partial E_\eps$ and $\Cext{x} = \{y_1, y_2\}$ with $y_1 - x = -(y_2 - x)$. Furthermore, let $\cG(x)$ be a local boundary representation at $x$ with radius $r > 0$, let $\xi \in \Xext{x}$ be an extremal outward direction, and let $g_{\xi, y_1}, g_{\xi, y_2} \in \cG(x)$ be as in~\eqref{Eq_Canonical_LBR}.
Then exactly one of the cases (i) and (ii) below holds true:
\begin{itemize}
  \item[(i)] \emph{(sharp-type)} There exists some $r > 0$, for which $g_{\xi, y_1}(s) \neq g_{\xi, y_2}(s)$ for all $s \in (0, r)$, and
  \begin{equation} \label{Eq_Directed_Cone_Repr}
    U_{r}(x, \xi) \setminus E_\eps = V_\xi \cap U_{r}(x, \xi) = \bigcup_{0 < s < r} x + s \big(\alpha(s), \beta(s) \big)_{S^1},
  \end{equation}
  where $V_\xi$ is the unique connected component of $\R^2 \setminus E_\eps$ intersecting $U_{r}(x, \xi)$, $\alpha(s), \beta(s) \in S^1$ for all $s \in (0, r)$ and $\alpha(s), \beta(s) \to \xi$ as $s \to 0$.
  \item[(ii)] \emph{(chain-type)} There exists a sequence $(s_n)_{n=1}^\infty \subset \R_+$ with the following properties:
  \begin{itemize}
  \item[(a)] $s_n \to 0$ and $g_{\xi, y_1}(s_n) = g_{\xi, y_2}(s_n)$ for all $n \in \N$. We denote this common value by $x_n$.
  \item[(b)]
      There exists some $r > 0$ and a sequence $(V_n)_{n=1}^\infty \subset U_r(x, \xi)$ of disjoint connected components of $\R^2 \setminus E_\eps$ with $\mathrm{dist}_H (x, V_n) \to 0$ as $n \to \infty$
      and $x_n \in \partial V_n$ for all $n \in \N$.
   \item[(c)]
     $x_n \in \Sing$ for each $n \in \N$, with
    \begin{equation} \label{Eq_Limits_of_contributor_vectors}
    \lim_{n \to \infty} \left \langle \frac{y_n^{(1)} - x_n}{\eps}, \frac{y_n^{(2)} - x_n}{\eps} \right \rangle = -1,
    \end{equation}
    where $\Cext{x_n} = \big\{y_n^{(1)}, y_n^{(2)}\big\}$ for all $n \in \N$.
  \end{itemize}
\end{itemize}
\end{prop}

\begin{proof}
Clearly, either there exists some $r > 0$, for which $g_{\xi, y_1}(s) \neq g_{\xi, y_2}(s)$ for all $s \in (0, r)$, or else there exists a sequence $(q_n)_{n=1}^\infty \subset \R_+$ with $q_n \to 0$, for which $g_{\xi, y_1}(q_n) = g_{\xi, y_2}(q_n)$ for all $n \in \N$, and these cases are mutually exclusive. The proof amounts to showing that in the former case, representation~\eqref{Eq_Directed_Cone_Repr} is valid for some connected component $V_\xi \subset \R^2 \setminus E_\eps$ and arc-segments $(\alpha(s), \beta(s))_{S^1}$, and in the latter, to identifying the prescribed sequences $(s_n)_{n=1}^\infty \subset \R_+$ and $(V_n)_{n=1}^\infty \subset U_{r}(x, \xi)$, as well as confirming the limit~\eqref{Eq_Limits_of_contributor_vectors} and that $x_n \in \partial V_n$ for all $n \in \N$.

Consider for $i \in \{1,2\}$ the continuous functions $f^{\xi, y_i}: [0, r] \to \R$ for which
\begin{align*}
g_{\xi, y_1}(s) &= x + s\xi + f^{\xi, y_1}(s)\frac{x - y_1}{\eps}, \\
g_{\xi, y_2}(s) &= x + s\xi + f^{\xi, y_2}(s)\frac{x - y_2}{\eps},
\end{align*}
see Proposition~\ref{Prop_local_representation_exists}.
The assumption $x - y_1 = -(x - y_2)$ implies that the vector 
representing the difference at $s \in (0, r)$ between the graphs $g_{\xi, y_1}([0, r])$ and $g_{\xi, y_2}([0, r])$, is given by
\begin{equation} \label{Eq_Difference_between_boundary_graphs}
g_{\xi, y_2}(s) - g_{\xi, y_1}(s) = - \left( f^{\xi, y_1}(s) + f^{\xi, y_2}(s) \right) \frac{x - y_1}{\eps}.
\end{equation}

{(i)} We start by assuming that there exists some $r > 0$, for which $g_{\xi, y_1}(s) \neq g_{\xi, y_2}(s)$ for all $s \in (0, r)$ which implies $\alpha_\xi(s) := f^{\xi, y_1}(s) + f^{\xi, y_2}(s) \neq 0$ for all $s \in (0, r)$. Due to continuity, this implies either
\begin{equation*}
\textrm{(1)} \quad \alpha_\xi(s) > 0 \,\, \textrm{for all} \,\, s \in (0, r), \qquad \textrm{or} \qquad
\textrm{(2)} \quad \alpha_\xi(s) < 0 \,\, \textrm{for all} \,\, s \in (0, r).
\end{equation*}
Note that (1) would contradict the assumption $\xi \in \Xext{x}$, so that (2) necessarily holds true. Furthermore, by the definition of the local boundary representation, and since $y_1 - x = -(y_2 - x)$, it follows that $z \in U_r(x, \xi) \cap \big( \R^2 \setminus E_\eps \big)$ if and only if $z = x + s\xi + t(x - y_1)/\eps$ where $s > 0$, $s^2 + t^2 < r^2$ and
\begin{equation} \label{Eq_complement_inequality_first}
f^{\xi, y_1}(s) < t < -f^{\xi, y_2}(s).
\end{equation}
Hence, the average $h_{\xi} : [0, r] \to \R^2$, given by
\begin{equation} \label{Eq_average_of_functions_g}
h_{\xi}(s) := \frac{g_{\xi, y_1}(s) + g_{\xi, y_2}(s)}{2} = x + s\xi + \left( f^{\xi, y_1}(s) - \frac{\alpha_\xi(s)}{2} \right) \frac{x - y_1}{\eps},
\end{equation}
satisfies $h_{\xi}(s) \in \R^2 \setminus E_\eps$ for all $s \in (0, r)$.
Although we have defined the function $h_\xi$ in~\eqref{Eq_average_of_functions_g} in terms of the contributor $y_1$, we could have equally well chosen $y_2$ due to symmetry. Noting that $h_\xi$ is continuous, that $h_\xi((0, r)) \subset \R^2 \setminus E_\eps$, and $x = h_\xi(0)$, it follows that there exists a connected component $V_\xi \subset \R^2 \setminus E_\eps$ for which $h_\xi((0, r)) \subset V_\xi$ and $x \in \partial V_\xi$. Emulating the reasoning presented in the proof of Lemma~\ref{Lemma_Unique_Connected_Component} we may confirm that $V_\xi$ is the only connected component of $\R^2 \setminus E_\eps$ that intersects $U_{r}(x, \xi)$.

To obtain the representation in~\eqref{Eq_Directed_Cone_Repr}, consider for each $s \in (0, r)$ the unit vectors $h(s), \alpha(s), \beta(s) \in S^1$, given by
\begin{equation} \label{Eq_Def_Unit_Vectors}
h(s) := \frac{h_\xi(s) - x}{\norm{h_\xi(s) - x}}, \quad \alpha(s) := \frac{g_{\xi, y_1}(s) - x}{\norm{g_{\xi, y_1}(s) - x}}, \quad \beta(s) := \frac{g_{\xi, y_2}(s) - x}{\norm{g_{\xi, y_2}(s) - x}}.
\end{equation}
Proposition~\ref{Prop_Tangents_are_Defined} and Lemma~\ref{Lemma_Convergence_of_Contributing_Points_of_Sequences_of_Boundary_Points}(i) imply that the tangential directions on the boundary $\partial E_\eps$ coincide with the extremal outward directions $\xi \in \Xext{x}$ at each boundary point $x \in \partial E_\eps$. We can hence assume $r$ to be small enough such that the distances
\[
\cD_h(s) := \norm{h_\xi(s) - x}, \quad \cD_\alpha(s) := \norm{g_{\xi, y_1}(s) - x}, \quad \cD_\beta(s) := \norm{g_{\xi, y_2}(s) - x}
\]
appearing in the divisors in~\eqref{Eq_Def_Unit_Vectors} are all strictly increasing functions of $s$ on the interval $(0, r)$. By definition, we furthermore have
\[
\mathrm{max} \, \big \{\cD_h^{-1}(s), \cD_\alpha^{-1}(s), \cD_\beta^{-1}(s) \big \} \leq s
\]
for all $s \in (0, r)$. It follows that for each $s \in (0, r)$
\begin{align*}
h_\xi\big(\cD_h^{-1}(s)\big) &= x + s h\big(\cD_h^{-1}(s)\big) \\
&\in x + s \big(\alpha(\cD_\alpha^{-1}(s)), \beta(\cD_\beta^{-1}(s)) \big)_{S^1} \subset V_\xi \cap U_{r}(x, \xi).
\end{align*}
In addition, we can assume Proposition~\ref{Prop_local_contribution} to apply 
with $\delta := \eps/2$. From this it follows for $C_s := s \left(S^1 \setminus \big(\alpha(s), \beta(s) \big)_{S^1} \right) \cap U_{r}(x, \xi)$ that
\[
x + C_s \subset \overline{B_\eps \left(E \cap B_{\eps/2} (\Cext{x})\right)} \subset E_\eps
\]
for all $s \in (0, r)$. Hence
\[
U_{r}(x, \xi) \setminus E_\eps = V_\xi \cap U_{r}(x, \xi) = \bigcup_{0 < s < r} x + s \big(\alpha(s), \beta(s) \big)_{S^1}.
\]

{(ii)} Assume then that there exists a sequence $(q_n)_{n=1}^\infty \subset \R_+$ for which $g_{\xi, y_1}(q_n) = g_{\xi, y_2}(q_n)$ for all $n \in \N$ and $q_n \to 0$. This situation corresponds to the chain-type geometry characteristic of chain singularities (types S6--S8; see Definition~\ref{Def_classification_of_singularities} and Figure~\ref{Figure_Types_of_Singularities}).
Note that since $x \in \partial E_\eps$ and $\xi \in \Xext{x}$, there exists for all $s > 0$ some $0 < \lambda < s$, for which
\[
\alpha_\xi(\lambda) = f^{\xi, y_1}(\lambda) + f^{\xi, y_2}(\lambda) < 0.
\]
One can thus define two new sequences $(s_n)_{n=1}^\infty \subset \R_+$ and $(p_n)_{n=1}^\infty \subset \R_+$ inductively as follows. First, choose some $\lambda_1 \in (0, q_1)$ with $f^{\xi, y_1}(\lambda_1) + f^{\xi, y_2}(\lambda_1) < 0$ and define
\begin{align*}
s_1 &:= \sup \left\{s \, : \, s > \lambda_1 \,\, \textrm{and} \,\, f^{\xi, y_1}(\lambda) + f^{\xi, y_2}(\lambda) < 0 \,\, \textrm{for all} \,\, \lambda \in [\lambda_1, s] \right\}, \\
p_1 &:= \inf \left\{s \, : \, s < \lambda_1 \,\, \textrm{and} \,\, f^{\xi, y_1}(\lambda) + f^{\xi, y_2}(\lambda) < 0 \,\, \textrm{for all} \,\, \lambda \in [s, \lambda_1] \right\}.
\end{align*}
For the induction step, assume we have already chosen the points $s_1, \ldots, s_{n-1}$ and $p_1, \ldots, p_{n-1}$ for some $n \in \N$. One can then choose some $\lambda_n \in (0, \textrm{min}\{p_{n-1}, q_n\})$ with $f^{\xi, y_1}(\lambda_n) + f^{\xi, y_2}(\lambda_n) < 0$, and define
\begin{align*}
s_n &:= \sup \left\{s \, : \, s > \lambda_n \,\, \textrm{and} \,\, f^{\xi, y_1}(\lambda) + f^{\xi, y_2}(\lambda) < 0 \,\, \textrm{for all} \,\, \lambda \in [\lambda_n, s] \right\}, \\
p_n &:= \inf \left\{s \, : \, s < \lambda_n \,\, \textrm{and} \,\, f^{\xi, y_1}(\lambda) + f^{\xi, y_2}(\lambda) < 0 \,\, \textrm{for all} \,\, \lambda \in [s, \lambda_n] \right\}.
\end{align*}
Then $p_n < s_n \leq p_{n-1} < s_{n-1}$ for all $n \in \N$ and $f^{\xi, y_1}(s) + f^{\xi, y_2}(s) = 0$ for all $s \in (s_n)_{n=1}^\infty \cup (p_n)_{n=1}^\infty$. Also, by definition, $f^{\xi, y_1}(s) + f^{\xi, y_2}(s) < 0$ for all $s \in (p_n, s_n)$ and $n \in \N$. This implies that for each $n \in \N$ the open set
\[
V_n := \left\{ \tau g_{\xi, y_1}(s) +  (1 - \tau) g_{\xi, y_2}(s) \, : \, \tau \in (0,1), \, s \in (p_n, s_n) \right\}
\]
is connected and satisfies $V_n \subset \R^2 \setminus E_\eps$ and $x_n := g_{\xi, y_1}(s_n) = g_{\xi, y_2}(s_n) \in \partial V_n$. In addition $V_n \cap V_m = \varnothing$ whenever $n \neq m$, and for every $r > 0$ there exists some $N \in \N$, for which $0 < p_n < s_n < r$ for all $n \geq N$. It thus follows from Propositions~\ref{Prop_Tangents_are_Defined} and Lemma~\ref{Lemma_Convergence_of_Contributing_Points_of_Sequences_of_Boundary_Points}(i) that $\mathrm{dist}_H (x, V_n) \to 0$ as $n \to \infty$.

Since $x_n := g_{\xi, y_1}(s_n) = g_{\xi, y_2}(s_n)$ for each $n \in \N$, there exist for $i \in \{1,2\}$ extremal contributors $y_n^{(i)} \in \Cext{x_n}$, for which $y_n^{(i)} \in B_{\eps/2}(y_i)$. This, together with Lemma~\ref{Lemma_Convergence_of_Contributing_Points_of_Sequences_of_Boundary_Points}(ii)(a), implies $y_n^{(i)} \to y_i$ for $i \in \{1,2\}$ as $n \to \infty$, and consequently
\[
  \lim_{n \to \infty} \left \langle \frac{y_n^{(1)} - x_n}{\eps}, \frac{y_n^{(2)} - x_n}{\eps} \right \rangle = -1. \qedhere
\]
\end{proof}

As a consequence of Proposition~\ref{Prop_Sharp_Singularity_Types} we obtain the following characterisation for inaccessible boundary points $x$, which are defined by the property that $x \notin \partial V$ for all connected components $V$ of the complement $\R^2  \setminus E_\eps$.

\begin{cor}[{Inaccessible singularities}] \label{Cor_Inaccessible_Singularities}
Let $E \subset \R^2$ and $x \in \partial E_\eps$. Then $x \notin \partial V$ for all connected components $V$ of the complement $\R^2 \setminus E_\eps$ if and only if $x$ is a one-sided chain singularity (type S6) or a chain-chain singularity (type S7).
\end{cor}

\begin{proof}
{(i)} Assume first that $x \notin \partial V$ for all connected components $V$ of the complement $\R^2 \setminus E_\eps$. Proposition~\ref{Prop_regular_and_wedge_complement_cone} then implies that $x$ is not a wedge and $x \notin \Unp{E}$, so that the extremal contributors $y_1, y_2 \in \Cext{x}$ satisfy $y_1 - x = -(y_2 - x)$.

In case $x$ has only one extremal outward direction $\xi \in \Xext{x}$, Proposition~\ref{Prop_local_contribution} implies the existence of some $r < \eps/2$ for which
\[
B_r(x) \setminus E_\eps \subset U_r(x, \xi).
\]
By assumption, there cannot exist a connected component $V_\xi$ described in case~(i) of Proposition~\ref{Prop_Sharp_Singularity_Types}, so that case~(ii) is necessarily satisfied. Hence, $x$ is a (one-sided) chain singularity.

Similarly, if there exist extremal outward directions $\xi_1, \xi_2$ with $\xi_1 = -\xi_2$, Proposition~\ref{Prop_Sharp_Singularity_Types} again rules out case (i) for each one of them. Consequently, $x$ fulfils the definition of a chain-chain singularity.

{(ii)} Assume then that $x$ is either a chain singularity (type S6) or a chain-chain singularity (type S7). In the former case $\Xext{x} = \{\xi\}$ for some $\xi \in S^1$, and Proposition~\ref{Prop_local_contribution} again implies the existence of some $r < \eps/2$ for which $B_r(x) \setminus E_\eps \subset U_r(x, \xi)$. We aim to deduce a contradiction by assuming there exists a connected $V \subset U_r(x, \xi) \setminus E_\eps$ for which $x \in \partial V$. Note that every $z \in V$ has a representation
\begin{equation} \label{Eq_representation}
z = x + s\xi + t(s) \frac{x - y_1}{\eps}
\end{equation}
for some $s = s(z) > 0$ and $t(s) \in \R$. Now, choose some $s_0 > 0$ and $t(s_0)$ so that
\[
z_0 := x + s_0\xi + t(s_0) \frac{x - y_1}{\eps} \in V.
\]
Given that $V$ is connected and $x \in \partial V$, it follows from $z_0 \in V$ that there exists a path $\gamma : [0,1] \to V$ for which $\gamma(0) = x$ and $\gamma(1) = z_0$. Following~\eqref{Eq_representation}, we may write
\[
\gamma(u) = x + s(u)\xi + t(s(u)) \frac{x - y_1}{\eps},
\]
where the coordinate $s(u)$ depends continuously on $u \in [0,1]$. Since $\gamma(u) \notin E_\eps$ for all $u \in [0,1]$, 
the coordinate $t(s(u))$ satisfies
\[
t(s(u)) \in \big(f^{\xi, y_1}(s(u)), -f^{\xi, y_2}(s(u))\big)
\]
for all $u \in [0,1]$, where the functions $f^{\xi, y_i}$ are as in Proposition~\ref{Prop_local_representation_exists}. However, Proposition~\ref{Prop_Sharp_Singularity_Types} implies the existence of some $0 < q < s_0$ for which $f^{\xi, y_1}(q) = -f^{\xi, y_2}(q)$, and since $s(\cdot)$ is continuous as a function of $u$, there exists some $u_q$ for which $s(u_q) = q$. For this coordinate we thus obtain the contradiction
\[
t(s(u_q)) = t(q) \in \big(f^{\xi, y_1}(q), -f^{\xi, y_2}(q)\big) = \varnothing.
\]

In case $x$ is a chain-chain singularity (type S7), one may again follow the above reasoning to deduce that the existence of a connected component $V$ of $\R^2 \setminus E_\eps$ with $x \in \partial V$ would contradict the existence of a sequence $s_n \to 0$ with $f^{\xi_i, y_1}(s_n) = -f^{\xi_i, y_2}(s_n)$, which is on the other hand guaranteed for both $\xi_1, \xi_2 \in \Xext{x}$ by Proposition~\ref{Prop_Sharp_Singularity_Types}.
\end{proof}

\begin{remark} \label{Remark_Inaccessible_Points}
Corollary~\ref{Cor_Inaccessible_Singularities} states that it is impossible for a chain singularity (type S6) or a chain-chain singularity (type S7) $x$ to lie on the boundary of any connected component $V \subset \R^2 \setminus E_\eps$, even though a sequence $(V_n)_{n=1}^\infty$ of connected components $V_n \subset \R^2 \setminus E_\eps$ converges to $x$ in Hausdorff distance. This is the motivation for the terminology of \emph{inaccessible} singularities and can be seen as an analogue of the distinction between accessible and inaccessible points in a Cantor set $C \subset [0,1]$, where inaccessible points do not lie on the boundary of any of the countably many removed open intervals $[a_n, b_n] \subset [0,1]$. See also Example~\ref{Example_Pos_Measure_Set_of_Chain_Sing} and the discussion after Definition~\ref{Def_classification_of_singularities} above.
\end{remark}

\begin{prop}[{Characterisation of chain singularities}] \label{Prop_Topological_Characterisation_of_Chain_Singularities}
Let $E \subset \R^2$, let $x \in \partial E_\eps$ and let $\cG(x)$ be a local boundary representation at $x$ with the functions $g_{\xi, y} \in \cG(x)$ as in~\eqref{Eq_Canonical_LBR}. Then the following are equivalent:
\begin{itemize}
  \item[(i)] $x$ is a chain singularity (type S6, S7 or S8).
  \item[(ii)] There exists a sequence $(V_n)_{n=1}^\infty$ of mutually disjoint connected components $V_n \subset \R^2 \setminus E_\eps$
  for which $\mathrm{dist}_H (x, V_n) \to 0$ as $n \to \infty$.
  \item[(iii)] There exists a sequence $(x_n)_{n=1}^\infty$ of singularities 
  on $\partial E_\eps$ for which $x_n \to x$ and
  \begin{equation} \label{Eq_Extremal_Contr_Separate_2}
    \lim_{n \to \infty} \left\langle \frac{y_n^{(1)} - x_n}{\eps}, \frac{y_n^{(2)} - x_n}{\eps} \right \rangle = -1,
  \end{equation}
  where $\Cext{x_n} = \big\{y_n^{(1)}, y_n^{(2)} \big\}$ for each $n \in \N$.
  \item[(iv)] The extremal contributors $\Cext{x} = \{y_1, y_2\}$ satisfy $y_1 - x = -(y_2 - x)$ and there exist some $\xi \in \Xext{x}$ and corresponding functions $g_{\xi, y_1},  g_{\xi, y_2} \in \cG(x)$ for which $g_{\xi, y_1}(s_n) = g_{\xi, y_2}(s_n)$ for all $n \in \N$ for some sequence $(s_n)_{n=1}^\infty \subset \R_+$ with $s_n \to 0$ as $n \to \infty$.
\end{itemize}
\end{prop}

\begin{proof}
We begin by showing that (ii) $\Rightarrow$ (iv) $\Rightarrow$ (iii) $\Rightarrow$ (ii). Since clearly (i) implies (iii), the result then follows by showing that (iii) $\wedge$ (iv) $\Rightarrow$ (i).

{(ii) $\Rightarrow$ (iv).} Assume there exists a sequence $(V_n)_{n=1}^\infty \subset \R^2 \setminus E_\eps$ of mutually disjoint connected components of the complement $\R^2 \setminus E_\eps$ with $\mathrm{dist}_H (x, V_n) \to 0$ as $n \to \infty$. It follows then from Proposition~\ref{Prop_regular_and_wedge_complement_cone} that $x \notin \Unp{E}$ and it cannot be a wedge, which implies $y_1 - x = -(y_2 - x)$ for the extremal contributors $\Cext{x} = \{y_1, y_2\}$.

For each $n \in \N$, choose a point $v_n \in V_n$, and define $\xi_n := (v_n - x)/\norm{v_n - x}$. Due to compactness, one can choose a subsequence $(v_{n_k})_{k=1}^\infty$ for which $\xi_{n_k} \to \xi \in S^1$. Since $v_n \in \R^2 \setminus E_\eps$ for all $n \in \N$, it follows from Lemma~\ref{Lemma_Basic_Properties_of_Sequences_Converging_onto_Boundary} that $\langle y_i - x, \xi \rangle = 0$ for $i \in \{1,2\}$, so that $\xi \in \Xext{x}$.

Assume contrary to the claim that there exists $\delta > 0$ for which $g_{\xi, y_1}(s) \neq g_{\xi, y_2}(s)$ for all $s \in (0, \delta)$. According to Proposition~\ref{Prop_Sharp_Singularity_Types} this implies
\begin{equation} \label{Eq_cone_repr}
U_{\delta}(x, \xi) \setminus E_\eps = V \cap U_{\delta}(x, \xi) = \bigcup_{0 < \rho < \delta} \rho (\alpha_\rho, \beta_\rho)_{S^1},
\end{equation}
where $V$ is the unique connected component of the complement $\R^2 \setminus E_\eps$ that intersects $U_\delta(x, \xi)$. Since $\xi_{n_k} \to \xi$ as $k \to \infty$, equation \eqref{Eq_cone_repr} now implies $v_{n_k} \in V$ for large $k \in \N$, which in turn contradicts the assumption that the sets $V_n$ are connected and mutually disjoint. Hence, no such $\delta$ can exist, and (iv) follows.

{(iv) $\Rightarrow$ (iii).} Case (ii) in Proposition~\ref{Prop_Sharp_Singularity_Types} now holds true, and directly implies (iii) here.

{(iii) $\Rightarrow$ (ii).} Write $\Cext{x} = \{y_1, y_2\}$ and define $\xi_n := (x_n - x)/\norm{x_n - x}$. Due to Lemma~\ref{Lemma_Convergence_of_Contributing_Points_of_Sequences_of_Boundary_Points}(i) there exists a subsequence $(x_{n_k})_{k=1}^\infty$, for which $\xi_{n_k} \to \xi \in \Xext{x}$ as $k \to \infty$. Furthermore, since $\Cext{x_n} = \big\{y_n^{(1)}, y_n^{(2)} \big\}$ for all $n \in \N$, Lemma~\ref{Lemma_Convergence_of_Contributing_Points_of_Sequences_of_Boundary_Points}(ii)(a) together with~\eqref{Eq_Extremal_Contr_Separate_2} implies the existence of a further subsequence $(x_m)_{m=1}^\infty \subset (x_{n_k})_{k=1}^\infty$ for which $y_m^{(i)} \to y_i \in \Cext{x}$ for $i \in \{1,2\}$. The extremal contributors $y_1, y_2$ at $x$ therefore satisfy $y_1 - x = -(y_2 - x)$.

To complete the argument, we show that for any $\delta > 0$ there exists some $0 < s_m  < \delta$ for which $x_m = g_{\xi, y_1}(s_m) = g_{\xi, y_2}(s_m)$. Once this is established, the statement follows from Proposition~\ref{Prop_Sharp_Singularity_Types}. Since $x_m \to x$ as $m \to \infty$, there exists for all $\delta > 0$ some $M \in \N$ for which $\norm{x_m - x} \leq \delta$ for all $m > M$. It follows that for each $m > M$ there exists some $s_m \leq \delta$ such that $x_m = g_{\xi, y_i}(s_m)$ for at least one of the contributors $y_1, y_2$. But since $y_1 - x = -(y_2 - x)$ and $y_m^{(i)} \to y_i \in \Cext{x}$ for $i \in \{1,2\}$ as $m \to \infty$, this implies $x_m = g_{\xi, y_1}(s_m) = g_{\xi, y_2}(s_m)$ for all sufficiently large $m > M$.

{(iii) $\wedge$ (iv) $\Rightarrow$ (i).} Since $y_1 - x = -(y_2 - x)$ for the extremal contributors $y_1, y_2 \in \Cext{x}$, it follows from Lemma~\ref{Lemma_OD_versus_dot_product_ineq} that the extremal outward directions $\xi_1, \xi_2 \in \Xext{x}$ satisfy either $\xi_1 = \xi_2$ or $\xi_1 = - \xi_2$. In the former case $x$ is a one-sided chain singularity (type S6). In the latter case we may assume without loss of generality that $\xi_n := (x_n - x) / \norm{x_n - x} \to \xi_1$. Proposition~\ref{Prop_Sharp_Singularity_Types} then states that for some $\delta > 0$, the boundary subset $\partial E_\eps \cap U_\delta(x, \xi_2)$ exhibits either sharp-type or chain-type geometry and that these cases are mutually exclusive. In the former case $x$ is a sharp-chain singularity (type S8), in the latter a chain-chain singularity (type S7).
\end{proof}

We employ Proposition~\ref{Prop_Topological_Characterisation_of_Chain_Singularities} to show that our characterisation of smooth points in Definition~\ref{Def_Smooth_Points} coincides with the property of lying on a $C^1$-smooth curve. By a \emph{curve} we mean the image $\Gamma = \gamma([0,1])$ of a continuous, injective map $\gamma : [0,1] \to \R^2$.

\begin{prop}[{Characterisation of smooth points}] \label{Prop_Characterisation_of_smooth_points}
Let $E \subset \R^2$ and $x \in \partial E_\eps$. Then $x$ is smooth in the sense of Definition~\ref{Def_Smooth_Points} if and only if there exists a $C^1$-curve $\Gamma$ for which $\Gamma = \partial E_\eps \cap \overline{B_\delta(x)}$ for some $\delta > 0$.
\end{prop}

\begin{proof}
%
{(i)} Assume first that $x$ is smooth in the sense of Definition~\ref{Def_Smooth_Points}. Then $x \in \Unp{E}$, and Proposition~\ref{Prop_structure_of_set_of_outward_directions} implies that the set of extremal outward directions at $x$ satisfies $\Pext{x} = \{\xi, -\xi\}$ for some $\xi \in S^1$. In addition, let $\cG(x) = \{g_{\xi,y}, g_{-\xi,y}\}$ be a local boundary representation at $x$ with radius $r > 0$. Since $x$ is smooth, there exists some $0 < \delta \leq r$ for which $\overline{B_\delta(x)} \cap \partial E_\eps \subset \Unp{E}$. In particular, $x \in \Unp{E}$, and it follows from Proposition~\ref{Prop_local_representation_exists} that
\[
\Gamma := g_{\xi, y}([0, s_\delta]) \cup g_{-\xi, y}([0, s_\delta']) = \partial E_\eps \cap \overline{B_\delta(x)}
\]
for some $s_\delta, s_\delta' \in (0, \delta]$. We will show that the functions $g_{\xi, y}$ and $g_{-\xi, y}$ are differentiable on the above intervals.

To this end, let $s \in [0, s_\delta]$ and consider the point $z := g_{\xi, y}(s)$. Then the set of extremal outward directions at $z$ satisfies
\begin{equation} \label{Eq_extremal_directions_at_z}
\Xext{z} = \{\xi_{z}, -\xi_{z}\}
\end{equation}
for some $\xi_{z} \in S^1$. According to Proposition~\ref{Prop_Tangents_are_Defined}, the extremal outward directions coincide with tangential directions on the boundary. This, together with~\eqref{Eq_extremal_directions_at_z}, implies that $g_{\xi,y}$ is differentiable at $s$. Since the argument $s \in [0, s_\delta]$ above was arbitrary, it follows that $g_{\xi_x,y}$ is differentiable on $[0, s_\delta]$. The differentiability of $g_{-\xi_x,y}$ on $[0, s_\delta']$, as well as the differentiability of $\Gamma$ at $x$, follow with the same reasoning, and the claim follows.





{(ii)} Let then $\Gamma$ be a $C^1$-curve for which $\Gamma = \partial E_\eps \cap \overline{B_\delta(x)}$ for some $\delta > 0$. Consider now some $z \in \partial E_\eps \cap \overline{B_\delta(x)}$. Since $\Gamma$ is $C^1$-smooth, the correspondence between tangents and extremal outward directions given by Proposition~\ref{Prop_Tangents_are_Defined} implies that $z$ is not a wedge (type S1) or a sharp singularity (types S2--S3). On the other hand, since $z \in \Gamma$ and $\Gamma$ is connected as a curve, Proposition~\ref{Prop_Topological_Characterisation_of_Chain_Singularities} implies that $z$ cannot be a chain singularity (types S6--S8). Hence it follows from Theorem~\ref{Thm_Main_1} below that $z$ is either a smooth point or a shallow singularity (types S4--S5), which implies $z \in \Unp{E}$. The same argument applies to all $z \in \partial E_\eps \cap \overline{B_r(x)}$, which means that $x$ is smooth in the sense of Definition~\ref{Def_Smooth_Points}.
\end{proof}

\subsection{Proof of Theorem~\ref{Thm_Main_1}} \label{Sec_Proof_of_Thm_1}
We conclude this section with the proof of our first main result, a classification of boundary points on $\partial E_\eps$. We restate the result here for the convenience of the reader.

\begin{theorem_without_2}[{Classification of boundary points}] 
Let $E \subset \R^2$ be compact, $\eps > 0$ and let $x \in \partial E_\eps$ be a boundary point of $E_\varepsilon$ that is not smooth. Then $x$
belongs to precisely one of the eight categories of singularities given in Definition~\ref{Def_classification_of_singularities}.
\end{theorem_without_2}

\begin{proof}
Note first that for any $u,v \in S^1$ either $u = v$, $u = -v$, or $u \notin \{v, -v\}$. Hence we obtain the following categorisation of boundary point types according to the orientation of the extremal outward directions $\Xext{x} = \{\xi_1, \xi_2\}$:
\begin{align*}
\xi_1 &\notin \{\xi_2, -\xi_2\}:  &\textrm{type S1} \\
\xi_1 &= \xi_2: &\textrm{types S2 and S6} \\
\xi_1 &= -\xi_2:  &\textrm{smooth points and types S3--S5, S7--S8}
\end{align*}
These are due to Proposition~\ref{Prop_structure_of_set_of_outward_directions} for $x \in \Unp{E}$ and Definition~\ref{Def_classification_of_singularities} for $x \notin \Unp{E}$, and they correspond to the cases ({\footnotesize A})--({\footnotesize C}) illustrated in Figure~\ref{Figure_Three_Basic_Cases}. It follows immediately that if $x$ is a wedge (type S1), it cannot be of any other type, and vice versa. In addition, of all the defined types of boundary points, only the shallow singularities (types S4 and S5) and smooth points satisfy $\Cext{x} = \{y\}$ for some $y \in \partial E$, and these types are by definition mutually exclusive.

Hence, it suffices to show that the remaining types S2--S3 and S6--S8, corresponding to case ({\footnotesize C}) in Figure~\ref{Figure_Three_Basic_Cases}, are mutually exclusive. For all these types, the set of extremal contributors $\Cext{x} = \{y_1, y_2\}$ satisfies $y_1 - x = -(y_2 - x)$. Proposition~\ref{Prop_Sharp_Singularity_Types} then states that for each $\xi \in \Xext{x}$ either
\begin{itemize}
\item[(i)] there exists a connected component $V \subset \R^2 \setminus E_\eps$ and $r > 0$ for which $U_r(x, \xi) \setminus E_\eps = V \cap U_r(x, \xi)$ and $x \in \partial V$, or
\item[(ii)] there exists a sequence of singularities $(x_n)_{n=1}^\infty$ in $\Sing$ with $x_n \to x$ as $n \to \infty$ and $(x_n - x)/\norm{x_n - x} \to \xi$, and
\[
\lim_{n \to \infty }\left \langle \frac{y_n^{(1)} - x_n}{\eps}, \frac{y_n^{(2)} - x_n}{\eps} \right \rangle = -1,
\]
\end{itemize}
and that these situations are mutually exclusive. In other words, for each extremal outward direction $\xi \in \Xext{x}$, the intersection $\partial E_\eps \cap U_r(x, \xi)$ exhibits either sharp-type or chain-type geometry (see Definition~\ref{Def_classification_of_singularities}). In case $\xi_1 = \xi_2$, the point $x$ is hence either a sharp (type S2) or a chain (type S6) singularity, and in case $\xi_2 = -\xi_1$, it is either a sharp-sharp (type S3), a chain-chain (type S7), or a sharp-chain (type S8) singularity, and all these cases are mutually exclusive.
\end{proof}
\chapter{Topological Structure of the Set of Singularities} \label{Sect_Topo_Structure}
Since the categories of boundary points given in Definition~\ref{Def_classification_of_singularities} define a partition of the boundary, it makes sense to inquire on their cardinalities and topological structure. Our second main result, Theorem~\ref{Thm_Main_2}, states that for any compact $E \in \R^2$ and $\eps > 0$, the sets of wedges (type S1), sharp singularities (types S2, S3 and S8) and one-sided chain singularities (type S6) on $\partial E_\eps$ are at most countably infinite. This does not hold in general for the sets of shallow-shallow singularities (type S5) or chain-chain singularities (type S7), which may even have a positive one-dimensional Hausdorff measure on the boundary, see \cite{Ferry_When_epsilon_boundaries, Fu_Tubular_neighborhoods}. In Section~\ref{Subsec_Topology_of_Chain} we show that the set $\cC(\partial E_\eps)$ of chain singularities is nevertheless nowhere dense, and hence small in the topological sense.

\section{Cardinalities of Sets of Singularities} \label{Sec_Cardinality_of_Singularities}
In order to prove the above-mentioned results on the cardinalities of the sets of singularities, we proceed by treating one by one the cases of wedges, sharp singularities, and one-sided shallow and chain singularities. We begin with the following general result on the geometry of accumulating singularities. It is essentially a corollary of Lemma~\ref{Lemma_Convergence_of_Contributing_Points_of_Sequences_of_Boundary_Points} on the asymptotic behaviour of sequences of boundary points.

\begin{lemma}[{Geometry of accumulating singularities}] \label{Lemma_Accum_Sing_Shallow}
Let $E \subset \R^2$ and let $(x_n)_{n=1}^\infty \subset \Sing$ be a sequence of pair-wise disjoint singularities with $x_n \to x \in \partial E_\eps$. Let $\Xext{x_n} = \big \{\xi_n^{(1)}, \xi_n^{(2)} \big \}$ for each $n \in \N$. Then $\lim_{n\to\infty} \big | \big \langle \xi_n^{(1)}, \xi_n^{(2)} \big \rangle \big | = 1$.
\end{lemma}

\begin{proof}
Due to Lemma~\ref{Lemma_Convergence_of_Contributing_Points_of_Sequences_of_Boundary_Points}(i) we can assume that the limit
\[
\xi := \lim_{n\to\infty} (x_n - x) / \norm{x_n - x}
\]
exists. Note that depending on the geometry at the singularities $x_n$, each of the extremal outward directions $\xi_n^{(i)}$ for $i \in \{1,2\}$ may be more aligned with $\xi$ than $-\xi$, or vice versa.
However, for each $i \in \{1,2\}$, Lemma~\ref{Lemma_Convergence_of_Contributing_Points_of_Sequences_of_Boundary_Points}(ii)(b) implies that there exist sequences of coefficients $\big(a_n^{(i)}\big)_{n=1}^\infty \in \{1,-1\}^\N$ for which $\big | \big \langle \xi_n^{(i)}, \xi \big \rangle \big | = \big \langle a_n^{(i)} \xi_n^{(i)}, \xi \big \rangle$ and
\begin{equation} \label{Eq_scalar_prod_limits}
\lim_{n\to\infty} \big \langle a_n^{(1)} \xi_n^{(1)}, \xi \big \rangle = \lim_{n\to\infty} \big \langle a_n^{(2)} \xi_n^{(2)}, \xi \big \rangle = 1.
\end{equation}
Let $\theta_n^{(i)} \geq 0$ be the angle for which $\big \langle a_n^{(i)} \xi_n^{(i)}, \xi \big \rangle = \cos \big(\theta_n^{(i)} \big)$. Equation~\eqref{Eq_scalar_prod_limits} implies $\theta_n^{(i)} \to 0$ for $i \in \{1,2\}$, and in particular there exists some $N \in \N$ for which
\[
\big | \big \langle \xi_n^{(1)}, \xi_n^{(2)} \big \rangle \big | = \big \langle a_n^{(1)} \xi_n^{(1)}, a_n^{(2)} \xi_n^{(2)} \big \rangle
\]
for all $n \geq N$. For each $n \in \N$, define
\[
\theta_n^{\textrm{max}} := \max \big\{\theta_n^{(1)}, \theta_n^{(2)} \big\}.
\]
Then
\begin{equation} \label{Eq_Direction_Estimate}
\big \langle a_n^{(1)} \xi_n^{(1)}, a_n^{(2)} \xi_n^{(2)} \big \rangle = \cos \big( \theta_n^{(1)} +  \theta_n^{(2)} \big)
  \geq \cos \left( 2 \theta_n^{\textrm{max}} \right) \rightarrow 1,
\end{equation}
as $n \to \infty$, and the result follows.
\end{proof}

A particular consequence of Lemma~\ref{Lemma_Accum_Sing_Shallow} is that accumulating wedges (type S1) become increasingly acute/obtuse as they approach a limit point, with the angles
\[
\theta_n := \sphericalangle \big(\xi_n^{(1)}, \xi_n^{(2)} \big)
\]
between the extremal outward directions approaching an asymptotic value
\[
\theta = \lim_{n\to\infty} \theta_n \in \{0, \pi\}.
\]
It follows that for any fixed $p > 0$, there can only exist finitely many wedges whose sharpness deviates from these asymptotic values by more than $p$, which in turn implies that the total number of wedges can at most be countably infinite. Lemma~\ref{Lemma_Countable_Number_of_Wedges} below makes this argument precise.

\begin{lemma}[{Number of wedges}] \label{Lemma_Countable_Number_of_Wedges}
For any compact $E \subset \R^2$ and $\eps > 0$, the number of wedges (type S1) on $\partial E_\eps$ is at most countably infinite.
\end{lemma}

\begin{proof}
We begin by showing that for each $p \in (0,1)$, the subset $A(p) \subset \partial E_\eps$ defined by
\[
A(p) := \big \{x \in \partial E_\eps \, : \, \big | \big \langle \xi^{(1)}, \xi^{(2)} \big \rangle \big | \leq p \,\,\,\, \textrm{for} \,\,\,\, \xi^{(1)}, \xi^{(2)} \in \Xext{x} \big \}
\]
contains only finitely many points. Assume contrary to this that for some $p \in (0,1)$ the set $A(p)$ contains infinitely many points. This implies that there exists a pair-wise disjoint sequence $(x_n)_{n=1}^\infty$ with $x_n \in A(p)$ for all $n \in \N$, and due to compactness of $\partial E_\eps$ we may assume that $x_n \to x \in \partial E_\eps$ as $n \to \infty$. Writing $\Xext{x_n} = \big \{\xi_n^{(1)}, \xi_n^{(2)} \big \}$ for each $n \in \N$, Lemma~\ref{Lemma_Accum_Sing_Shallow} then implies $\lim_{n\to\infty} \big | \big \langle \xi_n^{(1)}, \xi_n^{(2)} \big \rangle \big | = 1$, contradicting the assumption that $\big | \big \langle \xi_n^{(1)}, \xi_n^{(2)} \big \rangle \big | \leq p < 1$ for all $n \in \N$.

By definition, each wedge $x \in \partial E_\eps$ belongs to the set $A(p)$ for some $p \in (0,1)$. Hence the union
\[
A := \bigcup_{n = 1}^\infty A \big(1 - n^{-1} \big)
\]
contains all the wedges on $\partial E_\eps$. According to the reasoning above, each of the sets $A (1 - n^{-1} )$ contains only finitely many points, from which the result follows.
\end{proof}

We next show that for a given connected component $U$ of the complement $\R^2 \setminus E_\eps$, there can only exist finitely many sharp singularities on $\partial U$. This essentially follows from Propositions~\ref{Prop_local_representation_exists} and~\ref{Prop_Topological_Characterisation_of_Chain_Singularities} which imply that any convergent sequence of pairwise disjoint sharp singularities $x_n \in \partial U$ is associated with a sequence $(U_n)_{n=1}^\infty$ of pairwise disjoint connected components of the complement $\R^2 \setminus E_\eps$, aligned with one of the extremal outward directions $\xi \in \Xext{x}$, satisfying $x_n  \in \partial U_n$ for all $n \in \N$. 
Our proof by contradiction amounts to showing that there can be no fixed connected component $U \subset \R^2 \setminus E_\eps$ for which $x = \lim_{n \to \infty} x_n \in \partial U$ and $U_n \neq U$ for infinitely many $n \in \N$.

\begin{lemma}[{Number of sharp singularities}] \label{Lemma_Finitely_Many_Sharp_Singularities_per_Component}
Let $E \subset \R^2$ be compact. For any connected component $U$ of the complement $\R^2 \setminus E_\eps$, the number of sharp singularities (types S2, S3 and S8) on the boundary $\partial U$ is finite.
\end{lemma}

\begin{proof}
Assume that the claim fails for some connected $U \subset \R^2 \setminus E_\eps$. Since $\partial U$ is compact, this implies the existence of a pair-wise disjoint sequence $(x_n)_{n=1}^\infty \subset \partial U$ of sharp singularities with $x_n \to x \in \partial U$. We may furthermore assume that the sequence is ordered so that
\begin{equation} \label{Eq_Ordered_Sequence}
\norm{x_{n+1} - x} < \norm{x_n - x}
\end{equation}
for all $n \in \N$, and that the limit $\xi := \lim_{n \to \infty} (x_n - x) / \norm{x_n - x}$ exists. Write $\Cext{x_n} = \big\{y_n^{(1)}, y_n^{(2)}\big\}$. It follows from the definition of a sharp singularity that
\begin{equation} \label{Eq_Sequence_of_Singularities}
y_n^{(1)} - x_n = -\big( y_n^{(2)} - x_n \big)
\end{equation}
for all $n \in \N$. According to Proposition~\ref{Prop_Topological_Characterisation_of_Chain_Singularities}, also the extremal contributors $y_1, y_2 \in \Cext{x}$ satisfy $y_1 - x = -(y_2 - x)$, and $y_n^{(i)} \to y_i$ for $i \in \{1,2\}$ as $n \to \infty$. Let $\cG(x)$ be the local boundary representation at $x$, with radius $r > 0$, given by Proposition~\ref{Prop_local_representation_exists}. Consider for $i \in \{1,2\}$ the functions $g_{\xi, y_i} \in \cG(x)$,
\[
g_{\xi, y_i}(s) := x + s \xi + f^{\xi, y_i}(s)\frac{x - y_i}{\eps},
\]
where the functions $f^{\xi, y_i}: [0, r] \to \R$ are continuous. Due to \eqref{Eq_Sequence_of_Singularities}, and since $x_n \to x$, there exists a sequence $(s_n)_{n=1}^\infty \subset \R_+$ with $s_n \to 0$ and some $N \in \N$, for which $n > N$ implies $x_n = g_{\xi, y_1}(s_n) = g_{\xi, y_2}(s_n)$ and consequently $f^{\xi, y_1}(s_n) + f^{\xi, y_2}(s_n) = 0$. Inequality \eqref{Eq_Ordered_Sequence} implies $s_{n+1} < s_n < s_{n-1}$ for all $n > N$, and one can define the open sets
\begin{align*}
S_n &:= \left\{s \in (s_{n+1}, s_{n-1}) \, : \, f^{\xi, y_1}(s) + f^{\xi, y_2}(s) < 0 \right\}, \\
U_n &:= \left\{ \tau g_{\xi, y_1}(s) +  (1 - \tau) g_{\xi, y_2}(s) \, : \, \tau \in (0,1), \, s \in S_n \right\}.
\end{align*}
For each $n > N$, the set $U_n$ is contained in the interior $\textrm{int} \,R_n$ of the closed rectangle
\begin{align*}
R_n &:= \Big\{ x + s\xi + t\frac{x - y_1}{\eps} \, : \, s_{n+1} \leq s \leq s_{n-1}, \\
&\phantom{=} \qquad \quad \inf_{s \in S_n}\left\{f^{\xi, y_1}(s)\right\} \leq t \leq -\inf_{s \in S_n}\left\{f^{\xi, y_2}(s)\right\} \Big\}.
\end{align*}
By definition, $s_n \in S_n$ for all $n > N$. From this it follows that $x_n \notin \partial V$ for any open $V \subset R_n^c \setminus E_\eps$. On the other hand, $x_n \in \partial U$ for all $n \in \N$. This implies $U \subset \textrm{int} \,R_n$ for all $n \in \N$, since $U$ is connected and $\partial R_n \setminus E_\eps = \varnothing$ by definition. However, given that $x \in \partial U \cap R_n^c$, this leads to the contradiction $U \subset R_n \cap R_n^c$ for all $n > N$.
\end{proof}

In~\cite[Remark 6.6]{Rataj_Zajicek_On_the_structure_of_sets_with_PR} it is noted that for any planar compact set with positive reach, the number of points of type $T^3$ on the boundary is finite. For $\eps$-neighbourhoods, points of type $T^3$ correspond to sharp singularities (type S2), see~\cite[Definition 6.3]{Rataj_Zajicek_On_the_structure_of_sets_with_PR}. Lemma~\ref{Lemma_Finitely_Many_Sharp_Singularities_per_Component} can therefore be seen as an analogue of this result in the context of $\eps$-neighbourhoods.

Lemmas~\ref{Lemma_countably_many_S4_singularities} and~\ref{Lemma_countably_many_S6_singularities} below state that the sets of one-sided shallow singularities (type S4) and chain singularities (type S6) are both at most countably infinite. The argument in both cases rests on the observation that for a finite sum $M := \sum_{x \in A} m_x < \infty$ of non-negative real numbers $m_x$, indexed by some possibly uncountable set $A$, the index subset $A^+ := \{x \in A \, : \, m_x > 0\}$, corresponding to the positive elements in the sum, is at most countably infinite.\footnote{This follows from the observation that the set $A_n := \{x  \in A \, : \, m_x > 1/n \}$ is finite for each $n \in \N$ and hence $A^+ := \{x \in A \, : \, m_x > 0 \} = \bigcup_{n \in \N} A_n$ is countable as a countable union of finite sets.}
In the case of shallow singularities, the numbers being summed will represent lengths (one-dimensional Hausdorff measures) $m_x := \cH^1(I_x)$ of boundary segments $I_x \subset \partial E_\eps$. In the case of chain singularities they will stand for surface areas $m_x := \cH^2(A_x)$ of open subsets $A_x \subset \mathrm{int} \, E_\eps$. In each case, these numbers will be strictly positive for every $x$ by definition, which implies that the underlying index sets, corresponding to the sets of singularities in question, are themselves at most countably infinite.

\begin{lemma}[{Number of one-sided shallow singularities}] \label{Lemma_countably_many_S4_singularities}
For a compact set $E \subset \R^2$, the number of one-sided shallow singularities (type S4) on $\partial E_\eps$ is at most countably infinite.
\end{lemma}

\begin{proof}
Let $x \in \partial E_\eps$ be a one-sided shallow singularity. Then, by definition, there exists $\xi \in \Xext{x}$ and $\delta_x > 0$ for which $J_x(\delta_x) := U_{\delta_x}(x, \xi) \cap \partial E_\eps \subset \Unp{E}$. 
In particular, it follows that if $z \neq x$ is another one-sided shallow singularity, with the corresponding set $J_z(\delta_z)$, then the sets $J_x(\delta_x/2)$ and $J_z(\delta_z/2)$ satisfy
\begin{equation} \label{Eq_The_sets_J_x_are_disjoint}
J_x(\delta_x/2) \cap J_z(\delta_z/2) = \varnothing.
\end{equation}

Let $\cG(x)$ be a local boundary representation at $x$ with radius $r > 0$, and let $(\xi,y) \in \Pext{x}$ be such that $J_x(\delta_x/2) = g_{\xi, y}([0,s_x])$ for some $s_x \in (0, \delta)$, see Proposition~\ref{Prop_local_representation_exists}.
According to Proposition~\ref{Prop_LBR_Lipschitz}, the function $g_{\xi, y}$ is $2/\sqrt{3}$-Lipschitz, so that the length $\cH^1(J_x(\delta_x/2))$ of the curve $J_x(\delta_x/2)$ satisfies
\[
  0 < s_x \leq \cH^1(J_x(\delta_x/2)) \leq \frac{2}{\sqrt{3}}s_x,
\]
see for instance \cite[Proposition 2.49]{Ambrosio_et_al_Functions_of_Bounded_Variation}. Denote by $W$ the set of one-sided shallow singularities. Then, due to~\eqref{Eq_The_sets_J_x_are_disjoint}, the length of the collection $I := \bigcup_{x \in W} J_x(\delta_x/2)$ satisfies
\begin{equation} \label{Eq_Finite_Measure_on_the_Boundary}
  \cH^1(I) = \sum_{x \in W} \cH^1(J_x(\delta_x/2)) \leq \frac{2}{\sqrt{3}} \sum_{x \in W} s_x < \cH^1(E_\eps) < \infty.
\end{equation}
The finiteness of $\cH^1(E_\eps)$ was established already by Erd\H{o}s, see~\cite[Section 6]{Erdos_Some_remarks}. According to the counting argument preceding the statement of the result, the estimate \eqref{Eq_Finite_Measure_on_the_Boundary} implies that the set $W$ can be at most countably infinite.
\end{proof}

The following example demonstrates that the set of shallow-shallow singularities (type S5) can be dense and have positive Hausdorff measure on the boundary $\partial E_\eps$. The idea is to construct a suitably jagged function on the interval $[0,1]$, and interpret its graph as a subset of the boundary $\partial E_\eps$ of a certain set $E \subset \R^2$.

\begin{example}[{Dense, full measure set of shallow-shallow singularities}] \label{Ex_Shallow_can_be_dense}
Consider a bounded, increasing function $\alpha: [0, 1] \to \R$ that is discontinuous at every rational number $q \in \Q \cap [0,1]$ but continuous at every irrational number $p \in [0,1] \setminus \Q$.\footnote{One way to construct such a function is to write $\Q \cap [0,1] = \big \{ q_n \, : \, n \in \N \big \}$, define $N(x) := \{n  \, : \, q_n \leq x\}$, take any positive summable sequence $(a_n)_{n=1}^\infty$, and set $\alpha(x) := \sum_{n \in N(x)} a_n$ for all $x \in [0,1]$. This way $\alpha$ is increasing on $[0,1]$, has a jump of amplitude $a_n$ at each rational $x = q_n$, is continuous at every $x \in [0,1] \setminus \Q$, and satisfies $\alpha(1) = \sum_{n=1}^\infty a_n < \infty$.}
~As an almost everywhere continuous bounded function, every such $\alpha$ is Riemann-integrable, and its monotonicity implies that the integral function $I_\alpha(x) := \int_0^x \alpha(s) ds$ is convex. Most significantly for our example, $I_\alpha$ has a well-defined derivative at every irrational $p \in [0,1] \setminus \Q$, but not at any rational $q \in \Q \cap [0,1]$.

For any $\eps > 0$, one may thus interpret the graph $G := \big\{ (s, I_\alpha(s)) \, : \, s \in [0,1] \big\}$ as a subset of an boundary $\partial E_\eps$ as follows. Since the one-sided derivatives
\[
D^\pm I_\alpha(s) := \lim_{h \to \pm 0} \frac{I_\alpha(s + h) - I_\alpha(s)}{h}
\]
exist at every $s \in [0,1]$, one can define for each $x(s) := (s, I_\alpha(s))$ the corresponding contributors $y^-(s), y^+(s) \in \Pi_E(x(s))$ by setting
\[
y^\pm(s) := \big(s + a^\pm(s), \, I_\alpha(s) - b^\pm(s) \big),
\]
where for each $s \in [0,1]$
\[
a^\pm(s) := \frac{\eps D^\pm I_\alpha(s)}{\sqrt{1 + \big[D^\pm I_\alpha(s)\big]^2}} \qquad \textrm{and} \qquad b^\pm(s) := \frac{\eps}{\sqrt{1 + \big[D^\pm I_\alpha(s)\big]^2}}.
\]
A direct computation shows that $D^\pm I_\alpha(s) = b^\pm(s) / a^\pm(s)$ and $\norm{x(s) - y^\pm(s)} = \eps$ for all $s \in [0,1]$. Furthermore, the convexity of $I_\alpha$ implies that for each $y^\pm(s)$ the ball $\overline{B_\eps(y^\pm(s))}$ intersects $G$ only at the corresponding $x(s)$, which implies that $G \subset \partial E_\eps$ for the set $E := \{ y^\pm(s) \, : \, s \in [0,1] \}$.

By construction, $y^+(p) = y^-(p)$ for the irrational $p \in [0,1] \setminus \Q$. On the other hand, $y^+(q) \neq y^-(q)$ for all rational $q \in \Q \cap [0,1]$.
Hence, $x(q)$ is a wedge for each rational $q$, and $x(p) \in \Unp{E}$ for every irrational $p$. This implies that the points $x(p)$ are in fact shallow-shallow singularities (type S5). Due to the continuity of the integral function $I_\alpha$, these points form a dense set on $G$. Given that $I_\alpha$ is convex and absolutely continuous as an integral function, and the derivative $\alpha$ is bounded, $I_\alpha$ is in fact Lipschitz continuous. 
It then follows from the basic properties of Hausdorff measure (see for instance \cite[Proposition 2.49]{Ambrosio_et_al_Functions_of_Bounded_Variation}) and the rectifiability of $\partial E_\eps$ (see \cite[Proposition 2.3]{Rataj_Winter_On_Volume} and \cite[Corollary 3.3]{Falconer_The_Geometry}) that $\cH^1([0,1]) \leq \cH^1(G) < \infty$. Hence, the set $P := \big \{ (p, I_\alpha(p)) \, : \, p \in [0,1] \setminus \Q \big \}$ of shallow-shallow singularities has full measure on $G$.
\end{example}

\begin{lemma}[{Number of one-sided chain singularities}] \label{Lemma_countably_many_S6_singularities}
For a compact set $E \subset \R^2$, the number of one-sided chain singularities (type S6) on $\partial E_\eps$ is at most countably infinite.
\end{lemma}

\begin{proof}
Write $C$ for the set of one-sided chain singularities on $\partial E_\eps$. We argue that there exists a collection $\{A_x\}_{x \in C}$ of pair-wise disjoint open sets $A_x \subset \textrm{int}\,E_\eps$, indexed by $C$.
The result then follows from the counting argument discussed in the lead-up to Lemma~\ref{Lemma_countably_many_S4_singularities} above.

For each $x \in C$, the set of outward directions is a singleton $\Xi_x(E_\eps) = \{\xi\}$. Let $x \in C$ and let $\cG(x)$ be a local boundary representation at $x$ with radius $r > 0$, so that for $i \in \{1,2\}$ the functions $g_{\xi, y_i} \in \cG(x)$ are given by
\[
g_{\xi, y_i}(s) = x + s\xi + f^{\xi, y_i}(s)\frac{x - y_i}{\eps}
\]
for some continuous functions $f^{\xi, y_i}: [0, r] \to \R$. It follows from the definition of one-sided chain singularities that the extremal contributors $y_1, y_2 \in \Cext{x}$ satisfy $y_1 - x = -(y_2 - x)$  for all $x \in C$. Since $-\xi \notin \Xi_x(E_\eps)$, Proposition~\ref{Prop_local_contribution} implies that there are two possibilities:
\begin{itemize}
\item[(i)] there exists some non-extremal contributor $y \in \Pi_E(x)$ for which $\langle y - x, \xi \rangle < 0$, or else
\item[(ii)] 
the set
$\overline{B_\eps(E \cap U_\delta(y_1, -\xi))} \cap \overline{B_\eps(E \cap U_\delta(y_2, -\xi))} \setminus \{x\}$ is non-empty for all $\delta > 0$. 
\end{itemize}
In both cases, there exists some $p(x) < 0$ for which
\begin{equation} \label{Eq_One_Sided_Chain_Interior}
Q(x) := \Big \{x + s\xi + t\frac{x - y_1}{\eps} \, : \, (s, t) \in (p(x), 0) \times (-\eps/2, \eps/2) \Big \} \subset \mathrm{int} \, E_\eps .
\end{equation}
One can then define $A_x := U_{p(x)/3}(x, -\xi)$. Note that for all $x \in C$, the set $A_x$ is open and has a positive surface area $\cH^2(A_x) > 0$. Furthermore, $A_x \subset Q(x) \subset \textrm{int} \,E_\eps$ and it follows from the construction that if $z \neq x$ is any other one-sided chain singularity, its distance from $x$ satisfies $|z - x| \geq p(x)$, implying $A_x \cap A_z = \varnothing$. The sets $A_x$ are thus pair-wise disjoint, open and contained in some bounded ball $B_R(0)$ due to the compactness of $E_\eps$. Hence the sum $\sum_{x \in C} \cH^2(A_x)$ of their surface areas is finite, from which the result follows by the counting argument discussed in the lead-up to Lemma~\ref{Lemma_countably_many_S4_singularities}.
\end{proof}

\subsection{Proof of Theorem~\ref{Thm_Main_2}}
We conclude this chapter with the proof of Theorem~\ref{Thm_Main_2}, which combines Lemmas~\ref{Lemma_Countable_Number_of_Wedges}--\ref{Lemma_countably_many_S4_singularities} and~\ref{Lemma_countably_many_S6_singularities} into one statement.

\begin{theorem_without_2}[{Countable sets of singularities}]
Let $E \subset \R^2$ be compact and let $\eps > 0$. Then
the corresponding $\eps$-neighbourhood boundary $\partial E_\eps$ contains at most a countably infinite number of wedges (type S1), sharp singularities (types S2, S3, S8), one-sided shallow singularities (type S4) and chain singularities (type S6).
\end{theorem_without_2}

\begin{proof}
Consider the collection $\{U_i\}_{i \in I}$ of the connected components of the complement $\R^2 \setminus E_\eps$. Since $E$ is assumed to be compact, $E \subset B_R(0)$ for some $R > 0$. It follows that all but one, say $U_j$, of the connected components $U_i$ are bounded, so that
\[
\bigcup_{i \in I \setminus \{j\}} U_i \subset B_R(0).
\]
Following the counting argument discussed before the statement of Lemma~\ref{Lemma_countably_many_S4_singularities}, this implies that the index set $I$ is at most countably infinite. By definition, every sharp singularity $x \in \partial E_\eps$ (types S2, S3, S8) satisfies $x \in \bigcup_{i \in I} \partial U_i$. It follows then from Lemma~\ref{Lemma_Finitely_Many_Sharp_Singularities_per_Component} that the set of sharp singularities on $\partial E_\eps$ is countable as a countable union of finite sets. Finally, Lemmas~\ref{Lemma_Countable_Number_of_Wedges}, \ref{Lemma_countably_many_S4_singularities} and~\ref{Lemma_countably_many_S6_singularities} guarantee that the number of wedges (type S1) and one-sided shallow (type S4) and chain (type S6) singularities, respectively, are at most countably infinite.
\end{proof}

\section{Chain Singularities Form a Totally Disconnected Set} \label{Subsec_Topology_of_Chain}
We conclude our analysis of the singularities on the boundary by showing that the set $\cC(\partial E_\eps)$ of chain singularities (types S6--S8) is closed and totally disconnected. This implies that $\cC(\partial E_\eps)$ is nowhere dense, meaning that it is small in the topological sense, even though it may have a positive one-dimensional Hausdorff measure on the boundary.

Before presenting the proof of the above result, we provide a concrete example of a set $E \subset \R^2$ and $\eps > 0$ for which the one-dimensional Hausdorff measure of the set of chain-chain singularities on $\partial E_\eps$ is positive. Essentially, we analyse~\cite[Example 2.2]{Rataj_Winter_On_Volume} from the geometric point of view.

\begin{example}[{A set of chain singularities with positive measure}] \label{Example_Pos_Measure_Set_of_Chain_Sing}
Let $C \subset [0,1]$ be a 'fat' Cantor set (a Cantor set with positive one-dimensional Hausdorff measure) and consider the set $E := \big\{(s,t) \in \R^2 \, : \, s \in C, \, t \in \{0, 1\}  \big\}$. The Cantor set is obtained by removing from the interval $[0,1]$ a certain countable collection $\cI := \{I_n \, : \, n \in \N\}$ of open subintervals. By construction, $C$ is totally disconnected and thus contains no intervals. Since it is uncountable, most of the points $s \in C$ do not lie on the boundary of any of the removed intervals. Denote the collection of these points by $C^* := C \setminus \bigcup_{n  \in \N} \partial I_n$. By construction, every $s \in C^*$ is however an accumulation point of $\bigcup_{n\in\N}I_n$. Since the sets $I_n$ are open, it follows that for $\eps = 1/2$ the sets $V_n := \{ (s, 1/2) \, : \, s \in I_n \}$ satisfy $V_n \subset \R^2 \setminus E_\eps$ for all $n \in \N$. Thus, for every point $x \in A := \{ (s, 1/2) \, : \, s \in C^* \} \subset \partial E_\eps$ there exists a sequence $(w_m)_{m=1}^\infty \subset \R^2 \setminus E_\eps$, where $w_m = (s_m, 1/2) \in V_m$ for some $s_m \in I_{n(m)}$, and $s_m  \to s$ as $m \to \infty$. We can also assume that $I_{n(m)} \neq I_{n(m')}$ whenever $m \neq m'$. For the connected components $W_m \supset V_m$ of the complement $\R^2 \setminus E_\eps$, this implies $W_m \neq W_{m'}$ whenever $m \neq m'$. It is easy to see that condition (ii) in Proposition~\ref{Prop_Topological_Characterisation_of_Chain_Singularities} thus holds true for all $x \in A$ so that $A \subset \cC(\partial E_\eps)$. Hence, $\cH^1(\cC(\partial E_\eps)) \geq \cH^1(A) = \cH^1(C^*) > 0$ due to the translation invariance of Hausdorff measure. 
\end{example}

\subsection{Proof of Theorem~\ref{Thm_Main_3}}
The proof of our third main result builds on many of the results presented in the previous sections. To show that the set of chain singularities is closed, we combine Proposition~\ref{Prop_regular_and_wedge_complement_cone} regarding the connectedness of the complement $\R^2 \setminus E_\eps$ near wedges and $x \in \Unp{E}$ with the characterisation of chain singularities provided by Proposition~\ref{Prop_Topological_Characterisation_of_Chain_Singularities}. The second part of the proof also makes use of the basic results established in Section~\ref{Subsec_Properties_of_Outward_Dir}.

\begin{theorem_without_2}[{The set of chain singularities is closed and totally disconnected}]
For any compact set $E \subset \R^2$ and $\eps > 0$, the set $\cC(\partial E_\eps)$ of chain singularities is closed and totally disconnected.
\end{theorem_without_2}

\begin{proof}
We begin by showing that the complement $\partial E_\eps \setminus \cC(\partial E_\eps)$ is open. To this end, consider some $x \in \partial E_\eps \setminus \cC(\partial E_\eps)$. If $x$ is a wedge (type S1) or if $x \in \Unp{E}$, Proposition~\ref{Prop_regular_and_wedge_complement_cone} implies that there exists some neighbourhood $B_r(x)$ and a connected subset $V_x \subset \R^2 \setminus E_\eps$ for which
\begin{equation} \label{Eq_Connected_Surroundings}
B_r(x) \setminus E_\eps = B_r(x) \cap V_x.
\end{equation}
One the other hand, Proposition~\ref{Prop_Topological_Characterisation_of_Chain_Singularities} states that each chain singularity $x \in \cC(\partial E_\eps)$ is associated with a sequence $(V_n)_{n=1}^\infty \subset \R^2 \setminus E_\eps$ of disjoint connected components of the complement $\R^2 \setminus E_\eps$, for which $\mathrm{dist}_H(x, V_n) \to 0$ as $n \to \infty$. Equation~\eqref{Eq_Connected_Surroundings} hence implies that $B_r(x) \cap \cC(\partial E_\eps) = \varnothing$. Similarly, for a sharp singularity (type S2) or a sharp-sharp singularity (type S3), Proposition~\ref{Prop_Sharp_Singularity_Types} implies the existence of a neighbourhood $B_r(x)$ for which $B_r(x) \cap \cC(\partial E_\eps) = \varnothing$. Hence $\partial E_\eps \setminus \cC(\partial E_\eps)$ is open on the boundary.

To demonstrate that $\cC(\partial E_\eps)$ is totally disconnected, we show that for any two chain singularities $x, z \in \cC(\partial E_\eps)$ there exist disjoint open sets $A_x, A_z \subset \R^2$ for which $x \in A_x, z \in A_z$ and $\cC(\partial E_\eps) \subset (A_x \cup A_z) \cap \partial E_\eps$. More specifically, we will consider for each $x \in \cC(\partial E_\eps)$ and $s_1 \leq 0 \leq s_2$ the sets
\begin{equation} \label{Eq_Sets_A_x(s_1_s_2)}
A_{x, \xi}(s_1, s_2) := \Big \{x + s\xi + t \frac{x - y_1}{\eps} \, : \, s_1 < s < s_2, \,\, -\eps/2 < t < \eps/2 \Big \}
\end{equation}
and show that for each $z \in \cC(\partial E_\eps) \setminus \{x\}$ there exist $\xi \in \Xext{x}$ and $s_1 < 0 < s_2$ for which $\partial A_{x, \xi}(s_1, s_2) \cap \cC(\partial E_\eps) = \varnothing$ and $z \notin A_{x, \xi}(s_1, s_2)$.

Given that $x$ is a chain singularity, Proposition~\ref{Prop_Sharp_Singularity_Types} implies that for $r > 0$ the boundary region $\partial E_\eps \cap U_r(x, \xi)$ exhibits chain-type geometry near $x$ for at least one extremal outward direction $\xi \in \Xext{x}$. We begin by assuming that $\xi$ is such a direction, and consider the corresponding sets $A_{x, \xi}(0, s)$ for $s > 0$. Writing $R(z) := \norm{z - x}$, our aim is to find some $s < R(z)/2$ for which $\partial A_{x, \xi}(0, s) \cap \cC(\partial E_\eps) = \{ x \}$.
To this end, let $\cG(x)$ be a local boundary representation at $x$ with radius $r > 0$ and let $f^{\xi,y_1}, f^{\xi, y_2} :  [0, r] \to \R$ be the continuous functions for which
\[
g_{\xi, y_i}(s) = x + s\xi + f^{\xi, y_i}(s)\frac{x - y_i}{\eps}
\]
for every $g_{\xi, y_i} \in \cG(x)$, $i \in \{1,2\}$ and all $s \in [0, r]$. As argued in the proof of Proposition~\ref{Prop_Sharp_Singularity_Types} (ii), there exist sequences $(s_n)_{n=1}^\infty \subset \R_+$ and $(p_n)_{n=1}^\infty \subset \R_+$ for which
\begin{itemize}
\item $s_n \to 0$ and $p_n \to 0$, as $n \to \infty$,
\item $p_n < s_n \leq p_{n-1} < s_{n-1}$ for all $n \in \N$,
\item $f^{\xi, y_1}(s) + f^{\xi, y_2}(s) = 0$ for all $s \in (s_n)_{n=1}^\infty \cup (p_n)_{n=1}^\infty$, and
\item $f^{\xi, y_1}(s) + f^{\xi, y_2}(s) < 0$ for all $s \in (p_n, s_n)$ and $n \in \N$.
\end{itemize}
It follows that the open set
\[
V_n := \big \{ \tau g_{\xi, y_1}(s) +  (1 - \tau) g_{\xi, y_2}(s) \, : \, \tau \in (0,1), \, s \in (p_n, s_n) \big \}
\]
is a connected component of the complement $\R^2 \setminus E_\eps$ for all $n \in \N$. Consequently there exists some $N \in \N$ for which $\mathrm{dist}_H(V_n, x) < R(z) / 2$ whenever $n \geq N$. The definition of the sets $V_n$ implies that for each $n \geq N$, the boundary point $x_n^{(1)} := g_{\xi, y_1}\big (\frac{p_n + s_n}{2} \big ) \in \partial V_n \subset \partial E_\eps$ has an outward direction aligned with the vector
\begin{align*}
\eta_n^{(1)} &:= g_{\xi, y_2}\left( \frac{p_n + s_n}{2} \right) - g_{\xi, y_1} \left( \frac{p_n + s_n}{2} \right) \\
&= -\left( f^{\xi, y_1}\left( \frac{p_n + s_n}{2}\right) + f^{\xi, y_2}\left( \frac{p_n + s_n}{2}\right) \right) (x - y_1)
\end{align*}
An expression analogous to the above, obtained by replacing the roles of $y_1$ and $y_2$, holds true for $x_n^{(2)}$ and $\eta_n^{(2)}$, defined similarly.
We claim that there exists some $n \geq N$, for which $x_n^{(i)}  \notin \cC(\partial E_\eps)$ for $i \in \{1,2\}$. Assume this were not the case. Then it follows from the definition of chain-singularities and Proposition~\ref{Prop_structure_of_set_of_outward_directions} that for at least one $i \in \{1,2\}$ we have $\eta_n^{(i)} / \big \Vert \eta_n^{(i)} \big \Vert \in \Xext{x_n^{(i)}}$ for infinitely many $n \geq N$. By virtue of the definition of the sets $V_n$ as regions between the graphs $g_{\xi, y_i}$, and since $x_n^{(i)} \in \partial V_n \subset \partial E_\eps$ for all $n \in \N$, it follows from Proposition~\ref{Prop_Tangents_are_Defined} that
\[
\xi = \lim_{n \to \infty} \frac{x_n^{(i)} - x}{\big \Vert x_n^{(i)} - x \big \Vert}
\]
for $i \in \{1,2\}$. But then, due to Lemma~\ref{Lemma_Convergence_of_Contributing_Points_of_Sequences_of_Boundary_Points} (ii)(b), we have
\[
0 = \left \langle \frac{x - y_i}{\eps}, \, \xi \right \rangle = \lim_{n \to \infty} \left \langle \frac{\eta_n^{(i)}}{\big \Vert \eta_n^{(i)} \big \Vert}, \xi \right \rangle = 1,
\]
which is impossible. Thus, there exists some $n \in \N$ for which $x_n^{(i)}  \notin \cC(\partial E_\eps)$ for $i \in \{1,2\}$, which in turn implies that $\partial A_{x, \xi} \big(0, (p_n + s_n) / 2 \big) \cap \cC(\partial E_\eps) = \{ x \}$, since
\[
\partial A_{x, \xi} \left(0, \frac{p_n + s_n}{2}\right) \cap \partial E_\eps = \big\{x, x_n^{(1)}, x_n^{(2)}\big\}.
\]
To conclude the proof we consider one by one the cases of one-sided chain (type S6), chain-chain (type S7) and sharp-chain (type S8) singularities, and identify the sets $A_x$ and $A_z$ mentioned in the beginning of the proof.

{(i)} Assume $x$ is a one-sided chain singularity (type S6), so that $\Xext{x} = \{\xi\}$ for some $\xi \in S^1$. By the argument presented above, there exists some $0 < s_2 < R(z) / 2$ for which $\partial A_{x, \xi} \big(0, s_2 \big) \cap \cC(\partial E_\eps) = \{ x \}$. On the other hand, according to the reasoning presented in the proof of Lemma~\ref{Lemma_countably_many_S6_singularities}, the set
\begin{equation} \label{Eq_quadrangle_in_the_interior_of_E_sub_eps}
Q_p(x) := \Big \{x + s\xi + t\frac{x - y_1}{\eps} \, : \, p < s < 0, \, -\eps/2 < t < \eps/2 \Big \}
\end{equation}
satisfies $Q_p(x) \subset \mathrm{int} \, E_\eps$ for some $p < 0$ (see equation~\eqref{Eq_One_Sided_Chain_Interior}). By setting $s_1 := -\mathrm{min} \big\{p, R(z) / 2 \big\}$ it follows then that $\partial A_{x,\xi}(s_1, s_2) \cap  \cC(\partial E_\eps) = \varnothing$ and we may define $A_x := A_{x,\xi}(s_1, s_2)$ and $A_z := \big( \R^2 \setminus \overline{A_x} \big )$.

{(ii)} Assume then that $x$ is a chain-chain singularity (type S7) with $\Xext{x} = \{\xi, -\xi\}$ for some $\xi \in S^1$. Since now both of the extremal outward directions $\xi$ and $-\xi$ are associated with chain-type geometry, one can again utilise the argument above in order to choose for $\xi_1 := \xi$ and $\xi_2 := -\xi$ the corresponding $s_1, s_2 > 0$ for which $\partial A_{x, \xi_i} \big(0, s_i \big) \cap \cC(\partial E_\eps) = \{ x \}$ and $s_i < R(z) / 2$ for $i \in \{1,2\}$. By setting $s_3 = \mathrm{min} \{s_1, s_2\}$ we may define $A_x :=  A_{x,\xi}(-s_3, s_3)$ and $A_z := \big( \R^2 \setminus \overline{A_x} \big )$.

{(iii)} Finally, assume that $x$ is a sharp-chain singularity (type S8) with $\Xext{x} = \{\xi, -\xi\}$ for some $\xi \in S^1$. We may assume that $\xi$ is associated with chain-type geometry, so that once again we have $\partial A_{x, \xi} \big(0, s_2 \big) \cap \cC(\partial E_\eps) = \{ x \}$ for some $0 < s_2 < R(z) / 2$ due to the arguments presented above. For the direction $-\xi$, Proposition~\ref{Prop_Sharp_Singularity_Types}(i) implies that there exists some $r > 0$ for which
\[
U_{r}(x, -\xi) \setminus E_\eps = V \cap U_{r}(x, -\xi) = \bigcup_{0 < s < r} x + s \big(\alpha(s), \beta(s) \big)_{S^1},
\]
where $V$ is the unique connected component of $\R^2 \setminus E_\eps$ intersecting $U_{r}(x, -\xi)$, $\alpha(s), \beta(s) \in S^1$ for all $s \in (0, r)$ and $\alpha(s), \beta(s) \to -\xi$ as $s \to 0$. It follows then from the definition of chain singularity that $U_r(x, -\xi) \cap \cC(\partial E_\eps)  = \varnothing$. Finally, by setting $s_1 = -\mathrm{min}   \{r/2, \eps/2, R(z)/2 \}$ we may thus define $A_x :=  A_{x,\xi}(s_1, s_2)$ and $A_z := \big( \R^2 \setminus \overline{A_x} \big )$, which completes the proof.
\end{proof}

The set $\cI$ of inaccessible singularities (types S6 and S7) inherits the properties of being totally disconnected and nowhere dense, but it may generally fail to be closed. Since the set $\cC(\partial E_\eps)$ on the other hand is compact and separable as a subset of $\R^2$, it follows from the Cantor-Bendixson Theorem (see~\cite[Thm.\ 6.4]{Kechris_Classical_Descriptive}) that whenever the cardinality of the chain-chain singularities (type S7) is uncountable, the set $\cC(\partial E_\eps)$ can be written as a disjoint union $\cC(\partial E_\eps) = C \cup P$, where $C$ is homeomorphic to the Cantor set and $P$ is countable. For further information on totally disconnected spaces, see \cite{Kechris_Classical_Descriptive}.
\chapter{Complements of {\larger{\textepsilon}}-neighbourhoods as Sets with Positive Reach} \label{Sec_Pos_Reach}
In this chapter we characterise, in terms of local topology around boundary singularities, those $\eps$-neighbourhoods $E_\eps$ whose complement $\overline{\R^2 \setminus E_\eps}$ is a set with positive reach, see Definition~\ref{Def_Positive_Reach}. To reach this, we make use of~\cite[Corollary 3.4]{Fu_Tubular_neighborhoods}, which implies that the complement $\overline{\R^2 \setminus E_\eps} = d_E^{-1}([\eps, \infty))$ has positive reach at the regular values $\eps$ of the distance function
\[
d_E(x) := \inf_{y \in E}\norm{x - y}.
\]
We combine this with Ferry's geometric characterisation~\eqref{Eq_Def_Critical_Points_Ferry} of the critical points of $d_E$ and then interpret the situation in terms of our classification of boundary points, given in Theorem~\ref{Thm_Main_1}.

It follows from Theorem~\ref{Thm_Main_1} that condition~\eqref{Eq_Def_Critical_Points_Ferry} is satisfied at a boundary point $x \in \partial E_\eps$ if and only if there exist extremal contributors $y_1, y_2 \in \Cext{x}$ for which
\begin{equation} \label{Eq_regularity_condition_again}
y_1 - x = -(y_2 - x).
\end{equation}
This corresponds to the local geometry type ({\footnotesize C}) in Figure~\ref{Figure_Three_Basic_Cases}, and means that $x$ is either a sharp singularity (types S2 and S3) or a chain singularity (types S6--S8). It follows from this that if $E \subset \R^2$ has diameter less than $2\eps$, then the boundary $\partial E_\eps$ contains no critical points of the distance function $d_E$. The set $\overline{\R^2 \setminus E_\eps}$ then has positive reach according to~\cite[Corollary 3.4]{Fu_Tubular_neighborhoods}. We use this observation, in conjunction with the local boundary representations given by Proposition~\ref{Prop_local_representation_exists}, to show that for any compact $E \subset \R^2$, the local reach~\eqref{Eq_Def_Reach} of the set $\overline{\R^2 \setminus E_\eps}$ is positive everywhere, with the possible exception of sharp and chain singularities, where the local reach may be zero.

\begin{lemma} \label{Lemma_Positive_Local_Reach_at_UNP_and_wedge_points}
Let $E \subset \R^2$ be compact, $\eps > 0$ and let $x \in \partial E_\eps$ be either a smooth point, a wedge (type S1) or a shallow singularity (types S4 and S5). Then $\mathrm{reach}\big(\overline{\R^2 \setminus E_\eps}, x \big) > 0$.
\end{lemma}

\begin{proof}
Let $\cG(x)$ be a local boundary representation at $x$ with radius $r > 0$. We first argue that for each extremal contributor $y \in \Cext{x}$, there exists some $\delta(y) > 0$ for which the set $A(y) := E \cap B_{\eps / 2}(y)$ satisfies
\begin{equation} \label{Eq_local_reach_for_non_chain_point_01}
\mathrm{reach}\big(\overline{\R^2 \setminus B_\eps (A(y))}, x \big) = \delta(y).
\end{equation}
Since $\mathrm{diam}\big( A(y) \big) < 2\eps$, equation~\eqref{Eq_regularity_condition_again} cannot be satisfied for any boundary point $x \in \partial B_\eps (A_y)).$ Condition~\eqref{Eq_Def_Critical_Points_Ferry} thus guarantees that the boundary $\partial B_\eps (A(y))$ contains no critical points of the distance function $d_{A(y)}$. Thus, \cite[Corollary 3.4]{Fu_Tubular_neighborhoods} implies that the set $\overline{\R^2 \setminus B_\eps (A(y))}$ has positive reach, which in particular implies~\eqref{Eq_local_reach_for_non_chain_point_01} for some $\delta(y) > 0$.

If $x \in \Unp{E}$ (i.e.~if $x$ is a smooth point or a shallow singularity), there is a unique extremal contributor $y$ so that $\Cext{x} = \{y\}$. It then follows from Proposition~\ref{Prop_local_contribution} that in the neighbourhood $B_r(x)$ corresponding to the local boundary representation $\cG(x)$ we have
\[
B_r(x) \cap \overline{\R^2 \setminus E_\eps} = B_r(x) \cap \overline{\R^2 \setminus B_\eps(A(y))}.
\]
Due to~\eqref{Eq_local_reach_for_non_chain_point_01} we know that $B_{\delta(y)}(x) \subset \mathrm{Unp} \big(\overline{\R^2 \setminus B_\eps (A(y))} \big)$, see Definition~\ref{Def_Positive_Reach}. Reducing the radius of this neighbourhood around $x$ further to $\rho(x) := \textrm{min}\{r, \delta(y)\} / 3$ guarantees that for all $z \in B_{\rho(x)}(x)$ we in fact have
\[
\Pi_{\overline{\R^2 \setminus E_\eps}}(z) = \Pi_{\overline{\R^2 \setminus B_\eps(A(y))}}(z),
\]
where $\Pi$ is the metric projection~\eqref{Def_metric_projection}. Hence, $\mathrm{reach}\big(\overline{\R^2 \setminus E_\eps}, x \big) \geq \rho(x)$ whenever $x \in \Unp{E}$.

Assume then that $x$ is a wedge singularity. The set of extremal pairs at $x$ is then of the form $\Pext{x} = \{(\xi_1, y_1), (\xi_2, y_2)\}$, with $\xi_1 \neq \pm\xi_2$. Proposition~\ref{Prop_local_contribution} implies that
\[
B_r(x) \cap \overline{\R^2 \setminus E_\eps} = B_r(x) \cap \bigcap_{i = 1,2} \overline{\R^2 \setminus B_\eps(A(y_i))},
\]
where the radius $r > 0$ corresponds to the local boundary representation $\cG(x)$. Moreover, it follows from Proposition~\ref{Prop_local_contribution} that the set $B_r(x) \cap E_\eps$ can be decomposed into the disjoint union
\begin{align*}
B_r(x) \cap E_\eps &= \Big( U_r(x, \xi_1) \cap B_\eps(A(y_1)) \Big) \cup
						\Big( U_r(x, \xi_2) \cap B_\eps(A(y_2)) \Big) \\
						&\phantom{=} \qquad \qquad
						\cup \Big( B_r(x) \setminus \big(U_r(x, \xi_1) \cup U_r(x, \xi_2) \big) \Big).
\end{align*}
Let $\rho(x) := \textrm{min}\{r, \delta(y_1), \delta(y_2)\} / 3$, where $\delta(y_i)$ is given by~\eqref{Eq_local_reach_for_non_chain_point_01} for $i \in \{1,2\}$. Similarly to the case $x \in \Unp{E}$ described above, restricting to the smaller neighbourhood $B_{\rho(x)}(x)$ allows us to infer that
\[
\Pi_{\overline{\R^2 \setminus E_\eps}}(z) = \Pi_{\overline{\R^2 \setminus B_\eps(A(y_i))}}(z)
\]
for $i \in \{1,2\}$ and all $z \in U_{\rho(x)}(x, \xi_i) \cap B_\eps(A(y_i))$. For these $z$, we thus have $z \in \mathrm{Unp} \big( \overline{\R^2 \setminus E_\eps} \big)$, due to~\eqref{Eq_local_reach_for_non_chain_point_01}. In case $z \in B_{\rho(x)}(x) \setminus \big(U_{\rho(x)}(x, \xi_1) \cup U_{\rho(x)}(x, \xi_2) \big)$ we have $\Pi_{\overline{\R^2 \setminus E_\eps}}(z) = \{x\}$, which again implies $z \in \mathrm{Unp} \big( \overline{\R^2 \setminus E_\eps} \big)$. It thus follows that $\mathrm{reach}\big(\overline{\R^2 \setminus E_\eps}, x \big) \geq \rho(x) > 0$.
\end{proof}

We now proceed to characterise those boundary points $x \in \partial E_\eps$ which satisfy $\mathrm{reach} \big(\overline{\R^2 \setminus E_\eps}, x \big) = 0$.

\begin{prop} \label{Prop_Positive_local_reach_means_locally_connected_complement}
Let $E \subset \R^2$ be compact, $\eps > 0$ and let $x \in \partial E_\eps$. Then $\mathrm{reach}\big(\overline{\R^2 \setminus E_\eps}, x \big) = 0$ if and only if the set $B_r(x) \cap \overline{\R^2 \setminus E_\eps}$ is disconnected for all $r > 0$.
\end{prop}

\begin{proof}
According to Theorem~\ref{Thm_Main_1}, each boundary point belongs to exactly one of the following categories in terms of its local geometry (see Figure~\ref{Figure_Three_Basic_Cases} for an illustration):
\begin{itemize}
\item[(a)] $x \in \Unp{E}$ (smooth points and shallow singularities),
\item[(b)] $\Xext{x} = \{\xi_1, \xi_2\}$ with $\xi_1 \notin \{\xi_2, -\xi_2\}$ (wedges), or
\item[(c)] there exist $y_1, y_2 \in \Cext{x}$ for which $y_1 - x = -(y_2 - x)$ (sharp singularities and chain singularities).
\end{itemize}
For $x \in \partial E_\eps$ in categories (a) and (b), Lemma~\ref{Lemma_Positive_Local_Reach_at_UNP_and_wedge_points} implies $\mathrm{reach}\big(\overline{\R^2 \setminus E_\eps}, x \big) > 0$, and Proposition~\ref{Prop_regular_and_wedge_complement_cone} guarantees the existence of some $r > 0$ for which $B_r(x) \cap \overline{\R^2 \setminus E_\eps}$ is connected.
Thus, we need to verify the claim for boundary points in category (c). Note that in the case of chain singularities, the set $B_r(x) \cap \big( \R^2 \setminus E_\eps \big)$ is by definition disconnected for all $r > 0$, but that this is not necessarily true for the set $B_r(x) \cap \overline{\R^2 \setminus E_\eps}$. The proof below hinges on this fact, which intuitively expresses the observation that there may or may not be a non-zero distance between the disjoint connected components $V_n \subset \R^2 \setminus E_\eps$, defined in Proposition~\ref{Prop_Topological_Characterisation_of_Chain_Singularities}, which constitute the complement of the set $E_\eps$ near the point $x$.

Let $\xi \in \Xext{x}$. Then, if $x$ belongs to category (c) above, the half-ball $U_r(x, \xi)$ exhibits either sharp-type or chain-type geometry as defined in Proposition~\ref{Prop_Sharp_Singularity_Types}(i)-(ii). The same holds true for $U_r(x, -\xi)$, unless $x$ is a one-sided sharp singularity or a one-sided chain singularity. In this latter case, $U_r(x, -\xi) \subset \R^2 \setminus E_\eps$. This implies $\Pi_{\overline{\R^2 \setminus E_\eps}}(z) = \{x\}$ for $z \in U_{r/3}(x, -\xi)$, so that $U_{r/3}(x, -\xi) \subset \mathrm{Unp}(\R^2 \setminus E_\eps)$. To prove the claim, it thus suffices to analyse the geometry of $E_\eps$ inside the half-balls $U_r(x, \xi)$ for $\xi \in \Xext{x}$.

Assume now that for some $\xi \in \Xext{x}$, the set $U_r(x, \xi) \cap \overline{\R^2 \setminus E_\eps}$ is disconnected for all $r > 0$. The same then holds true also for the set $U_r(x, \xi) \cap \big( \R^2 \setminus E_\eps \big)$, which consequently exhibits chain-type geometry as defined in Proposition~\ref{Prop_Sharp_Singularity_Types}(ii).
Let $\cG(x)$ be a local boundary representation at $x$ with radius $r > 0$, let $\xi \in \Xext{x}$ and let $g_{\xi, y_1}, g_{\xi, y_2} \in \cG(x)$, see Proposition~\ref{Prop_local_representation_exists}.
For $i \in \{1,2\}$, consider the continuous functions $f^{\xi, y_i}: [0, r] \to \R$ for which
\begin{equation} \label{Eq_LBR_for_Pos_Reach}
g_{\xi, y_i}(s) = x + s\xi + f^{\xi, y_i}(s)\frac{x - y_i}{\eps}.
\end{equation}
Since $y_1 - x = -(y_2 - x)$, it follows from the definition of the local boundary representation that $z \in U_r(x, \xi) \cap \big( \R^2 \setminus E_\eps \big)$ if and only if $z = x + s\xi + t(x - y_1)/\eps$,
where $s > 0$, $s^2 + t^2 < r^2$ and
\begin{equation} \label{Eq_complement_inequality}
f^{\xi, y_1}(s) < t < -f^{\xi, y_2}(s).
\end{equation}
As argued in the proof of Proposition~\ref{Prop_Sharp_Singularity_Types}(ii), the above characterisation in terms of~\eqref{Eq_LBR_for_Pos_Reach} and~\eqref{Eq_complement_inequality}, together with the disconnectedness of $U_r(x, \xi) \cap \big( \R^2 \setminus E_\eps \big)$ for all $r > 0$, implies that there exist sequences $(s_n)_{n=1}^\infty \subset \R_+$ and $(p_n)_{n=1}^\infty \subset \R_+$ for which
\begin{itemize}
\item $s_n \to 0$ and $p_n \to 0$, as $n \to \infty$,
\item $p_n < s_n \leq p_{n-1} < s_{n-1}$ for all $n \in \N$,
\item $f^{\xi, y_1}(s) + f^{\xi, y_2}(s) = 0$ for all $s \in (s_n)_{n=1}^\infty \cup (p_n)_{n=1}^\infty$, and
\item $f^{\xi, y_1}(s) + f^{\xi, y_2}(s) < 0$ for all $s \in (p_n, s_n)$ and $n \in \N$.
\end{itemize}
It follows then for all $n \in \N$ that $g_{\xi, y_1}(s_n), g_{\xi, y_1}(p_n) \in \partial E_\eps$ and the open set
\begin{equation} \label{Eq_Sets_V_n_in_the_complement}
V_n := \big \{ \tau g_{\xi, y_1}(s) +  (1 - \tau) g_{\xi, y_2}(s) \, : \, \tau \in (0,1), \, s \in (p_n, s_n) \big \}
\end{equation}
is a connected component of the complement $\R^2 \setminus E_\eps$. Crucially, for each $n \in \N$, the union $\overline{V_n} \cup \overline{V_{n+1}}$ is a connected set if and only if $\overline{V_n} \cap \overline{V_{n+1}} \neq \varnothing$, which is equivalent to $s_{n+1} = p_n$. Geometrically, this means that the closures of the consecutive sets $V_n$ and $V_{n+1}$ touch at the point $g_{\xi, y_1}(s_{n+1}) = g_{\xi, y_1}(p_n)$.

The above reasoning implies that the union $\bigcup_{n \geq M} \overline{V_n}$ is disconnected for all $M \in \N$. This in turn implies that $0$ is an accumulation point of the set
\begin{equation} \label{Eq_indices_of_interior_points}
S := \big\{s \in (0, r) \, : \, f^{\xi, y_1}(s) + f^{\xi, y_2}(s) > 0 \big\}.
\end{equation}
Consequently, there exist subsequences $(s_{n_k})_{k = 1}^\infty$ and $(p_{n_k})_{k=1}^\infty$ for which the pairs $(s_{n_k}, p_{n_k - 1})$ satisfy $s_{n_k} < p_{n_k - 1}$ for all $k \in \N$. Geometrically, the proper inequality here indicates a gap of positive length between the consecutive connected components $V_{n_k}$ and $V_{n_k - 1}$. Now, for each $k \in \N$, consider the midpoint
\[
x_k := \frac{g_{\xi, y_1}(s_{n_k}) + g_{\xi, y_1}(p_{n_k - 1})}{2} \in \mathrm{int} \, E_\eps.
\]
Let $\rho := \norm{g_{\xi, y_1}(s_{n_k}) - g_{\xi, y_1}(p_{n_k - 1})} / 2$. Then
\[
B_\rho(x_k) \cap \overline{\R^2 \setminus E_\eps} = \{g_{\xi, y_1}(s_{n_k}), g_{\xi, y_1}(p_{n_k - 1})\}.
\]
It follows that
\[
\Pi_{\overline{\R^2 \setminus E_\eps}}(x_k) = \{g_{\xi, y_1}(s_{n_k}), g_{\xi, y_1}(p_{n_k - 1})\},
\]
so that $x_k \notin \mathrm{Unp} \big(\overline{\R^2 \setminus E_\eps} \big)$.
Since $x_k \to x$ as $k \to \infty$, this implies $\mathrm{reach}\big(\overline{\R^2 \setminus E_\eps}, x \big) = 0$ as claimed.

To prove the implication in the other direction, assume there exists some $r > 0$ for which the sets $U_r(x, \xi) \cap \overline{\R^2 \setminus E_\eps}$ are connected for each $\xi \in \Xext{x}$. According to Proposition~\ref{Prop_Sharp_Singularity_Types}, these sets exhibit either sharp-type or chain-type geometry. For sharp-type geometry, the set $U_r(x, \xi) \cap \overline{\R^2 \setminus E_\eps}$ consists of a single connected set, which implies that the graphs $g_{\xi, y_i}((0, r])$, $i \in \{1,2\}$ do not touch. In the case of chain-type geometry, these sets are composed of a (countable) collection of sets of type~\eqref{Eq_Sets_V_n_in_the_complement} which thus must form an unbroken chain. Hence, the set of indices~\eqref{Eq_indices_of_interior_points} contains only isolated points. Geometrically this means that the graphs $g_{\xi, y_i}([0, r])$, $i \in \{1,2\}$ do not cross transversally.

Analogously to the proof of Lemma~\ref{Lemma_Positive_Local_Reach_at_UNP_and_wedge_points}, Proposition~\ref{Prop_local_contribution} implies for each $\xi \in \Xext{x}$ the decomposition
\[
U_r(x, \xi) \cap E_\eps = \Big( U_r(x, \xi) \cap B_\eps(A(y_1)) \Big) \cup
						\Big( U_r(x, \xi) \cap B_\eps(A(y_2)) \Big).
\]
Now, as argued in the proof of Lemma~
\ref{Lemma_Positive_Local_Reach_at_UNP_and_wedge_points}, it follows from~\cite[Corollary 3.4]{Fu_Tubular_neighborhoods} that for $i \in \{1,2\}$ the set $A(y_i) := E \cap B_{\eps / 2}(y_i)$ satisfies
\begin{equation} \label{Eq_local_reach_for_non_chain_point_01_again}
\mathrm{reach}\big(\overline{\R^2 \setminus B_\eps (A(y_i))}, x \big) > 0. 
\end{equation}
Since the graphs $g_{\xi, y_i}([0, r])$ do not cross, we furthermore have
\[
U_r(x, \xi) \cap \partial B_\eps(A(y_i)) = \{ g_{\xi, y_i}(s) \, : \, s \in [0,r]\}
\]
for $i \in \{1,2\}$. Hence,
\[
\Pi_{\overline{\R^2 \setminus E_\eps}}(z) = \Pi_{\overline{\R^2 \setminus B_\eps (A(y_i))}}(z)
\]
whenever $z \in U_{r/3}(x, \xi) \cap B_\eps(A(y_i))$ for $i \in \{1,2\}$. It then follows from~\eqref{Eq_local_reach_for_non_chain_point_01_again} that these projections are singletons, so that 
$\textrm{reach} \big( \overline{\R^2 \setminus E_\eps} \big) > 0$.
\end{proof}

The proof of Proposition~\ref{Prop_Positive_local_reach_means_locally_connected_complement} shows that for all sharp singularities (types S2 and S3) and even for certain chain singularities (types S6--S8), the local $\textrm{reach}\big( \overline{\R^2 \setminus E_\eps}, x \big)$ is strictly positive. Combining this with Lemma~\ref{Lemma_Positive_Local_Reach_at_UNP_and_wedge_points} we obtain the following corollary.

\begin{cor}[{Non-chain singularities have positive reach}] \label{Cor_Positive_Reach_everywhere_except_Chain_Sing}
Let $E \subset \R^2$ be compact, $\eps > 0$ and let $x \in \partial E_\eps$. If $x$ is not a chain singularity (types S6--S8), then $\mathrm{reach}\big(\overline{\R^2 \setminus E_\eps}, x \big) > 0$, and there exists some $r > 0$ for which the set $B_r(x) \cap \overline{\R^2 \setminus E_\eps}$ is connected.
\end{cor}

Around sharp singularities, and those chain singularities that have positive local reach, the geometry of the set $\overline{\R^2 \setminus E_\eps}$ corresponds to the geometric types $T^2$ (for sharp-sharp (type S3) and chain-chain (type S7) singularities) and $T^3$ (for sharp (type S2) and chain (type S6) singularities) presented in~\cite[Definition 6.3]{Rataj_Zajicek_On_the_structure_of_sets_with_PR}. In particular, it follows that the existence of critical points of the distance function $d_E$ on the boundary $\partial E_\eps$ does not directly imply that the set $\overline{\R^2 \setminus E_\eps}$ fails to have positive reach.

\subsection{Proof of Theorem~\ref{Thm_Main_4}} We now provide the proof of our fourth main result, Theorem~\ref{Thm_Main_4}, which characterises, in terms of a local connectedness property, those $\eps$-neighbourhoods whose complement $\overline{\R^2 \setminus E_\eps}$ is a set with positive reach.

\begin{theorem_without_2}[{Local geometry of $\eps$-neighbourhoods whose complements have positive reach}] Let $E \subset \R^2$ be compact and $\eps > 0$. Then the following are equivalent:
\begin{itemize}
\item[(i)] the complement $\overline{\R^2 \setminus E_\eps}$ is a set with positive reach,
\item[(ii)] for each $x \in \partial E_\eps$, there exists some $r := r(x) > 0$ for which the set $B_r(x) \cap \overline{\R^2 \setminus E_\eps}$ is connected.
\item[(iii)] for each chain singularity $x \in \cC(\partial E_\eps)$, there exists some $r := r(x) > 0$ for which the set $B_r(x) \cap \overline{\R^2 \setminus E_\eps}$ is connected.
\end{itemize}
\end{theorem_without_2}

\begin{proof}
(i) $\Rightarrow$ (ii): Assume the set $\overline{\R^2 \setminus E_\eps}$ has positive reach. Then, by definition,
\[
\inf_{x \in \partial E_\eps} \mathrm{reach}\big(\overline{\R^2 \setminus E_\eps}, x \big) > 0.
\]
In particular, there exist no points $x \in \partial E_\eps$ for which $\mathrm{reach}\big(\overline{\R^2 \setminus E_\eps}, x \big) = 0$. Then, (ii) follows from Proposition~\ref{Prop_Positive_local_reach_means_locally_connected_complement} .

(ii) $\Rightarrow$ (iii): Trivial.

(iii) $\Rightarrow$ (ii): Follows directly from Corollary~\ref{Cor_Positive_Reach_everywhere_except_Chain_Sing}.

(ii) $\Rightarrow$ (i): Assume that (i) fails. This implies that there exists a sequence $(x_n)_{n=1}^\infty \subset \partial E_\eps$ with
\begin{equation} \label{Eq_limit_of_reach}
\lim_{n\to\infty} \mathrm{reach}\big(\overline{\R^2 \setminus E_\eps}, x_n \big) = 0.
\end{equation}
Since $\partial E_\eps$ is compact, there exists a convergent subsequence $(x_{n_k})_{k=1}^\infty$ with $x_0 := \lim_{k\to\infty} x_{n_k} \in \partial E_\eps$.
Now, combining (ii) with Proposition~\ref{Prop_Positive_local_reach_means_locally_connected_complement} implies that there exists some $r := r(x_0) > 0$ for which
\begin{equation} \label{Eq_reach_at_limit_point_x}
\mathrm{reach}\big(\overline{\R^2 \setminus E_\eps}, x_0 \big) > r.
\end{equation}
According to~\eqref{Eq_limit_of_reach} there exists some $k_0 \in \N$ such that for all $k > k_0$ we have simultaneously $\norm{x_{n_k} - x_0} < r/2$ and
\begin{equation} \label{Eq_reach_upper_bound}
\mathrm{reach}\big(\overline{\R^2 \setminus E_\eps}, x_{n_k} \big) < r/4.
\end{equation}
On the other hand, for a fixed $K > k_0$, it follows from $\norm{x_{n_K} - x} < r/2$ and~\eqref{Eq_reach_at_limit_point_x} that
\[
\mathrm{reach}\big(\overline{\R^2 \setminus E_\eps}, x_{k_N} \big) \geq r/2,
\]
since $B_{r/2}(x_{k_N}) \subset B_r(x_0)$. This contradicts~\eqref{Eq_reach_upper_bound} and finishes the proof.
\end{proof}
\chapter{Jordan Curves on the Boundary} \label{Sec_Boundaries_as_Jordan_Curves}
In this chapter we show that the boundary $\partial E_\eps$ is the disjoint union of a possibly uncountable set $\cI$ of inaccessible singularities and an at most countably infinite collection $J$ of Jordan curves. By definition, inaccessible singularities do not lie on the boundary of any connected component of the complement. These correspond to \emph{chain} (type S6) and \emph{chain-chain singularities} (type S7), see Figure~\ref{Figure_Theorem_1}({\footnotesize B}). To prove Theorem~\ref{Thm_global_structure}, it thus suffices to obtain a representation for the boundaries of connected components of the complement as a union of Jordan curves. 
Intuitively, these Jordan curves can be constructed by connecting adjacent simple curves that represent local boundary segments. There is an essentially unique way of connecting such curve representations around all types of boundary points, except sharp-sharp singularities, where the boundary curves 'split', see Figure~\ref{Figure_Topologically_Different_Sharps}. However, it turns out that for each connected component $U$ of the complement $\R^2 \setminus E_\eps$, there exists a unique way of connecting the curves around sharp-sharp singularities that results in a representation of the boundary $\partial U$ as a finite union of Jordan curves.

\section{The Role of Sharp-sharp Singularities in Boundary Topology}
Let $U$ be a connected component of the complement $\R^2 \setminus E_\eps$ and let $Q$ denote the set of sharp-sharp singularities on $\partial U$. According to Lemma~\ref{Lemma_Finitely_Many_Sharp_Singularities_per_Component}, the set $Q$ is finite, and Proposition~\ref{Prop_Sharp_Singularity_Types} implies that each $x \in Q$ lies on the boundary of at most two connected components of the complement $\R^2 \setminus E_\eps$.
Hence, the set $Q$ can be written as the disjoint union
\begin{equation} \label{Eq_Two_Sharp_Sing_Types}
Q = Q_1 \cup Q_2,
\end{equation}
where $Q_1$ contains those sharp-sharp singularities that lie on the boundary of a unique connected component $U$ of the complement $\R^2 \setminus E_\eps$, and $Q_2$ contains 
the sharp-sharp singularities for which there exist two disjoint connected components $U, V \in \R^2 \setminus E_\eps$ with $x \in \partial U \cap \partial V$. The fact that both $Q_1$ and $Q_2$ may be non-empty is the reason sharp-sharp singularities play a central role for the topological structure of the boundary $\partial U$.

\begin{figure}[h]
      \centering
      \captionsetup{margin=0.75cm}
                \includegraphics[width = \textwidth]{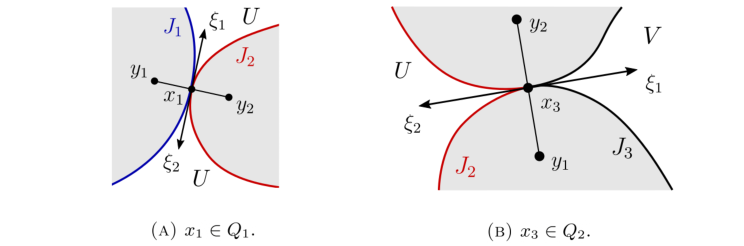} \\[0mm]   
                \caption{Topology of the complement $\R^2 \setminus E_\eps$ around sharp-sharp singularities. The figures here are close-ups of the points $x_1$ and $x_3$ in Figure~\ref{Figure_Theorem_1}({\footnotesize A}). ({\footnotesize A}) The sharp-sharp singularity $x_1 \in Q_1$ lies on the boundary of a single connected component $U$ of the complement $\R^2 \setminus E_\eps$. Here the curves $\Gamma_{x_1}^{1}$ (blue) and $\Gamma_{x_1}^{2}$ (red) pass through $x_1$ along the boundary subsets $J_1$ and $J_2$. ({\footnotesize B}) The point $x_3 \in Q_2$ lies on the boundary of two different connected components $U, V \subset \R^2 \setminus E_\eps$. The curve $\Gamma_{x_3}$ (red) bounces back at $x_3$ along the boundary subset $J_2$.}
                \label{Figure_Topologically_Different_Sharps}
\end{figure}

The boundary $\partial U$ of a single connected component $U$ may thus consist of more than one Jordan curve for two reasons. First, $U$ need not be simply connected, and can thus surround several connected components $A \subset E_\eps$. The boundary of each such $A$ contains at least one Jordan curve. Second, if $\partial U$ contains sharp-sharp singularities of type $Q_1$, more than one Jordan curve is needed to represent the boundary. Any single curve containing the boundary subset $B_r(x) \cap \partial U$ for any $x \in Q_1$ would intersect itself at $x$, and would thus not be a Jordan curve.

\section{Boundaries of Connected Components of the Complement}
Let the sets $U$ and $Q = Q_1 \cup Q_2$ be as above. According to Proposition~\ref{Prop_local_representation_exists}, each $x \in \partial U$ has a neighbourhood $B_x := B_r(x)$,
within which the local boundary subset $\partial U \cap \overline{B_x}$ can be represented as a union of either two or four graphs of continuous functions. We show in Lemma~\ref{Lemma_Finite_Repr_for_Boundary} below that it is possible to use these functions to represent the boundary subset $\partial U \cap \overline{B_x}$ with either one (for $x \in \partial U \setminus Q_1$) or two (for $x \in Q_1$) simple curves.

Before stating the lemma, we define the basic terminology regarding curves. We give the definition in a general setting, but we will only be dealing with plane curves, i.e.~curves in $\R^2$.

\begin{definition}[{Curve, simple curve, Jordan curve}] \label{Def_Jordan_Curve}
A \emph{curve} is the image $\Gamma = \gamma([a,b])$ of a closed interval $[a,b]$ under some continuous map $\gamma : [a,b] \to \R^d$. If $\gamma$ is injective, we call $\Gamma$ a \emph{simple curve}, and if injectivity is violated only at the endpoints, i.e. $\gamma(a) = \gamma(b)$, then we say that $\Gamma$ is a \emph{simple closed curve} or a \emph{Jordan curve}.
\end{definition}

\begin{lemma}[{Boundaries of connected components of the complement as unions of simple curves}] \label{Lemma_Finite_Repr_for_Boundary}
Let $E \subset \R^2$ be compact, let $\eps > 0$ and let $U$ be a connected component of the complement $\R^2 \setminus E_\eps$. For each $x \in \partial U$, let $B_x := B_r(x)$ be the neighbourhood in~\eqref{Eq_local_boundary_repr_condition}, corresponding to the local boundary representation $\cG(x)$ given by Proposition~\ref{Prop_local_representation_exists}. In addition, denote by $Q$ the set of sharp-sharp singularities on $\partial U$, and let the subsets $Q_1, Q_2 \subset Q$ be as in~\eqref{Eq_Two_Sharp_Sing_Types}.
Then
\begin{itemize}
\item[(i)] for each $x \in Q_1$ there exist continuous functions $\gamma_x^{(1)}, \gamma_x^{(2)} : [0,1] \to \partial U$ for which the images $\Gamma_x^{(i)} := \gamma_x^{(i)}([0,1])$ with $i \in \{1,2\}$ are simple curves, $\partial U \cap \overline{B_x} = \Gamma_x^{(1)} \cup \Gamma_x^{(2)}$ and $\Gamma_x^{(1)} \cap \Gamma_x^{(2)} = \{x\}$;
\item[(ii)] for each $x \in \partial U \setminus Q_1$ there exists a continuous function $\gamma_x : [0,1] \to \partial U$, for which the image $\Gamma_x := \gamma_x([0,1])$ is a simple curve and $\partial U \cap \overline{B_x}= \Gamma_x$.
\end{itemize}
\end{lemma}

\begin{proof}

\begin{figure}[h]
      \centering  \vspace{-1mm}
      \captionsetup{margin=0.75cm}
                \includegraphics[width = \textwidth]{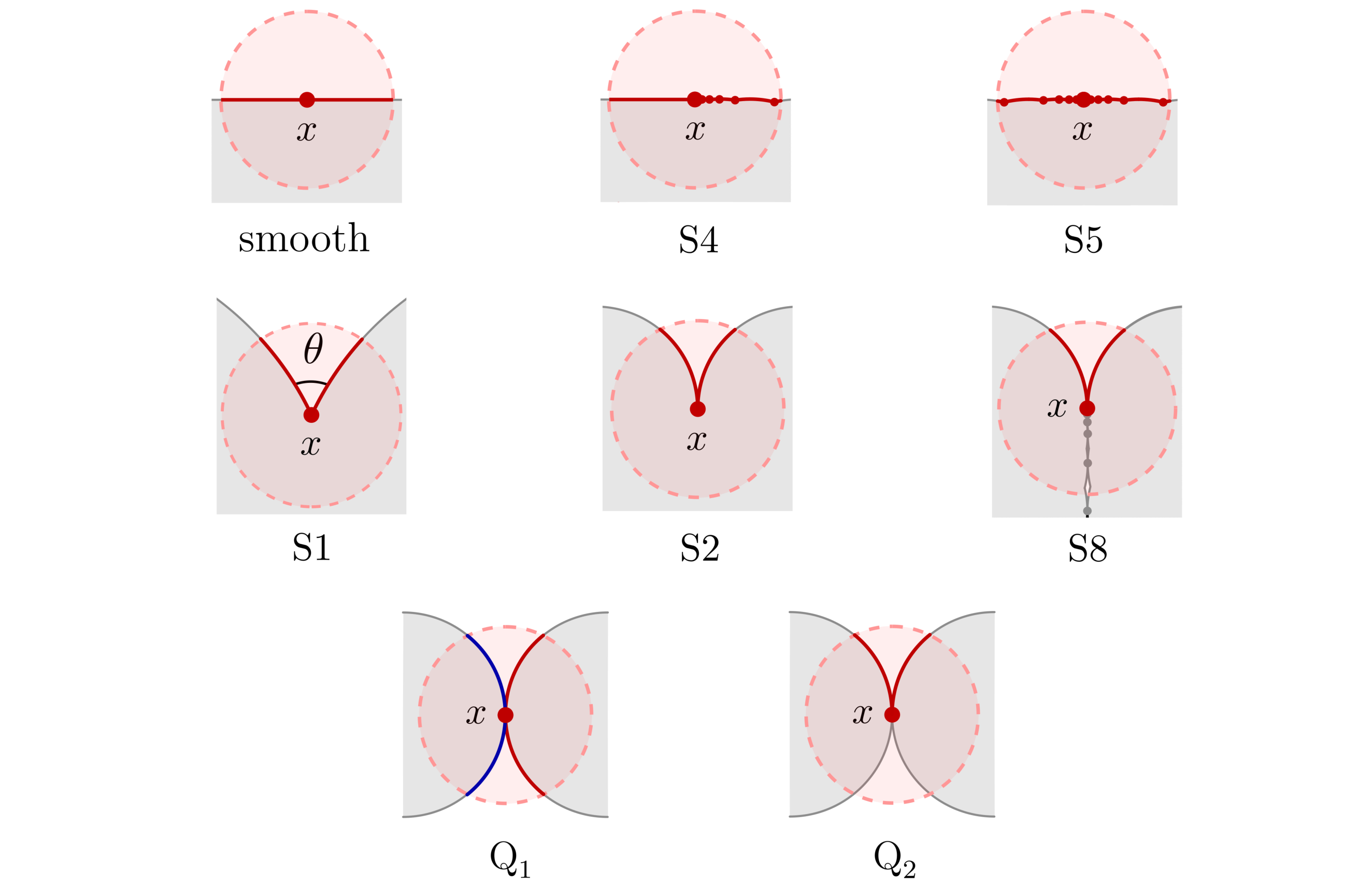} \\[0mm]   
                \caption{Local curve representations. In each picture, $x$ represents the boundary point in question, the grey region represents the set $E_\eps$, and white regions the complement $\R^2 \setminus E_\eps$. The red dashed balls correspond to the neighbourhoods $B_r(x)$ associated with the local boundary representations $\cG(x)$, see \eqref{Eq_local_boundary_repr_condition}. The dark red curves (and the dark blue curve for $x \in Q_1$) represent the simple curve representations $\Gamma_x = \gamma_x([0,1])$ of $\partial U \cap \overline{B_r(x)}$. Chain (type S6) and chain-chain (type S7) singularities are absent as these do not appear on the boundaries of connected components of the complement, see Corollary~\ref{Cor_Inaccessible_Singularities}. The types $Q_1$ and $Q_2$ together represent sharp-sharp (type S3) singularities. Note that points of type $Q_1$ are the only ones for which two simple curves are needed in a local representation of the boundary of a fixed connected component $U$ of the complement $\R^2 \setminus E_\eps$, see also Figure~\ref{Figure_Topologically_Different_Sharps}.
                 }
                \label{Figure_Local_Curve_Representations}
\end{figure}

We divide the proof into two steps.
In the first step, we define the local curve representations of the boundary $\partial U$ around sharp-sharp singularities $x \in Q$. We also show that the images $\Gamma_x$ and $\Gamma_x^{(i)}$ are simple (i.e. non-self-intersecting) curves, and that $\Gamma_x^{(1)} \cap \Gamma_x^{(2)} \cap B_r(x) = \{x\}$ for sufficiently small radii $r > 0$.
In the second step, we define curve representations, analogous to those defined in Step~1, for the other types of boundary points $x \in \partial U \setminus Q$.

\emph{Step 1.} Consider a boundary point $x \in \partial U$, and let $\cG(x) = \{g_{\xi,y} \, : \, (\xi,y) \in \Pext{x} \}$ be a local boundary representation at $x$, with the functions $g_{\xi,y}$ defined as in~\eqref{Eq_Canonical_LBR}.
By the definition of sharp-sharp singularities, the set of extremal outward directions for each $x \in Q$ satisfies 
$\Xext{x} = \{\xi, -\xi\}$ for some $\xi \in S^1 \subset \R^2$, see Definition~\ref{Def_classification_of_singularities}. Proposition~\ref{Prop_Sharp_Singularity_Types} furthermore states that there exist connected components $V_\xi, V_{-\xi}$ of the complement $\R^2 \setminus E_\eps$, for which 
\begin{equation} \label{Eq_The_Two_Touching_Components}
\partial U \cap \overline{B_x} \subset \big(\partial V_\xi \cup \partial V_{-\xi} \big) \cap \overline{B_x}.
\end{equation}
Since $Q \subset \partial U$, we have $U \in \{V_\xi, V_{-\xi}\}$ and we may define $\xi$ so that $U = V_\xi$. The correct way of defining the local curve representation near $x$ now depends on whether $x \in Q_1$ or $x \in Q_2$, where $Q = Q_1 \cup Q_2$ as in~\eqref{Eq_Two_Sharp_Sing_Types}. We deal with these cases separately as follows:
\begin{itemize}
  \item[(i)] Let $x \in Q_1$. This implies that $U$ touches the point $x$ from both directions, so that $U =  V_\xi = V_{-\xi}$. This corresponds to the situation in Figure~\ref{Figure_Topologically_Different_Sharps}({\footnotesize A}). Then
\begin{equation} \label{Eq_Curve_Representation_for_Sharp-Sharp_Q1}
\partial U \cap \overline{B_x} = \bigcup_{i \in \{1,2\}} \Big( \,g_{\xi, y_i}\big([0, s_{\xi,y_i}]\big) \cup g_{-\xi, y_i}\big([0, s_{-\xi_i,y_i}]\big) \hspace{-0.5mm}\Big).
\end{equation}
In this case, we associate a curve segment with each extremal contributor $y_1, y_2 \in \Cext{x}$ by
defining for each $i \in \{1,2\}$ the continuous maps $\gamma_x^{(i)} : \left[0, 1\right] \to \partial U \cap \overline{B_x}$ through
\begin{equation} \label{Eq_Sharp_Path_Definition}
\gamma_x^{(i)}(s) := \left\{\begin{array}{ll} \vspace{-0mm}
        g_{-\xi, y_i}\big(s_{-\xi, y_i} - sL^{(i)}\big),  & \text{if } s \in \Big[0, \frac{s_{-\xi, y_i}}{L^{(i)}}\Big),\vspace{1.5mm} \\
        g_{\xi, y_i}\big(sL^{(i)} - s_{-\xi, y_i}\big),  & \text{if } s \in \Big[\frac{s_{-\xi, y_i}}{L^{(i)}}, 1\Big], \vspace{1mm}\\
        \end{array} \right.
\end{equation}
where $g_{-\xi, y_i}, g_{\xi, y_i} \in \cG (x)$, and $L^{(i)} := s_{\xi, y_i} + s_{-\xi, y_i}$ with $s_{\xi, y_i}$ and $s_{-\xi, y_i}$ as in~\eqref{Eq_Curve_Representation_for_Sharp-Sharp_Q1}. We thus obtain $\partial U \cap \overline{B_x} = \Gamma_x^{(1)} \cup \Gamma_x^{(2)}$, where $\Gamma_x^{(i)} := \gamma_x^{(i)}([0,1])$ for $i \in \{1,2\}$.


\item[(ii)] Let $x \in Q_2$. Then $U = V_\xi \neq V_{-\xi}$ so that the graphs $g([0, s_{-\xi, y_1}])$ and $g([0, s_{-\xi, y_2}])$ 
intersect the local boundary subset $\partial U \cap B_x$ only at $x$. They are thus not relevant for a local representation of $\partial U$ at $x$. This corresponds to the situation in Figure~\ref{Figure_Topologically_Different_Sharps}({\footnotesize B}). Consequently, we have
\begin{equation} \label{Eq_Curve_Representation_for_Sharp-Sharp_Q2}
\partial U \cap \overline{B_x} = g_{\xi, y_1} \big([0, s_{\xi,y_1}] \big) \cup g_{\xi, y_2} \big([0, s_{\xi,y_2}] \big)
\end{equation}
for $g_{\xi, y_1}, g_{\xi, y_2} \in \cG (x)$. We thus define
\begin{equation} \label{Eq_Sharp_Path_Definition_Q2}
\gamma_x(s) := \left\{\begin{array}{ll} \vspace{-0mm}
        g_{\xi, y_1}\big(s_{\xi, y_1} - sL\big),  & \text{if } s \in \Big[0, \frac{s_{\xi, y_1}}{L}\Big),\vspace{1.5mm} \\
        g_{\xi, y_2}\big(sL - s_{\xi, y_1}\big),  & \text{if } s \in \Big[\frac{s_{\xi, y_1}}{L}, 1\Big], \vspace{1mm}\\
        \end{array} \right.
\end{equation}
where $g_{\xi, y_1}, g_{\xi, y_2} \in \cG (x)$, and $L := s_{\xi, y_1} + s_{\xi, y_2}$ with $s_{\xi, y_1}$ and $s_{\xi, y_2}$ as in~\eqref{Eq_Curve_Representation_for_Sharp-Sharp_Q2}. In this case, we obtain a representation using only one curve, since $\partial U \cap \overline{B_x} = \Gamma_x :=  \gamma_x([0,1])$.
\end{itemize}

For $x \in Q_1$, we define $\Gamma_x^{(i)} := \gamma_x^{(i)}([0,1])$ for $i \in \{1,2\}$, where the functions $\gamma_x^{(i)}$ are as in~\eqref{Eq_Sharp_Path_Definition}. Then $\partial U \cap \overline{B_x} = \Gamma_x^{(1)} \cup \Gamma_x^{(2)}$, and it follows directly from the definition of the functions $g_{\xi,y} \in \cG(x)$ that the image $\Gamma_x^{(i)}$ is a simple curve for each $i \in \{1,2\}$. Note furthermore that sharp-sharp singularities by definition exhibit sharp-type geometry in the direction of both extremal outward directions $\xi_1, \xi_2 \in \Xext{x}$, see Proposition~\ref{Prop_Sharp_Singularity_Types}(i). Proposition~\ref{Prop_Sharp_Singularity_Types} thus implies the existence of some $r_1, r_2 > 0$ for which $g_{\xi_i, y_1}(s) \neq g_{\xi_i, y_2}(s)$ for all $s \in (0, r_i)$ and $i \in \{1,2\}$. Hence, $\Gamma_x^{(1)} \cap \Gamma_x^{(2)} \cap B_r(x) = \{x\}$ for all $r < \min\{r_1, r_2\}$. By choosing $r < \min\{r_1, r_2\}$ for the local boundary representation in the first place, we obtain the equivalent property $\Gamma_x^{(1)} \cap \Gamma_x^{(2)}  = \{x\}$.

For $x \in Q_2$, we define $\Gamma_x := \gamma_x([0,1])$, where the function $\gamma_x$ is as in~\eqref{Eq_Sharp_Path_Definition_Q2}. The fact that $\partial U  \cap \overline{B_x} = \Gamma_x$  follows from the definition of the curve $\gamma_x$.
Similarly to the above, Proposition~\ref{Prop_Sharp_Singularity_Types} implies that $g_{\xi, y_1}(s) \neq g_{\xi, y_2}(s)$ for sufficiently small $s \in (0, r)$. It follows that the curve $\Gamma_x$ is non-self-intersecting in every sufficiently small neighbourhood $B_x$.

\emph{Step 2.}
%
We divide $\partial U \setminus Q$ into the following two subsets:
\begin{align}
R &:= \big\{x \in \partial U \, : \, x \,\, \textrm{is a wedge or } x \in \Unp{x} \big\}, \label{Eq_Sets_R_and_P_1} \\
P &:= \big\{x \in \partial U \, : \, x \,\, \textrm{is a sharp or a sharp-chain singularity} \big\}.\label{Eq_Sets_R_and_P_2}
\end{align}
According to Corollary~\ref{Cor_Inaccessible_Singularities}, chain and chain-chain singularities (types S6 and S7) do not appear on the boundaries of connected components of the complement $\R^2 \setminus E_\eps$, which implies $\partial U = R \cup P \cup Q$. Furthermore, it follows from Theorem~\ref{Thm_Main_1} that the sets $R$, $P$ and $Q$ are disjoint.

Assume first that $x \in R$, so that either $x \in \Unp{x}$ or $x$ is a wedge. Then $\Pext{x} = \{(\xi_1, y_1), (\xi_2, y_2)\}$, where for $y_1 = y_2$ in case $x \in \Unp{x}$. It follows immediately from Lemma~\ref{Lemma_Unique_Connected_Component} and Proposition~\ref{Prop_regular_and_wedge_complement_cone} that $\overline{B_x} \setminus E_\eps = \overline{B_x} \cap U$, so that
\begin{equation}  \label{Eq_wedge_Unp_expression}
\partial U \cap \overline{B_x} = g_{\xi_1, y_1}\big([0, s_{\xi_1,y_1}]\big) \cup g_{\xi_2, y_2}\big([0, s_{\xi_2,y_2}]\big),
\end{equation}
where $g_{\xi_1, y_1}, g_{\xi_2, y_2} \in \cG (x)$.
By normalising the arguments with $L := s_{\xi_1,y_1} + s_{\xi_2,y_2}$, one can thus define a function $\gamma_x : [0,1] \to \partial U$ analogously to~\eqref{Eq_Sharp_Path_Definition_Q2}, so that $\partial U \cap \overline{B_x} = \Gamma_x := \gamma_x([0,1])$. The fact that $\Gamma_x$ defines a simple curve in some neighbourhood $B_r(x)$ follows directly from Proposition~\ref{Prop_regular_and_wedge_complement_cone}.

Let then $x \in P$, so that $\Xext{x} = \{\xi, -\xi\}$ for some $\xi \in S^1$. In the following, we denote by $U_r(x, v)$ an open $x$-centered half-ball of radius $r$, oriented in the direction of $v \in S^1$, see~\eqref{Def_oriented_half_ball}. For both sharp and sharp-chain singularities, there exists a unique extremal outward direction $v \in \{\xi, -\xi\}$, towards which the singularity exhibits sharp-type geometry, as defined in Proposition~\ref{Prop_Sharp_Singularity_Types}(i). This means that there exists a unique connected component $V_\xi$ of the complement $\R^2 \setminus E_\eps$, and $s_\xi > 0$, for which $U_{s_\xi}(x, \xi) \setminus E_\eps = U_{s_\xi}(x, \xi) \cap V_\xi$.

Moreover, in the case of sharp-chain singularities, one can use the local boundary representation~\eqref{Eq_Canonical_LBR} and Proposition~\ref{Prop_Sharp_Singularity_Types} to show that there exist no connected components $V$ of the complement $\R^2 \setminus E_\eps$,
%
for which 
$x \in \partial (V \cap U_{r}(x, -\xi))$.
The detailed argument can be found in the proof of Corollary~\ref{Cor_Inaccessible_Singularities}. For one-sided sharp singularities, this follows immediately from their definition, which states that for some $\delta > 0$, the set $B_\delta(x) \setminus E_\eps$ is a connected set.

Hence, for both sharp and sharp-chain singularities there exists some $r > 0$ for which
\begin{equation} \label{Eq_Sharp_Singularity_Expression}
\partial U \cap \overline{B_r(x)} = g_{\xi, y_1}\big([0, s_{\xi, y_1}]\big) \cup g_{\xi, y_2}\big([0, s_{\xi, y_2}]\big),
\end{equation}
where $g_{\xi, y_1}, g_{\xi, y_2} \in \cG (x)$ are as in Proposition~\ref{Prop_local_representation_exists}. One may thus once more define the function $\gamma_x : [0,1] \to \partial U$ analogously to~\eqref{Eq_Sharp_Path_Definition_Q2} by concatenating the images $g_{\xi, y_i}\big([0, s_{\xi, y_i}]\big)$ in~\eqref{Eq_Sharp_Singularity_Expression} and normalising the argument with $L := s_{\xi_1,y_1} + s_{\xi_2,y_2}$, so that $\partial U \cap \overline{B_x} = \Gamma_x := \gamma_x([0,1])$.
\end{proof}

It follows from Lemma~\ref{Lemma_Finite_Repr_for_Boundary} that the boundary $\partial U$ of every connected component $U$ of the complement $\R^2 \setminus E_\eps$ can be covered by the (infinite) collection
\[
\cM := \big\{\Gamma_x \, : \, x \in \partial U \setminus Q_1\big\} \cup  \big\{\Gamma_x^{(i)} \, : \, x \in Q_1, \, i \in \{1,2\}\big\}
\]
of simple curves. The compactness of $E$ implies that $\partial U$ is compact as well, and $\cM$ thus always has a finite subcover. We show next that $\cM$ moreover has a finite, order two\footnotemark~ subcover $\cM^*$ that contains all the curves $\Gamma_x^{(i)}$, with $i \in \{1,2\}$, corresponding to sharp-sharp singularities $x \in Q_1$.

\begin{lemma}[{Finite subcover of order two}] \label{Lemma_Order_Two_Cover}
Let $E \subset \R^2$ be compact, let $\eps > 0$ and let $U$ be a connected component of the complement $\R^2 \setminus E_\eps$. In addition, let the functions $\gamma_x$, $\gamma_x^{(1)}, \gamma_x^{(2)}  : [0,1] \to \partial U$ and the corresponding simple curves $\Gamma_x := \gamma_x([0,1])$ and $\Gamma_x^{(i)} := \gamma_x^{(i)}([0,1])$ with $i \in \{1,2\}$, be as in Lemma~\ref{Lemma_Finite_Repr_for_Boundary}.
Then there exists a finite subset $X^* \subset \partial U \setminus Q_1$ for which the collection
\begin{equation} \label{Eq_Cover_M_of_partial_U}
\cM^* := \big\{\Gamma_x \, : \, x \in X^*\big\} \cup  \big\{\Gamma_x^{(i)} \, : \, x \in Q_1, \, i \in \{1,2\}\big\}
\end{equation}
of simple curves is a finite, minimal, order two cover of $\partial U$.\footnotemark[\value{footnote}]
\end{lemma}

\footnotetext{A cover $\cU = \{ U_\alpha \, : \, \alpha \in A \}$ of a set $X \subset \bigcup_{\alpha \in A} U_\alpha$, indexed by the set $A$, is said to be of \emph{order $n$}, if the set $A(x) := \{ \alpha \in A \, : \, x \in U_\alpha \}$ contains at most $n$ elements for all $x \in X$, and it is said to be \emph{minimal}, if no $U \in \cU$ may be removed so that $\cU \setminus \{U\}$ is still a cover for $X$.}

\begin{proof}
Let the sets $R, P \subset \partial U$ be as in~\eqref{Eq_Sets_R_and_P_1} and~\eqref{Eq_Sets_R_and_P_2}. We now use the curve representations $\Gamma_x$ for $x \in R \cup P \cup Q_2$ and $\Gamma_x^{(i)}$ for $x \in Q_1$ given by Lemma~\ref{Lemma_Finite_Repr_for_Boundary} in order to construct an order two cover for the boundary $\partial U$. To begin with, consider the open cover $\cU_0 := \{B_x \, : \, x \in \partial U\}$. Since $\partial U$ is compact, there exists a finite subcover $\cU_1 = \{B_x \, : \, x \in X_0\} \subset \cU_0$, corresponding to some finite subset $X_0 \subset \partial U$. We argue that in fact there exists a further subcover $\cU \subset \cU_1$ of order two. To see this, consider some boundary point
\begin{equation} \label{Eq_Order_Three_Cover}
x \in \partial U \cap B_{x_1} \cap B_{x_2} \cap B_{x_3},
\end{equation}
where $B_{x_k} \in \cU_1$ for $k \in \{1,2,3\}$. We claim that in this situation one of the sets $B_{x_k}$ can always be removed from the collection so that $\cU_1 \setminus B_{x_k}$ is still a cover for $\partial U$. As presented above, for each $k \in \{1,2,3\}$ the set $X \cap \overline{B_{x_k}}$ has a parametrisation as the curve $\Gamma_{x_k} = \gamma_{x_k}([0,1])$. Hence, there exists a continuous bijection $h : [0,1] \to \partial U \cap \bigcup_{k \in \{1,2,3\}} \overline{B_{x_k}}$, for which $\partial U \cap B_{x_k} = h\left((a_k, b_k)\right)$ for some open intervals $\left(a_k, b_k\right) \subset [0,1]$ and $k \in \{1,2,3\}$. Equation \eqref{Eq_Order_Three_Cover} is then equivalent to
\[
h^{-1}(x) \in \bigcap_{k \in \{1,2,3\}} (a_k, b_k).
\]
On the other hand, regardless of the order in which the points $a_k, b_k$ lie on $[0,1]$---as long as $a_k < b_k$ for all $k \in \{1,2,3\}$---one of the intervals $(a_k, b_k)$ is always contained in the union of the other two, and can therefore be removed without affecting the cover. One can thus remove any redundant balls $B_x$ from the cover $\cU_1$ and obtain a minimal, order two subcover $\cU_2$.

To complete the proof, we still need to ensure that none of the balls $B_x$ corresponding to sharp-sharp singularities $x \in Q_1$ were removed from the initial cover $\cU_0$ along the pruning process. To this end, note that according to Lemma~\ref{Lemma_Finitely_Many_Sharp_Singularities_per_Component} the set $Q_1$ is finite. This implies that we can assume the radii of the balls $B_x$ to have been initially chosen sufficiently small such that $B_x \cap Q_1 = \varnothing$ for all $x \in \partial U \setminus Q_1$ and $B_x \cap Q_1 = \{x\}$ for $x \in Q_1$. This way, the inclusion of every ball $B_x$ with $x \in Q_1$ is necessary in any subcover of $\cU_0$.

Hence, we have arrived at the desired finite, minimal subcover $\cU_2$ of order two, in the form
\begin{equation} \label{Eq_Order_Two_Cover}
\cM^* := \big\{\Gamma_x \, : \, x \in X^*\big\} \cup \big\{\Gamma_x^{(i)} \, : \, x \in Q_1, \, i \in \{1,2\}\big\},
\end{equation}
where the set $X^*$ is defined by $X^* = \{x \in \partial U \setminus Q_1  \, : \, B_x \in \cU_2\}$.
\end{proof}

The significance of the order two property of the cover obtained in Lemma~\ref{Lemma_Order_Two_Cover}, as well as the requirement that all the curves corresponding to $x \in  Q_1$ are included in the cover, both stem from the need to ensure that every boundary segment in the cover has a unique successor and predecessor segment. This makes it is possible to construct Jordan curves on the boundary by following the boundary along adjacent curve segments.

\begin{figure}[h]
      \centering
      \captionsetup{margin=0.75cm}
      \vspace{0mm}
                \includegraphics[width = \textwidth]{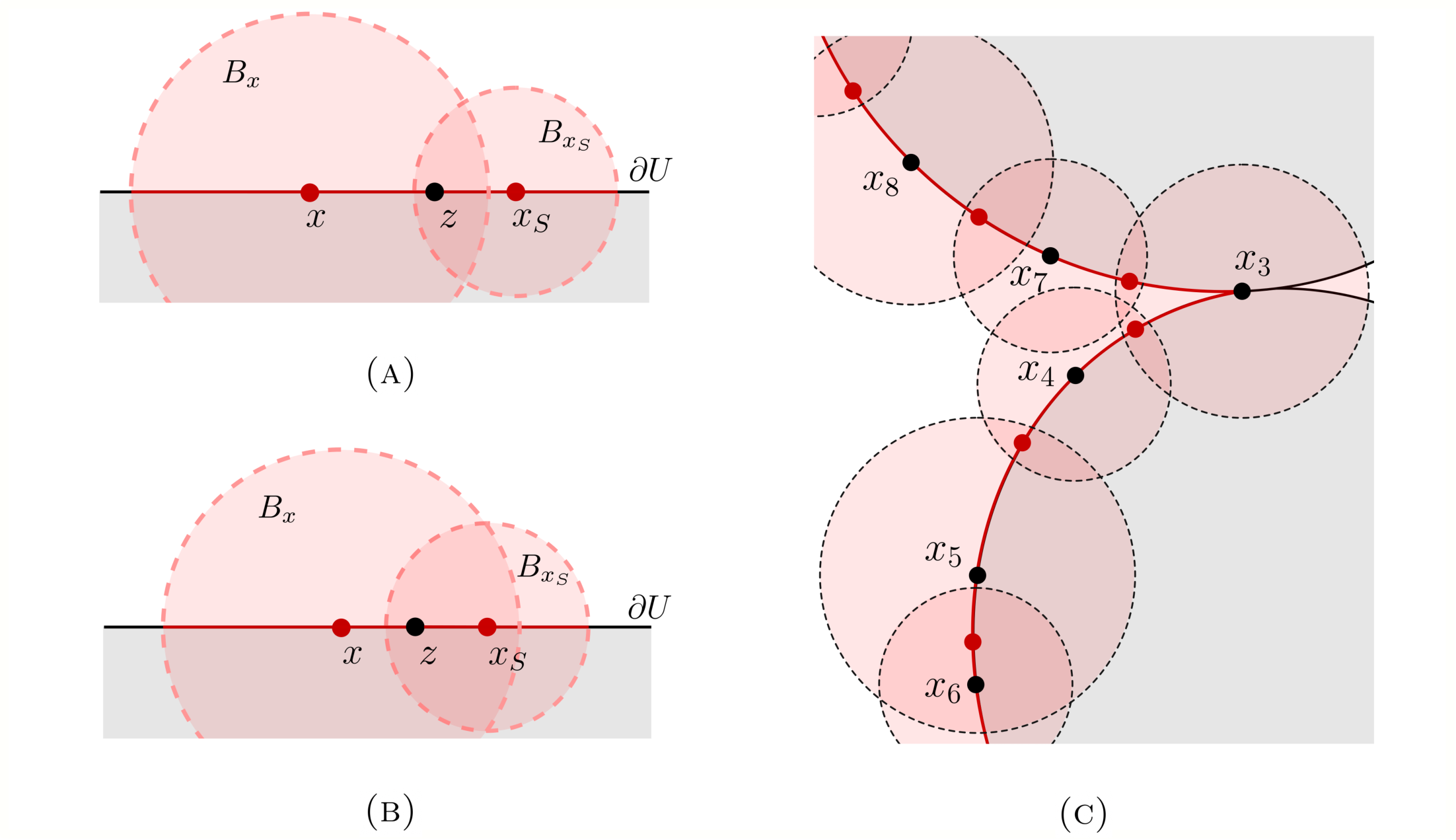} \\[0mm]
       \caption{Tracking the boundary $\partial U$ along a finite collection of adjacent balls and corresponding simple curves. In ({\footnotesize A}) and ({\footnotesize B}), the points $x$ and $x_{S}$ correspond to simple curves $\Gamma_x, \Gamma_{x_S} \in \cM^*$. Here $\Gamma_{x_S}$ is the successor of $\Gamma_x$, and due to minimality of $\cM^*$, no point $z \in \Gamma_x \cap \Gamma_{x_S}$ belongs to any other curve $\Gamma \in \cM^*$. It is not ruled out that $x_{S} \in \Gamma_x$, as in ({\footnotesize B}).
                In ({\footnotesize C}), which is a close-up of Figure~\ref{Figure_Theorem_1}~({\footnotesize A}), a longer curve is formed by concatenating such overlapping local representations.}
                \label{Figure_Connecting_the_Dots}
\end{figure}

\begin{prop}[{Boundaries of connected components of the complement as finite unions of Jordan curves}] \label{Prop_Jordan_curve_boundaries}
Let $E \subset \R^2$ be compact, let $\eps > 0$ and let $U$ be a connected component of the complement $\R^2 \setminus E_\eps$. Then $\partial U = \bigcup_{i = 1}^M J_i$, where $M \in \N$ and each $J_i$ is a Jordan curve. The representation is unique up to parametrisation, and for any $i \neq j$, the intersection $J_i \cap J_j$ contains at most one point.
\end{prop}

\begin{proof}
We divide the proof into three steps. First, we use the finite, order two cover given by Lemma~\ref{Lemma_Order_Two_Cover} to construct the Jordan curves on the boundary. We then argue that this representation is unique up to parametrisation, and finish by showing that the intersection of any two Jordan curves in the representation is either a singleton or the empty set. The last two facts are essentially implied by the Jordan Curve Theorem and the connectedness of the set $U$.\footnote{The Jordan Curve Theorem states that for any set $J \subset \R^2$ that is homeomorphic to the unit circle $S^1 \subset \R^2$, the complement $\R^2 \setminus J$ has precisely two connected components $A_1$ and $A_2$, one of which is bounded and the other unbounded, with the common boundary $\partial A_1 = \partial A_2 = J$.}

\emph{Step 1: Construction of Jordan Curves.} Let
\begin{equation} \label{Eq_Order_Two_Cover_again}
\cM^* := \big\{\Gamma_x \, : \, x \in X^*\big\} \cup \big\{\Gamma_x^{(i)} \, : \, x \in Q_1, \, i \in \{1,2\}\big\}
\end{equation}
be the finite, minimal, order two cover of $\partial U$ given by Lemma~\ref{Lemma_Order_Two_Cover}. The sets $X^*$ and $Q_1$ are as in~\eqref{Eq_Order_Two_Cover}, and $\Gamma_x$ and $\Gamma_x^{(i)}$ are simple curves. Since $\cM^*$ is minimal and order two, each curve $\Gamma$ in the collection
\[
\cA := \big\{\Gamma_x \, : \, x \in X^*\big\}
\]
intersects exactly two other simple curves
\[
\Gamma_P, \Gamma_S \in \cM^* \setminus \{\Gamma\},
\]
which we call the \emph{predecessor} and \emph{successor}, respectively. It is possible that $\Gamma_P = \Gamma_S$.
The parametrisations of individual curves $\Gamma \in \cA$ can thus be combined to form longer curves, see Figure~\ref{Figure_Connecting_the_Dots}. Some of these eventually loop back onto themselves, forming Jordan curves, while others connect to curves
\[
\Gamma \in \cB  := \big\{\Gamma_x^{(i)} \, : \, x \in Q_1, \, i \in \{1,2\}\big\}.
\]
By construction, $\Gamma_x^{(1)} \cap \Gamma_x^{(2)} = \{x\}$ for each $x \in Q_1$, and in addition, each $\Gamma_x^{(i)} \in \cB$ intersects precisely two other curves
\[
\Gamma_P, \Gamma_S \in \cM^*  \setminus \big\{\Gamma_x^{(1)}, \Gamma_x^{(2)} \big\}.
\]
One can thus start from any $\Gamma \in \cM^*$ and follow the boundary along uniquely defined successor and predecessor curves. The order two property and the finiteness of the collection $\cM^*$ ensure that every such extended curve ultimately returns to $\Gamma$ without self-intersections, thus forming a Jordan curve. Furthermore, there are only finitely many such curves.

\vspace{1mm}

\begin{figure}[h]
      \centering
      \captionsetup{margin=0.75cm}
      \vspace{0mm}
                \includegraphics[width = \textwidth]{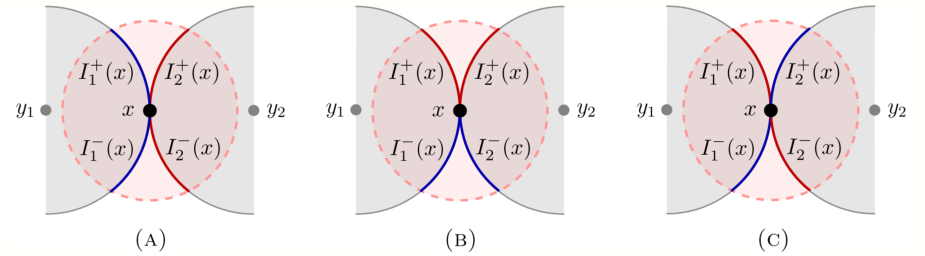} \\[0mm]
      \caption{Schematic illustration of the three ways of representing the boundary $\partial U$ near a sharp-sharp singularity $x \in Q_1$ as the union of two simple curves. 
Since $I_i^\pm(x) \subset \Gamma_x^{(i)} \subset \partial U$ for $i \in \{1,2\}$, the point $x$ must lie at the intersection of two different Jordan curves. In ({\footnotesize A}) the red and blue curves are subsets of the curves $\Gamma_x^{(1)}$ and $\Gamma_x^{(2)}$ defined in Lemma~\ref{Lemma_Finite_Repr_for_Boundary} for $x \in Q_1$. In ({\footnotesize B}) 
the coloured curves bounce back at $x$ and in ({\footnotesize C}) they cross at $x$. 
It turns out that only case ({\footnotesize A}) admits the continuation of such local curves into a representation of $\partial U$ as Jordan curves.}
                \label{Figure_Jordan_curve_options}
\end{figure}

\emph{Step 2: Uniqueness of Representation.} Assume there exists another representation of $\partial U$ as a union of Jordan curves. Since there is only one way to concatenate the curves $\Gamma \in \cA$, the only way to obtain a representation that differs from the one constructed in Step 1 is to define the Jordan curves differently at the junction points $x \in Q_1$, where four simple curves meet at the same point. For each $x \in Q_1$ and $i \in \{1,2\}$ we write $\Gamma_x^{(i)} = I_i^+(x) \cup I_i^-(x)$ where the $I_i^\pm(x)$ are simple curves that intersect each other pairwise only at $x$. The Jordan curves constructed in Step 1 are obtained by arriving to each $x \in Q_1$ via either $I_1^-(x)$ or $I_2^-(x)$ and leaving along $I_1^+(x)$ or $I_2^+(x)$, respectively. This corresponds to following either the blue or the red curve in Figure~\ref{Figure_Jordan_curve_options}({\footnotesize A}). Assume then that for some $x \in Q_1$, a Jordan curve $J$ arrives at $x$ via $I_1^-(x)$ and leaves via either $I_2^-(x)$ or $I_2^+(x)$. This corresponds to following the blue curve either in Figure~\ref{Figure_Jordan_curve_options}({\footnotesize B}) or in Figure~\ref{Figure_Jordan_curve_options}({\footnotesize C}), respectively. The resulting curve would then not coincide with any of the Jordan curves constructed in Step 1. We show that both possibilities lead to a contradiction.

Assume first that $J$ bounces back at $x$ along $I_2^-(x)$, as in Figure~\ref{Figure_Jordan_curve_options}({\footnotesize B}).
Let $A_1, A_2$ be the connected components of $\R^2 \setminus J$, for which $\partial A_1 = \partial A_2 = J$.
By definition, $J$ has no self-intersections, which implies that it intersects the curves $I_1^+(x)$ and $I_2^+(x)$ only at $x$.
It follows that $J \cap B_r(x) = I_1^-(x) \cup I_2^-(x)$, where $r > 0$ is the radius of the local boundary representation at $x$, see Proposition~\ref{Prop_local_representation_exists}.
Consequently, $J$ divides the ball neighbourhood $B_r(x)$ into two connected components $C_1, C_2 \subset B_r(x)$, which satisfy $C_i \subset A_i$ for $i \in \{1,2\}$. Let $\xi \in \Xext{x}$ be the extremal outward direction for which $\partial U \cap U_r(x, \xi) = I_1^-(x) \cup I_2^-(x)$ for some $r > 0$. Choose some $x(\xi) \in U \cap U_r(x, \xi)$ and $x(-\xi) \in U \cap U_r(x, -\xi)$. Now, since $U$ is connected and open, there exists a path $\gamma : [0,1] \to U$ connecting $x(\xi)$ and $x(-\xi)$. But this is impossible, because it would imply that $A_1 \cup A_2$ is connected.

Assume then that $J$ crosses $x$ along $I_2^+(x)$, as in Figure~\ref{Figure_Jordan_curve_options}({\footnotesize C}). Since $J$ cannot cross itself, it follows that another Jordan curve $J_1 \subset \partial U$ contains the curves $I_2^-(x)$ and $I_1^+(x)$. As before, let $A_1, A_2$ be the connected components of $\R^2 \setminus J$, for which $\partial A_1 = \partial A_2 = J$. Now $J_1 \cap A_i \neq \varnothing$ for both $i \in \{1,2\}$. But this contradicts the fact that the connected set $U$ must be a subset of either $A_1$ or $A_2$.

\emph{Step 3: Pairwise Intersections.} Assume, contrary to the claim, that there exist two Jordan curves $J_1, J_2  \subset \partial U$ and some $x_1, x_2 \in J_1 \cap J_2$ with $x_1 \neq x_2$. We first note that if it were the case that $J_1 \cap B_r(x) = J_2 \cap B_r(x)$ for all $x \in Q_1$ and some $r = r(x) > 0$, this would imply $J_1 = J_2$. Therefore, and because $J_1 \cap J_2 \neq \varnothing$, there must exist some $x \in Q_1 \cap (J_1 \cap J_2)$ for which $J_1 \cap B_x \neq J_2 \cap B_x$. We may assume $x = x_1$.
One may then construct a new Jordan curve $J^* \subset \partial U$ by following $J_1$ from $x_1$ to $x_2$ and returning back to $x_1$ via $J_2$.
Note that, because of this rerouting of the curve, the local geometry of $J^*$ around $x_1$ corresponds to either case ({\footnotesize B}) or ({\footnotesize C}) in Figure~\ref{Figure_Jordan_curve_options}.
Let $A_1, A_2$ be the connected components of $\R^2 \setminus J^*$, for which $\partial A_1 = \partial A_2 = J^*$. Since $U$ is connected, either $U \subset A_1$ or $U \subset A_2$. But this contradicts the fact that $J^*$ divides the neighbourhood $B_r(x_1)$ into two disjoint connected components $C_1 \subset A_1$ and $C_2 \subset A_2$, for which $C_i \cap U \neq \varnothing$ for both $i \in \{1,2\}$.
\end{proof}

We now prove the main result of this chapter, Theorem~\ref{Thm_global_structure}, by combining Proposition~\ref{Prop_Jordan_curve_boundaries} and Corollary~\ref{Cor_Inaccessible_Singularities}.

\begin{theorem_without_2} \label{Thm_global_structure_text}
For a compact set $E \subset \R^2$ and $\eps > 0$, the boundary $\partial E_\eps$ is a disjoint union
\[
\partial E_\eps = \cI \cup J,
\]
where $\cI$ is the set of inaccessible singularities and $J = \bigcup_{i \in I} J_i$ is a countable (possibly finite) union of Jordan curves $J_i$. Furthermore, there is a unique representation with the property that each Jordan curve $J_i$ satisfies $J_i \subset \partial U$ for some connected component $U$ of the complement $\R^2 \setminus E_\eps$.
\end{theorem_without_2}

\begin{proof}
Denote by $\cU$ the collection of connected components of the complement $\R^2 \setminus E_\eps$. Corollary~\ref{Cor_Inaccessible_Singularities} asserts that each boundary point $x \in \partial E_\eps$ lies on the boundary $\partial U$ of some $U \in \cU$ if and only if it is not a chain (type S6) or chain-chain (type S7) singularity. Proposition~\ref{Prop_Jordan_curve_boundaries} then implies that
\begin{equation} \label{Eq_Boundary_expression}
\partial E_\eps = \cI \cup \bigcup_{U \in \cU} \bigg( \bigcup_{n=1}^{N_U} J_n(U) \bigg),
\end{equation}
where $N_U \in \N$ for each $U \in \cU$ and $J_n(U) \subset \partial U$ is a Jordan curve for all $1 \leq n \leq N_U$. For each $U \in \cU$, the collection $\{J_n(U) \, : \, n \in N_U\}$ is furthermore unique up to parametrisation.

Due to the compactness of $E_\eps$, the complement $\R^2 \setminus E_\eps$ contains only one unbounded connected component. All other connected components of the complement $\R^2 \setminus E_\eps$ are contained in some ball $B_M(0)$, centered at the origin, with radius $M < \infty$. Since they are open, each of them contains an open ball. Hence, the Lebesgue measure of each such component $U$ satisfies $m(U) > 0$. However, for each $n \in \N$, there can be only finitely many $U$ with $m(U) > 1/n$. It follows from this that there are at most countably many connected components of the complement, which also implies that the number of Jordan curves $J_n(U)$ in \eqref{Eq_Boundary_expression} is countable.
\end{proof}
\chapter{Uniform Rectifiability and Ahlfors Regularity}
\label{Sec_Ahlfors_regularity}
In this chapter we recall the definitions of rectifiability and Ahlfors regularity, and show that the Jordan curve subsets of the boundary are
uniformly rectifiable.
We also give an example of an $\eps$-neighbourhood whose boundary is not an Ahlfors-regular set.

Recall that a \emph{curve} in $\R^2$ is the image of a closed interval under some continuous map $\gamma: [a,b] \to \R^2$, see Definition~\ref{Def_Jordan_Curve}.
The \emph{length} of a curve $\Gamma$ is given by 
\[
\cL(\Gamma) := \sup \sum_{k=1}^m |\gamma(t_k) - \gamma(t_{k-1}) |,
\]
where the supremum is taken over all finite subsets $\{t_0, \ldots, t_m\}$ of the interval $[a,b]$. We say that $\Gamma$ is a \emph{rectifiable curve}, if $\cL(\Gamma) < \infty$. Equivalently, $\Gamma$ is a rectifiable curve if there exists some Lipschitz map $f : [a,b] \to \R^2$ for which $\Gamma = f([a,b])$.
The notion of a \emph{rectifiable set}~\cite{David_Semmes_Analysis_of,Mattila_Geometry_of_Sets_and_Measures} generalises the above definition.


\begin{definition}[{$1$-rectifiable set}] \label{Def_Rectifiability}
A set $A \subset \R^d$ is \emph{$1$-rectifiable} (or simply \emph{rectifiable}), if $\cH^1(A) < \infty$ and there exists a countable family of Lipschitz functions $f_i : \R \to \R^d$ for which $\cH^1 \big( A \setminus \bigcup_{i\in\N}f_i(\R) \big) = 0$.
\end{definition}
%
%
It was shown in~\cite[Proposition 2.3]{Rataj_Winter_On_Volume} that for any bounded set $E \subset \R^d$ and $d \geq 1$, both $\partial E_\eps$ and $\partial E_{<\eps}$ are $(d-1)$-rectifiable sets.

Let $A \subset \R^2$ and $x \in A$. The limit
\begin{equation} \label{Eq_Def_Density}
\Theta(A, x) := \lim_{r \to 0} \frac{\cH^1(A \cap B_r(x))}{2r},
\end{equation}
when it exists, is called the ($1$-dimensional) $\emph{density}$ of $A$ at $x$. A basic result in geometric measure theory is that if $\cH^1(A) < \infty$, then $A$ is rectifiable if and only if the density~\eqref{Eq_Def_Density} exists at almost all $x \in A$, and is equal to $1$. Some authors 
call sets with this property regular sets. The slightly different notion of Ahlfors-regularity quantifies in a scale-free fashion the Hausdorff-dimensionality of the set $A$.

\begin{definition}[{Ahlfors regularity}] \label{Def_Ahlfors_Regularity}
A set $A \subset \R^d$ is \emph{Ahlfors-regular} (or simply \emph{regular}) with dimension $n$, if $A$ is closed and there exist constants $a,b > 0$ so that
\begin{equation} \label{Eq_Ahlfors_Regularity_Condition}
a r^n \leq \cH^n \big( A \cap B_r(x)  \big) \leq b r^n
\end{equation}
for all $x \in A$ and $0 < r < \textrm{diam}(A)$.\footnote{The requirement $r < \textrm{diam}(A)$ allows for the lower bound in~\eqref{Eq_Ahlfors_Regularity_Condition} to be satisfied with some constant $a > 0$ also for sets $A$ that have finite $n$-dimensional measure.}
\end{definition}
%
%
The notion of \emph{uniform rectifiability}~\cite{David_Semmes_Analysis_of} refers to a family of closely related geometric properties that, when satisfied by some set $A$, guarantee the boundedness of certain integral operators defined on the space $L^2(A)$ of square-integrable functions.
A thorough review of uniform rectifiability is provided in the introductory chapter of~\cite{David_Semmes_Analysis_of}. We take as the definition the following characterisation, whcih is given in~\cite[Section 15.23]{Mattila_Geometry_of_Sets_and_Measures}.

\begin{definition}[{Uniform rectifiability}] \label{Def_uniform_rect}
A set $A$ is \emph{uniformly rectifiable}, if (i) it is closed, (ii) it is Ahlfors regular and (iii) there exists some curve $\Gamma$ and some $C > 0$ for which $A \subset \Gamma$ and $\Gamma$ satisfies for all $x \in \Gamma$ and all $r > 0$ the inequality
\begin{equation} \label{Eq_Local_Measure_Upper_Bound}
\cH^1\big(\Gamma \cap B_r(x)\big) \leq Cr.
\end{equation}
\end{definition}

We begin by showing that for any compact $E \subset \R^2$, the Jordan curve subsets of the boundary $\partial E_\eps$ are Ahlfors regular.


\begin{prop}[{Ahlfors regularity of Jordan curve boundary subsets}] \label{Prop_Jordan_curves_are_Ahlfors_regular}
Let $E \subset \R^2$ be compact, $\eps > 0$ and let $U$ be a connected component of the complement $\R^2 \setminus E_\eps$. Let $J$ be a Jordan curve subset of the boundary $\partial U$. Then $J$ is Ahlfors regular.
\end{prop}

\begin{proof}
The idea is to first identify a suitable 'anchor' radius $r_0 > 0$ which will depend on the particular set $E$ and radius $\eps$. After this, we establish the required constants $a,b > 0$ for the lower and upper bounds in~\eqref{Eq_Ahlfors_Regularity_Condition} by considering separately the cases $r < r_0$ and $r \geq r_0$.


Let $J$ be a Jordan curve subset of the boundary $\partial U$, where $U$ is a connected component of the complement $\R^2 \setminus E_\eps$. The appropriate anchor radius $r_0$ can be defined on the basis of the finite, minimal, order two cover of $J$, whose existence was shown in Lemma~\ref{Lemma_Order_Two_Cover}. In~\eqref{Eq_Order_Two_Cover_again_and_again} below, the sets $X^*$ and $Q_1$ and the simple curves $\Gamma_x$ and $\Gamma_x^{(i)}$ are as in~\eqref{Eq_Order_Two_Cover}. According to Lemma~\ref{Lemma_Order_Two_Cover}, the collection
\begin{equation} \label{Eq_Order_Two_Cover_again_and_again}
\cM^* := \big\{\Gamma_x \, : \, x \in X^*\big\} \cup \big\{\Gamma_x^{(i)} \, : \, x \in Q_1, \, i \in \{1,2\}\big\}
\end{equation}
is a finite, minimal, order two cover of the Jordan curve $J$.  Since $\cM^*$ is minimal and order two, each curve $\Gamma$ in the collection $\cA := \big\{\Gamma_x \, : \, x \in X^*\big\}$
intersects exactly two other simple curves in $\cM^* \setminus \{\Gamma\}$.
The key observation is that since the union $X^* \cup Q_1$ is finite, there exists some radius $r_0 > 0$ with the following property: for any boundary point $x \in J$, there exists a corresponding curve $\Gamma_x \in \cA$ for which $J \cap B_r(x) = \Gamma_x \cap B_r(x)$ whenever $r \leq r_0$. We now identify the constants $a,b > 0$ for the lower and upper bounds in~\eqref{Eq_Ahlfors_Regularity_Condition}

\textbf{Lower bound.} Let $r_0$ be as defined above. Let $x \in J$, and consider first some radius $r \leq r_0$. Then, as argued above, there exists some curve $\Gamma_x \in \cA$ for which $J \cap B_r(x) = \Gamma_x \cap B_r(x)$.
Note that each $\Gamma \in \cA$ is constructed from the graphs in a local boundary representation around some $z \in J$, see Lemma~\ref{Lemma_Finite_Repr_for_Boundary} and Figure~\ref{Figure_Local_Curve_Representations}. 
The set $\Gamma_x \cap B_{r}(x)$ thus consists of two disjoint, continuous curves that connect the center point $x$ of the $r$-radius ball $B_r(x)$ to its circumference $\partial B_r(x)$. From this we obtain the lower bound 
\begin{equation} \label{Eq_curve_length_lower_bound_1}
\cH^1(J \cap B_r(x)) = \cH^1(\Gamma_x \cap B_r(x)) \geq 2r.
\end{equation}

Assume then that $r_0 \leq r \leq \textrm{diam}(J)$. In this case we have
\begin{equation} \label{Eq_curve_length_lower_bound_2}
\frac{\cH^1(J \cap B_r(x))}{r} \geq \frac{\cH^1(J \cap B_{r_0}(x))}{\textrm{diam}(J)} \geq \frac{2r_0}{\textrm{diam}(J)}.
\end{equation}
Setting $a := \textrm{min}\big(2, 2r_0 / \textrm{diam}(J)\big)$, the lower bounds~\eqref{Eq_curve_length_lower_bound_1} and~\eqref{Eq_curve_length_lower_bound_2}
imply
\[
\cH^1(J \cap B_r(x)) \geq ar
\]
for all $r \in \big(0, \textrm{diam}(J)\big)$.

\textbf{Upper bound.}
Let $r_0$ be as defined above. Let $x \in J$ and consider first the case $r < r_0$. Then, as argued for the lower bound above, $J \cap B_r(x) = \Gamma_x \cap B_r(x)$ for some $\Gamma_x \in \cA$. Let $\cG(x)$ be a local boundary representation at $x$. By construction, the curve $\Gamma_x$ is the union of either two or four curves that are obtained as images $g_{\xi, y}([0, s_{\xi,s}])$ under Lipschitz-continuous functions $g_{\xi, y} \in \cG(x)$, where $(\xi, y) \in \Pext{x}$ are extremal pairs and $0 < s_{\xi,s} \leq r_0$ for each $(\xi, y)$. 
According to Proposition~\ref{Prop_LBR_Lipschitz}, the functions $g_{\xi, y}$ are $2/\sqrt{3}$ Lipschitz, so that
\begin{equation} \label{Eq_Lipschitz_upper_bound_for_J}
\cH^1(J \cap B_r(x)) = \cH^1(\Gamma_x \cap B_r(x)) \leq \sum_{j=1}^4 \frac{2r}{\sqrt{3}} = \frac{8r}{\sqrt{3}}.
\end{equation}
Assume then that $r \geq r_0$. In this case we have
\begin{equation} \label{Eq_Ahlfors_upper_bound_for_J}
\frac{\cH^1(J \cap B_r(x))}{r} \leq \frac{\cH^1(J)}{r_0}.
\end{equation}
Note that $\cH^1(J) \leq \cH^1(\partial E_\eps) <  \infty$ and define $b := \max\big\{8/\sqrt{3}, \cH^1(J) / r_0\big\}$.\footnote{
The finiteness of the length of the boundary is one of the very first properties established about $\eps$-neighbourhoods, see~\cite[Section 6]{Erdos_Some_remarks}.} Thus, it follows from the estimates~\eqref{Eq_Lipschitz_upper_bound_for_J} and~\eqref{Eq_Ahlfors_upper_bound_for_J} that
\[
\cH^1(J \cap B_r(x)) \leq br
\]
for all $r \in \big(0, \textrm{diam}(J)\big)$.
%
%
%
\end{proof}

We next leverage Proposition~\ref{Prop_Jordan_curves_are_Ahlfors_regular} in order to show that for the boundary $\partial E_\eps$ of the $\eps$-neighbourhood of any compact $E \subset \R^2$
there exists some curve $\Gamma$ that
satisfies the upper bound~\eqref{Eq_Local_Measure_Upper_Bound} for some constant $C > 0$.

\begin{lemma}[{Curve covering of the boundary}] \label{Lemma_Curve_covering_of_the_boundary}
Let $E \subset \R^2$ be compact and $\eps > 0$. Then there exists a curve $\Gamma \supset \partial E_\eps$ and some constant $C > 0$ for which the upper bound~\eqref{Eq_Local_Measure_Upper_Bound} is satisfied for all $r > 0$ and all $x \in \Gamma$.
\end{lemma}

\begin{proof}
Since $E$ is compact, there exists a finite collection of compact sets $A_k$ with diameter at most $\eps/2$ and for which
$E = \bigcup_k A_k$. According to~\cite[Lemma 1]{Brown_Sets_of_constant} the boundary of each $A_k$ is a Jordan curve $J_k$. It then follows from Proposition~\ref{Prop_Jordan_curves_are_Ahlfors_regular} that there exist constants $C_k$ for which
\begin{equation} \label{Eq_Curve_covering_upper_bound}
\cH^1(J_k \cap B_r(x)) \leq C_k r
\end{equation}
for all $x \in J_k$ and $0 < r \leq \textrm{diam}(J_k)$. In fact, we may assume that~\eqref{Eq_Curve_covering_upper_bound} is valid for $0 < r \leq \textrm{diam}(\partial E_\eps)$ for all $k$, see~\eqref{Eq_Ahlfors_upper_bound_for_J}.
The individual Jordan curves $J_k$ can be arranged into a sequence in which consecutive Jordan curves are connected to each other by a straight line segment $L_k$. The union $\bigcup_k \big(J_k \cup L_k\big)$ thus forms a curve $\Gamma$ for which $\bigcup_k J_k \subset \Gamma$. According to~\cite[Theorem 1]{Brown_Sets_of_constant}, $\partial E_\eps \subset \bigcup_k J_k$ so that $\Gamma$ covers the boundary. Since only one connecting line segment is needed for each Jordan curve, it follows that a ball neighbourhood $B_r(x)$ with $r > 0$ around any $x \in \partial E_\eps$ intersects at most $N = \textrm{card}\big(\{J_k\}\big)$ line segments. The same holds for the Jordan curves $J_k$, so that
\[
\cH^1(\partial E_\eps \cap B_r(x)) \leq rN\big(\max_k\{C_k\} + 2\big). \qedhere
\]
\end{proof}

Note that Lemma~\ref{Lemma_Curve_covering_of_the_boundary} makes no assumptions regarding the types of singularities that may appear on the boundary $\partial E_\eps$. From this, we obtain the following result.

\begin{theorem_without_2}
[{Sufficient condition for uniform rectifiability}] \label{Thm_suff_cond_for_uni_rect_text}
Let $E \subset \R^2$ be compact, let $\eps > 0$, and assume that the boundary $\partial E_\eps$ contains no chain or chain-chain singularities. Then $\partial E_\eps$ is uniformly rectifiable.
\end{theorem_without_2}

\begin{proof}
Since $E$ is compact and the boundary $\partial E_\eps$ is assumed to contain no chain singularities, there must exist only finitely many Jordan curve subsets $J_k \subset \partial E_\eps$. In addition, every boundary point lies on some $J_k$. According to Proposition~\ref{Prop_Jordan_curves_are_Ahlfors_regular}, each $J_k$ is Ahlfors regular with its associated lower and upper bound constants $a_k$ and $b_k$ in~\eqref{Eq_Ahlfors_Regularity_Condition}. As argued in the proof of the previous lemma, the constants $b_k$ can be assumed to hold for all $0 < r \leq \textrm{diam}(\partial E_\eps)$ for all $k$, see~\eqref{Eq_Ahlfors_upper_bound_for_J}.
Since there are only finitely many Jordan curves, we may define $a := \min_k\{a_k\}$ and $b := \max_k\{b_k\}$. It follows that $\partial E_\eps$ is Ahlfors regular. Since it is closed, the claim now follows from Lemma~\ref{Lemma_Curve_covering_of_the_boundary}.
\end{proof}

\section{Counterexample to Uniform Rectifiability}
We now construct an example of an $\eps$-neighbourhood $E_\eps \subset \R^2$ whose boundary $\partial E_\eps$ fails to satisfy the lower bound in the Ahlfors regularity condition~\eqref{Eq_Ahlfors_Regularity_Condition} at a chain singularity. We obtain the underlying set $E$ as the union of two copies of a point sequence on the positive real line, embedded in the plane at a distance of $2\eps$ from each other.

\begin{example} \label{Ex_counterexample_to_Ahlfors}
Define the sequences $(a_n)_{n=1}^\infty$ and $(b_n)_{n=1}^\infty$ by setting $a_n := 2^{-n^2}$ and $b_n := 2a_n$ 
for all $n \in \N$. Let $A := \bigcup_{n=1}^\infty [b_n, a_{n-1}]$. Fix $\eps > 0$ and define
\begin{align*}
E^+ &:= \big\{ (x, \eps) \in \R^2 \, : \, x \in A \cup [-1, 0] \big\}, \\
E^- &:= \big\{ (x, -\eps) \in \R^2 \, : \, x \in A \cup [-1, 0] \big\}.
\end{align*}
Finally, let $E = E^+ \cup E^-$ and consider the $\eps$-neighbourhood $E_\eps$. The interval $[-1,0]$ is included in the definition of the sets $E^{\pm}$ above for simplicity, so that the part of the boundary $\partial E_\eps$ that intersects small neighbourhoods of the origin $(0,0)$ lies entirely on the right half plane.

We now show that the origin is a chain singularity on the boundary $\partial E_\eps$, and that the lower bound in the Ahlfors regularity condition~\eqref{Eq_Ahlfors_Regularity_Condition} fails at this point. To this end, consider a ball neighbourhood $B(r) := B_r\big((0,0)\big)$ of radius $r>0$ around the origin, and let $x \in \partial E_\eps \cap B(r)$ be some boundary point inside this neighbourhood. Then $x \in \partial B_\eps(z)$ for some $z = (z_1, z_2) \in E$, in which $z_1 \in [0,r]$ and $z_2 \in \{\eps, -\eps\}$. Hence, as $r$ decreases, the tangential direction of the boundary $\partial E_\eps$ at $x$ approaches horizontal, which implies in particular that the absolute values of the (one-sided) derivatives of the local boundary representation are bounded from above by some $K(r) > 0$ with $\lim_{r \to 0} K(r) = 0$. It follows that the Hausdorff measure of the boundary subset $\partial E_\eps \cap B(r)$ satisfies the upper bound
\begin{equation} \label{Eq_Boundary_Length_Upper_Bound}
\cH^1 \left( \partial E_\eps \cap B(r) \right) \leq 2 K(r) \ell \big( [0,r] \setminus A \big),
\end{equation}
in which $\ell \big( [0,r] \setminus A \big)$ is the length of the one-dimensional set $[0,r] \setminus A \subset \R_+$. We now show that
\[
\liminf_{r \to 0} \frac{\cH^1 \left( \partial E_\eps \cap B(r) \right)}{r} = 0,
\]
which implies that $\partial E_\eps$ is not Ahlfors regular. Due to~\eqref{Eq_Boundary_Length_Upper_Bound}, it suffices to show that
\[
\lim_{n \to \infty} \frac{\ell \big( [0,a_n] \setminus A \big)}{a_n} = 0.
\]
Note first that for each $n \in \N$,
\begin{align}
\ell \big( [0,a_n] \setminus A \big) &= a_n - \sum_{k=n}^\infty \ell \big([b_{k+1}, a_{k}] \big)
= \sum_{k=n+1}^\infty \ell \big([a_{k}, b_{k}]\big)
= \sum_{k=n+1}^\infty a_k.
\end{align}
For each $n \in \N$ the ratio of consecutive terms satisfies $R_n := a_{n+1} / a_n = 2^{-1 - 2n}$, so that $\lim_{n \to \infty} R_n = 0$ and
\[
\sum_{k=n+1}^\infty a_k = 2^{-(n+1)^2}\big(1 + R_{n+1} + R_{n+1}R_{n+2} + \ldots \big) < 2^{-(n+1)^2} \frac{R_{n+1}}{1 - R_{n+1}}.
\]
This implies
\[
\frac{\ell \big( [0,a_n] \setminus A \big)}{a_n} = \frac{2^{-(n+1)^2}}{2^{-n^2}} \frac{R_{n+1}}{1 - R_{n+1}}
= \frac{R_{n+1}^2}{1 - R_{n+1}} \longrightarrow 0
\]
as $n \to \infty$, which completes the argument.
\end{example}

Example~\ref{Ex_counterexample_to_Ahlfors} shows that boundaries of $\eps$-neighbourhoods that contain chain singularities are generally not uniformly rectifiable. However, it is also easy to construct examples of boundaries of $\eps$-neighbourhoods that do satisfy the lower bound in~\eqref{Eq_Ahlfors_Regularity_Condition} even though they contain chain singularities. The mere existence of chain singularities therefore does not in itself rule out the possibility that the boundary of an $\eps$-neighbourhood is uniformly rectifiable. Thus, in light of Theorem~\ref{Thm_suff_cond_for_uni_rect_text}, the uniform rectifiability of the boundary hinges on whether the lower bound in the Ahlfors regularity condition~\eqref{Eq_Ahlfors_Regularity_Condition} is satisfied at the chain singularities appearing on the boundary.

\chapter{Curvature} \label{Sec_Existence_of_Curvature}
In this final chapter we show that curvature is defined almost everywhere on Jordan curves $J \subset \partial U \subset  \partial E_\eps$, where $U$ is some connected component of the complement $\R^2 \setminus E_\eps$. It is not clear from the outset that curvature can be defined on these sets even almost everywhere. It was shown in Example~\ref{Ex_Shallow_can_be_dense} that shallow singularities (types S4 and S5), which are defined as accumulation points of increasingly shallow wedge-type singularities, may lie densely on boundary segments that have positive one-dimensional Hausdorff measure. It follows that even though Propositions~\ref{Prop_local_representation_exists} and~\ref{Prop_LBR_Lipschitz} provide a local representation of the boundary as Lipschitz-continuous functions, the first derivatives of these function generally exist only almost everywhere.


However, since there are no cusp singularities on the boundary, the tangent function on the boundary $\partial U$ may only have jumps in one direction. In the opposite direction, the rate of change of the tangential direction is bounded from below by the curvature of an $\eps$-radius ball, since every boundary point $x \in \partial U$ lies on the boundary of such a ball. We use these properties to show that the one-sided tangent functions are locally of bounded variation, and therefore have a derivative at almost every point. One can then use these to obtain the (signed) curvature of the boundary.

\section{Existence of Curvature via Bounded Variation} \label{Sec_Existence of Curvature via Bounded Variation}
Let $U$ be a connected component of the complement $\R^2 \setminus E_\eps$ and let $J \subset \partial U$ be one of its Jordan curve subsets, see Theorem~\ref{Thm_global_structure}. In general, $J$ has a well-defined tangent only at almost every point. However, Proposition~\ref{Prop_Tangents_are_Defined} asserts that the directional (left and right) tangents coincide with extremal outward directions, which exist at every $x \in J$ according to Proposition~\ref{Prop_structure_of_set_of_outward_directions}. Lemma~\ref{Lemma_Order_Two_Cover} asserts that $J$ can be covered by a finite collection of simple curves that are constructed from the graphs of local boundary representations, given by Proposition~\ref{Prop_local_representation_exists}. Hence, in order to establish the existence of curvature on the Jordan curve $J$, it suffices to work with these local representations.

In this section, we prove the final main result of the paper, Theorem~\ref{Thm_existence_of_curvature}, which states that curvature is well-defined almost everywhere on the Jordan curve subsets $J$ of the boundary $\partial E_\eps$. To prove the theorem, we apply a criterion for bounded variation to the local boundary representations of $J$. We present here the proof of the main theorem, and postpone the more technical proof of the criterion to Section~\ref{Sec_A Lower Bound for Differences}.

\subsection{Boundary Points in a Local Coordinate System} \label{Sec_Local_Coordinate_System}
A local boundary representation at $x \in \partial E_\eps$ is a finite collection $\cG(x)$ of continuous functions $g_{\xi, y} : [0, r] \to \R^2$, for some $0 < r < \eps / 2$, one for each extremal pair\footnote{For each boundary point $x \in \partial E_\eps$, the set of \emph{extremal pairs} $\Pext{x}$ consists of all the pairs $(\xi, y)$ of extremal outward directions $\xi \in \Xext{x}$ and extremal contributors $y \in \Cext{x}$ for which $\langle x - y, \xi \rangle = 0$.} $(\xi, y) \in \Pext{x}$, with the property that
\begin{equation} \label{Eq_local_boundary_repr_condition_02}
  \partial E_\eps \cap \overline{B_r(x)} = \bigcup_{(\xi, y) \in \Pext{x}} g_{\xi, y}\left(A_{\xi,y}\right)
\end{equation}
for some closed sets $A_{\xi, y} \subset [0, r]$, see Proposition~\ref{Prop_local_representation_exists}. More precisely, for each extremal pair $(\xi,y) \in \Pext{x}$ there exists a continuous function $f^{\xi, y} : [0, r] \to \R$ for which
\begin{equation} \label{Eq_Canonical_LBR_02}
g_{\xi, y}(s) = x + s \xi + f^{\xi, y}(s)\frac{x - y}{\eps}.
\end{equation}
According to Corollary~\ref{Cor_Inaccessible_Singularities}, chain singularities do not appear on the boundary $\partial U$ for connected components $U$ of the complement $\R^2 \setminus E_\eps$. It follows then from Proposition~\ref{Prop_local_representation_exists} that for each $x \in \partial U$, the sets $A_{\xi, y}$ in~\eqref{Eq_local_boundary_repr_condition_02} are intervals $A_{\xi, y} = [0, s_{\xi, y}]$ for some $s_{\xi,y} \in [0, r]$.

\begin{figure}[h]
      \centering
      \captionsetup{margin=0.75cm}
      \vspace{0mm}
                \includegraphics[width = \textwidth]{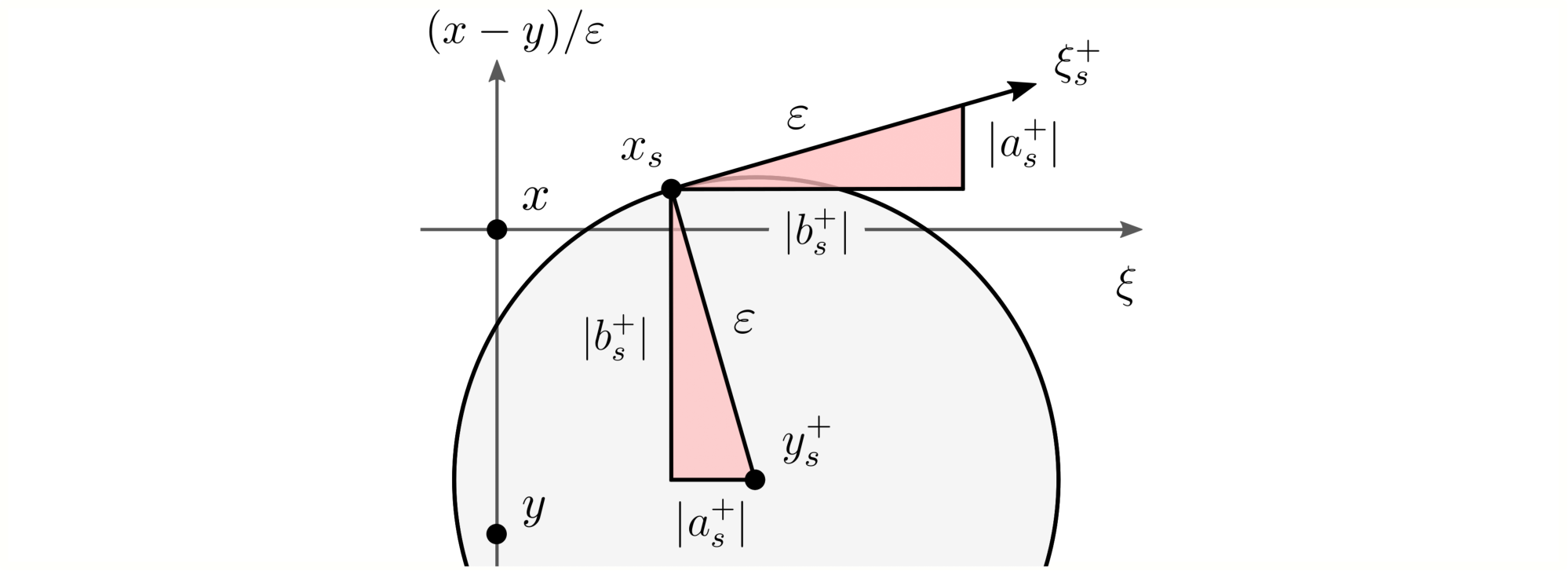} \\[0mm]   
                \caption{The relationship between the tangential direction $\xi_s^+$, the slope $D^+ f^{\xi, y}(s)$ and the extremal contributor $y_s^+$ at $x_s = g_{\xi,y}(s) = \big(s, f^{\xi,y}(s)\big)$ for $s \in [0, r]$. 
                 Here $a_s^+ < 0$, and the slope at $x_s$ satisfies $ D^+ f^{\xi, y}(s) = -a_s^+/b_s^+ = -a_s^+ / \sqrt{\eps^2 - (a_s^+)^2} =: p\big(a_s^+\big)$.} 
                \label{Figure_Local_Coordinate_System}
\end{figure}

Consider a fixed $x \in J$ and an extremal pair $(\xi,y) \in \Pext{x}$. To facilitate the subsequent analysis, we choose local coordinates given by the orthogonal unit vectors $\xi$ and $(x - y)/\eps$, and place the origin of this coordinate system at $x$. However, to simplify notation, we assume without loss of generality that this coordinate system coincides with the standard Euclidean coordinate system, so that $x = (0, 0)$, $\xi = (1,0)$ and $(x - y)/\eps = (0,1)$. For each $s \in [0, s_{\xi, y}]$, we define
\[
x_s := g_{\xi,y}(s) = \big(s, f^{\xi,y}(s)\big) \in J,
\]
where $f^{\xi,y} : [0, s_{\xi,y}] \to \R$ is as in~\eqref{Eq_Canonical_LBR_02}. The boundary subset $J \cap \overline{B_r(x)}$ can thus be expressed as the union of graphs 
\[
J \cap \overline{B_r(x)} = \bigcup_{(\xi, y) \in \Pext{x}} \big\{\big(s, f^{\xi,y}(s)\big) \, : \, s \in [0, s_{\xi,y}] \big\},
\]
where each extremal pair $(\xi, y) \in \Pext{x}$ defines its corresponding local coordinates.

According to Proposition~\ref{Prop_local_contribution}, for each $s \in [0, s_{\xi, y}]$ the extremal contributors $y_s \in \Cext{x_s}$ satisfy $y_s \in B_{\eps/2}(y)$, where $y \in \Cext{x}$ is the extremal contributor corresponding to the extremal pair $(\xi, y) \in \Pext{x}$, see Proposition~\ref{Prop_local_representation_exists}. This implies that for each $s \in [0, s_{\xi, y}]$ the extremal outward directions $\xi_s \in \Xext{x_s}$ deviate only slightly from the $\xi$-axis in the local coordinate system. We call them the right and left extremal outward direction at $x_s$, and denote them by $\xi_s^+$ and $\xi_s^-$. According to Proposition~\ref{Prop_Tangents_are_Defined}, these coincide with the tangential directions at $x_s \in J$. In the local coordinates, the slopes of the one-sided tangents of the graph $\big\{\big(s, f^{\xi,y}(s)\big) \, : \, s \in [0, s_{\xi,y}] \big\}$ are given by the the left and right derivatives
\begin{equation} \label{Eq_Def_first_derivatives}
D^\pm f^{\xi, y}(s) := \lim_{h \to \pm 0} \frac{f^{\xi, y}(s + h) - f^{\xi, y}(s)}{h}.
\end{equation}
These are related to the extremal outward directions $\xi_s^\pm = \big(\xi_1^\pm(s), \xi_2^\pm(s)\big) \in \Xext{x_s}$ through
\begin{equation} \label{Eq_Derivative_vs_Outward_Direction_Correspondence}
D^\pm  f^{\xi, y}(s) = \frac{\xi_2^\pm(s)}{\xi_1^\pm(s)}, 
\end{equation}
see Figure~\ref{Figure_Local_Coordinate_System}.

\subsection{A Condition for Bounded Variation} \label{Sec_A_Condition_for_BV}
In order to prove the existence of curvature almost everywhere on $J$, we consider a local boundary representation $\cG(x)$ around an arbitrary boundary point $x \in J$ and show that for each extremal pair $(\xi,y) \in \Pext{x}$, the second derivative of the function $f^{\xi,y}$ in~\eqref{Eq_Canonical_LBR_02} exists almost everywhere on the interval $[0, s_{\xi,y}]$. We begin by stating the following general criterion for bounded variation that applies to bounded functions on a closed interval. 

\begin{lemma}[{A criterion for bounded variation}] \label{Lemma_bounded_variation}
Let let $f: [a,b] \to \R$ be a bounded function and assume there exists some $q > 0$ for which
\begin{equation} \label{curvature_inequality}
f(s + h) - f(s) \geq -qh
\end{equation}
for all $s \in [a,b)$ and all $h > 0$ that satisfy $s + h \in [a, b]$. Then $f$ is of bounded variation on $[a,b]$.
\end{lemma}

\begin{proof}
Since $f$ is bounded, there exists some $m > 0$ for which $f(s) \in [-m, m]$ for all $s \in [a,b]$. Assume now contrary to the claim that $f$ is not of bounded variation.
Then there exists a partition $a = s_0 < s_1 < \ldots < s_K = b$ of the interval $[a,b]$ for which
\begin{equation} \label{inequality_variation_sum}
S := \sum_{j=1}^K |f(s_j) - f(s_{j-1})| > 2(p + m),
\end{equation}
where $p := \max \{q(b-a), 2m \}$.
Consider the index sets
\begin{align*}
J^+ &:= \big\{j \in \{1, \ldots, K\} \, : \, f(s_j) - f(s_{j-1}) > 0 \big\}, \\
J^- &:= \{1, \ldots, K\} \setminus J^+,
\end{align*}
that correspond to contributions from positive and negative differences in~\eqref{inequality_variation_sum}, respectively.
Then $S = S^+ - S^-$, where
\begin{equation} \label{Eq_decomposed_sum}
S^+ :=  \sum_{j \in J_+} \big( f(s_j) + f(s_{j-1}) \big) \quad \textrm{and} \quad  S^- := \sum_{j \in J^-} \big( f(s_j) - f(s_{j-1}) \big).
\end{equation}
It follows from~\eqref{curvature_inequality} that $S^- \geq -q(b-a) \geq -p$, which together with~\eqref{inequality_variation_sum} implies the lower bound $S^+ > p + 2m$. Combining these bounds 
leads to the contradiction
\begin{align*}
f(b) 
	  &= f(a) + S^+ + S^-
	  > -m + (p + 2m) - p
	  = m. \qedhere
\end{align*}
\end{proof}

The main step towards proving Theorem~\ref{Thm_existence_of_curvature} is to show that the one-sided derivatives $D^\pm f^{\xi, y}$ satisfy condition~\eqref{curvature_inequality} of Lemma~\ref{Lemma_bounded_variation} for a certain $q > 0$ and are thus of bounded variation on the interval $[0, s_{\xi,y}]$.

\begin{prop}[{A lower bound for differences of right derivatives}] \label{Prop_Lower_Bound_for_Curvature}
Let $E \in \R^2$ be a compact set and let $x \in J \subset \partial E_\eps$ where $J$ is a Jordan curve. Let $\cG(x)$ be a local boundary representation at $x$, where each $g_{\xi, y} \in \cG(x)$ satisfies
\[
g_{\xi, y}(s) = x + s\xi + f^{\xi, y}(s)\eps^{-1}(x - y)
\]
for some continuous function $f^{\xi, y} : [0, s_{\xi, y}] \to \R$. Then the left and right derivatives $D^\pm f^{\xi, y}$ satisfy the inequality 
\begin{equation} \label{Eq_Lower_bound_for_dir_derivatives}
D^\pm f^{\xi, y}(s + h) - D^\pm f^{\xi, y}(s) \geq - \frac{8h}{3\sqrt{3}\eps}
\end{equation}
for all $s \in [0, s_{\xi,y}]$ and $s, h > 0$ with $0 \leq s + h \leq s_{\xi, y}$.
\end{prop}

The proof of Proposition~\ref{Prop_Lower_Bound_for_Curvature} is the most technical part of the proof of Theorem~\ref{Thm_existence_of_curvature}, and is presented separately in Section~\ref{Sec_A Lower Bound for Differences}. We now proceed to combine Proposition~\ref{Prop_Lower_Bound_for_Curvature} and Lemma~\ref{Lemma_bounded_variation} into a general statement about the existence of curvature on the boundary $\partial E_\eps$.
We denote by $\cH^1$ the one-dimensional Hausdorff measure, and by $\cI$ the set of inaccessible singularities, see Corollary~\ref{Cor_Inaccessible_Singularities}.

\setcounter{theorem_without}{1}

\begin{theorem_without_2}[{Existence of curvature}]
Let $E \subset \R^2$ be compact and let $\eps > 0$. Then for $\mathcal{H}^1$-almost all $x \in \partial E_\eps \setminus \cI$, the (signed) curvature $\kappa$ exists and is given by the formula
\begin{equation} \label{Eq_curvature_of_function_graph}
\kappa(s_x) = \frac{\frac{d^2}{ds^2} f^{\xi, y}(s_x)}{\left(1 + \left(\frac{d}{ds}f^{\xi,y}(s_x)\right)^2 \right)^{3/2}},
\end{equation}
where the coordinates $s_x$ and $f^{\xi, y}(s_x)$ are associated with a local boundary representation $\cG(z)$ at some $z \in \partial E_\eps \setminus \cI$, such that $x = z + s_x \xi + f^{\xi, y}(s_x)(z - y)/\eps$ for some $(\xi,y) \in \Pext{z}$.
\end{theorem_without_2}

\begin{proof}
According to Theorem \ref{Thm_global_structure}, $\partial E_\eps \setminus \cI$ is a countable union $\cJ$ of Jordan curves. Since each $J \in \cJ$ is compact, there exists a finite collection $Z$ of points $z \in J$ and corresponding boundary representations $\cG(z)$, for which
\[
J = \bigcup_{z \in Z} \left(\bigcup_{(\xi, y) \in \Pext{z}} g_{\xi, y}([0, s_{\xi,y}])\right).
\]
The closed intervals $[0, s_{\xi, y}]$ are as in Proposition~\ref{Prop_local_representation_exists}. It suffices to show that for each $z \in Z$, curvature exists outside a $\cH^1$-negligible set on each curve $g_{\xi,y}([0, s_{\xi,y}])$, where $(\xi,y) \in \Pext{z}$.

Let $z \in Z$ and let $g_{\xi,y} \in \cG(z)$ for some $(\xi,y) \in \Pext{z}$. According to Proposition~\ref{Prop_local_representation_exists} there exists a continuous function $f^{\xi,y} : [0, s_{\xi,y}] \to \R$ for which
\[
g_{\xi, y}(s) = z + s \xi + f^{\xi,y}(s) \frac{z - y}{\eps}.
\]
Since the one-sided derivatives $D^\pm f^{\xi, y}(s)$ are related to the extremal outward directions $\xi_s^\pm$ via~\eqref{Eq_Derivative_vs_Outward_Direction_Correspondence}, it follows that $D^+ f^{\xi, y}(s) = D^- f^{\xi, y}(s)$ whenever $\xi_s^+ = -\xi_s^-$. This implies that the one-sided derivatives agree on $[0, s_{\xi,y}]$ apart from an at most countably infinite set, since it follows from Theorem~\ref{Thm_Main_2} and Corollary~\ref{Cor_Inaccessible_Singularities} that the set
\begin{align*}
J \setminus \big\{x \in J \, : \, \Xext{x} &= \{\xi, -\xi\} \,\, \textrm{for some} \,\, \xi \in S^1 \big\} \\ &= J \setminus \Unp{E},
\end{align*}
where $\Unp{E}$ is as in Definition~\ref{Def_Contributor}, is at most countably infinite.

We next confirm that the assumptions of Lemma~\ref{Lemma_bounded_variation} are satisfied for the left and right derivatives $D^\pm f^{\xi, y}(s)$ on the interval $[0, s_{\xi,y}]$. By the definition of a local boundary representation, the extremal contributors $y_s^\pm \in \Cext{x_s}$ satisfy $y_s^\pm \in B_{\eps/2}(y)$, where $y \in \Cext{x}$ corresponds to the extremal pair $(\xi, y)$. This imposes an upper bound on the angle between the $\xi$-axis and the corresponding extremal outward directions $\xi_s^\pm \in \Xext{x_s}$, which in turn implies a bound on the left and right derivatives $D^\pm f^{\xi, y}$ on $[0, s_{\xi,y}]$ through the relationship~\eqref{Eq_Derivative_vs_Outward_Direction_Correspondence}. Due to inequality~\eqref{Eq_Lower_bound_for_dir_derivatives} in Proposition~\ref{Prop_Lower_Bound_for_Curvature}, one can thus apply Lemma~\ref{Lemma_bounded_variation}, which implies that the one-sided derivatives $D^\pm f^{\xi, y}$ are of bounded variation on $[0, s_{\xi,y}]$.

Thus, each of the functions $D^\pm f^{\xi, y}$ has a (two-sided) derivative $\frac{d}{ds}D^\pm f^{\xi, y}$ almost everywhere on $[0, s_{\xi,y}]$.\footnote{Every function of bounded variation can be written as the difference of two non-decreasing functions, and consequently has a finite derivative at almost every point. For details, see for instance~\cite[p. 331]{Kolmogorov_Fomin_Introductory_Real_Anal}} Since $D^+ f^{\xi, y} = D^- f^{\xi, y}$ outside a countable set, this furthermore implies $\frac{d}{ds}D^+ f^{\xi, y} = \frac{d}{ds}D^- f^{\xi, y}$ almost everywhere. Consequently, apart from a set $W \subset [0, s_{\xi,y}]$ of zero Hausdorff measure, both the first and second (two-sided) derivatives of $f^{\xi, y}$ exist, and define curvature on
the graph $\big\{ \big(s, f^{\xi,y}(s) \big) \, : \, s \in [0, s_{\xi,s}] \big\}$ via~\eqref{Eq_curvature_of_function_graph}.

According to Proposition~\ref{Prop_LBR_Lipschitz}, the function $s \mapsto \big(s, f^{\xi,y}(s)\big)$ is $2/\sqrt{3}$-Lipschitz on $[0,s_{\xi, y}]$. This implies $\cH^1\left(f^{\xi,y}(W)\right) \leq 2 \cH^1(W) /\sqrt{3} = 0$, since the Hausdorff measure of a Lipschitz transformation is bounded from above by the corresponding Lipschitz constant, see~\cite[Prop.\ 2.49]{Ambrosio_et_al_Functions_of_Bounded_Variation}. It follows that curvature exists outside a $\cH^1$-negligible set on each curve $g_{\xi,y}([0, s_{\xi,y}])$.
\end{proof}

\section{A Lower Bound for Differences of Tangential Directions} \label{Sec_A Lower Bound for Differences}
In this section we prove Proposition~\ref{Prop_Lower_Bound_for_Curvature}, which represents the key step in the proof of Theorem~\ref{Thm_existence_of_curvature}. Proposition~\ref{Prop_Lower_Bound_for_Curvature} expresses the geometric observation that it is impossible for the boundary $\partial E_\eps$ to curve inwards more than a certain threshold, implied by the radius $\eps > 0$.

Before proceeding with the proof, 
we establish some notation. We consider the local boundary representation $\cG(x)$ around $x \in J$, where $J \subset \partial E_\eps$ is a Jordan curve component of the boundary. We work in the local coordinates corresponding to an extremal pair $(\xi,y) \in \Pext{x}$, as defined in Section~\ref{Sec_Local_Coordinate_System}. For each $s \in [0, s_{\xi, y}]$, define
\[
x_s := g_{\xi,y}(s) = x + s \xi + f^{\xi,y}(s) \frac{x - y}{\eps} = \big(s, f^{\xi,y}(s)\big) \in J,
\]
where $f^{\xi,y} : [0, s_{\xi,y}] \to \R$ is as in~\eqref{Eq_Canonical_LBR_02}. Each of the extremal contributors 
\[
y_s^\pm = \big(y_1^\pm(s), y_2^\pm(s) \big) \in \Cext{x_s}
\]
lies at the center of an $\eps$-radius circle\footnote{Due to the orientation of the extremal contributor $y \in \Cext{x}$ relative to $x$, these circles lie below the graph $g_{\xi,y}([0,s_{\xi,y}])$ in the $\big(\xi, (x-y)/\eps\big)$-coordinates, and are thus uniquely defined for all $s \in [0, s_{\xi,y}]$.} $B_\eps \big(y_s^\pm \big)$, whose tangent at $x_s \in \partial B_\eps \big(y_s^\pm \big)$ coincides with the respective extremal outward direction
\[
\xi_s^\pm = \big(\xi_1^\pm(s), \xi_2^\pm(s)\big) \in \Xext{x_s},
\]
see Definition~\ref{Def_Extremal_Outward_Directions_and_Contributors} and Proposition~\ref{Prop_Tangents_are_Defined}. Hence, for each $s \in [0, s_{\xi,y}]$,
\begin{equation} \label{Eq_definition_of_a_and_b}
x_s - y_s^\pm = \big( {-\eps} \xi_2^\pm(s), \eps \xi_1^\pm(s) \big) =: \big(a_s^\pm, b_s^\pm \big).
\end{equation}
On the other hand, as indicated in~\eqref{Eq_Derivative_vs_Outward_Direction_Correspondence}, the slopes of the one-sided tangents at $x_s$ satisfy
\begin{equation} \label{Eq_Definition_of_eps_curvature_function}
D^\pm f^{\xi, y}(s) = \frac{\xi_2^\pm(s)}{\xi_1^\pm(s)} = \frac{\eps \xi_2^\pm(s)}{\sqrt{\eps^2 - \big[ \eps \xi_2^\pm(s) \big]^2}} = \frac{-a_s^\pm}{\sqrt{\eps^2 - \big[a_s^\pm\big]^2}} =: p\big(a_s^\pm \big).
\end{equation}
We call the function $p : (-\eps, \eps) \to \R$ in~\eqref{Eq_Definition_of_eps_curvature_function}, given by $p(s) := -s / \sqrt{\eps^2 - s^2}$, the \emph{slope function}.

In order to prove Proposition~\ref{Prop_Lower_Bound_for_Curvature}, we establish for all $s \in [0, s_{\xi,y}]$ and $h > 0$ with $0 \leq s + h \leq s_{\xi, y}$ the double inequality
\begin{equation} \label{Eq_Lower_bound_for_dir_derivatives_again}
D^\pm f^{\xi, y}(s + h) - D^\pm f^{\xi, y}(s) \geq p\left(a_s^\pm + h\right) - p\left(a_s^\pm \right) \geq - \frac{8h}{3\sqrt{3}\eps},
\end{equation}
where the coordinates $a_s^\pm$ 
are defined in~\eqref{Eq_definition_of_a_and_b}. 
We initially prove \eqref{Eq_Lower_bound_for_dir_derivatives_again} for the right derivative $D^+f^{\xi,y}$, and deduce from this the analogous inequality also for the left derivative $D^-f^{\xi,y}$. Since $D^+ f^{\xi, y}(s) = p\big(a_s^+\big)$, the key to showing inequality~\eqref{Eq_Lower_bound_for_dir_derivatives_again} is to demonstrate that $D^+ f^{\xi, y}(s + h) \geq p(a_s^+ + h)$ whenever $s, h \geq 0$ and $0 \leq s + h \leq s_{\xi, y}$. We divide the proof into the following steps, of which the first three correspond to Lemmas~\ref{Lemma_Lower_Bound_I}--\ref{Lemma_Lower_Bound_III} and step (iv) is included in the proof of Proposition~\ref{Prop_Lower_Bound_for_Curvature} below:
\begin{enumerate}
\item[(i)] identify a lower bound $k(T,h)$ for $D^+ f^{\xi, y}(s+h)$ in terms of $h$ and the difference
\[
T = T(s,h) := f^{\xi,y}(s + h) - f^{\xi,y}(s);\] 
\item[(ii)] show that the bound $k_h(T) := k(T, h)$ obtained in step (i) is increasing in $T$ for a fixed $h$;
\item[(iii)] show that there exists $\widehat{T} := \widehat{T}(s,h) \leq T(s,h)$ for which $k\big(\widehat{T}, h\big) =  p(a_s^+ + h)$;
\item[(iv)] combine steps (i)--(iii) to obtain the inequality
\[D^+ f^{\xi, y}(s + h) \geq k(T,h) \geq k\big(\widehat{T}, h\big) = p(a_s^+ + h).\]
\end{enumerate}

We begin by establishing a lower bound $k(T,h)$ for $D^+ f^{\xi, y}(s+h)$ in terms of $T := T(s,h)$ and the $T$-dependent distances
\begin{equation} \label{Eq_Def_of_P_and_A}
P(T) := \frac{\sqrt{h^2 + T^2}}{2} \quad \textrm{and} \quad A(T) := \sqrt{\eps^2 - P^2(T)} = \sqrt{\eps^2 - \frac{h^2 + T^2}{4}}.
\end{equation}

\begin{lemma}[{The functional form of the lower bound for the right derivative}] \label{Lemma_Lower_Bound_I}
Let $E \subset \R^2$ be a compact set, $\eps > 0$ and let $x \in J \subset \partial E_\eps$ where $J$ is a Jordan curve. Let $\cG(x)$ be a local boundary representation at $x$, where each $g_{\xi, y} \in \cG(x)$ satisfies
\[
g_{\xi, y}(s) = x + s\xi + f^{\xi, y}(s)\eps^{-1}(x - y)
\]
for some continuous function $f^{\xi, y} : [0, s_{\xi, y}] \to \R$. Then for all $s, h \geq 0$ with $0 \leq s + h \leq s_{\xi, y}$, the right derivative $D^+ f^{\xi,y}$ satisfies the inequality
\begin{equation} \label{Eq_Lower_estimate_for_derivative}
D^+f^{\xi,y}(s+h) \geq k(T,h) := \frac{T A(T) - h P(T)}{h A(T) + T P(T)},
\end{equation}
where $T := f^{\xi,y}(s + h) - f^{\xi,y}(s)$ and the distances $P(T), A(T)$ are as in~\eqref{Eq_Def_of_P_and_A}.
\end{lemma}

\begin{proof}
For each $s \in [0, s_{\xi,y}]$ we write $a_s^+ := -\eps\xi_2^+(s)$ and $b_s^+ := \eps\xi_1^+(s)$, where $\xi_s^+ = \big(\xi_1^+(s), \xi_2^+(s) \big)$ is the right extremal outward direction at $x_s$.
Fix then some $s, h \geq 0$ with $0 \leq s + h \leq s_{\xi, y}$ and consider the corresponding boundary points $x_s, x_{s+h} \in J$ whose local coordinates are given by
\begin{equation} \label{Eq_points_x_s_and_x_s+h}
x_s := \left(s, f^{\xi, y}(s)\right), \quad x_{s+h} := \left(s+h, f^{\xi, y}(s+h)\right).
\end{equation}
It follows from~\eqref{Eq_definition_of_a_and_b} that the right extremal contributors $y_s^+ \in \Cext{x_s}$ and $y_{s+h}^+ \in \Cext{x_{s+h}}$ satisfy
\begin{align} \label{Eq_extremal_contributors_y_s^+_and_y_s+h^+_ver_2}
y_s^+ &= \left(s - a_s^+, f^{\xi, y}(s) - b_s^+ \right), \\ \nonumber
y_{s+h}^+ &= \left(s + h - a_{s+h}^+, f^{\xi, y}(s+h) - b_{s+h}^+ \right).
\end{align}

In local coordinates, the right derivative $D^+ f^{\xi,y}(s+h)$ is the slope of the extremal outward direction $\xi_{s+h}^+$,
which is by definition perpendicular to the vector $x_{s+h} - y_{s+h}^+$. Since $y_{s+h}^+$ is a contributor and thus in $E$, it must by definition lie outside the $\eps$-radius ball $B_\eps(x_s)$ around the boundary point $x_s \in J$. One can therefore obtain a lower bound for the right derivative $D^+ f^{\xi,y}(s+h)$ by considering how far one could rotate the contributor $y_{s+h}^+$ clockwise around the point $x_{s+h}$ before it enters the ball $B_\eps(x_s)$.
This would happen at the point $y_{s+h}^*$, whose distance from both $x_s$ and $x_{s+h}$ equals $\eps$, see Figure~\ref{fig:minimal_tangent_at_x_(s+h)}. These three points thus form an isosceles triangle whose base is the line segment
\begin{equation} \label{Eq_line_segment_S}
S := \big\{ (1-\varphi)x_s + \varphi x_{s+h} \, : \, \varphi \in [0,1] \big\}.
\end{equation}
Consequently, the orthogonal projection of the apex $y_{s+h}^*$ of this triangle onto the line segment $S$ lands on the mid-point $x_{s+h}^* := (x_s + x_{s+h})/2$ of $S$. Together, the points $x_{s+h}^*$, $y_{s+h}^*$ and $x_{s+h}$ in turn form a right triangle $\Omega$, whose legs $x_{s+h} - x_{s+h}^*$ and $x_{s+h}^* - y_{s+h}^*$ satisfy
\[
\norm{x_{s+h} - x_{s+h}^*} = P(T) \quad \textrm{and} \quad \norm{x_{s+h}^* - y_{s+h}^*} = A(T).
\]

\vspace{2mm}
\begin{figure}[h]
      \centering
      \captionsetup{margin=0.75cm}
      \vspace{0mm}
                \includegraphics[width = \textwidth]{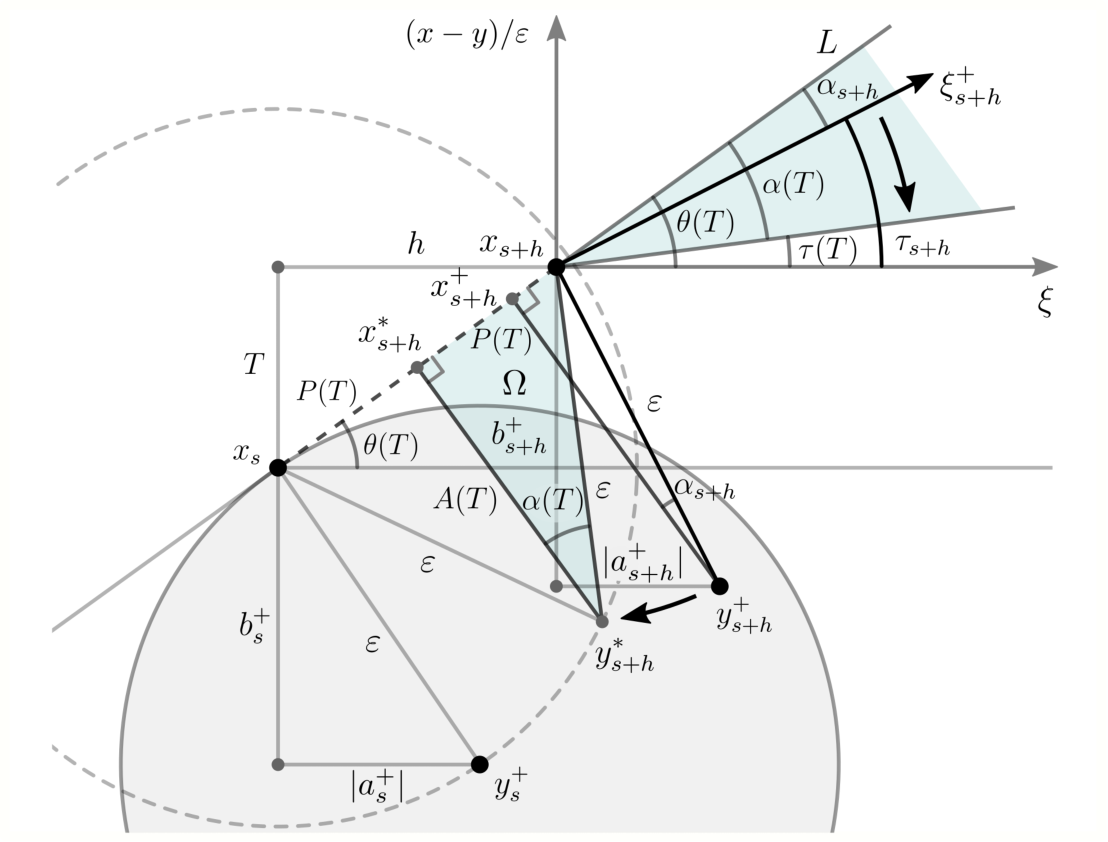} \\[3mm]   
                \caption{
                An illustration of the geometric components in the proof of Lemma~\ref{Lemma_Lower_Bound_I}. One can obtain a lower bound $k(T,h) := \tan \tau(T)$ for the slope $D^+ f^{\xi,y}(s+h) = \tan \tau_{s+h} = \xi_2^+(s+h) / \xi_1^+(s+h)$ by considering how far one could rotate the right extremal contributor $y_{s+h}^+$ clockwise around the point $x_{s+h}$ before it enters the ball $B_\eps(x_s)$. The line segment $S := \big\{ (1-\varphi)x_s + \varphi x_{s+h} \, : \, \varphi \in [0,1] \big\}$ is marked as a dashed line. Note that the illustration here does not strictly speaking apply in the situation of Lemma~\ref{Lemma_Lower_Bound_I} because the $\xi$-coordinate of $x_{s+h}$ differs from that of $x_s$ by more than $\eps/2$. This artistic liberty was taken in order to improve readability through the increase of the relevant angles.}
                \label{fig:minimal_tangent_at_x_(s+h)}
\end{figure}

We use the above geometric relationships to obtain an explicit lower bound for $D^+ f^{\xi,y}(s+h)$ in terms of the lengths $P(T)$ and $A(T)$.
As noted above, $y_{s+h}^+ \notin B_\eps\left(x_s\right)$. This implies the inequality
\begin{align*}
0 &\leq \norm{y_{s + h}^+ - x_s}^2 - \eps^2 \\
         &= \left(\left(s + h - a_{s+h}^+\right) - s\right)^2 + \left(\left(f^{\xi,y}(s + h) - b_{s + h}^+\right) - f^{\xi,y}(s) \right)^2 - \eps^2 \\
         &= h^2 + \left(f^{\xi,y}(s + h) - f^{\xi,y}(s) \right)^2 - 2 \left(h a_{s+h}^+ + \left(f^{\xi,y}(s + h) - f^{\xi,y}(s) \right) b_{s+h}^+ \right).
\end{align*}
Writing $T := f^{\xi,y}(s + h) - f^{\xi,y}(s)$, the above inequality can be expressed more concisely as
\begin{equation} \label{Eq_contributor_inequality}
ha_{s+h}^+ + Tb_{s+h}^+ \leq (h^2 + T^2) / 2.
\end{equation}
Since $x_{s+h} - y_{s+h}^+ = \big(a_{s+h}^+, b_{s+h}^+ \big)$ and $x_{s+h} - x_s = (h, T)$, \eqref{Eq_contributor_inequality} is furthermore equivalent to
\begin{equation} \label{Eq_Minimal_scalar_product_ineq}
\left\langle x_{s+h} - y_{s+h}^+, \frac{x_{s+h} - x_s}{\norm{x_{s+h} - x_s}} \right\rangle \leq \frac{\norm{x_{s+h} - x_s}}{2} = P(T).
\end{equation}

Geometrically, inequality~\eqref{Eq_Minimal_scalar_product_ineq} expresses the fact that the length of the orthogonal projection of the vector $x_{s+h} - y_{s+h}^+$ onto the line
\[
L := x_s + \mathrm{span}\{x_{s+h} - x_s\}
\]
is at most half the distance $\norm{x_{s+h} - x_s}$. Note that when $\xi_{s+h}^+$ is parallel to $L$, the scalar product on the left-hand side of~\eqref{Eq_Minimal_scalar_product_ineq} vanishes, and for steeper slopes it becomes negative. However, since we seek to obtain a lower bound for $D^+ f^{\xi,y}(s+h) = \tan \tau_{s+h}$, where $\tau_{s+h}$ is the angle between $\xi_{s+h}^+$ and the $\xi$-axis, it is sufficient to restrict the analysis to angles $\tau_{s+h}$ that are smaller than the angle $\theta(T)$ between the line $L$ and the $\xi$-axis.\footnote{Additionally, as pointed out also in the proof of Theorem~\ref{Thm_existence_of_curvature}, the definition of a local boundary representation imposes an upper bound on the angle $\tau_{s+h}$ between the extremal outward direction $\xi_{s+h}^+$ and the $\xi$-axis, since the extremal contributor $y_{s+h}^+ \in \Cext{x_{s+h}}$ lies for all $s,h$ within the ball $B_{\eps/2}(y)$, where $y \in \Cext{x}$.}

Consider now the vector $x_{s+h}^+ - y_{s+h}^+$, where $x_{s+h}^+ := \mathrm{proj}_L\left(y_{s+h}^+\right)$ is the orthogonal projection of $y_{s+h}^+$ onto the line $L$. It follows from~\eqref{Eq_Minimal_scalar_product_ineq} that
\[
\norm{x_{s+h} - x_{s+h}^+} \leq P(T),
\]
and consequently
\[
\norm{x_{s+h}^+ - y_{s+h}^+}^2 = \eps^2 - \norm{x_{s+h} - x_{s+h}^+}^2 \geq \eps^2 - \norm{x_{s+h} - x_{s+h}^*}^2 = A^2(T).
\]
This implies that the angle $\alpha_{s+h}$ at $y_{s+h}^+$ between the vectors $x_{s+h} - y_{s+h}^+$ and $x_{s+h}^+ - y_{s+h}^+$ satisfies
\begin{equation} \label{Eq_tangent_inequality}
\tan \alpha_{s+h} = \frac{\norm{x_{s+h} - x_{s+h}^+}}{\norm{x_{s+h}^+ - y_{s+h}^+}} \leq \frac{P(T)}{A(T)} = \tan \alpha(T),
\end{equation}
where $\alpha(T)$ is the angle at $y_{s+h}^*$  between the vectors $x_{s+h} - y_{s+h}^*$ and $x_{s+h}^* - y_{s+h}^*$. Since the vectors $x_{s+h} - y_{s+h}^+$ and $\xi_{s+h}^+$ are perpendicular, it furthermore follows that
\begin{equation} \label{Eq_tangent_angle}
\tau_{s+h} = \theta(T) - \alpha_{s+h}.
\end{equation}
This, together with~\eqref{Eq_tangent_inequality}, leads to the lower bound
\begin{align} \label{Eq_Lower_Bound_k(T,h)}
D^+ f^{\xi, y}(s+h) &= \tan \tau_{s+h} = \tan \big(\theta(T) - \alpha_{s+h}\big) \\ \nonumber
                &\geq \tan \big(\theta(T) - \alpha(T)\big)
                = \frac{\tan \theta(T) - \tan \alpha(T)}{1 + \tan \theta(T) \tan \alpha(T)} \\ \nonumber
                &= \frac{T A(T) - h P(T)}{h A(T) + T P(T)} = k(T, h),
\end{align}
where the third last equality is due to the standard formula for the tangent of a sum of angles.
\end{proof}

Intuitively, Lemma~\ref{Lemma_Lower_Bound_I} describes how the $\eps$-neighbourhood geometry imposes a lower bound $k(T, h)$ on the tangential direction $D^+ f^{\xi,y}(s + h)$, and how this bound depends on the increment $h$ and the difference $T := f^{\xi,y}(s+h) - f^{\xi,y}(s)$. We show next that if the point $x_s$ and the increment $h > 0$ are fixed, $k(T, h)$ depends monotonically on $T$.

\begin{lemma}[{Monotonicity of the lower bound for the right derivative}] \label{Lemma_Lower_Bound_II}
For a fixed $h \in [0, 2\eps)$ and the corresponding functions
\[
P(t) := \frac{\sqrt{h^2 + t^2}}{2} \quad \textrm{and} \quad A(t) := \sqrt{\eps^2 - P^2(t)} = \sqrt{\eps^2 - \frac{h^2 + t^2}{4}},
\]
the lower bound function
\[
t \mapsto k(t, h) := \frac{t A(t) - h P(t)}{h A(t) + t P(t)}\]
in~\eqref{Eq_Lower_estimate_for_derivative} is increasing on the interval $\big( \hspace{-1mm}-\hspace{-0.8mm}\sqrt{4\eps^2 - h^2}, \sqrt{\eps^2 - (\eps - h)^2}\hspace{0.6mm}\big)$.
\end{lemma}

\begin{proof}
Note first that $A(t)$ is defined when $P(t) \leq \eps$, which is equivalent to $|t| \leq \sqrt{4\eps^2 - h^2}$.
We compute the sign of the derivative $\frac{d}{dt} k(t,h)$ for a fixed $h \geq 0$. Write
\[
k_1(t) := t A(t) - h P(t) \quad \textrm{and} \quad k_2(t) := h A(t) + t P(t),
\]
so that $k(t, h) = k_1(t) / k_2(t)$. Then
\begin{align*}
\frac{d}{dt} k(t,h) &= \frac{k_1'(t) k_2(t) - k_1(t) k_2'(t)}{k_2^2(t)} \\
           &= \frac{h\left(A^2(t) + P^2(t) \right) + \left(h^2 + t^2 \right) \left( A'(t) P(t) - P'(t) A(t) \right)}{ \left( h A(t) + t P(t)\right)^2}, 
\end{align*}
where
\[
P'(t) = t \left(2\sqrt{h^2 + t^2}\right)^{-1} \quad \textrm{and} \quad A'(t) = -t \left(4\sqrt{\eps^2 - \frac{h^2 + t^2}{4}}\right)^{-1}.
\]
Thus, $\frac{d}{dt} k(t,h) > 0$ if and only if
\begin{align*}
0 &< h\left(A^2(t) + P^2(t) \right) + \left(h^2 + t^2 \right) \big( A'(t) P(t) - P'(t) A(t) \big) \\[2mm]
   &= h\eps^2 + \left(h^2 + t^2 \right) \frac{-t \big(h^2 + t^2 \big) - 4t\left(\eps^2 - \frac{h^2 + t^2}{4}\right)}{8\sqrt{\left(h^2 + t^2\right)\left(\eps^2 - \frac{h^2 + t^2}{4}\right)}} \\[1mm]
   &= \eps^2 \left( h - t \frac{\sqrt{\frac{h^2 + t^2}{4}}}{\sqrt{\eps^2 - \frac{h^2 + t^2}{4}}} \right)
   = \eps^2 \left( h - \frac{t P(t)}{\sqrt{\eps^2 - P^2(t)}} \right).
\end{align*}
A direct computation shows that this inequality is satisfied if and only if $t < \sqrt{\eps^2 - (\eps - h)^2}$, where the right hand side is defined for $h \in [0, 2\eps]$.
\end{proof}

The next lemma provides a lower bound $\widehat{T}(s,h)$ for the difference $T(s,h) = f^{\xi,y}(s+h) - f^{\xi,y}(s)$ for each $s$ and $h$.
In the proof of Proposition~\ref{Prop_Lower_Bound_for_Curvature} below, we combine this lower bound with Lemma~\ref{Lemma_Lower_Bound_II} and show that $k\big(\widehat{T}, h\big)$ is in fact a lower bound for the right derivative $D^+ f^{\xi,y}(s+h)$.

\begin{lemma}[{An explicit lower bound}] \label{Lemma_Lower_Bound_III}
Let $E \subset \R^2$ be a compact set, $\eps > 0$ and let $x \in J \subset \partial E_\eps$ where $J$ is a Jordan curve. Let $\cG(x)$ be a local boundary representation at $x$, where each $g_{\xi, y} \in \cG(x)$ satisfies $g_{\xi, y}(s) = x + s\xi + f^{\xi, y}(s)\eps^{-1}(x - y)$ for some continuous function $f^{\xi, y} : [0, s_{\xi, y}] \to \R$. In addition, let $T = T(s, h) := f^{\xi,y}(s+h) - f^{\xi,y}(s)$, and let
\[
k(T,h) := \frac{T A(T) - h P(T)}{h A(T) + T P(T)}
\]
as in~\eqref{Eq_Lower_estimate_for_derivative}. For each $s \in [0, s_{\xi,y}]$, write $(a_s^+, b_s^+) := \big(-\eps \xi_2^+(s), \eps \xi_1^+(s)\big)$, where $\xi_s^+ = \big(\xi_1^+(s), \xi_2^+(s) \big)$ is the right extremal outward direction at the boundary point $x_s = \left(s, f^{\xi, y}(s)\right)$, see~\eqref{Eq_points_x_s_and_x_s+h}.
Then
\[
\widehat{T}(s,h) := -b_s^+ + \sqrt{\eps^2 - \big(h + a_s^+\big)^2} \leq T(s,h)
\]
for all $s, h \geq 0$ with $0 \leq s + h \leq s_{\xi, y}$. Furthermore, $\widehat{T} = \widehat{T}(s,h)$ satisfies 
\begin{equation} \label{Eq_pah}
k\big(\widehat{T}, h\big) = p(a_s^+ + h) := \frac{-\left( a_s^+ + h \right)}{\sqrt{\eps^2 - \left( a_s^+ + h \right)^2}},
\end{equation}
where $p$ is the slope function in~\eqref{Eq_Definition_of_eps_curvature_function}.
\end{lemma}

\begin{proof}
Let $s, h \geq 0$ with $0 \leq s + h \leq s_{\xi, y}$, and consider the corresponding boundary points $x_s, x_{s+h} \in J$ whose local coordinates are given by
\begin{equation} \label{Eq_points_x_s_and_x_s+h_ver_2}
x_s := \left(s, f^{\xi, y}(s)\right), \quad x_{s+h} := \left(s+h, f^{\xi, y}(s+h)\right).
\end{equation}
It follows from~\eqref{Eq_definition_of_a_and_b} that the right extremal contributors $y_s^+ \in \Cext{x_s}$ and $y_{s+h}^+ \in \Cext{x_{s+h}}$ satisfy
\begin{align} \label{Eq_extremal_contributors_y_s^+_and_y_s+h^+}
y_s^+ &= \left(s - a_s^+, f^{\xi, y}(s) - b_s^+ \right), \\ \nonumber
y_{s+h}^+ &= \left(s + h - a_{s+h}^+, f^{\xi, y}(s+h) - b_{s+h}^+ \right).
\end{align}
In particular, the boundary point $x_{s+h}$ and contributor $y_s^+$ satisfy
\begin{align*} 
x_{s+h} - y_s^+ &= \big(s + h - (s - a_s^+), f^{\xi, y}(s,h) - \big(f^{\xi,y}(s) - b_s^+ \big) \big) \\
 &= \big(h + a_s^+, T + b_s^+ \big).
\end{align*}
This implies $\big(h + a_s^+\big)^2 + \big(T + b_s^+\big)^2 = \norm{x_{s+h} - y_s^+}^2 \geq \eps^2$, since $x_{s + h} \notin B_\eps(y_s^+)$.
Rearranging the terms in this inequality gives the desired estimate
\begin{equation} \label{Eq_vertical_inequality}
\widehat{T} := -b_s^+ + \sqrt{\eps^2 - \left(a_s^+ + h \right)^2} \leq T. 
\end{equation}



\begin{figure}[h]
      \centering
      \captionsetup{margin=0.75cm}
      \vspace{0mm}
                \includegraphics[width = \textwidth]{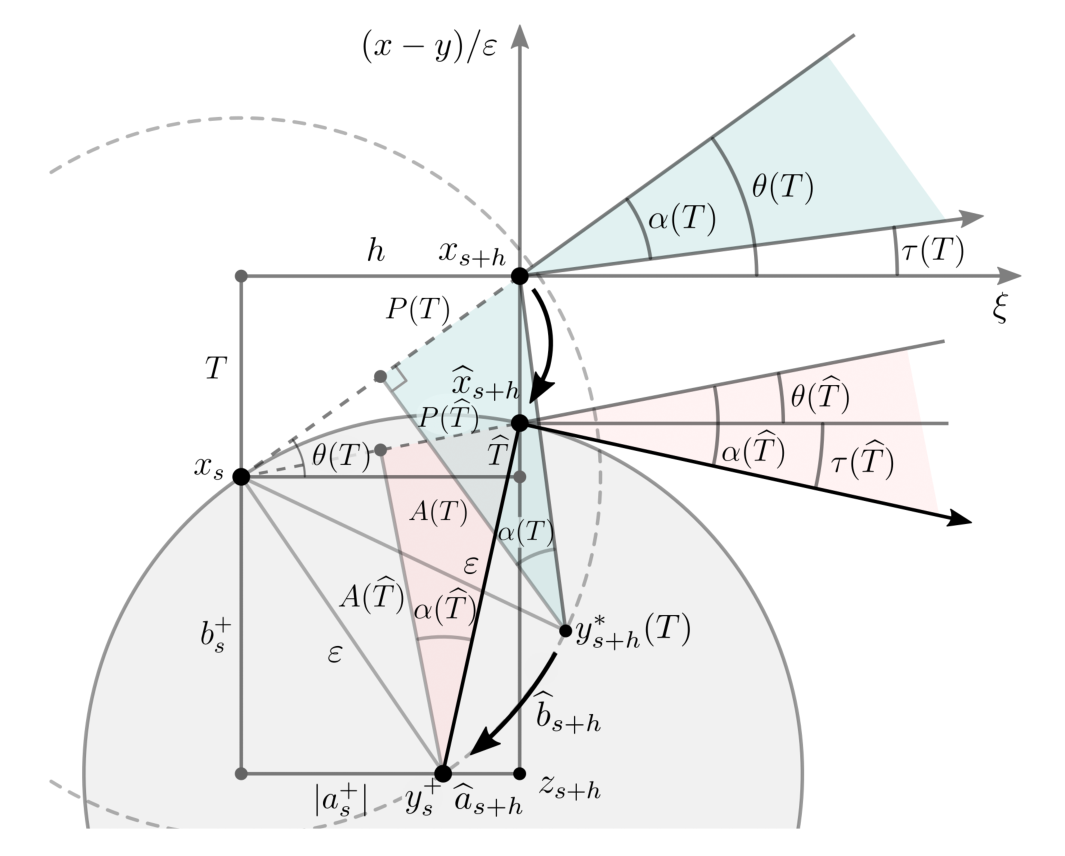} \\[0mm]
                \caption{Dependence of the function $k(T,h) = \tan \tau(T)$ on $T = f^{\xi, y} (s + h) - f^{\xi,y}(s)$. As $T$ decreases, the corresponding point $y_{s+h}^*(T)$ on the circumference $\partial B_\eps(x_s)$ moves clockwise towards $y_s^+$. This results in a decrease in the slope $\tan \tau(T)$, where $\tau(T) := \theta(T) - \alpha(T)$. At the smallest possible value $\widehat{T}$, one obtains the lower bound $k(\widehat{T}, h) = \tan \tau (\widehat{T}) = p(a_s^+ + h)$, which geometrically corresponds to the point $y_{s+h}^*(\widehat{T}) = y_s^+$.}
                \label{fig:tangent_for_widehat(T)}
\end{figure}

We now show that $k\big(\widehat{T}, h\big) = p(a_s^+ + h)$, where $p: (-\eps, \eps) \to \R$ is the slope function in~\eqref{Eq_Definition_of_eps_curvature_function}.
It follows from the definition of $\widehat{T}$ that the point $\widehat{x}_{s+h} := \big(s + h, f^{\xi, y}(s) + \widehat{T}\big)$ lies on the boundary $\partial B_\eps\left(y_s^+\right)$. Furthermore, $\widehat{x}_{s+h} - y_s^+ = \big( \widehat{a}_{s+h}, \widehat{b}_{s+h} \big)$, where\footnote{By construction, $b_s^+$ is always positive, while $a_s^+$ may either be positive or negative depending on the slope $D^+ f^{\xi, y}(s)$, see~\eqref{Eq_definition_of_a_and_b}.}
\[
\widehat{a}_{s+h} := a_s^+ + h, \quad \textrm{and} \quad  \widehat{b}_{s+h} := \sqrt{\eps^2 - \left( a_s^+ + h \right)^2}.
\]
Now, define $z_{s+h} := \big( s + h, f^{\xi, y}(s) - b_s^+ \big)$ so that $z_{s+h}$ shares its $\xi$--coordinate with $\widehat{x}_{s+h}$ and its $\eps^{-1}(x-y)$--coordinate with $y_s^+$. Recall that for each $h$, the function $T \mapsto k\big(T,h\big)$ is defined by $k\big(T,h\big) = \tan \big( \theta(T) - \alpha(T) \big)$ where the $T$-dependent angles $\theta(T)$ and $\alpha(T)$ are defined as in the proof of Lemma~\ref{Lemma_Lower_Bound_I}, see~\eqref{Eq_Lower_Bound_k(T,h)}. Then 
\begin{align*} \label{Eq_tangent_of_tau_widehat(T)}
k\big(\widehat{T}, h\big) &= \tan \big(\theta(\widehat{T}) - \alpha(\widehat{T}) \big) = \tan \tau ( \widehat{T} ) \\ \nonumber 
  &= \frac{-\widehat{a}_{s+h}}{\widehat{b}_{s+h}} 
  = \frac{-\big( a_s^+ + h \big)}{\sqrt{\eps^2 - \big( a_s^+ + h \big)^2}} \\
  &= p \big( a_s^+ + h \big),
\end{align*}
where $\tau(\widehat{T})$ is the angle at $\widehat{x}_{s+h}$ between the vectors $y_s^+ - \widehat{x}_{s+h}$ and $z_{s+h} - \widehat{x}_{s+h}$, see Figure~\ref{fig:tangent_for_widehat(T)}.
\end{proof}

Combining Lemmas~\ref{Lemma_Lower_Bound_I}--\ref{Lemma_Lower_Bound_III}, we are now ready to prove our second main result. 

\begin{proof}[Proof of Proposition~\ref{Prop_Lower_Bound_for_Curvature}]
Let $s, h > 0$ with $0 \leq s + h \leq s_{\xi,y}$ and let $T := f^{\xi,y}(s+h) - f^{\xi,y}(s)$. Define
\[
k(T,h) := \frac{T A(T) - h P(T)}{h A(T) + T P(T)}
\]
as in~\eqref{Eq_Lower_estimate_for_derivative} and let $\widehat{T} \leq T$ be the lower bound, given by Lemma~\ref{Lemma_Lower_Bound_III}, for which $k\big(\widehat{T},h \big) = p(a_s^+ + h)$. The slope function $p : (-\eps, \eps) \to \R$ is given by $p(s) := -s / \sqrt{\eps^2 - s^2}$ as in~\eqref{Eq_Definition_of_eps_curvature_function}.

We want to use the monotonicity of the map $T \mapsto k(T,h)$ in order to establish for each $h$ the inequality $k(T, h) \geq k\big(\widehat{T},h \big)$. For this, we need to show that $T$ satisfies for all $h$ the assumption
\[
-\sqrt{4\eps^2 - h^2} < T < \sqrt{\eps^2 - (\eps - h)^2}
\]
in Lemma~\ref{Lemma_Lower_Bound_II}. This is obtained immediately by combining the Lipschitz property\footnote{Note that the function $f^{\xi,y}$ here corresponds to the orthonormal coordinate system $\big(\xi, \eps^{-1}(x-y) \big)$, and is thus scaled by a factor of $\eps$ compared to the corresponding function $f^{\xi,y}$ in the statement of Proposition~\ref{Prop_LBR_Lipschitz}, see equations~\eqref{Eq_Canonical_LBR_02} and~\eqref{Eq_Canonical_LBR}. Hence, the Lipschitz constant here is $1/\sqrt{3}$ rather than $1/\sqrt{3}\eps$.}
\[
|T| = |f^{\xi,y}(s+h) - f^{\xi,y}(s)| \leq h/\sqrt{3},
\]
given by Proposition~\ref{Prop_LBR_Lipschitz}, with the inequalities
\[
\frac{h}{\sqrt{3}} \leq \sqrt{\eps^2 - (\eps - h)^2} < \sqrt{4\eps^2 - h^2},
\]
which follow from $0 \leq h \leq s_{\xi,y} \leq \eps/2$. The results in Lemmas~\ref{Lemma_Lower_Bound_I}--\ref{Lemma_Lower_Bound_III} thus imply that
\begin{equation} \label{Eq_inequality_for_right_derivatives}
D^+ f^{\xi, y}(s+h) \geq k(T, h) \geq k\big(\widehat{T},h \big) = p(a_s^+ + h)
\end{equation}
whenever $s, h \geq 0$ and $0 \leq s + h \leq s_{\xi, y}$. Furthermore, according to~\eqref{Eq_Definition_of_eps_curvature_function}, the slope of $f^{\xi,y}$ at $s$ satisfies $D^+ f^{\xi, y}(s) =  p\left( a_s^+ \right)$, which together with~\eqref{Eq_inequality_for_right_derivatives} implies
\begin{equation} \label{Eq_Derivative_Inequality}
D^+ f^{\xi, y}(s+h) - D^+ f^{\xi, y}(s) \geq p(a_s^+ + h) - p(a_s^+).
\end{equation}

We now establish this inequality also for the left derivatives $D^- f^{\xi, y}$.
Recall from~\eqref{Eq_Definition_of_eps_curvature_function} that
\[
D^- f^{\xi, y}(s) = \frac{\xi_2^-(s)}{\xi_1^-(s)}
\]
for all $s \in [0, s_{\xi,y}]$, where $\xi_s^- = \big(\xi_1^-(s), \xi_2^-(s) \big) \in \Xext{x_s}$ is the left extremal outward direction at $x_s$. Consider a sequence $(x_{s_n})_{n=1}^\infty \subset J$ where $s_n \to s$ from below, so that $x_{s_n} \to x_s$ and
\begin{equation} \label{Eq_below_approaching_sequence}
\frac{x_{s_n} - x_s}{\norm{x_{s_n} - x_s}} \to \xi_s^-.
\end{equation}
It follows from Lemma~\ref{Lemma_Convergence_of_Contributing_Points_of_Sequences_of_Boundary_Points} (ii)(b) that the right extremal outward directions $\xi_{s_n}^+ \in \Xext{x_{s_n}}$ satisfy $\xi_{s_n}^+ \to -\xi_s^-$ as $n \to \infty$. Since $D^+ f^{\xi, y}(s_n) =  \xi_2^+(s_n) / \xi_1^+(s_n)$ for all $n \in \N$, this implies
%
%
%
\begin{equation} \label{Eq_left_der_as_a_limit_of_right_ders}
\lim_{n\to\infty} D^+ f^{\xi, y}(s_n) = \lim_{n\to\infty} \frac{\xi_2^+(s_n)}{\xi_1^+(s_n)}  = \frac{-\xi_2^-(s)}{-\xi_1^-(s)} = D^- f^{\xi, y}(s).
\end{equation}
This limit is independent of the choice of the sequence $(s_n)_{n = 1}^\infty$ in~\eqref{Eq_below_approaching_sequence}. Since the slope function $p$ is continuous on $[0, \eps/2]$, it follows from~\eqref{Eq_left_der_as_a_limit_of_right_ders} and~\eqref{Eq_inequality_for_right_derivatives} that for all $s, h \geq 0$ and $0 \leq s + h \leq s_{\xi, y}$,
\begin{align} \label{Eq_inequality_for_incr_of_left_der}
D^- f^{\xi, y}(s + h) &= \lim_{\varphi \to 0-} D^+ f^{\xi, y}(s + h + \varphi) \\ \nonumber
  &\geq \lim_{\varphi \to 0-} p(a_s^+ + h + \varphi) \\ \nonumber
  &= p(a_s^+ + h).
\end{align}
Due to the characterisation of the sets of extremal outward directions given in Proposition~\ref{Prop_structure_of_set_of_outward_directions}, there can be no cusp singularities on the boundary $J$. Given the relationship~\eqref{Eq_Definition_of_eps_curvature_function} between extremal outward directions and the one-sided derivatives, this implies that
\[
D^- f^{\xi, y}(s) \leq D^+ f^{\xi, y}(s) = p(a_s^+)
\]
for all $s \in [0, s_{\xi,y}]$. Combining this inequality with~\eqref{Eq_inequality_for_right_derivatives} and~\eqref{Eq_inequality_for_incr_of_left_der} yields the analogue of~\eqref{Eq_Derivative_Inequality} for the left derivatives,
\begin{equation} \label{Eq_Left_Derivative_Inequality}
D^- f^{\xi, y}(s+h) - D^- f^{\xi, y}(s) 
\geq p(a_s^+ + h) - p(a_s^+).
\end{equation}

We conclude the proof by obtaining a lower bound for the difference $p(a_s^+ + h) - p(a_s^+)$. A direct calculation shows that the derivative of $p$ is given by
\[
p'(s) = \frac{p(s) + p^3(s)}{s}.
\]
Since $p$ is an odd function, it follows that $p'$ is even and non-positive. Then
\[
p'_\mathrm{min} := \min_{|s| \leq s_{\xi, y}} p'(s) \geq \min_{|s| \leq \eps / 2} p'(s) = p'\left(\frac{\eps}{2}\right) = -\frac{8}{3\sqrt{3}\eps},
\]
which implies
\begin{align*}
p(a_s^+ + h) - p(a_s^+) &= \int_{a_s^+}^{a_s^+ + h} p'(s) ds
    \geq \int_{a_s^+}^{a_s^+ + h} p'_\mathrm{min} ds
    \geq -\frac{8h}{3\sqrt{3}\eps}. 
\end{align*}
The result now follows from~\eqref{Eq_Derivative_Inequality} and~\eqref{Eq_Left_Derivative_Inequality}.
\end{proof}


\backmatter
\bibliographystyle{amsalpha}
\bibliography{epsneighbourhoods}
\printindex

\end{document}